\documentclass{amsart}
\usepackage{array,amssymb,amsxtra,latexsym,graphics,psfrag,epsfig,ifthen,mathtools,mathrsfs}
\usepackage[all]{xy}
\SelectTips {cm}{}
\usepackage{bbm,bm}
\usepackage{rotating}

\usepackage[colorlinks,final,backref=page,hyperindex]{hyperref}
\usepackage{xcolor}

\newcommand{\red}[1]{\textcolor{red}{#1}}

\newcommand{\black}[1]{\textcolor{black}{#1}}

\newdimen\secnumwidth
\makeatletter
\renewcommand{\tocsection}[3]{%
\indentlabel{\@ifnotempty{#2}{\ignorespaces#1\space \hbox to\secnumwidth{#2.\hfil}\quad}}#3}
\makeatother

\usepackage[newitem,newenum,neverdecrease]{paralist}
\makeatletter
\def\subsection{\@startsection{subsection}{2}%
  \z@{.5\linespacing\@plus.7\linespacing}{-.5em}%
  {\normalfont\bfseries\mathversion{bold}}}
\if@plflushright
 
\else
 
\fi
\renewcommand{\@asparaenum@}{%
\expandafter\list\csname label\@enumctr\endcsname{%
\usecounter{\@enumctr}%
\labelwidth\z@
\labelsep.5em
\leftmargin\z@
\parsep\parskip
\itemsep\z@
\topsep\z@
\partopsep\parskip
\itemindent\parindent
\advance\itemindent\labelsep
\def\makelabel##1{\upshape ##1}}}
\makeatother
\usepackage{longtable}
\theoremstyle{plain}
\newtheorem{theorem}{Theorem}%[section]
\newtheorem{proposition}[theorem]{Proposition}
\newtheorem{lemma}[theorem]{Lemma}
\newtheorem{corollary}[theorem]{Corollary}
\newtheorem*{theorem*}{Theorem}

\theoremstyle{definition}

\newtheorem{convention}[theorem]{Convention}

\theoremstyle{remark}
\newtheorem*{remark}{Remark}

\newtheoremstyle{italics}% name
{}% Space above, empty = `usual value'
{}% Space below
{}% Body font
{}% Indent amount (empty = no indent, \parindent = para indent)
{\itshape}% Thm head font
{.}% Punctuation after thm head
{.5em}% Space after thm head: " " = normal interword space; \newline = linebreak
{}% Thm head spec (can be left empty, meaning `normal'

\newtheorem*{firstproof}{First proof of Theorem~\ref{t:exp-log-iH}}
\newtheorem*{secondproof}{Second proof of Theorem~\ref{t:exp-log-iH}}

\newcommand{\pand}{\text{ and }}
\newcommand{\qand}{\quad\text{and}\quad}
\newcommand{\qqand}{\qquad\text{and}\qquad}
\newcommand{\qor}{\quad\text{or}\quad}

\newcommand{\id}{\mathrm{id}}
\newcommand{\onto}{\twoheadrightarrow}
\newcommand{\into}{\hookrightarrow}
\newcommand{\iso}{\cong}
\newcommand{\isoto}{\stackrel{\cong}{\longrightarrow}}
\newcommand{\xyinc}{\ar@{^{(}->}}
\newcommand{\xyrinc}{\ar@{_{(}->}}
\newcommand{\xyonto}{\ar@{->>}}
\newcommand{\xytwo}{\ar@{<->}}

\newcommand{\map}[1]{\xrightarrow{#1}}

\newcommand{\abs}[1]{\lvert#1\rvert} %for absolute value

\newcommand{\bdot}{\bm\cdot}
\newcommand{\pmat}[1]{\begin{pmatrix}#1\end{pmatrix}} %for round matrix

\newcommand{\Kb}{\Bbbk}

\newcommand{\Nb}{\mathbb{N}}
\newcommand{\Qb}{\mathbb{Q}}
\newcommand{\Rb}{\mathbb{R}}
\newcommand{\Fb}{\mathbb{F}}

\DeclareMathOperator{\End}{End}
\DeclareMathOperator{\Hom}{Hom}

\DeclareMathOperator{\apode}{\textsc{s}}
\DeclareMathOperator{\area}{sch}
\DeclareMathOperator{\supp}{supp}
\DeclareMathOperator{\dist}{dist}
\DeclareMathOperator{\im}{im}
\DeclareMathOperator{\mx}{\mathrm{max}}
\DeclareMathOperator{\can}{\mathrm{can}}
\DeclareMathOperator{\trace}{\mathrm{tr}}

\newcommand{\Sr}{\mathrm{S}} %symmetric group
\newcommand{\opp}[1]{\overline{#1}} %for opposite
\newcommand{\len}{l} % for length (degree) of a composition or partition
\newcommand{\maxflat}{\widehat{I}}
\newcommand{\minflat}{\{I\}}
\newcommand{\ifac}{
\begin{picture}(3,5)(0,0)
\put(0,0){\textup{!}}\put(1.5,4.8){\circle{3}}
\end{picture}
} %cyclic factorial
\newcommand{\acyc}[1]{#1 !} %acyclic orientations

\newcommand{\eigmult}[1]{k_{#1}} %%%eigenvalue multiplicity indexed by a flat. Used for cumulant in this paper
\newcommand{\pos}[1]{{#1}_+} %for positive part (has not been used uniformly)
\newcommand{\act}{\triangleright} %characteristic operation
\newcommand{\facematrix}{\Psi} %%%action of faces on a cocomm bimonoid.

\newcommand{\Fset}{\mathsf{set^{\times}}}
\newcommand{\Veck}{\mathsf{Vec}_\Kb} 
\newcommand{\Set}{\mathsf{Set}}
\newcommand{\Ss}{\mathsf{Sp}} % set species
\newcommand{\Ssk}{\mathsf{Sp}_\Kb} %vector species

\newcommand{\Mo}[1]{\mathsf{Mon}(#1)} %monoids
\newcommand{\Co}[1]{\mathsf{Comon}(#1)} %comonoids
\newcommand{\Bo}[1]{\mathsf{Bimon}(#1)} %bimonoids

\newcommand{\B}[1]{\mathtt{#1}} % bases elements
\newcommand{\BH}{\B{H}} %
\newcommand{\BM}{\B{M}} %
\newcommand{\BP}{\B{P}} %
\newcommand{\BQ}{\B{Q}} %
\newcommand{\BHh}{\Hat{\B{H}}} % Bayer-Diaconis-Loday elements

\newcommand{\rP}{\mathrm{P}}
\newcommand{\rQ}{\mathrm{Q}}
\newcommand{\rH}{\mathrm{H}}

\newcommand{\rE}{\mathrm{E}}
\newcommand{\rL}{\mathrm{L}}
\let\rPi=\Pi %flats
\let\rSig=\Sigma
\newcommand{\rSigh}{\widehat{\rSig}} %decompositions = weak compositions
\newcommand{\rG}{\mathrm{G}} %graphs
\newcommand{\rGP}{\mathrm{GP}} %generalized permutahedra
\newcommand{\SP}{\mathfrak{p}} %standard permutahedron

\newcommand{\thh}{\mathbf{h}} 
\newcommand{\tg}{\mathbf{g}}
\newcommand{\tk}{\mathbf{k}}
\newcommand{\tp}{\mathbf{p}} 
\newcommand{\tq}{\mathbf{q}}
\newcommand{\tr}{\mathbf{r}}
\newcommand{\ta}{\mathbf{a}}
\newcommand{\tb}{\mathbf{b}} 
\newcommand{\tc}{\mathbf{c}}
\newcommand{\td}{\mathbf{d}}
\newcommand{\tm}{\mathbf{m}} %module
\newcommand{\tone}{\mathbf{1}}
\newcommand{\wU}{\mathbf{U}}
\newcommand{\wX}{\mathbf{X}}
\newcommand{\wE}{\mathbf{E}} %exponential species
\newcommand{\wL}{\mathbf{L}} %linear orders
\newcommand{\tLie}{\mathbf{Lie}} %Lie species
\newcommand{\tPi}{\mathbf{\Pi}} %flats
\newcommand{\tSig}{\mathbf{\Sigma}} %faces
\newcommand{\tSigh}{\mathbf{\widehat{\Sigma}}} % decompositions
\newcommand{\tG}{\mathbf{G}} %graphs
\newcommand{\tcG}{\mathbf{cG}} %connected graphs
\newcommand{\tGP}{\mathbf{GP}} 

\let\isoflat=\psi
\let\isolinear=\psi
\let\isograph=\varphi

\newcommand{\GL}{\mathrm{GL}}
\newcommand{\Un}{\mathrm{U}}
\newcommand{\Mat}{\mathrm{M}}

\newcommand{\FC}{\bm{\mathrm{f}}}

\newcommand{\contra}[1]{#1^{\vee}} 
\newcommand{\Kc}{\mathcal{K}}
\newcommand{\Kcb}{\overline{\Kc}}

\newcommand{\cKcb}{\contra{\Kcb}}

\newcommand{\Tc}{\mathcal{T}}
\newcommand{\Tcq}{\Tc_q}
\newcommand{\Sc}{\mathcal{S}}
\newcommand{\Pc}{\mathcal{P}} %primitive element functor
\newcommand{\Qc}{\mathcal{Q}} %indecomposables
\newcommand{\Uc}{\mathcal{U}} %universal enveloping algebra
\newcommand{\Nc}{\mathcal{N}}
\newcommand{\freelie}{\mathcal{L}ie}
\newcommand{\Jc}{\mathcal{J}}

\newcommand{\iH}{\mathcal{H}} %internal Hom 
\newcommand{\iE}{\mathcal{E}} %internal End

\newcommand{\Ser}{\mathscr{S}} %series
\newcommand{\Gls}{\mathscr{G}} %group-like series
\newcommand{\Prs}{\mathscr{P}} %primitive series
\newcommand{\Exs}{\mathscr{E}} %exponential series
\newcommand{\degmap}{\textsc{n}} %number operator
\newcommand{\euler}[1]{\mathtt{E}_{#1}} %Euler idempotents
\newcommand{\univh}{\Hat{\mathtt{G}}}
\newcommand{\univ}{\mathtt{G}} %exponential of Euler
\newcommand{\dynkin}{\mathtt{D}} %Dynkin
\newcommand{\pdynkin}{\dynkin} %partial Dynkin
\newcommand{\cancur}[1]{\{#1\}} %canonical curve %change?
\newcommand{\onepar}[1]{\langle #1\rangle} %one parameter subgroup (curve)
\newcommand{\conv}{\ast} %convolution

\newcommand{\type}[2]{\mathsf T_{#1}(#2)}
\newcommand{\ordi}[2]{\mathsf O_{#1}(#2)}
\newcommand{\expo}[2]{\mathsf E_{#1}(#2)}
\DeclareMathAlphabet{\mathpzc}{OT1}{pzc}{m}{it}
\newcommand{\varx}{\mathpzc{x}} % for variable in power series. 
\def\llb{{[\![}}
\def\rrb{{]\!]}}

\begin{document}

\title[{H}opf monoids in species]{{H}opf monoids in the category of species}

\author[M.~Aguiar]{Marcelo Aguiar}
\address{Department of Mathematics\\
Texas A\&M University\\
College Station, TX 77843}
\email{maguiar@math.tamu.edu}
\urladdr{http://www.math.tamu.edu/~maguiar}
\thanks{Aguiar supported in part by NSF grant DMS-1001935.}

\author[S.~Mahajan]{Swapneel Mahajan}
\address{Department of Mathematics\\
Indian Institute of Technology Mumbai\\
Powai, Mumbai 400 076\\ India}
\email{swapneel@math.iitb.ac.in}
\urladdr{http://www.math.iitb.ac.in/~swapneel}

\date{\today}

%16T30 Hopf algebra connections with combinatorics
%18D10 Monoidal categories
%18D35 Structured objects in a category
%20B30 Symmetric groups
%20F55 Reflection and Coxeter groups


\begin{abstract}
A Hopf monoid (in Joyal's category of species)
is an algebraic structure akin to that of a Hopf algebra.
We provide a self-contained introduction 
to the theory of Hopf monoids in the category of species. 
Combinatorial structures which compose and decompose give rise to Hopf monoids. 
We study several examples of this nature.
We emphasize the central role played in the theory 
by the Tits algebra of set compositions. 
Its product is tightly knit with the Hopf monoid axioms, and
its elements constitute universal operations on connected Hopf monoids.
We study analogues of the classical Eulerian and Dynkin idempotents and
discuss the Poincar\'e-Birkhoff-Witt and Cartier-Milnor-Moore theorems for Hopf monoids.
\end{abstract}

\maketitle

\tableofcontents

\section*{Introduction}\label{s:intro}

The theory of species and its applications to enumerative combinatorics
was developed by Joyal in~\cite{Joy:1981} and 
further advanced in~\cite{BerLabLer:1998}.
The basic theory of Hopf monoids in species was laid out in~\cite[Part~II]{AguMah:2010}.
The account we provide here includes most of the topics discussed there
and some more recent results from~\cite{AguArd:2012,ABT:2012,AguLau:2012, AguMah:2012}, 
as well as a good number of new results and constructions.

The category of species carries a symmetric monoidal structure, the Cauchy product.
We study the resulting algebraic structures, 
such as monoids, Hopf monoids, and Lie monoids. 
They are analogous to algebras, Hopf algebras, and
Lie algebras (and their graded counterparts).

Species will be denoted by letters such as $\tp$ and $\tq$,
monoids by $\ta$ and $\tb$, comonoids by $\tc$ and $\td$,
bimonoids (and Hopf monoids) by $\thh$ and $\tk$, and Lie monoids by $\tg$. 

A species $\tp$ consists of a family of vector spaces $\tp[I]$ indexed by finite sets $I$
and equivariant with respect to bijections $I\to J$.
A monoid is a species $\ta$ together with a family of product maps
\[
\mu_{S,T} : \ta[S] \otimes \ta[T] \to \ta[I]
\]
indexed by pairs $(S,T)$ of disjoint sets whose union is $I$.
They are subject to the usual axioms (associativity and unitality).
Iterations of these maps yield higher product maps
\[
\mu_{F} : \ta[S_1] \otimes \dots \otimes \ta[S_k] \to \ta[I]
\]
indexed by decompositions $F=(S_1,\dots,S_k)$ of $I$ 
(sequences of pairwise disjoint sets whose union is $I$).
When decompositions $F$ and $G$ are related by refinement, the higher product
maps $\mu_F$ and $\mu_G$ are related by associativity.

The notion of a comonoid in species is defined dually.
The higher coproduct maps $\Delta_F$ are again indexed by set decompositions.
Decompositions of finite sets go on to play a defining role in the theory.

Given two decompositions $F$ and $G$ of the same set $I$,
there is a natural way to refine $F$ using $G$ in order 
to obtain another decomposition $FG$ of $I$.
We refer to $FG$ as the Tits product of $F$ and $G$.
This turns the set of all decompositions of $I$ into a monoid 
(in the usual sense of the term).
We refer to the algebra of this monoid as the Tits algebra of decompositions of $I$.

A bimonoid is a species $\thh$ which is at the same time a monoid and a comonoid,
with the structures linked by the usual compatibility axiom.
The Tits product appears implicitly in this axiom.
This becomes evident when the compatibility 
among the higher product and coproduct maps is spelled out.
The axiom then involves two arbitrary decompositions $F$ and $G$ of $I$,
and their Tits products $FG$ and $GF$.

For any bimonoid $\thh$, the set of all endomorphisms of species of $\thh$ 
is an algebra under convolution. If the identity of $\thh$ is invertible,
then $\thh$ is a Hopf monoid and the inverse its antipode. 
There is a formula for the antipode in terms of the higher (co)product maps
(in which decompositions necessarily enter).

Complete definitions of a bimonoid and Hopf monoid are given in Section~\ref{s:hopf}.
The discussion on higher (co)products is deferred to Section~\ref{s:higher}, with the
groundwork on set decompositions and the Tits product laid out in Section~\ref{s:prelim}.

The category of species supports other operations in addition to the Cauchy product
and one such is the Hadamard product, discussed in Section~\ref{s:hadamard}.
The corresponding internal Hom and duality functors are also studied.
The Hadamard product preserves Hopf monoids. 
This is one aspect in which the theory of Hopf monoids 
differs from that of graded Hopf algebras. 

We pay particular attention to set-theoretic Hopf monoids
(in which the role of vector spaces is played by sets)
and connected Hopf monoids (in which $\thh[\emptyset]=\Kb$, the base field).
These are discussed in Sections~\ref{s:set} and~\ref{s:connected}.
The latter includes a discussion of canonical antipode formulas 
(Takeuchi's and Milnor and Moore's) and the formulation of the \emph{antipode problem},
a question of a general nature and concrete combinatorial interest.
The primitive part $\Pc(\tc)$ of a connected comonoid $\tc$ 
is also introduced at this point.

The free monoid $\Tc(\tq)$ on any species $\tq$ is studied in Section~\ref{s:free}.
We concentrate mainly on the case when $\tq$ is positive: $\tq[\emptyset]=0$.
In this situation, the monoid $\Tc(\tq)$ is connected and carries a canonical
Hopf monoid structure.
Corresponding discussions for commutative monoids and Lie monoids are 
given in Sections~\ref{s:freecom} and~\ref{s:freelie}.
The free commutative monoid on a species $\tq$ is denoted $\Sc(\tq)$
and the universal enveloping monoid of a Lie monoid $\tg$ by $\Uc(\tg)$. 
They are Hopf monoids.
Theorem~\ref{t:freeness} proves that the Hadamard product of two connected Hopf monoids is free, 
provided one of the factors is free.
In the case when both factors are free, 
Theorem~\ref{t:had-free-monoid} provides a basis for the Hadamard product 
in terms of the bases of the factors. 
Similar results hold in the commutative case.

Many examples of Hopf monoids are given in Section~\ref{s:examples}.
Most of them are connected and at least part of their structure is set-theoretic.
They arise from various combinatorial structures 
ranging from linear orders to generalized permutahedra, 
which then get organized into Hopf monoids
and related to one another by means of morphisms of such. 
Solutions to the antipode problem are provided in most cases.
Among these, we highlight those in Theorems~\ref{t:apode-graph}
and~\ref{t:apode-gp}.

Two additional important examples are studied in Section~\ref{s:faces}.
The bimonoid $\tSigh$ has a linear basis $\BH_{F}$ indexed by set decompositions $F$.
It is not connected and in fact not a Hopf monoid. Its companion
$\tSig$ has a linear basis indexed by set compositions (decompositions
into nonempty sets) and is a connected Hopf monoid. They are both free.
There is a surjective morphism of Hopf monoids $\upsilon:\tSigh\to\tSig$ which
removes the empty blocks from a decomposition.
Each space $\tSigh[I]$ and $\tSig[I]$ is an algebra under the Tits product,
and the interplay between this and the Hopf monoid structure is central 
to the theory. This important point was made in~\cite{AguMah:2006,AguMah:2010,PatSch:2006},
and is discussed extensively here.

To a species $\tq$ one may associate a vector space $\Ser(\tq)$ of series,
studied in Section~\ref{s:series}.
A series $s$ of $\tq$ is a collection of elements $s_I\in\tq[I]$, equivariant with
respect to bijections $I\to J$.
If $\ta$ is a monoid, then $\Ser(\ta)$ is an algebra, and
formal functional calculus can be used to substitute any series $s$ of $\ta$ 
with $s_\emptyset=0$ into a formal power series $a(\varx)$ 
to obtain another series $a(s)$ of $\ta$.
In particular, one can define $\exp(s)$ and $\log(s)$, 
the exponential and logarithm of $s$, and
these two operations are inverses of each other in their domains of definition.
If $\thh$ is a connected bimonoid, they restrict to inverse bijections
\[
\xymatrix@C+15pt{
\Prs(\thh) \ar@<0.6ex>[r]^-{\exp} & \ar@<0.6ex>[l]^-{\log} 
\Gls(\thh)
}
\]
between primitive series and group-like series of $\thh$. 
The former constitute a Lie algebra and the latter a group.

The characteristic operations are discussed in Section~\ref{s:action}.
They are analogous to the familiar Hopf powers $h\mapsto h^{(p)}$ 
from the theory of Hopf algebras. 
The role of the exponent $p$ is played by a decomposition $F$,
with the Tits monoid replacing the monoid of natural numbers.
For any decomposition $F=(S_1,\dots,S_k)$ of $I$, we refer to the composite
\[
\thh[I] \map{\Delta_F} \thh[S_1] \otimes \dots \otimes \thh[S_k] \map{\mu_F} \thh[I]
\]
as the characteristic operation of $\BH_{F}$ on $\thh[I]$, where $\thh$ is a bimonoid.
Extending by linearity, 
each element of $\tSigh[I]$ induces a linear endomorphism of $\thh[I]$,
and thus a series of $\tSigh$ induces a species endomorphism of $\thh$. 
If $\thh$ is connected, then the characteristic operations factor through $\upsilon$,
and a series of $\tSig$ induces a species endomorphism of $\thh$. 
Among the most important results in this section we mention Theorems~\ref{t:can-action}
and~\ref{t:act-to-prim}. 
The former states that, when the bimonoid $\thh$ is cocommutative, 
the characteristic operations endow each space $\thh[I]$ with
a left $\tSigh[I]$-module structure, and similarly for $\tSig[I]$ if $\thh$
is connected. 
Under the same hypotheses, the latter theorem states that 
the operation of a primitive element $z\in\Pc(\tSig)[I]$ maps $\thh[I]$ to $\Pc(\thh)[I]$. 
Moreover, if the coefficient of $\BH_{(I)}$ in $z$ is $1$, then
$z$ is an idempotent element of the Tits algebra $\tSig[I]$, and its operation is
a projection onto $\Pc(\thh)[I]$. 
The basis element $\BH_{(I)}$ of $\tSig[I]$ is indexed by the composition with one block. 
The characteristic operations of some important series of $\tSig$
on connected bimonoids are given in Table~\ref{table:euler-log}.

Section~\ref{s:convpower} deals with the analogues of certain classical idempotents. 
They are elements (or series) of $\tSig$.
For each integer $k\geq 0$, the Eulerian idempotent $\euler{k}$ is
a certain series of $\tSig$, with $\euler{1}$ being primitive.
The latter is the logarithm of the group-like series $\univ$ defined by $\univ_I := \BH_{(I)}$. 
Since $\univ$ operates as the identity, $\euler{1}$ operates as its logarithm.
It follows that the logarithm of the identity projects any
cocommutative connected bimonoid onto its primitive part.
This result is obtained both in Corollary~\ref{c:exp-log-iH} and in
Corollary~\ref{c:char-op-1euler}.
Another important primitive series of $\tSig$ is $\dynkin$,
the Dynkin quasi-idempotent. It operates (on products of primitive elements)
as the left bracketing operator. 
The analogue of the classical Dynkin-Specht-Wever theorem is given in Corollary~\ref{c:DSW}. 
This section also discusses the Garsia-Reutenauer idempotents, reviewed below.

\begin{table}[htp]
\caption{}
\label{table:euler-log}
\centering
\renewcommand{\arraystretch}{1.3}
\begin{tabular}{c|c|c}
Series of $\tSig$ & Characteristic operation on $\thh$ &\\\hline
$\euler{1}$ & $\log\id$ & Corollary~\ref{c:char-op-1euler}\\
$\euler{k}$ & $\frac{1}{k!}(\log\id)^{\conv k}$ & Proposition~\ref{p:conv-power-log-id}\\
$\BH_p$ & $\id^{\conv p}$ & Formula~\eqref{e:conv-power-id-conn}\\
$\BH_{-1}$ & $\apode$ (antipode) & Formula~\eqref{e:char-apode}\\ 
$\dynkin$ & $\apode\conv \degmap$ & Corollary~\ref{c:dynkin-action} 
\end{tabular}
\end{table}

The Poincar\'e-Birkhoff-Witt (PBW) and 
Cartier-Milnor-Moore (CMM) theorems for Hopf monoids in species
are discussed in Section~\ref{s:structure}.
It is a nontrivial fact that any Lie monoid $\tg$ embeds in $\Uc(\tg)$.
This is a part of the PBW theorem;
the full result is given in Theorem~\ref{t:pbw}.
In a bimonoid $\thh$, 
the commutator preserves the primitive part $\Pc(\thh)$.
The CMM theorem (Theorem~\ref{t:cmm}) states that the functors
\[
\xymatrix@C+20pt{
\{\text{positive Lie monoids}\} \ar@<0.6ex>[r]^-{\Uc} & \ar@<0.6ex>[l]^-{\Pc} 
\{\text{cocommutative connected Hopf monoids}\}
}
\]
form an adjoint equivalence.
%The proof employs the Garsia-Reutenauer idempotents.

The structure of the Tits algebra $\tSig[I]$ has been studied in detail in the literature.
It follows from works of Bidigare~\cite{Bid:1997} and Brown~\cite{Bro:2000} that
$\tSig[I]$ admits a complete family of orthogonal idempotents indexed by partitions of $I$.
In fact there is a canonical choice for such a family,
which we call the Garsia-Reutenauer idempotents (Theorem~\ref{t:idempotent-family}).
The action of these idempotents provides a canonical decomposition of each space
$\thh[I]$ (Theorem~\ref{t:euler-decomp}).
The Garsia-Reutenauer idempotent indexed by the partition with one block 
is the first Eulerian idempotent, 
hence its action is a projection onto the primitive part.
This leads to the important result that 
for any cocommutative connected Hopf monoid $\thh$,
there is an isomorphism of comonoids between $\Sc(\Pc(\thh))$ and $\thh$
(Theorem~\ref{t:pbw+cmm}).
This is closely related to (and in fact used to deduce) the PBW and CMM theorems.
PBW states that $\Sc(\Pc(\thh))\cong \Uc(\Pc(\thh))$ as comonoids,
and CMM includes the statement that $\Uc(\Pc(\thh))\cong \thh$ as bimonoids.

Section~\ref{s:genfun} collects a number of results on the dimension sequence $\dim_{\Kb} \thh[n]$
of a finite-dimensional connected Hopf monoid $\thh$. The structure imposes conditions
on this sequence in the form of various polynomial inequalities its entries must
satisfy. For example, Theorem~\ref{t:ordi-boolean} states that the Boolean
transform of the sequence must be nonnegative, and Theorem~\ref{t:binomial-set}
that if the Hopf monoid is set-theoretic (and nontrivial),
then the same is true for the binomial transform.
The latter result follows from Theorem~\ref{t:lagrange},
the analogue of Lagrange's theorem for Hopf monoids.

Among the topics discussed in~\cite[Part~II]{AguMah:2010} but left out here,
we mention the notion of cohomology and deformations for set-theoretic
comonoids, the construction of the cofree comonoid on a positive species,
the construction of the free Hopf monoid on a positive comonoid, 
the notion of species with balanced operators, and
the geometric perspective that adopts the braid hyperplane arrangement
as a central object. The latter is particularly important in regard to generalizations
of the theory which we are currently undertaking.

The notion of Hopf monoid parallels the more familiar one
of graded Hopf algebra and the analogy manifests itself throughout our discussion.
Much of the theory we develop has a counterpart
for graded Hopf algebras, with Hopf monoids being richer due to the underlying
species structure.
To illustrate this point, consider the characteristic operations.
For connected Hopf monoids, they are indexed by elements of the Tits algebra,
while for connected graded Hopf algebras they are indexed
by elements of Solomon's descent algebra. 
The symmetric group acts on the former, and the latter is the invariant subalgebra. 
In some instances, the theory for graded Hopf algebras
is not widely available in the literature.
We develop the theory for Hopf monoids only, but hope that this paper is useful
to readers interested in Hopf algebras as well.
Some constructions are specific to the setting of Hopf monoids
and others acquire special form.
There are ways to build graded Hopf algebras from Hopf monoids.
These are studied in~\cite[Part~III]{AguMah:2010},
but not in this paper.

\subsection*{Acknowledgements} 

We warmly thank the editors for their
interest in our work and for their support during the preparation of the
manuscript. We are equally grateful to the referees for a prompt and
detailed report with useful comments and suggestions.

\section{Preliminaries on compositions and partitions}\label{s:prelim}

Decompositions of sets play a prominent role in the theory of Hopf monoids.
Important properties and a good amount of related notation are discussed here.
In spite of this, it is possible and probably advisable to proceed with
the later sections first
and refer to the present one as needed. 

\subsection{Set compositions}\label{ss:compositions}

Let $I$ be a finite set.
A \emph{composition} of $I$ is a finite sequence $(I_1,\ldots,I_k)$
%A \emph{composition} of a finite set $I$ is a finite sequence $(I_1,\ldots,I_k)$
of disjoint nonempty subsets of $I$ such that
\[
I= \bigcup_{i=1}^k I_i.
\]
The subsets $I_i$ are the \emph{blocks} of the composition.
We write $F\vDash I$ to indicate that $F=(I_1,\ldots,I_k)$ is a composition of~$I$.

There is only one composition of the empty set (with no blocks).

Let $\rSig[I]$ denote the set of compositions of $I$.

\subsection{Operations}\label{ss:oper}

Let $F=(I_1,\dots,I_k)$ be a composition of $I$.

The \emph{opposite} of $F$ is the composition
\[
\opp{F} = (I_k,\dots,I_1).
\]

Given a subset $S$ of $I$,
let $i_1<\cdots<i_j$ be the subsequence of $1<\cdots<k$
consisting of those indices $i$ for which $I_i\cap S\neq\emptyset$.
The \emph{restriction} of $F$ to $S$ is the composition of $S$ defined by
\[
F|_S = (I_{i_1}\cap S,\ldots,I_{i_j}\cap S).
\]
Let $T$ be the complement of $S$ in $I$.
We say that $S$ is \emph{$F$-admissible} if for each $i=1,\ldots,k$, either
\[
I_i\subseteq S \qor I_i\subseteq T.
\]
Thus $S$ is $F$-admissible if and only if $T$ is, and
in this case, $F|_S$ and $F|_T$ are complementary subsequences of $F$.

We write
\begin{equation}\label{e:complementary}
I=S\sqcup T
\end{equation}
to indicate that $(S,T)$ is an ordered pair of complementary subsets of $I$,
as above. Either subset $S$ or $T$ may be empty.

Given $I=S\sqcup T$ and compositions $F=(S_1,\dots,S_j)$ of $S$
and $G=(T_1,\dots,T_k)$ of $T$, their \emph{concatenation}
\[
F\cdot G := (S_1,\dots,S_j,T_1,\dots,T_k)
\]
is a composition of $I$. A \emph{quasi-shuffle} of $F$ and $G$ is a
composition $H$ of $I$ such that
$H|_S=F$ and $H|_T=G$.
In particular, each block of $H$ is either a block of $F$, a block of $G$,
or a union of a block of $F$ and a block of $G$. 

\subsection{The Schubert cocycle and distance}\label{ss:schubert}

Let $I=S\sqcup T$ and $F=(I_1,\ldots,I_k)\vDash I$. 
The \emph{Schubert cocycle} is defined by
\begin{equation}\label{e:schubert-comp}
\area_{S,T}(F) := \abs{\{(i,j)\in S\times T \mid \text{$i$ is in a strictly later block of $F$ than $j$}\}}.
\end{equation}
Alternatively,
\[
\area_{S,T}(F) = \sum_{1 \leq i < j \leq k} \abs{I_i \cap T}\,\abs{I_j \cap S}.
\]

Let $F'=(I'_1,\ldots,I'_n)$ be another composition of $I$. The \emph{distance}
between $F$ and $F'$ is
\begin{equation}\label{e:dist}
\dist(F,F') := \sum_{\substack{i<j\\ m>l}}\,\abs{I_i \cap I'_m} \,\abs{I_j \cap I'_l}.
\end{equation}
In the special case when $F$ and $F'$ consist of the same blocks
(possibly listed in different orders), the previous formula simplifies to
\begin{equation}\label{e:dist-supp}
\dist(F,F') = \sum_{(i,j)} \,\abs{I_i}\, \abs{I_j},
\end{equation}
where the sum is over those pairs $(i,j)$ such that $i<j$ and
$I_i$ appears after $I_j$ in $F'$. In particular,
\begin{equation}\label{e:dist-opp}
\dist(F,\opp{F}) = \sum_{1\leq i<j\leq k} \abs{I_i}\, \abs{I_j}.
\end{equation}
Note also that
\[
\area_{S,T}(F) = \dist(F,G),
\]
where $G=(S,T)$.

\subsection{Linear orders}\label{ss:linear-prelim}

When the blocks are singletons, a composition of $I$ amounts to a linear order on $I$.
We write $\ell=i_1\cdots i_n$ for the linear order on $I=\{i_1,\ldots,i_n\}$ for which
$i_1<\cdots<i_n$.

Linear orders are closed under opposition, restriction, and concatenation. 
The opposite of $\ell=i_1\cdots i_n$ is $\opp{\ell}=i_n\cdots i_1$.
If $\ell_1=s_1\cdots s_i$ and $\ell_2=t_1\cdots t_j$, 
their concatenation is 
\[
\ell_1\cdot \ell_2=s_1\cdots s_i\, t_1\cdots t_j.
\]
The restriction $\ell|_S$ of a linear order $\ell$ on $I$ to $S$
is the list consisting of the elements of $S$ written in the order
in which they appear in $\ell$. 
We say that $\ell$ is a \emph{shuffle} of $\ell_1$ and $\ell_2$ if $\ell|_S=\ell_1$ and 
$\ell|_T=\ell_2$.

Replacing a composition $F$ of $I$ for a linear order $\ell$ on $I$ in
equations~\eqref{e:schubert-comp} and~\eqref{e:dist} we obtain
\begin{gather}\label{e:schubert-linear}
\area_{S,T}(\ell) = \abs{\{(i,j)\in S\times T \mid \text{$i>j$
according to $\ell$}\}},\\
\label{e:dist-linear}
\dist(\ell,\ell') = \abs{\{(i,j)\in I\times I \mid \text{$i<j$ according to $\ell$ and $i>j$ according to $\ell'$}\}}.
\end{gather}

\subsection{Refinement}\label{ss:refinement}

The set $\rSig[I]$ is a partial order under \emph{refinement}: 
we say that $G$ refines $F$ and write
$F\leq G$ if each block of $F$ is obtained by merging a number of contiguous blocks of $G$.
The composition $(I)$ is the unique minimum element, and linear orders are the maximal elements. 

Let $F=(I_1,\ldots,I_k)\in \rSig[I]$. 
There is an order-preserving bijection
\begin{equation}\label{e:res-conc}
\{G\in\rSig[I] \mid F\leq G\} \longleftrightarrow \rSig[I_1]\times\cdots\times\rSig[I_k],
\quad G\longmapsto (G|_{I_1},\ldots, G|_{I_k}).
\end{equation}
The inverse is given by concatenation:
\[
 (G_1,\ldots,G_k) \mapsto G_1\cdots G_k.
\]

Set compositions of $I$ are in bijection with flags of subsets of $I$ via
\[
(I_1,\ldots,I_k) \mapsto (\emptyset\subset I_1\subset I_1\cup I_2\subset \cdots \subset
I_1\cup \cdots\cup I_k = I).
\]
Refinement of compositions corresponds to inclusion of flags. In this manner $\rSig[I]$ is
a lower set of the Boolean poset $2^{2^I}$, and hence a meet-semilattice.

\subsection{The Tits product}\label{ss:faceprod}

Let $F=(S_1,\ldots,S_p)$ and $G=(T_1,\ldots,T_q)$ be two compositions of $I$.
Consider the pairwise intersections 
\[
A_{ij} := S_i\cap T_j
\]
for $1\leq i\leq p$, $1\leq j\leq q$. 
%Schematically,
A schematic picture is shown below.
\begin{equation}\label{e:pqsets}
\setlength{\unitlength}{.9 pt}
\quad
\begin{gathered}
\begin{picture}(100,90)(20,0)
\put(50,40){\oval(100,80)}
\put(0,55){\dashbox{2}(100,0){}}
\put(0,25){\dashbox{2}(100,0){}}
\put(45,63){$S_1$}
\put(45,8){$S_p$}
\end{picture}
\quad
\begin{picture}(100,90)(10,0)
\put(50,40){\oval(100,80)}
\put(25,0){\dashbox{2}(0,80){}}
\put(75,0){\dashbox{2}(0,80){}}
\put(8,35){$T_1$}
\put(82,35){$T_q$}
\end{picture}
\quad
\begin{picture}(100,90)(0,0)
\put(50,40){\oval(100,80)}
\put(0,55){\dashbox{2}(100,0){}}
\put(0,25){\dashbox{2}(100,0){}}
\put(25,0){\dashbox{2}(0,80){}}
\put(75,0){\dashbox{2}(0,80){}}
\put(6,62){$A_{11}$}
\put(77,62){$A_{1q}$}
\put(6,8){$A_{p1}$}
\put(77,8){$A_{pq}$}
\end{picture}
\end{gathered}
\end{equation}
The \emph{Tits product} $FG$ is the composition
obtained by listing the nonempty intersections $A_{ij}$ in lexicographic order of the indices $(i,j)$:
\begin{equation}\label{e:tits}
FG := (A_{11},\ldots,A_{1q},\ldots, A_{p1},\ldots, A_{pq}),
\end{equation}
where it is understood that empty intersections are removed.

The Tits product is associative and unital; it turns the set $\rSig[I]$
into an ordinary monoid, that we call the \emph{Tits monoid}.
The unit is $(I)$. If $\abs{I}\geq 2$, it is not commutative.
In fact, 
\[
FG = GF \iff \text{$F$ and $G$ admit a common refinement.}
\]
The following properties hold, for all compositions $F$ and $G$.
\begin{itemize}
\item $F \leq FG$.
\item $F\leq G$ $\iff$ $FG=G$.
\item If $G \leq H$, then $FG \leq FH$.
\item If $C$ is a linear order, then $CF=C$ and $FC$ is a linear order.
\item $F^2=F$ and $F \opp{F} = F$.
\item $FGF = FG$.
\end{itemize}

The last property makes the monoid $\rSig[I]$ a \emph{left regular band}. 
Additional properties are given in~\cite[Proposition~10.1]{AguMah:2010}.

\begin{remark}
The product of set compositions may be seen as a special instance of an
operation that is defined in more general settings.
It was introduced by Tits in the context of Coxeter groups and buildings~\cite[Section 2.30]{Tit:1974}.
Bland considered it in the context of oriented matroids~\cite[Section~5, page~62]{Bla:1974}.
It appears in the book by Orlik and Terao on hyperplane arrangements~\cite[Definition~2.21 and Proposition~2.22]{OrLTer:1992}.
A good self-contained account is given by Brown~\cite[Appendices~A and~B]{Bro:2000}.
The case of set compositions arises from the symmetric group (the Coxeter
group of type A), or from the braid hyperplane arrangement.
\end{remark}

\subsection{Length and factorial}\label{ss:factorials}

Let $F=(I_1,\ldots,I_k)$ be a composition of $I$.
The \emph{length} of $F$ is its number of blocks:
\begin{equation}\label{e:Fdeg}
\len(F) := k
\end{equation}
and the \emph{factorial} is
\begin{equation}\label{e:Ffac}
F! := \prod_{i=1}^k \,\abs{I_i}!.
\end{equation}
The latter counts the number of ways of endowing each block of $F$
with a linear order. 

Given a set composition $G$ refining $F$, let 
\begin{equation}\label{e:FGdeg}
\len(G/F) := \prod_{i=1}^k n_i,
\end{equation}
where $n_i:=\len(G|_{I_i})$ is the number of blocks of $G$ that refine 
the $i$-th block of $F$. Let also
\begin{equation}\label{e:degree-fact}
(G/F)! = \prod_{i=1}^k n_i!.
\end{equation}
In particular, $\len(G/(I))=\len(G)$
and if $G$ is a linear order, then $(G/F)!=F!$.

\subsection{Set partitions}\label{ss:partitions}

A \emph{partition} $X$ of $I$ is an unordered collection $X$
of disjoint nonempty subsets of $I$ such that
\[
I = \bigcup_{B\in X} B.
\]
The subsets $B$ of $I$ which belong to $X$ are the \emph{blocks} of $X$.
We write $X\vdash I$ to indicate that $X$ is a partition of $I$.

There is only one partition of the empty set (with no blocks).

Let $\rPi[I]$ denote the set of partitions of $I$.

Given $I=S\sqcup T$ and partitions $X$ of $S$ and $Y$ of $T$, 
their \emph{union} is the partition $X\sqcup Y$ of $I$ 
whose blocks belong to either $X$ or $Y$.

Restriction, quasi-shuffles and admissible subsets are defined for set partitions as
for set compositions (Section~\ref{ss:oper}), disregarding the
order among the blocks.

Refinement is defined for set partitions as well: we set
$X\leq Y$ if each block of $X$ is obtained by merging a number of blocks of $Y$. The partition $\{I\}$ is the unique minimum element and the partition of $I$ into
singletons is the unique maximum element. We denote it by $\maxflat$.
The poset $\rPi[I]$ is a lattice. 

Lengths and factorials for set partitions $X\leq Y$ are defined as for set compositions:
\[
\len(X) := \abs{X}, \quad
X! := \prod_{B\in X} \abs{B}!, \quad
\len(Y/X) := \prod_{B\in X} n_B, \quad
(Y/X)! := \prod_{B\in X} (n_B)!.
\]
Here $n_B:=\len(Y|_B)$ is the number of blocks of $Y$ that refine the block $B$ of $X$.
In particular, $\len(Y/\{I\})=\len(Y)$ and $(\maxflat/X)!=X!$.

The \emph{cyclic factorial} of $X$ is
\begin{equation}\label{e:Xifac}
X\ifac := \prod_{B\in X} (\left\abs{B}-1\right)!.
\end{equation}
It counts the number of ways of endowing each block of $X$
with a cyclic order.

The M\"obius function of $\rPi[I]$ satisfies
\begin{equation}\label{e:mobiusPi}
\mu(X,Y) = (-1)^{\len(Y)-\len(X)}\, \prod_{B\in X} (n_B-1)!
\end{equation}
for $X\leq Y$, with $n_B$ as above. In particular,
\begin{equation}\label{e:partmobiusPi}
\mu(\minflat,X) = (-1)^{\len(X)-1} (\len(X) - 1)!
\qqand
\mu(X,\maxflat) = (-1)^{\abs{I}-\len(X)}\, X\ifac.
\end{equation}

\subsection{Support}\label{ss:support}

The \emph{support} of a composition $F$ of $I$ is the partition
$\supp F$ of $I$ obtained by forgetting the order among the blocks:
if $F=(I_1,\ldots,I_k)$, then 
\[
\supp F :=\{I_1,\ldots,I_k\}.
\]

The support preserves lengths and factorials:
\begin{gather*}
\len(\supp F) = \len(F), \qquad (\supp F)! = F!,\\
\len( \supp G/\supp F) = \len(G/F), \qquad (\supp G/\supp F)! = (G/F)!. 
\end{gather*}

If $F$ refines $G$, then $\supp F$ refines $\supp G$. Thus, the map
$\supp:\rSig[I]\to\rPi[I]$ 
is order-preserving. Meets are not preserved; for example, $(S,T)\wedge (T,S) = (I)$
but $\supp(S,T)=\supp(T,S)$. The support turns the Tits product into the join:
\[
\supp(FG) = (\supp F) \vee (\supp G).
\]
We have
\begin{equation}\label{e:supp-tits}
GF=G \iff \supp F \leq \supp G. 
\end{equation}
Both conditions express the fact that each block of $G$ is contained in a block of $F$.

\subsection{The braid arrangement}\label{ss:braid}

Compositions of $I$ are in bijection with \emph{faces} of the
\emph{braid hyperplane arrangement} in $\Rb^I$. 
Partitions of $I$ are in bijection with \emph{flats}.
Refinement of compositions corresponds to inclusion of faces, meet to
intersection, linear orders to \emph{chambers} (top-dimensional faces), 
and $(I)$ to the central face. 
When $S$ and $T$ are nonempty, the statistic $\area_{S,T}(F)$ counts the number of hyperplanes that separate the face $(S,T)$ from $F$. The factorial $F!$ is the number
of chambers that contain the face $F$. The Tits product $FG$ is, in a sense, the
face containing $F$ that is closest to $G$.

This geometric perspective is expanded in~\cite[Chapter~10]{AguMah:2010}.
It is also the departing point of far-reaching generalizations of the notions
studied in this paper. We intend to make this the subject of future work.

\subsection{Decompositions}\label{ss:decompositions}

A \emph{decomposition} of a finite set $I$ is a finite sequence
$(I_1,\ldots,I_k)$
of disjoint subsets of $I$ whose union is $I$. In this situation, we write
\[
I=I_1\sqcup\cdots\sqcup I_k.
\]
A composition is thus a decomposition in which each subset $I_i$ is nonempty.

Let $\rSigh[I]$ denote the set of decompositions of $I$.
If $\abs{I}=n$, there are $k^n$ decompositions into $k$ subsets.
Therefore, the set $\rSigh[I]$ is countably infinite.

With some care, most of the preceding considerations on compositions
extend to decompositions.
For $I=S\sqcup T$,
decompositions of $S$ may be concatenated with decompositions of $T$.
Let $F=(I_1,\ldots,I_k)$ be a decomposition of $I$. 
The restriction to $S$ is the decomposition
\[
F|_S := (S\cap I_1,\ldots,S\cap I_k)
\]
(the empty intersections are kept).
The set $\rSigh[I]$ is a monoid under a product defined as in~\eqref{e:tits}
(where now all $pq$ intersections are kept). 
The product is associative and unital.
We call it the Tits product on decompositions.

Consider the special case in which $I$ is empty.
There is one decomposition 
\[
\underbrace{(\emptyset,\ldots,\emptyset)}_{p},
\]
of the empty set for each nonnegative integer $p$,
with $p=0$ corresponding to the unique decomposition with no subsets. 
We denote it by $\emptyset^p$. Under concatenation,
\[
\emptyset^p\cdot\emptyset^q = \emptyset^{p+q},
\]
and under the Tits product
\[
\emptyset^p\emptyset^q = \emptyset^{pq}.
\]
Thus,
\begin{equation}\label{e:empty-monoid}
\rSigh[\emptyset]\cong \Nb
\end{equation}
with concatenation corresponding to addition of nonnegative integers
and the Tits product to multiplication.

Given a decomposition $F$, let
$\pos{F}$ denote the composition of $I$ obtained by removing
those subsets $I_i$ which are empty. This operation preserves concatenations,
restrictions, and Tits products:
\begin{equation}\label{e:dec-comp-prod}
\pos{(F\cdot G)} = \pos{F}\cdot\pos{G},\quad
\pos{(F|_S)} = \pos{F}|_S, \quad
\pos{(FG)} = \pos{F} \pos{G}.
\end{equation}

The notion of distance makes sense for decompositions. Since empty subsets
do not contribute to~\eqref{e:dist}, we have
\begin{equation}\label{e:dec-comp-dist}
\dist(F,F') = \dist(\pos{F},\pos{F}').
\end{equation}

The notion of refinement for decompositions deserves special attention.
Let $F=(I_1,\ldots,I_k)$ be a decomposition of a (possibly empty) finite set $I$.
We say that another decomposition $G$ of $I$ refines $F$, and write $F\leq G$, 
if there exists a sequence $\gamma=(G_1,\ldots,G_k)$, with $G_j$ a decomposition
of $I_j$, $j=1,\ldots,k$, such that $G=G_1\cdots G_k$ (concatenation). 
In this situation, we also say that $\gamma$ is a \emph{splitting} of the pair $(F,G)$.

If $F=\emptyset^0$, then $\gamma$ must be the empty sequence, and 
hence $G=\emptyset^0$. Thus, the only decomposition that refines 
$\emptyset^0$ is itself.

If $F$ and $G$ are compositions with $F\leq G$, 
then $(F,G)$ has a unique splitting, in view of~\eqref{e:res-conc}.
In general, however, the factors $G_j$ cannot be determined from the pair $(F,G)$,
and thus the splitting $\gamma$ is not unique.
For example, if
\[
F=\emptyset^2 \pand G=\emptyset^1,
\]
we may choose either
\[
(G_1=\emptyset^1 \pand G_2=\emptyset^0)
\qor
(G_1=\emptyset^0 \pand G_2=\emptyset^1)
\]
to witness that $F\leq G$. Note also that $\emptyset^1\leq\emptyset^0$ (for we may choose $G_1=\emptyset^0$) and in fact
\[
\emptyset^p\leq\emptyset^q
\]
for any $p\in\Nb_+$ and $q\in\Nb$. 

The preceding also shows that
refinement is not an antisymmetric relation. It is reflexive and transitive, though, and
thus a preorder on each set $\rSigh[I]$. 
We have, for $F\neq\emptyset^0$,
\[
F\leq G \iff \pos{F}\leq \pos{G}.
\]

Concatenation defines an order-preserving map
\begin{equation}\label{e:conc-dec}
\rSigh[I_1]\times\cdots\times\rSigh[I_k] \longrightarrow \{G\in\rSigh[I] \mid F\leq G\},
\quad (G_1,\ldots,G_k) \longmapsto G_1\cdots G_k.
\end{equation}
According to the preceding discussion, the map is surjective but not injective.
A sequence $\gamma$ in the fiber of $G$ is a
splitting of $(F,G)$.

\section{Species and Hopf monoids}\label{s:hopf}

Joyal's category of species~\cite{Joy:1981} provides the context for our work.
The Cauchy product furnishes it with a braided monoidal structure. We
are interested in the resulting algebraic structures, particularly that of
a Hopf monoid. This section presents the basic definitions and describes
these structures in concrete terms.

\subsection{Species}\label{ss:species}

Let $\Fset$ denote the category whose objects are finite sets and 
whose morphisms are bijections.
Let $\Kb$ be a field and let $\Veck$ denote the category whose objects are vector spaces over $\Kb$
and whose morphisms are linear maps.

A \emph{(vector) species} is a functor
\[
\Fset \longrightarrow \Veck.
\]

Given a species $\tp$, its value on a finite set $I$ is denoted by $\tp[I]$.
Its value on a bijection $\sigma:I\to J$ is denoted
\[
\tp[\sigma]:=\tp[I]\to\tp[J].
\]
We write $\tp[n]$ for the space $\tp[\{1,\ldots,n\}]$.
The symmetric group $\Sr_n$ acts on $\tp[n]$ by
\[
\sigma\cdot x := \tp[\sigma](x)
\]
for $\sigma\in\Sr_n$, $x\in\tp[n]$.

A morphism between species is a natural transformation of functors.
Let $f:\tp\to\tq$ be a morphism of species. It consists of a
collection of linear maps
\[
f_I: \tp[I]\to\tq[I],
\]
one for each finite set $I$, such that the diagram
\begin{equation}\label{e:mor-sp}
\begin{gathered}
\xymatrix{
\tp[I] \ar[r]^-{f_I} \ar[d]_{\tp[\sigma]} & \tq[I] \ar[d]^{\tq[\sigma]}\\
\tp[J] \ar[r]_-{f_J} & \tq[J]
}
\end{gathered}
\end{equation}
commutes, for each bijection $\sigma:I\to J$. 
Let $\Ssk$ denote the category of species.

A species $\tp$ is said to be \emph{finite-dimensional} if each vector space $\tp[I]$
is finite-dimensional. We do not impose this condition, although most examples
of species we consider are finite-dimensional.

A species $\tp$ is \emph{positive} if $\tp[\emptyset]=0$. The \emph{positive part} of 
a species $\tq$ is the (positive) species $\tq_+$ defined by
\[
\tq_+[I] := \begin{cases}
\tq[I] & \text{ if $I\neq\emptyset$,}\\
0 & \text{ if $I=\emptyset$.}
\end{cases}
\]

Given a vector space $V$, let $\tone_V$ denote the species defined by
\begin{equation}\label{e:tone}
\tone_V[I] := \begin{cases}
V & \text{if $I$ is empty,}\\
0 & \text{otherwise.}
\end{cases}
\end{equation}

\subsection{The Cauchy product}\label{ss:cauchy}

Given species $\tp$ and $\tq$, their \emph{Cauchy product} is the species 
$\tp \bdot \tq$ defined on a finite set $I$ by
\begin{equation}\label{e:cau} 
(\tp \bdot \tq)[I] := \bigoplus_{I = S \sqcup T} \tp[S] \otimes \tq[T].
\end{equation}
The direct sum is over all decompositions $(S,T)$ of $I$,
or equivalently over all subsets $S$ of $I$.
On a bijection $\sigma: I\to J$, the map
$(\tp \bdot \tq)[\sigma]$ is defined to be the direct sum of the maps
\[
\tp[S] \otimes \tq[T]\map{\tp[\sigma|_S]\otimes\tp[\sigma|_T]}
\tp[\sigma(S)] \otimes \tq[\sigma(T)]
\]
over all decompositions $(S,T)$ of $I$, 
where $\sigma|_S$ denotes the restriction of $\sigma$ to $S$.

The Cauchy product turns $\Ssk$ into a monoidal category.
The unit object is the species $\tone_{\Kb}$ as in~\eqref{e:tone}.

Let $q \in \Kb$ be a fixed scalar, possibly zero.
Consider the natural transformation
\[
\beta_q \colon \tp \bdot \tq \to \tq \bdot \tp
\]
which on a finite set $I$ is the direct sum of the maps
\begin{equation}\label{e:braiding}
\tp[S] \otimes \tq[T] \to \tq[T] \otimes \tp[S],
\qquad
x \otimes y \mapsto q^{\abs{S}\abs{T}} y \otimes x
\end{equation}
over all decompositions $(S,T)$ of $I$. The notation $\abs{S}$ stands
for the cardinality of the set $S$.

If $q$ is nonzero, then $\beta_q$ is a (strong) braiding
for the monoidal category $(\Ssk,\bdot)$.
In this case, the inverse braiding is $\beta_{q^{-1}}$,
and $\beta_q$ is a symmetry if and only if $q=\pm 1$.
The natural transformation $\beta_0$ is a lax braiding
for $(\Ssk,\bdot)$.

We consider monoids and comonoids in the monoidal category $(\Ssk,\bdot)$,
and bimonoids and Hopf monoids in the braided monoidal category $(\Ssk,\bdot,\beta_q)$.
We refer to the latter as $q$-\emph{bimonoids} and $q$-\emph{Hopf monoids}.
When $q=1$, we speak simply of \emph{bimonoids} and \emph{Hopf monoids}.
We expand on these notions in the following sections.

\subsection{Monoids}\label{ss:monoids}

The structure of a monoid $\ta$ consists of morphisms of species 
\[
\mu:\ta\bdot\ta\to\ta \qand \iota:\tone_{\Kb}\to\ta
\] 
subject to the familiar associative and unit axioms.
In view of~\eqref{e:cau}, the morphism $\mu$ consists of a collection of linear maps
\[
\mu_{S,T} : \ta[S]\otimes\ta[T] \to \ta[I],
\]
one for each finite set $I$ and each decomposition $(S,T)$ of $I$.
The unit $\iota$ reduces to a linear map
\[
\iota_\emptyset: \Kb \to \ta[\emptyset].
\]
The collection must satisfy the following \emph{naturality} condition. For each
bijection $\sigma:I\to J$, the diagram
\begin{equation}\label{e:nat}
\begin{gathered}
\xymatrix@C+40pt@R+1pc{
\ta[S]\otimes\ta[T] \ar[r]^-{\mu_{S,T}} \ar[d]_{\ta[\sigma|_S]\otimes\ta[\sigma|_T]} & \ta[I] \ar[d]^{\ta[\sigma]}\\
\ta[\sigma(S)]\otimes\ta[\sigma(T)] \ar[r]_-{\mu_{\sigma(S),\sigma(T)}} & \ta[J]
}
\end{gathered}
\end{equation}
commutes.

The associative axiom states that for each decomposition $I=R\sqcup S\sqcup T$, the diagram
\begin{gather}\label{e:assoc}
\begin{gathered}
\xymatrix@R+1pc@C+25pt{
\ta[R]\otimes\ta[S]\otimes\ta[T]\ar[r]^-{\id\otimes\mu_{R,S}}
\ar[d]_{\mu_{R,S}\otimes\id} & \ta[R]\otimes\ta[S\sqcup T]\ar[d]^{\mu_{R,S\sqcup T}}\\
\ta[R\sqcup S]\otimes \ta[T]\ar[r]_-{\mu_{R\sqcup S,T}} & \ta[I]
}
\end{gathered}
\end{gather}
commutes. 

The unit axiom states that for each finite set $I$, the diagrams
\begin{align}\label{e:unit}
&
\begin{gathered}
\xymatrix@=3pc{
\ta[I]\ar@{=}[rd] & \ta[\emptyset]\otimes\ta[I] \ar[l]_-{\mu_{\emptyset,I}}\\
& \Kb\otimes \ta[I] \ar[u]_{\iota_\emptyset\otimes\id_I}
}
\end{gathered}
& &
\begin{gathered}
\xymatrix@=3pc{
\ta[I]\otimes \ta[\emptyset] \ar[r]^-{\mu_{I,\emptyset}} & \ta[I]\ar@{=}[ld]\\
\ta[I] \otimes \Kb\ar[u]^{\id_I\otimes\iota_\emptyset} &
}
\end{gathered}
\end{align}
commute.

We refer to the maps $\mu_{S,T}$ as the \emph{product maps} of the monoid $\ta$.
The following is a consequence of associativity. For any
decomposition $I=S_1\sqcup\cdots\sqcup S_k$ with $k\geq 2$, there is a
unique map
\begin{equation}\label{e:iterated-mu}
\ta[S_1]\otimes\cdots\otimes\ta[S_k] \map{\mu_{S_1,\ldots,S_k}} \ta[I]
\end{equation}
obtained by iterating the product maps $\mu_{S,T}$ in any meaningful way.
We refer to~\eqref{e:iterated-mu} as the \emph{higher product maps} of $\ta$.
We extend the notation to the cases $k=1$ and $k=0$.
In the first case, the only subset in the decomposition is $I$ itself,
and we let the map~\eqref{e:iterated-mu} be the identity of $\ta[I]$.
In the second, $I$ is necessarily $\emptyset$ and the decomposition is
$\emptyset^0$;
we let $\mu_{\emptyset^0}$ be the unit map $\iota_\emptyset$.
Thus, the collection of higher product maps contains among others
the product maps $\mu_{S,T}$ as well as the unit map $\iota_\emptyset$. 

The Cauchy product of two monoids $\ta_1$ and $\ta_2$ is again a monoid.
The product 
\begin{equation}\label{e:cau-mon}
\mu_{S,T}: (\ta_1\bdot\ta_2)[S]\otimes(\ta_1\bdot\ta_2)[T] \to
(\ta_1\bdot\ta_2)[I] 
\end{equation}
is the sum of the following maps
\[
\xymatrix@C+10pt@R+1pc{
(\ta_1[S_1]\otimes\ta_2[S_2])\otimes(\ta_1[T_1]\otimes\ta_2[T_2]) \ar[r]^-{\id\otimes\beta_q\otimes\id} 
\ar@{.>}[rd]& 
(\ta_1[S_1]\otimes\ta_1[T_1])\otimes(\ta_2[S_2]\otimes\ta_2[T_2]) \ar[d]^-{\mu_{S_1,T_1}\otimes\mu_{S_2,T_2}}\\
 & \ta_1[S_1\sqcup T_1] \otimes\ta_2[S_2\sqcup T_2] 
}
\]
over all $S=S_1\sqcup S_2$ and $T=T_1\sqcup T_2$.

If $\ta$ is a monoid, then
$\ta[\emptyset]$ is an algebra with product $\mu_{\emptyset,\emptyset}$
and unit $\iota_\emptyset$.

\subsection{Comonoids}\label{ss:comonoids}

Dually, the structure of a comonoid $\tc$ consists of linear maps
\[
\Delta_{S,T}: \tc[I]\to\tc[S]\otimes\tc[T]
\qand
\epsilon_\emptyset: \tc[\emptyset] \to \Kb
\]
subject to the coassociative and counit axioms, plus naturality.
Given a decomposition $I=S_1\sqcup\cdots\sqcup S_k$, there is a unique map
\begin{equation}\label{e:iterated-delta}
\tc[I] \map{\Delta_{S_1,\ldots,S_k}} \tc[S_1]\otimes\cdots\otimes\tc[S_k]
\end{equation}
obtained by iterating the coproduct maps $\Delta_{S,T}$. 
For $k=1$ this map is defined to be the identity of $\tc[I]$,
and for $k=0$ to be the counit map $\epsilon_\emptyset$.

Comonoids are closed under the Cauchy product. If $\tc$ is a comonoid,
then $\tc[\emptyset]$ is a coalgebra.

\subsection{Commutative monoids}\label{ss:commutative}

A monoid is \emph{$q$-commutative} if the diagram
\begin{equation}\label{e:comm}
\begin{gathered}
\xymatrix@C-15pt{
\ta[S]\otimes\ta[T] \ar[rr]^-{\beta_q} \ar[rd]_{\mu_{S,T}} & &
\ta[T]\otimes\ta[S] \ar[ld]^{\mu_{T,S}}\\
& \ta[I] &
}
\end{gathered}
\end{equation}
commutes for all decompositions $I=S\sqcup T$.
When $q=1$, we speak simply of commutative monoids.
When $q\neq\pm 1$, $q$-commutative monoids are rare:
it follows from~\eqref{e:comm} that if $q^{2\abs{S}\abs{T}}\neq 1$, then $\mu_{S,T}=0$. 

If $\ta$ is $q$-commutative,
then its product $\mu:\ta\bdot\ta\to\ta$ is a morphism of monoids.

A comonoid is $q$-cocommutative if the dual to diagram~\eqref{e:comm} commutes.

\subsection{Modules, ideals, and quotients}\label{ss:modules}

Let $\ta$ be a monoid. 
A \emph{left $\ta$-module} is a species $\tm$ with a structure map 
\[
\lambda:\ta\bdot\tm \to \tm
\] 
which is associative and unital.
When made explicit in terms of the components 
\[
\lambda_{S,T}:\ta[S]\otimes\tm[T]\to\tm[I],
\]
the axioms are similar to~\eqref{e:nat}--\eqref{e:unit}.

The free left $\ta$-module on a species $\tp$ is $\tm:=\ta\bdot\tp$
with structure map
\[
\ta\bdot \tm = \ta\bdot\ta\bdot\tp \map{\mu\bdot\id} \ta\bdot\tp = \tm.
\]

Submonoids, (left, right, and two-sided) ideals, and quotients of $\ta$ can be defined similarly. 
Every monoid $\ta$ has a largest commutative quotient $\ta_{\mathrm{ab}}$.
Depending on the context, we refer to either $\ta_{\mathrm{ab}}$ or the
canonical quotient map $\ta\onto\ta_{\mathrm{ab}}$
as the \emph{abelianization} of $\ta$.

The dual notions exist for comonoids. 
In particular, every comonoid has a largest cocommutative subcomonoid, 
its \emph{coabelianization}.

\subsection{The convolution algebra}\label{ss:convolution}

The set $\Hom_{\Ssk}(\tp,\tq)$ of morphisms of species is a vector
space over $\Kb$ under
\[
(f+g)_I:=f_I+g_I \qand (c\cdot f)_I:=c f_I,
\]
for $f,g\in\Hom_{\Ssk}(\tp,\tq)$ and $c\in\Kb$.

Assume now that $\ta$ is a monoid and $\tc$ is a comonoid. 
The space $\Hom_{\Ssk}(\tc,\ta)$ is then an algebra over $\Kb$ 
under the \emph{convolution product}:
\begin{equation}\label{e:convolution}
(f\conv g)_I := \sum_{I=S\sqcup T} \mu_{S,T}(f_S\otimes g_T)\Delta_{S,T}.
\end{equation}
The unit is the morphism $u$ defined by $u=\iota\epsilon$. Explicitly,
\begin{equation}\label{e:unit-conv}
u_I := \begin{cases}
\iota_\emptyset\epsilon_\emptyset & \text{ if $I=\emptyset$,}\\
0 & \text{ otherwise.}
\end{cases}
\end{equation} 

If $\varphi:\ta\to\tb$ is a morphism of monoids and $\psi:\td\to\tc$ is a morphism of comonoids,
then 
\begin{equation}\label{e:mor-conv}
\Hom_{\Ssk}(\psi,\varphi): \Hom_{\Ssk}(\tc,\ta) \to \Hom_{\Ssk}(\td,\tb), \quad
f \mapsto \varphi f \psi
\end{equation}
is a morphism of algebras.

Let $\ta_i$ be a monoid and $\tc_i$ a comonoid, for $i=1,2$.
If $f_i,g_i:\tc_i\to\ta_i$ are morphisms of species, then
\begin{equation}\label{e:inter-conv}
(f_1\bdot f_2)\conv (g_1\bdot g_2) = (f_1\conv g_1)\bdot(f_2\conv g_2).
\end{equation}
The convolution product on the left is in $\Hom_{\Ssk}(\tc_1\bdot\tc_2,\ta_1\bdot\ta_2)$;
those on the right in $\Hom_{\Ssk}(\tc_i,\ta_i)$, $i=1,2$.
Here we employ the case $q=1$ of the Cauchy product of (co)monoids~\eqref{e:cau-mon}.

\subsection{Bimonoids and Hopf monoids}\label{ss:bim-hopf}

A $q$-bimonoid $\thh$ is at the same time a monoid and a comonoid with the two
structures linked by axioms~\eqref{e:comp}--\eqref{e:inverser} below.
They express the requirement that 
\[
\mu:\thh\bdot\thh\to\thh \qand \iota:\tone_{\Kb}\to\thh
\]
are morphisms of comonoids, or equivalently that
\[
\Delta:\thh\to\thh\bdot\thh \qand \epsilon:\thh\to\tone_{\Kb}
\] 
are morphisms of monoids.

Let $I=S_1\sqcup S_2=T_1\sqcup T_2$ be two decompositions of a finite set
and consider the resulting pairwise intersections:
\[
A:=S_1\cap T_1,\ B:=S_1\cap T_2,\ C:=S_2\cap T_1,\ D:=S_2\cap T_2,
\]
as illustrated below.
\begin{equation}\label{e:4sets}
\setlength{\unitlength}{.9 pt}
\begin{gathered}
\quad
\begin{picture}(100,90)(20,0)
\put(50,40){\oval(100,80)}
\put(0,40){\dashbox{2}(100,0){}}
\put(45,55){$S_1$}
\put(45,15){$S_2$}
\end{picture}
\quad
\begin{picture}(100,90)(10,0)
\put(50,40){\oval(100,80)}
\put(50,0){\dashbox{2}(0,80){}}
\put(20,35){$T_1$}
\put(70,35){$T_2$}
\end{picture}
\quad
\begin{picture}(100,90)(0,0)
\put(50,40){\oval(100,80)}
\put(0,40){\dashbox{2}(100,0){}}
\put(50,0){\dashbox{2}(0,80){}}
\put(20,55){$A$}
\put(70,55){$B$}
\put(20,15){$C$}
\put(70,15){$D$}
\end{picture}
\end{gathered}
\end{equation}
The compatibility axiom for $q$-bimonoids
states that diagrams~\eqref{e:comp}--\eqref{e:inverser} commute.
\begin{equation}\label{e:comp}
\begin{gathered}
\xymatrix@R+2pc@C-5pt{
\thh[A] \otimes \thh[B] \otimes \thh[C] \otimes \thh[D] \ar[rr]^{\id\otimes\beta_q\otimes\id} & &
\thh[A] \otimes \thh[C] \otimes \thh[B] \otimes \thh[D] \ar[d]^{\mu_{A,C}
\otimes \mu_{B,D}}\\
\thh[S_1] \otimes \thh[S_2] \ar[r]_-{\mu_{S_1,S_2}}\ar[u]^{\Delta_{A,B} \otimes
\Delta_{C,D}} & \thh[I] \ar[r]_-{\Delta_{T_1,T_2}} & \thh[T_1] \otimes
\thh[T_2]
}
\end{gathered}
\end{equation}
\begin{align}\label{e:unitr}
&
\begin{gathered}
\xymatrix@R+1pc@C+2pc{
\thh[\emptyset] \otimes \thh[\emptyset] \ar[r]^-{\epsilon_\emptyset
\otimes \epsilon_\emptyset}\ar[d]_{\mu_{\emptyset,\emptyset}} &
\Kb \otimes \Kb\\
\thh[\emptyset] \ar[r]_-{\epsilon_\emptyset} & \Kb \ar@{=}[u]\\
}
\end{gathered}
& & 
\begin{gathered}
\xymatrix@R+1pc@C+2pc{
\Kb \ar@{=}[d]\ar[r]^-{\iota_\emptyset} & \thh[\emptyset]
\ar[d]^{\Delta_{\emptyset,\emptyset}}\\
\Kb \otimes \Kb \ar[r]_-{\iota_\emptyset \otimes \iota_\emptyset} &
\thh[\emptyset] \otimes \thh[\emptyset]
}
\end{gathered}
\end{align}
\begin{equation}\label{e:inverser}
\begin{gathered}
\xymatrix@=2pc{
& \thh[\emptyset] \ar[rd]^{\epsilon_\emptyset}\\
\Kb \ar@{=}[rr]\ar[ru]^-{\iota_\emptyset} & & \Kb
}
\end{gathered}  
\end{equation}

If $\thh$ is a $q$-bimonoid, then $\thh[\emptyset]$ is an ordinary bialgebra
with structure maps $\mu_{\emptyset,\emptyset}$, $\iota_\emptyset$, $\Delta_{\emptyset,\emptyset}$ and $\epsilon_\emptyset$.

\smallskip

A $q$-Hopf monoid is a $q$-bimonoid along with
a morphism of species $\apode \colon \thh \to \thh$ (the \emph{antipode})
which is the inverse of the identity map
in the convolution algebra $\End_{\Ssk}(\thh)$ of Section~\ref{ss:convolution}. 
This requires the existence of an antipode
\[
\apode_\emptyset: \thh[\emptyset] \to \thh[\emptyset]
\]
for the bialgebra $\thh[\emptyset]$, and of a linear map
\[
\apode_I\colon \thh[I]\to\thh[I]
\]
for each nonempty finite set $I$ such that 
\begin{equation}
\label{e:apode}
\sum_{S\sqcup T=I} \mu_{S,T}(\id_S\otimes\apode_T)\Delta_{S,T} = 0
\qand
\sum_{S\sqcup T=I} \mu_{S,T}(\apode_S\otimes\id_T)\Delta_{S,T} = 0.
\end{equation}

\begin{proposition}\label{p:empty}
Let $\thh$ be a $q$-bimonoid.
\begin{enumerate}[(i)]
\item
Suppose $\thh$ is a $q$-Hopf monoid with antipode $\apode$.
Then $\thh[\emptyset]$ is a Hopf algebra with antipode $\apode_{\emptyset}$.

\item
Suppose $\thh[\emptyset]$ is a Hopf algebra and 
let $\apode_0$ denote its antipode. 
Then $\thh$ is a $q$-Hopf monoid with antipode $\apode$ given by
\[
\apode_{\emptyset} := \apode_0
\]
and
\begin{equation}\label{e:takeuchi-general}
\apode_I := \!\!\sum_{\substack{T_1\sqcup\dots \sqcup T_k = I\\ T_i\neq\emptyset\ k\geq 1}} \!\!\! (-1)^{k} \mu_{\emptyset,T_1,\emptyset,\dots,\emptyset,T_k,\emptyset}\bigl(\apode_0\otimes\id_{T_1}\!\otimes\apode_0\otimes\dots\otimes\apode_0\otimes\id_{T_k}\!\otimes\apode_0\bigr)
\Delta_{\emptyset,T_1,\emptyset,\dots,\emptyset,T_k,\emptyset}
\end{equation}
%{\multlinegap0pt
%\begin{multline}\label{e:takeuchi-general}
%\apode_I := \smash[b]{\sum_{\substack{T_1\sqcup\dots \sqcup T_k = I\\ T_i\neq\emptyset
%\ k\geq 1}}} (-1)^{k} \mu_{\emptyset,T_1,\emptyset,\dots,\emptyset,T_k,\emptyset}\bigl(\apode_0\otimes\id_{T_1}\otimes\apode_0\otimes\dots\otimes\apode_0\otimes\id_{T_k}\otimes\apode_0\bigr)\\
%\Delta_{\emptyset,T_1,\emptyset,\dots,\emptyset,T_k,\emptyset}
%\end{multline}}%
for each nonempty finite set~$I$.
\end{enumerate}
\end{proposition}

The sum is over all compositions of $I$. 
A proof of Proposition~\ref{p:empty} is given in~\cite[Proposition~8.10]{AguMah:2010}, when $q=1$.
The same argument yields the general case.

Thus, a $q$-Hopf monoid $\thh$ is equivalent to a $q$-bimonoid $\thh$ 
for which $\thh[\emptyset]$ is a Hopf algebra.

\begin{proposition}\label{p:apode-rev}
Let $\thh$ be a $q$-Hopf monoid with antipode $\apode$ and $I=S\sqcup T$. Then
the diagram
\begin{equation}\label{e:apode-rev-prod}
\begin{gathered}
\xymatrix@C+10pt{
\thh[S]\otimes\thh[T] \ar[r]^-{\apode_S\otimes\apode_T} \ar[d]_-{\mu_{S,T}} 
& \thh[S]\otimes\thh[T] \ar[r]^-{\beta_q}
& \thh[T]\otimes\thh[S]\ar[d]^{\mu_{T,S}}\\
\thh[I] \ar[rr]_-{\apode_I} & & \thh[I]
}
\end{gathered}
\end{equation}
commutes.
\end{proposition}

Thus, the antipode reverses products (the reversal involves the braiding $\beta_q$).
Similarly, it reverses coproducts.
These are general results for Hopf monoids in braided monoidal categories 
(see for instance~\cite[Proposition~1.22]{AguMah:2010}) and
hence they apply to $q$-Hopf monoids.
Other such results~\cite[Section~1.2]{AguMah:2010} yield the following.
\begin{itemize}
\item
If $\thh$ is a $q$-(co)commutative $q$-Hopf monoid, 
then its antipode is an involution: $\apode^2=\id$.

\item
Let $\thh$ and $\tk$ be $q$-Hopf monoids. 
A morphism of $q$-bimonoids $\thh\to\tk$
necessarily commutes with the antipodes, 
and is thus a morphism of $q$-Hopf monoids.

\item 
Let $\thh$ be a $q$-Hopf monoid and $\ta$ a monoid. 
If $f:\thh\to\ta$ is a morphism of monoids,
then it is invertible under convolution and its inverse is $f\apode$.
If $\ta$ is commutative, then
the set of morphisms of monoids from $\thh$ to $\ta$ is a group under convolution.

\item 
The dual statement for comonoid morphisms $f:\tc\to\thh$ holds.
The inverse of $f$ is $\apode f$.
\end{itemize}
A result of Schauenburg~\cite[Corollary~5]{Sch:1998} implies the following,
confirming the earlier observation that $q$-(co)commutativity is rare when $q\neq \pm 1$.
\begin{itemize}
\item 
If $\thh$ is a $q$-(co)commutative $q$-Hopf monoid and
$q^{2\abs{I}}\neq 1$, then $\thh[I]=0$.
\end{itemize}
%Schauenburg proves \beta_{\thh,\thh}^2=\id. 
%Apply this to the component (I,*) of \beta, 
%where * is any singleton disjoint from I. 
%Other components do not yield additional information.

\subsection{Lie monoids}\label{ss:liemon}

We consider Lie monoids in the symmetric monoidal category $(\Ssk,\bdot,\beta)$.
A Lie monoid structure on a species $\tg$ consists of a morphism of species 
\[
\tg\bdot\tg\to\tg
\]
subject to antisymmetry and the Jacobi identity.
The morphism consists of a collection of linear maps
\[
\tg[S]\otimes\tg[T] \to \tg[I] 
\qquad x\otimes y \mapsto [x,y]_{S,T}
\]
one for each finite set $I$ and each decomposition $I=S\sqcup T$,
satisfying the naturality condition~\eqref{e:nat} 
(with $[x,y]_{S,T}$ replacing $\mu_{S,T}(x\otimes y)$). 
We refer to $[x,y]_{S,T}$ as the \emph{Lie bracket} of $x\in\tg[S]$ and $y\in\tg[T]$.

The antisymmetry relation states that
\begin{equation}\label{e:antisym}
[x,y]_{S,T}+[y,x]_{T,S}=0
\end{equation}
for any decomposition $I=S\sqcup T$, $x\in\tg[S]$, and $y\in\tg[T]$. 

The Jacobi identity states that
\begin{equation}\label{e:jacobi}
 [[x,y]_{R,S},z]_{R\sqcup S,T}+[[z,x]_{T,R},y]_{T\sqcup R,S}+[[y,z]_{S,T},x]_{S\sqcup T,R}=0
\end{equation}
for any decomposition $I=R\sqcup S\sqcup T$, 
$x\in\tg[R]$, $y\in\tg[S]$, and $z\in\tg[T]$.

\smallskip

Every monoid $(\ta,\mu)$ gives rise to a Lie monoid $(\ta,[\,,\,])$
with Lie bracket defined by
\begin{equation}\label{e:underlying-lie}
[x,y]_{S,T} := \mu_{S,T}(x\otimes y) - \mu_{T,S}(y\otimes x).
\end{equation}
We refer to~\eqref{e:underlying-lie} as the \emph{commutator bracket}.
This does not require the monoid $\ta$ to possess a unit.
If $\ta$ is commutative, then the commutator bracket is zero.

\subsection{Algebras as monoids}\label{ss:hopfalg}

Let $V$ be a vector space over $\Kb$ and $\tone_V$ the species defined in~\eqref{e:tone}.
If $V$ is an ordinary algebra with unit $\iota:\Kb\to V$
and product $\mu:V\otimes V\to V$, then $\tone_V$
is a monoid with $\iota_\emptyset:=\iota$, $\mu_{\emptyset,\emptyset}:=\mu$ and all other $\mu_{S,T}=0$.
Moreover, $\tone_V[\emptyset]=V$ as algebras.

In the same manner, if $V$ is a coalgebra, bialgebra, Hopf algebra, or Lie algebra,
then $\tone_V$ is a comonoid, bimonoid, Hopf monoid, or Lie monoid.

There is another way to associate a species to a vector space $V$. 
Define $\wU_V$ by
\begin{equation}\label{e:uniform}
\wU_V[I] := V
\end{equation}
for every finite set $I$. 
On a bijection $\sigma:I\to J$, $\wU_V[\sigma]$ is the identity of 
$V$. If $V$ is an algebra over $\Kb$, then $\wU_V$ is a monoid with 
$\iota_\emptyset:=\iota$ and $\mu_{S,T}:=\mu$ for all $I=S\sqcup T$.
If $V$ is a coalgebra, bialgebra, Hopf algebra, or Lie algebra,
then $\wU_V$ is a comonoid, bimonoid, Hopf monoid, or Lie monoid.
When $V$ is a Hopf algebra with antipode $\apode$, the antipode of $\wU_V$ is simply
$\apode_I := (-1)^{\abs{I}} \apode$. This follows from either~\eqref{e:apode}
or~\eqref{e:takeuchi-general}.

\section{The Hadamard product}\label{s:hadamard}

Multiplying species term by term yields the Hadamard product.
The possibility of building Hopf monoids by means of this operation is
an important feature of the category of species. We study this and
related constructions in this section.

\subsection{Species under Hadamard product}\label{ss:hadamard}

The \emph{Hadamard product} of two species $\tp$ and $\tq$ is the species
$\tp\times\tq$ defined on a finite set $I$ by
\[
(\tp\times\tq)[I] := \tp[I]\otimes\tq[I],
\]
and on bijections similarly. 
This operation turns $\Ssk$ into a symmetric monoidal category. 
The unit object is the \emph{exponential} species $\wE$ defined by 
\begin{equation}\label{e:expsp}
\wE[I]:=\Kb
\end{equation}
for all $I$. The symmetry is simply
\[
 \tp[I] \otimes \tq[I] \to \tq[I] \otimes \tp[I],
\qquad
x \otimes y \mapsto y \otimes x.
\]

(Comparing with~\eqref{e:uniform}, we see that $\wE=\wU_{\Kb}$. 
In particular, $\wE$ is a Hopf monoid. 
Its structure is further studied in Section~\ref{ss:exp}.) 

The Hadamard product of two monoids $\ta$ and $\tb$ is again a monoid.
The product of $\ta\times\tb$ is defined by
\[
%\xymatrix@C+9pt{
\xymatrix@C-17pt{
(\ta\times\tb)[S]\otimes(\ta\times\tb)[T] \ar@{.>}[rrrr]^-{\mu_{S,T}} \ar@{=}[d] & & & &
(\ta\times\tb)[I] \ar@{=}[d]\\
(\ta[S]\otimes\tb[S])\otimes(\ta[T]\otimes\tb[T]) \ar[r]_-{\cong} & 
(\ta[S]\otimes\ta[T])\otimes(\tb[S]\otimes\tb[T]) \ar[rrr]_-{\mu_{S,T}\otimes\mu_{S,T}} & & &
\ta[I] \otimes\tb[I],
}
\]
where the first map on the bottom simply switches the middle two tensor factors.
If $\ta$ and $\tb$ are commutative, then so is $\ta\times\tb$.
Similar statements hold for comonoids.

\begin{proposition}\label{p:hadamard}
Let $p,q\in\Kb$ be arbitrary scalars.
If $\thh$ is a $p$-bimonoid and $\tk$ is a $q$-bimonoid,
then $\thh\times\tk$ is a $pq$-bimonoid.
\end{proposition}

In particular, if $\thh$ and $\tk$ are bimonoids ($p=q=1$), then so is $\thh\times\tk$.

A proof of Proposition~\ref{p:hadamard} is given in~\cite[Corollary~9.6]{AguMah:2010}.
The analogous statement for Hopf monoids holds as well, 
in view of item (ii) in Proposition~\ref{p:empty} and 
the fact that $(\thh\times\tk)[\emptyset]$ is the tensor
product of the Hopf algebras $\thh[\emptyset]$ and $\tk[\emptyset]$. 
However, there is no simple formula expressing the antipode of $\thh\times\tk$ 
in terms of those of $\thh$ and $\tk$.

\subsection{Internal Hom}\label{ss:iH}

Given species $\tp$ and $\tq$, let $\iH(\tp,\tq)$ be the
species defined by
\[
\iH(\tp,\tq)[I] := \Hom_{\Kb}(\tp[I],\tq[I]).
\]
The latter is the space of all linear maps from the space $\tp[I]$
to the space $\tq[I]$. If $\sigma:I\to J$ is a bijection and $f\in\iH(\tp,\tq)[I]$, 
then 
\[
\iH(\tp,\tq)[\sigma](f)\in\iH(\tp,\tq)[J]
\]
is defined as the composition
\begin{equation}\label{e:iHact}
\tp[J] \map{\tp[\sigma^{-1}]} \tp[I] \map{f} \tq[I] \map{\tq[\sigma]} \tq[J].
\end{equation}

Note that there is a canonical isomorphism
\[
\tq \cong \iH(\wE,\tq).
\]

For species $\tp$, $\tq$ and $\tr$, there is a natural isomorphism
\begin{equation}\label{e:iH}
\Hom_{\Ssk}(\tp\times\tq,\tr) \cong \Hom_{\Ssk}\bigl(\tp,\iH(\tq,\tr)\bigr).
\end{equation}
This says that the functor $\iH$ is the \emph{internal Hom} in the 
symmetric monoidal category $(\Ssk,\times)$ of species under Hadamard product.
There is a canonical morphism
\begin{equation}\label{e:iHHad}
\iH(\tp_1,\tq_1)\times\iH(\tp_2,\tq_2) \to \iH(\tp_1\times\tp_2,\tq_1\times\tq_2);
\end{equation}
it is an isomorphism if $\tp_1$ and $\tp_2$ are finite-dimensional.

We let
\[
\iE(\tp) := \iH(\tp,\tp).
\]

Let $\tc$ be a comonoid and $\ta$ a monoid. 
Then $\iH(\tc,\ta)$ is a monoid as follows. 
Given $f\in\iH(\tc,\ta)[S]$ and $g\in\iH(\tc,\ta)[T]$, their product 
\[
\mu_{S,T}(f\otimes g)\in\iH(\tc,\ta)[I]
\] 
is the composition
\begin{equation}\label{e:iHprod}
\tc[I] \map{\Delta_{S,T}} \tc[S]\otimes\tc[T]
\map{f\otimes g} \ta[S]\otimes\ta[T] \map{\mu_{S,T}} \ta[I].
\end{equation}
The unit map $\iota_\emptyset:\Kb\to\iH(\tc,\ta)[\emptyset]$ sends $1\in \Kb$
to the composition 
\begin{equation*}%\label{e:iHunit}
\tc[\emptyset] \map{\epsilon_\emptyset} \Kb \map{\iota_\emptyset} \tc[\emptyset].
\end{equation*}
Associativity and unitality for $\iH(\tc,\ta)$ follow from
(co)associativity and (co)unitality of $\ta$ (and $\tc$).

There is a connection between~\eqref{e:iHprod} and the convolution 
product~\eqref{e:convolution}; this is explained in Section~\ref{ss:seriH}.

Given $f\in\iH(\ta,\tc)[I]$, let $\Tilde{\Delta}(f)$ be the composition
\begin{equation}\label{e:iHcoprod}
\ta[S]\otimes\ta[T] \map{\mu_{S,T}} \ta[I] \map{f} \tc[I] \map{\Delta_{S,T}} 
\tc[S]\otimes\tc[T].
\end{equation}
Assume now that $\ta$ is finite-dimensional. 
We then have a canonical isomorphism
\[
\Hom_{\Kb}(\ta[S],\tc[S])\otimes \Hom_{\Kb}(\ta[T],\tc[T]) \map{\cong}
\Hom_{\Kb}(\ta[S]\otimes\ta[T],\tc[S]\otimes\tc[T]).
\]
Let 
\[
\Delta_{S,T}(f)\in\iH(\ta,\tc)[S]\otimes\iH(\ta,\tc)[T]
\] 
be the preimage of $\Tilde{\Delta}(f)$ under this isomorphism. 
Also, if $I=\emptyset$, let 
$\epsilon_\emptyset(f)\in\Kb$ be the preimage of the composition
\[
\Kb \map{\iota_\emptyset} \ta[\emptyset] \map{f} \tc[\emptyset] \map{\epsilon_\emptyset} \Kb
\]
under the canonical isomorphism
\[
\Kb \map{\cong} \Hom_{\Kb}(\Kb,\Kb).
\]
With these definitions, $\iH(\ta,\tc)$ is a comonoid.

Now let $\thh$ be a $p$-bimonoid and $\tk$ a $q$-bimonoid. 
Suppose $\thh$ is finite-dimensional. 
Combining the preceding constructions 
yields a $pq$-bimonoid structure on $\iH(\thh,\tk)$. 
In particular, for any $p$-bimonoid $\thh$, we obtain a $p^2$-bimonoid
structure on $\iE(\thh)$.

The analogous statements for Hopf monoids hold.

\subsection{Duality}\label{ss:dual}

The \emph{dual} of a species $\tp$ is the species $\tp^*$ defined by
\[
\tp^*[I] := \tp[I]^*.
\]
Equivalently, 
\[
\tp^*=\iH(\tp,\wE).
\]
Since $\wE$ is a Hopf monoid, the considerations of Section~\ref{ss:iH} apply.
Thus, the dual of a comonoid $\tc$ is a monoid with product
\[
\tc^*[S]\otimes\tc^*[T] = \tc[S]^*\otimes\tc[T]^* \to \bigl(\tc[S]\otimes\tc[T]\bigr)^* \map{(\Delta_{S,T})^*} \tc[I]^* = \tc^*[I],
\]
where the first arrow is canonical.
The dual of a finite-dimensional monoid ($q$-bimonoid, $q$-Hopf monoid)
is a comonoid ($q$-bimonoid, $q$-Hopf monoid).

A $q$-bimonoid (or $q$-Hopf monoid) $\thh$ is called \emph{self-dual}
if $\thh\cong\thh^*$ as $q$-bimonoids. In general, such an isomorphism is not unique.

 There are canonical morphisms of species
\[
\tp^* \times \tq^* \to (\tp\times\tq)^* 
\qqand
\tp^*\times\tq \to \iH(\tp,\tq)
\]
which are isomorphisms if either $\tp$ or $\tq$ is finite-dimensional. 
These are special cases of~\eqref{e:iHHad}.
These maps preserve the (co)monoid structures discussed in
Sections~\ref{ss:hadamard} and~\ref{ss:iH} (when present).
In particular, if $\thh$ is a $p$-bimonoid and $\tk$ a $q$-bimonoid,
the map
\[
\thh^*\times\tk \to \iH(\thh,\tk)
\]
is a morphism of $pq$-bimonoids. 
If $\thh$ is finite-dimensional, the isomorphism
\[
\thh^*\times\thh \to \iE(\thh)
\]
implies that the $p^2$-bimonoid $\iE(\thh)$ is self-dual.

\section{Set species and set-theoretic Hopf monoids}\label{s:set}

Many naturally-occurring Hopf monoids have a canonical basis indexed by combinatorial objects. 
Let $\thh$ be such a Hopf monoid and
suppose that the product and coproduct maps of $\thh$
preserve the canonical basis.
In this situation, the structure of $\thh$ can be described in a purely set-theoretic manner,
leading to the notion of a set-theoretic Hopf monoid.
We discuss these objects and
explain how they relate to Hopf monoids via linearization.

\subsection{Set species}\label{ss:set}

Let $\Set$ denote the category whose objects are arbitrary sets
and whose morphisms are arbitrary functions.
A \emph{set species} is a functor
\[
\Fset \longrightarrow \Set,
\]
where the category $\Fset$ is as in Section~\ref{ss:species}.
A morphism between set species is a natural transformation.
Let $\Ss$ denote the category of set species.

A set species $\rP$ is \emph{finite} 
if the set $\rP[I]$ is finite for each finite set $I$.
It is \emph{positive} if $\rP[\emptyset]=\emptyset$.

The Cauchy product of set species $\rP$ and $\rQ$ 
is the set species $\rP \bdot \rQ$ defined by
\begin{align} \label{e:cau-set} 
(\rP \bdot \rQ)[I] := \coprod_{I = S \sqcup T} \rP[S] \times \rQ[T],
\end{align}
 where $\times$ denotes the cartesian product of sets.
The Cauchy product turns $\Ss$ into a monoidal category. 
The unit object is the set species $1$
given by
\begin{align*}
1[I] &:= \begin{cases}
\{\emptyset\} & \text{if $I$ is empty,}\\
\emptyset & \text{otherwise.}
\end{cases}
\end{align*}
The set $\{\emptyset\}$ is a singleton. 

The natural transformation $\rP \bdot \rQ \to \rQ \bdot \rP$
obtained by interchanging the factors in the cartesian product is a symmetry.

\subsection{Set-theoretic monoids and comonoids}\label{ss:set-mon-com}

Monoids in $(\Ss,\bdot)$ can be described in terms similar to those in
Section~\ref{ss:monoids}: the structure involves a collection of maps
\[
\mu_{S,T} : \rP[S]\times\rP[T] \to \rP[I],
\]
one for each decomposition $I=S\sqcup T$, and a map 
$\iota_\emptyset:\{\emptyset\}\to \rP[\emptyset]$, 
subject to axioms analogous to~\eqref{e:assoc} and~\eqref{e:unit}. 
We refer to these objects as \emph{set-theoretic monoids}.

Given $x\in \rP[S]$
and $y\in \rP[T]$, let
\begin{equation}\label{e:set-prod}
x\cdot y\in \rP[I]
\end{equation}
denote the image of $(x,y)$ under $\mu_{S,T}$.
Also, let $e\in \rP[\emptyset]$ denote the image of $\emptyset$ 
under $\iota_\emptyset$.
The axioms for a set-theoretic monoid $\rP$
then acquire the familiar form
\begin{equation}\label{e:set-prod-asso}
x\cdot(y\cdot z)=(x\cdot y)\cdot z
\end{equation}
for all decompositions $I=R\sqcup S\sqcup T$ and $x\in \rP[R]$, $y\in \rP[S]$,
$z\in \rP[T]$, and
\begin{equation}\label{e:set-prod-unit}
x\cdot e= x =e\cdot x
\end{equation}
for all $x\in \rP[I]$. In particular, $\rP[\emptyset]$ is an ordinary monoid.

A set-theoretic monoid $\rP$ is commutative if
\[
x\cdot y = y\cdot x
\]
for all $I=S\sqcup T$, $x\in\rP[S]$, and $y\in\rP[T]$.

The situation for comonoids is different. 
The existence of the counit forces a comonoid 
$\rQ$ in $(\Ss,\bdot)$ to be concentrated on the empty set. Indeed, a morphism 
$\epsilon:\rQ\to 1$ entails maps $\epsilon_I:\rQ[I]\to 1[I]$, and hence we must have
$\rQ[I]=\emptyset$ for all nonempty $I$.
 
There is, nevertheless, a meaningful notion of \emph{set-theoretic comonoid}.
It consists, by definition, of a set species $\rQ$ together with a collection of maps
\[
\Delta_{S,T}: \rQ[I]\to \rQ[S]\times \rQ[T],
\]
one for each $I=S\sqcup T$, subject to the coassociative and counit axioms.
The former states that for each decomposition $I=R\sqcup S\sqcup T$, the diagram
\begin{gather}\label{e:set-coassoc}
\begin{gathered}
\xymatrix@R+1pc@C+25pt{
\rQ[R]\times\rQ[S]\times\rQ[T]\
 & \rQ[R]\times\rQ[S\sqcup T] \ar[l]_-{\id\times\Delta_{R,S}}
\\
\rQ[R\sqcup S]\times \rQ[T] \ar[u]^{\Delta_{R,S}\times\id} 
& \rQ[I] \ar[l]^-{\Delta_{R\sqcup S,T}} \ar[u]_{\Delta_{R,S\sqcup T}}
}
\end{gathered}
\end{gather}
commutes. The latter states that for each finite set $I$, the diagrams
\begin{align}\label{e:set-counit}
&
\begin{gathered}
\xymatrix@=3pc{
\rQ[I]\ar@{=}[rd] \ar[r]^-{\Delta_{\emptyset,I}} & \rQ[\emptyset]\times\rQ[I] 
\ar[d]^{\epsilon_\emptyset\times\id_I}\\
& \{\emptyset\}\times \rQ[I] 
}
\end{gathered}
& &
\begin{gathered}
\xymatrix@=3pc{
\rQ[I]\times \rQ[\emptyset] \ar[d]_{\id_I\times\epsilon_\emptyset}
& \rQ[I]\ar@{=}[ld] \ar[l]_-{\Delta_{I,\emptyset}}\\
\rQ[I] \times \{\emptyset\} &
}
\end{gathered}
\end{align}
commute, where $\epsilon_\emptyset$ denotes the unique map to the singleton 
$\{\emptyset\}$.

A set-theoretic comonoid $\rQ$ is cocommutative if for each $I=S\sqcup T$,
the diagram
\[
\xymatrix@C-15pt{
\rQ[S]\times\rQ[T] \ar[rr]^-{\cong} & & \rQ[T]\times\rQ[S]\\
& \rQ[I] \ar[ul]^{\Delta_{S,T}} \ar[ur]_{\Delta_{T,S}} &
}
\]
commutes.

Given $x\in \rQ[I]$, let
\begin{equation}\label{e:rest-cont}
(x|_S, x/_S)
\end{equation}
denote the image of $x$ under $\Delta_{S,T}$. Thus, $x|_S\in\rQ[S]$ while
$x/_S\in\rQ[T]$. We think intuitively of $x\mapsto x|_S$ as restricting the structure $x$ from $I$
to $S$, and of $x\mapsto x/_S$ as contracting or moding out $S$ from $x$ (so that
the result is a structure on $T$).

The axioms for a set-theoretic comonoid may then be reformulated as follows.
Coassociativity states that
\begin{equation}\label{e:rest-cont-asso}
(x|_{R\sqcup S})|_R=x|_R,
\quad
(x|_{R\sqcup S})/_R=(x/_R)|_S,
\quad
x/_{R\sqcup S}=(x/_R)/_S,
\end{equation}
for any decomposition $I=R\sqcup S\sqcup T$ and $x\in \rQ[I]$.
Counitality states that
\begin{equation}\label{e:rest-cont-counit}
x|_I=x=x/_{\emptyset}
\end{equation}
for any $x\in \rQ[I]$.

In particular, it follows from~\eqref{e:rest-cont-counit} that 
for $x\in \rQ[\emptyset]$ we have $\Delta_{\emptyset,\emptyset}(x)=(x,x)$.

A set-theoretic comonoid $\rQ$ is cocommutative if and only if
\[
x|_S = x/_T
\]
for every $I=S\sqcup T$ and $x\in\rQ[I]$.

\begin{proposition}\label{p:presheaf}
There is an equivalence between the category of cocommutative 
set-theoretic comonoids and
the category of presheafs on the category of finite sets and injections.
\end{proposition}

A \emph{presheaf} is a contravariant functor to $\Set$. A set-theoretic comonoid $\rQ$
becomes a presheaf by means of the restrictions
\[
\rQ[V] \to \rQ[U], \quad x \mapsto x|_U
\]
for $U\subseteq V$ (together with the action of $\rQ$ on bijections).
Proposition~\ref{p:presheaf} and related results are given in~\cite[Section~8.7.8]{AguMah:2010}. It originates in work of Schmitt~\cite[Section~3]{Sch:1993}.

\subsection{Set-theoretic bimonoids and Hopf monoids}\label{ss:set-bimon}

A \emph{set-theoretic bimonoid} $\rH$ is, by definition, a set-theoretic monoid
and comonoid such that the diagram
\begin{equation}\label{e:set-comp}
\begin{gathered}
\xymatrix@R+2pc@C-5pt{
\rH[A] \times \rH[B] \times \rH[C] \times \rH[D] \ar[rr]^{\cong} & &
\rH[A] \times \rH[C] \times \rH[B] \times \rH[D] \ar[d]^{\mu_{A,C}
\times \mu_{B,D}}\\
\rH[S_1] \times \rH[S_2] \ar[r]_-{\mu_{S_1,S_2}}\ar[u]^{\Delta_{A,B} \times
\Delta_{C,D}} & \rH[I] \ar[r]_-{\Delta_{T_1,T_2}} & \rH[T_1] \times
\rH[T_2]
}
\end{gathered}
\end{equation}
commutes for every choice of sets as in~\eqref{e:4sets}. The top map simply
interchanges the two middle terms. The condition may be reformulated as follows:
\begin{equation}\label{e:rest-cont-prod}
x|_A\cdot y|_C= (x\cdot y)|_{T_1}
\qand
x/_A\cdot y/_C = (x\cdot y)/_{T_1}
\end{equation}
for all $x\in \rH[S_1]$ and $y\in \rH[S_2]$. Note that the set-theoretic analogues 
of~\eqref{e:unitr} and~\eqref{e:inverser} hold automatically.

A \emph{set-theoretic Hopf monoid} $\rH$ is a set-theoretic bimonoid such that the
monoid $\rH[\emptyset]$ is a group. Its antipode is the map
\begin{equation}\label{e:set-ant}
\apode_\emptyset: \rH[\emptyset]\to\rH[\emptyset], \quad x \mapsto x^{-1}.
\end{equation}

\subsection{The Hadamard product}\label{ss:set-hadamard}

The Hadamard product of two set species $\rP$ and $\rQ$ is the set species
$\rP\times\rQ$ defined by
\[
(\rP\times\rQ)[I]=\rP[I]\times\rQ[I],
\]
where on the right-hand side $\times$ denotes the cartesian product of sets. 

Set-theoretic monoids are preserved under Hadamard products.
The same is true of set-theoretic comonoids, bimonoids and Hopf monoids.

% Linearization preserves Hadamard products:
% \[
% \Kb(\rP\times\rQ) \cong \Kb\rP \times \Kb\rQ.
% \]

\subsection{Linearization}\label{ss:lin}

Given a set $X$, let 
\[
\Kb X \qand \Kb^X
\]
denote the vector space with basis $X$, and the vector space of scalar functions on $X$,
respectively. We have $\Kb^X = (\Kb X)^*$.

\begin{convention}\label{con:bases}
We use 
\[
\{\BH_x \mid x\in X\} \qand \{\BM_x \mid x\in X\}
\]
to denote, respectively, the canonical basis of $\Kb X$ and the basis of $\Kb^X$ of characteristic functions ($\BM_x(y) := \delta(x,y)$). 
Thus, $\{\BH\}$ and $\{\BM\}$ are dual bases.
\end{convention}

The \emph{linearization functors}
\[
\Set \longrightarrow \Veck, \quad X \mapsto \Kb X 
\qand X\mapsto \Kb^X
\]
are covariant and contravariant, respectively. 
Composing a set species $\rP$ with the linearization functors we obtain
vector species $\Kb\rP$ and $\Kb^\rP=(\Kb\rP)^*$.

Linearization transforms cartesian products into tensor products.
Therefore, if $\rP$ is a set-theoretic monoid, then the species $\Kb\rP$ is a monoid
whose product is the linear map 
\[
\mu_{S,T}: \Kb\rP[S]\otimes \Kb\rP[T] \to \Kb\rP[I]
\quad \text{such that} \quad
\mu_{S,T}(\BH_x\otimes\BH_y) = \BH_{x\cdot y}
\]
for every $x\in\rP[S]$ and $y\in\rP[T]$.
If in addition $\rP$ is finite, then the species $\Kb^\rP$ is a comonoid
whose coproduct is the linear map 
\[
\Delta_{S,T}: \Kb^\rP[I] \to \Kb^\rP[S]\otimes \Kb^\rP[T] 
\quad \text{such that} \quad
\Delta_{S,T}(\BM_z) = \sum_{\substack{x\in\rP[S],\, y\in\rP[T]\\ x\cdot y =z}} \BM_{x} \otimes \BM_{y}
\]
for every $z\in\rP[I]$.
If $\rQ$ is a set-theoretic comonoid, then $\Kb\rQ$ is a comonoid
whose coproduct is the linear map 
\[
\Delta_{S,T}: \Kb\rQ[I] \to \Kb\rQ[S]\otimes \Kb\rQ[T] 
\quad \text{such that} \quad
\Delta_{S,T}(\BH_z) = \BH_{z|_S} \otimes \BH_{z/_S},
\]
and $\Kb^\rQ$ is a monoid whose product is the linear map
\[
\mu_{S,T}: \Kb^\rQ[S]\otimes \Kb^\rQ[T] \to \Kb^\rQ[I]
\quad \text{such that} \quad
\mu_{S,T}(\BM_x\otimes\BM_y) = \sum_{\substack{z\in \rQ[I]\\ z|_S=x,\, z/_S=y}}
 \BM_{z}.
\]
% The sum may be infinite but is a well-defined function in $\Kb^\rQ[I]$.

Similar remarks apply to bimonoids and Hopf monoids. 
If $\rH$ is a Hopf monoid in set species, then $\Kb\rH[\emptyset]$ is a group algebra,
and the antipode of $\Kb\rH$ exists by item (ii) in Proposition~\ref{p:empty}.
(Co)commutativity and Hadamard products are also preserved under linearization.

\section{Connected Hopf monoids}\label{s:connected}

The study of a Hopf monoid $\thh$ necessarily involves
that of the Hopf algebra $\thh[\emptyset]$.
A special yet nontrivial class of Hopf monoids consists of those 
for which this Hopf algebra is isomorphic to $\Kb$, the simplest Hopf algebra.
These are the connected Hopf monoids.

\subsection{Connected species and Hopf monoids}\label{ss:conn-hopf}

A species $\tp$ is \emph{connected} if 
\[
\dim_{\Kb} \tp[\emptyset]=1.
\]
A set species $\rP$ is connected if $\rP[\emptyset]$ is a singleton. 
In this case, the species $\Kb\rP$ is connected.

Connectedess is preserved under duality and Hadamard products.

In a connected monoid, the map $\iota_\emptyset$ is an isomorphism $\Kb\cong\ta[\emptyset]$,
and by~\eqref{e:unit} the resulting maps
\[
\ta[I] \cong \ta[I]\otimes\ta[\emptyset] \map{\mu_{I,\emptyset}} \ta[I]
\qand
\ta[I] \cong \ta[\emptyset]\otimes\ta[I] \map{\mu_{\emptyset,I}} \ta[I]
\]
are identities. Thus, to provide a monoid structure on a connected species
it suffices to specify the maps $\mu_{S,T}$ when $S$ and $T$ are nonempty.

A dual remark applies to connected comonoids $\tc$. It can be expressed as follows:
\begin{equation}\label{e:conn-unit}
\Delta_{I,\emptyset}(x) = x\otimes 1 
\qand
\Delta_{\emptyset,I}(x) = 1\otimes x,
\end{equation}
where $1$ denotes the element of $\tc[\emptyset]$ such that $\epsilon_\emptyset(1)=1$.

\begin{proposition}\label{p:conn-hopf}
A connected $q$-bimonoid is necessarily a $q$-Hopf monoid.
\end{proposition}

This follows from item (ii) in Proposition~\ref{p:empty}. 
The special consideration of formulas~\eqref{e:apode} and~\eqref{e:takeuchi-general} in the
connected situation is the subject of Sections~\ref{ss:milnor-moore} and~\ref{ss:takeuchi} below.

\begin{proposition}\label{p:hopf-split}
Let $\thh$ be a connected $q$-bimonoid. 
Let $I=S\sqcup T$, $x\in\thh[S]$ and $y\in\thh[T]$. Then
\begin{equation}\label{e:hopf-split}
\Delta_{S,T}\mu_{S,T}(x\otimes y) = x\otimes y
\qand
\Delta_{T,S}\mu_{S,T}(x\otimes y) = q^{\abs{S} \abs{T}} y\otimes x.
\end{equation}
\end{proposition}
These formulas can be deduced from the compatibility axiom~\eqref{e:comp} by making
appropriate choices of the subsets $S_i$ and $T_i$, together with~\eqref{e:conn-unit}.
These and related results are given in~\cite[Corollary~8.38]{AguMah:2010} (when $q=1$).

Proposition~\ref{p:hopf-split} fails for nonconnected bimonoids $\thh$.
For concrete examples, let $H$ be a nontrivial bialgebra and let $\thh=\tone_H$
as in Section~\ref{ss:hopfalg}, or consider the bimonoid $\tSigh$ of Section~\ref{ss:dec-bimonoid}.

\subsection{Milnor and Moore's antipode formulas}\label{ss:milnor-moore}

Suppose $\thh$ is a connected $q$-bimonoid. 
Define maps $\apode_I$ and $\apode'_I$ by induction on the finite set $I$ as follows.
Let
\[
\apode_{\emptyset} = \apode'_{\emptyset}
\]
be the identity of $\thh[\emptyset]=\Kb$, and for nonempty $I$,
\begin{align}\label{e:mm-antipode-r}
\apode_I &:= -\sum_{\substack{S\sqcup T=I\\ T\neq I}}
\mu_{S,T}(\id_S\otimes\apode_T)\Delta_{S,T},\\
\label{e:mm-antipode-l}
\apode'_I &:= -\sum_{\substack{S\sqcup T=I\\ S\neq I}}
\mu_{S,T}(\apode'_S\otimes\id_T)\Delta_{S,T}.
\end{align}

\begin{proposition}\label{p:mm-antipode}
We have that
\[
\apode=\apode'
\]
and this morphism is the antipode of $\thh$.
\end{proposition}

This follows from~\eqref{e:apode}. 
There is an analogous recursive expression for the antipode of a connected
Hopf algebra, due to Milnor and Moore~\cite[Proposition~8.2]{MilMoo:1965}.

\subsection{Takeuchi's antipode formula}\label{ss:takeuchi}

Takeuchi's formula expresses the antipode in terms of the 
higher products and coproducts:

\begin{proposition}\label{p:antipode-r}
Let $\thh$ be a connected $q$-Hopf monoid with antipode $\apode$.
Then
\begin{equation}\label{e:antipode-r}
\apode_I = \sum_{F\vDash I} (-1)^{\len(F)} \mu_{F}\Delta_{F}
\end{equation}
for any nonempty finite set $I$.
\end{proposition}

The sum runs over all compositions $F$ of $I$ and the
maps $\mu_{F}$ and $\Delta_{F}$
are as in~\eqref{e:iterated-mu} and~\eqref{e:iterated-delta}.
Proposition~\ref{p:antipode-r} may be deduced from either Proposition~\ref{p:empty}
or~\ref{p:mm-antipode},
or by calculating the inverse of $\id$ 
in the convolution algebra $\Hom(\thh,\thh)$ as
\[
\bigl(\iota\epsilon - (\iota\epsilon-\id)\bigr)^{-1} = \sum_{k\geq 0}(\iota\epsilon-\id)^{*k}.
\]
For more details, see~\cite[Proposition~8.13]{AguMah:2010}. 
There is an analogous formula for the antipode
of a connected Hopf algebra due to Takeuchi;
see the proof of~\cite[Lemma~14]{Tak:1971}
or~\cite[Lemma~5.2.10]{Mon:1993}.

\subsection{The antipode problem}\label{ss:ant-prob}

Cancellations frequently take place in Takeuchi's formula~\eqref{e:antipode-r}.
Understanding these cancellations is often a challenging combinatorial problem.
The antipode problem asks for an explicit, cancellation-free, formula for
the antipode of a given $q$-Hopf monoid $\thh$.
The problem may be formulated as follows:
given a linear basis of the vector space $\thh[I]$, we search for
the structure constants of $\apode_I$ on this basis.
If $\thh$ is the linearization of a set-theoretic Hopf monoid $\rH$,
then we may consider the basis $\rH[I]$ of $\thh[I]$. In this or other
cases, we may also be interested in other linear bases of $\thh[I]$, and the
corresponding structure constants.

Several instances of the antipode problem are solved in 
Sections~\ref{s:free}, ~\ref{s:freecom} and~\ref{s:examples} below.

\subsection{The primitive part}\label{ss:prim-ind}

Let $\tc$ be a connected comonoid.
The \emph{primitive part} of $\tc$ is the positive species $\Pc(\tc)$ defined by
\begin{equation}\label{e:def-prim}
\Pc(\tc)[I] := \bigcap_{\substack{I=S\sqcup T\\ S,T\neq\emptyset}}
\ker\bigl(\Delta_{S,T}:\tc[I]\to \tc[S]\otimes\tc[T]\bigr)
\end{equation}
for each nonempty finite set $I$.
An element $x\in\Pc(\tc)[I]$ is a \emph{primitive element} of $\tc[I]$.

\begin{proposition}\label{p:def-prim}
We have
\begin{equation}\label{e:gen-prim}
\Pc(\tc)[I] = \bigcap_{\substack{F\vDash I\\F\not=(I)}} \ker (\Delta_F).
\end{equation}
\end{proposition}
The intersection is over all compositions of $I$ with more than one block
and the map $\Delta_F$ as in~\eqref{e:iterated-delta}.
The result follows from~\eqref{e:def-prim} by coassociativity. 

Let $\td$ be another connected comonoid and consider the Hadamard product
$\td\times\tc$ (Section~\ref{ss:hadamard}).

\begin{proposition}\label{p:had-prim}
We have
\begin{equation}\label{e:had-prim}
\td\times\Pc(\tc) + \Pc(\td)\times\tc\subseteq \Pc(\td\times\tc).
\end{equation}
\end{proposition}
This follows since the coproduct $\Delta_{S,T}$ of $\td\times\tc$ is the tensor
product of the coproducts of $\td$ and $\tc$.

On primitive elements, the antipode is negation:

\begin{proposition}\label{p:prim-ant}
Let $\thh$ be a connected $q$-Hopf monoid and $x$ a primitive element of $\thh[I]$. Then
\begin{equation}\label{e:prim-ant}
\apode_I(x) = -x.
\end{equation}
\end{proposition}

This follows from either of Milnor and Moore's
formulas~\eqref{e:mm-antipode-r} or~\eqref{e:mm-antipode-l},
or also from~\eqref{e:gen-prim} and Takeuchi's formula~\eqref{e:antipode-r}. 

\begin{proposition}\label{p:prim-lie}
Let $\thh$ be a connected bimonoid. Then $\Pc(\thh)$
is a Lie submonoid of $\thh$ under the commutator bracket~\eqref{e:underlying-lie}.
\end{proposition}
\begin{proof}
Let $I=S_1\sqcup S_2$ and $x_i\in\Pc(\thh)[S_i]$ for $i=1,2$. 
We need to check that the element
\[
[x_1,x_2]_{S_1,S_2} = \mu_{S_1,S_2}(x_1\otimes x_2) - \mu_{S_2,S_1}(x_2\otimes x_1)
\]
is primitive. 
Let $I=T_1\sqcup T_2$ be another decomposition. 
We have that
\[
\Delta_{T_1,T_2}\mu_{S_1,S_2}(x_1\otimes x_2) =
\begin{cases}
x_1x_2\otimes 1 & \text{ if $T_1=I$ and $T_2=\emptyset$,}\\
1\otimes x_1x_2 & \text{ if $T_1=\emptyset$ and $T_2=I$,}\\
x_1\otimes x_2 & \text{ if $T_1=S_1$ and $T_2=S_2$,}\\
x_2\otimes x_1 & \text{ if $T_1=S_2$ and $T_2=S_1$,}\\
0 & \text{ in every other case.}\\
\end{cases}
\]
We have written $x_1x_2$ for $\mu_{S_1,S_2}(x_1\otimes x_2)$ and $1$ for $\iota_\emptyset(1)$. 
These follow from~\eqref{e:comp} and~\eqref{e:hopf-split};
in the last case we used the primitivity of the $x_i$.
It follows that if $T_1$ and $T_2$ are nonempty, 
then $\Delta_{T_1,T_2}\bigl([x_1,x_2]_{S_1,S_2}\bigr)=0$ as needed.
\end{proof}

\begin{remark}
The primitive part is the first component of the \emph{coradical filtration} of $\tc$.
See~\cite[Section~8.10]{AguMah:2010} for more information
and~\cite[Section~11.9.2]{AguMah:2010}
for a word on primitive elements in the nonconnected setting.
\end{remark}

\subsection{The indecomposable quotient}\label{ss:indec}

Let $\ta$ be a connected monoid.
The \emph{indecomposable quotient} of $\ta$ is the
positive species $\Qc(\ta)$ defined by
\begin{align}\label{e:def-indec}
\Qc(\ta)[I] & := \ta[I]\bigl/\Bigl(\sum_{\substack{I=S\sqcup T\\ S,T\neq\emptyset}} 
\im\bigl(\mu_{S,T}:\ta[S]\otimes\ta[T]\to\ta[I]\bigr)\Bigr)\\\notag
& = \ta[I]\bigl/\sum_{\substack{F\vDash I\\F\not=(I)}} \im (\mu_F).
\end{align}

Primitives and indecomposables are related by duality: $\Pc(\tc)^* \cong \Qc(\tc^*)$,
and if $\ta$ is finite-dimensional, then $\Qc(\ta)^*\cong\Pc(\ta^*)$.

Assume $\ta$ is finite-dimensional and let $\tc$ be a connected comonoid.
Consider the internal Hom $\iH(\ta,\tc)$ with the comonoid structure discussed
in Section~\ref{ss:iH}.

\begin{proposition}\label{p:iH-prim}
We have
\[
\iH\bigl(\ta,\Pc(\tc)\bigr) + \iH\bigl(\Qc(\ta),\tc\bigr) \subseteq \Pc\bigl(\iH(\ta,\tc)\bigr).
\]
\end{proposition}

\subsection{The Lagrange theorem}\label{ss:lagrange}

Let $\ta$ be a monoid. We consider modules and ideals as in Section~\ref{ss:modules}.
Given a submonoid $\tb$ of $\ta$, let $\tb_+\ta$ denote the right ideal
of $\ta$ generated by the positive part of $\tb$. 

\begin{theorem}\label{t:lagrange}
Let $\thh$ be a connected Hopf monoid and $\tk$ a Hopf submonoid.
Then there is an isomorphism of left $\tk$-modules
\[
\thh \cong \tk\bdot(\thh/\tk_+\thh).
\]
In particular, $\thh$ is free as a left $\tk$-module.
\end{theorem}

Theorem~\ref{t:lagrange} is proven in~\cite[Theorem~2.2]{AguLau:2012}.
Similar results for ordinary Hopf algebras are
well-known~\cite[Section~9.3]{Rad:2012}; these
include the familiar theorem of Lagrange from basic group theory.

We remark that versions of this result should exist for 
certain nonconnected Hopf monoids, as they do for Hopf algebras, but we
have not pursued this possibility.

\section{The free monoid}\label{s:free}

We review the explicit construction of the free monoid on a positive species,
following~\cite[Section~11.2]{AguMah:2010}. The free monoid carries
a canonical structure of $q$-Hopf monoid. 
We briefly mention the free monoid on an arbitrary species.

There is a companion notion of \emph{cofree comonoid} on a positive species
which is discussed in~\cite[Section~11.4]{AguMah:2010}, but not in this paper.

\subsection{The free monoid on a positive species}\label{ss:free}

Given a positive species $\tq$ and a composition $F=(I_1,\ldots,I_k)$ of $I$, write
\begin{equation}\label{e:sp-comp}
\tq(F) := \tq[I_1] \otimes \dots \otimes \tq[I_k].
\end{equation}
When $F$ is the unique composition of the empty set, we set 
\begin{equation}\label{e:sp-comp-empty}
\tq(F):=\Kb.
\end{equation}

Define a new species $\Tc(\tq)$ by
\[
\Tc(\tq)[I] := \bigoplus_{F\vDash I} \tq(F).
\]
A bijection $\sigma:I\to J$ transports a composition $F=(I_1,\dots,I_k)$ of $I$
into a composition $\sigma(F):=\bigl(\sigma(I_1),\dots,\sigma(I_k)\bigr)$ of $J$.
The map 
\[
\Tc(\tq)[\sigma]:\Tc(\tq)[I]\to\Tc(\tq)[J]
\] 
is the direct sum of the maps
\[
\tq(F) = \tq[I_1] \otimes \dots \otimes \tq[I_k] \map{\tq[\sigma|_{I_1}] \otimes \dots \otimes \tq[\sigma|_{I_k}]} 
\tq[\sigma(I_1)] \otimes \dots \otimes \tq[\sigma(I_k)] = \tq\bigl(\sigma(F)\bigr).
\]

In view of~\eqref{e:sp-comp-empty}, the species $\Tc(\tq)$ is connected.

Every nonempty set $I$ admits a unique composition with one block; namely, $F=(I)$.
In this case, $\tq(F)=\tq[I]$. This yields an embedding
$
\tq[I] \into \Tc(\tq)[I]
$
and thus an embedding of species 
\[
\eta_\tq: \tq \into \Tc(\tq).
\]
On the empty set, $\eta_\tq$ is (necessarily) zero.

Given $I=S\sqcup T$ and compositions $F\vDash S$ and
$G\vDash T$, there is a canonical isomorphism
\[
 \tq(F)\otimes \tq(G)\map{\cong} \tq(F\cdot G)
\]
obtained by concatenating the factors in~\eqref{e:sp-comp}.
The sum of these over all $F\vDash S$ and
$G\vDash T$ yields a map
\[
\mu_{S,T}: \Tc(\tq)[S]\otimes\Tc(\tq)[T] \to \Tc(\tq)[I].
\]
This turns $\Tc(\tq)$ into a monoid. In fact, $\Tc(\tq)$ is the \emph{free}
monoid on the positive species $\tq$, in view of the following result (a slight reformulation of~\cite[Theorem 11.4]{AguMah:2010}).

\begin{theorem}\label{t:dmunivp}
Let $\ta$ be a monoid, $\tq$ a positive species, and
$\zeta\colon \tq \to \ta$ a morphism of species.
Then there exists a unique morphism of monoids
$\hat{\zeta}\colon \Tc(\tq) \to \ta$ such that
\begin{equation*}%\label{e:dmunivp}
\begin{gathered}
\xymatrix@C+20pt{
\Tc(\tq) \ar@{.>}[r]^-{\hat{\zeta}} & \ta\\
 \tq \ar[ru]_{\zeta}\ar[u]^{\eta_\tq}
}
\end{gathered}
\end{equation*}
commutes.
\end{theorem}

The map $\hat{\zeta}$ is as follows. On the empty set, it is the unit map of $\ta$:
\[
\Tc(\tq)[\emptyset] = \Kb \map{\iota_\emptyset} \ta[\emptyset].
\]
On a nonempty set $I$, it is the sum of the maps
\[
\tq(F) = \tq[I_1]\otimes\cdots\otimes\tq[I_k] \map{\zeta_{I_1}\otimes\cdots\otimes\zeta_{I_k}}
\ta[I_1]\otimes\cdots\otimes\ta[I_k]
\map{\mu_{I_1,\ldots,I_k}} \ta[I],
\]
where the higher product $\mu_{I_1,\ldots,I_k}$ is as in~\eqref{e:iterated-mu}.

\smallskip 

When there is given an isomorphism of monoids $\ta\cong\Tc(\tq)$, we say that
the positive species $\tq$ is a \emph{basis} of the (free) monoid $\ta$.

\subsection{The free monoid as a Hopf monoid}\label{ss:freeHopf}

Let $q\in\Kb$ and $\tq$ a positive species. The species $\Tc(\tq)$ admits a
canonical $q$-Hopf monoid structure, which we denote by $\Tcq(\tq)$, as follows.

As monoids, $\Tcq(\tq)=\Tc(\tq)$. In particular, $\Tcq(\tq)$ and $\Tc(\tq)$ are the same
species. The comonoid structure depends on $q$. Given $I=S\sqcup T$,
the coproduct
\[
\Delta_{S,T}: \Tcq(\tq)[I]\to\Tcq(\tq)[S]\otimes\Tcq(\tq)[T]
\]
is the sum of the maps
\begin{align*}
\tq(F) & \to \tq(F|_S)\otimes \tq(F|_T)\\
x_1\!\otimes\!\cdots\!\otimes x_k & \mapsto 
\begin{cases} 
q^{\area_{S,T}(F)} (x_{i_1}\!\otimes\!\cdots\!\otimes x_{i_j})\!\otimes\! (x_{i'_1}\!\otimes\!\cdots\!\otimes x_{i'_k}) & \text{if $S$ is $F$-admissible,}\\
0 & \text{otherwise.}
\end{cases}
\end{align*}
Here $F=(I_1,\ldots,I_k)$, $x_i\in \tq[I_i]$ for each $i$, and
$\area_{S,T}(F)$ is as in~\eqref{e:schubert-comp}.
In the admissible case, we have written $F|_S=(I_{i_1},\ldots,I_{i_j})$ and
$F|_T=(I_{i'_1},\ldots,I_{i'_k})$.

The preceding turns $\Tcq(\tq)$ into a $q$-bimonoid. Since it is connected,
it is a $q$-Hopf monoid. We have
\begin{equation}\label{e:prim-contained}
\tq\subseteq \Pc\bigl(\Tcq(\tq)\bigr);
\end{equation}
more precisely, $\tq$ maps into the primitive part under the embedding $\eta_q$.
When $q=1$, $\Tc(\tq)$ is cocommutative, and
$\Pc\bigl(\Tc(\tq)\bigr)$ is the \emph{Lie submonoid} of $\Tc(\tq)$
generated by $\tq$, see Corollaries~\ref{c:prim-in-free-1} and~\ref{c:prim-in-free}.

We return to the situation of Theorem~\ref{t:dmunivp} in the case when the
monoid $\ta$ is in fact a $q$-Hopf monoid $\thh$. Thus, we are given a
morphism of species $\zeta:\tq\to\thh$ and we consider the morphism of monoids
$\hat{\zeta}:\Tcq(\tq)\to\thh$.

\begin{proposition}\label{p:dhmunivp}
Suppose that 
\[
\im(\zeta) \subseteq \Pc(\thh).
\]
In other words, $\zeta$ maps elements of $\tq$ to primitive elements of $\thh$.
Then $\hat{\zeta}$ is a morphism of $q$-Hopf monoids.
\end{proposition}

This is a special case of~\cite[Theorem~11.10]{AguMah:2010}; 
see also~\cite[Section~11.7.1]{AguMah:2010}. 
It follows from here and~\cite[Theorem~IV.1.2, item (i)]{Mac:1998}
 that the functors
\begin{equation}\label{e:TP}
\xymatrix@C+20pt{
\{\text{positive species}\} \ar@<0.6ex>[r]^-{\Tcq} & \ar@<0.6ex>[l]^-{\Pc} 
\{\text{connected $q$-Hopf monoids}\}
}
\end{equation}
form an adjunction (with $\Tcq$ being left adjoint to $\Pc$).

The antipode problem for $\Tcq(\tq)$ offers no difficulty.

\begin{theorem}
\label{t:free-apode}
The antipode of $\Tcq(\tq)$ is given by
\begin{equation}\label{e:free-apode}
\apode_I(x_1\otimes\cdots\otimes x_k) = q^{\dist(F,\opp{F})} (-1)^k x_k\otimes\cdots\otimes x_1.
\end{equation} 
\end{theorem}

Here $F=(I_1,\dots,I_k)\vDash I$, $x_i\in\tq[I_i]$ for $i=1,\ldots,k$,
and $\dist(F,\opp{F})$ is as in~\eqref{e:dist-opp}.
In particular, $\apode_I$ sends the summand $\tq(F)$ of $\Tcq(\tq)[I]$
to the summand $\tq(\opp{F})$.

To obtain this result, 
one may first note that any $x\in\tq[I]$ is primitive in $\Tc_q(\tq)[I]$,
so the result holds when $k=1$ by~\eqref{e:prim-ant},
then one may use the fact that the antipode reverses products~\eqref{e:apode-rev-prod}.
One may also derive it as a special case of~\cite[Theorem~11.38]{AguMah:2010}.

\begin{remark}
A free monoid $\Tc(\tq)$ may carry other $q$-Hopf monoid structures than the
one discussed above. In particular, it is possible to construct such a structure
from a positive comonoid structure on $\tq$; when the latter is trivial, we arrive 
at the $q$-Hopf monoid structure discussed above.
This is discussed in~\cite[Section~11.2.4]{AguMah:2010} but not in this paper.
\end{remark}

\subsection{The free monoid on an arbitrary species}\label{ss:free-arb}

Let $\tq$ be an arbitrary species (not necessarily positive).
The free monoid $\Tc(\tq)$ on $\tq$ exists~\cite[Example~B.29]{AguMah:2010}. 
It is obtained by the same construction as the one in Section~\ref{ss:free},
employing decompositions (in the sense of Section~\ref{ss:decompositions}) 
instead of compositions. 
Theorem~\ref{t:dmunivp} continues to hold.
One has 
\[
\Tc(\tq)[\emptyset] = \Tc(\tq[\emptyset]),
\] 
the free associative unital algebra on the vector space $\tq[\emptyset]$. 
Thus, $\Tc(\tq)$ is connected if and only if $\tq$ is positive. 

If $\tq[\emptyset]$ is a coalgebra,
it is possible to turn $\Tc(\tq)$ into a $q$-bimonoid as in Section~\ref{ss:freeHopf},
employing the notion of restriction of decompositions of Section~\ref{ss:decompositions}.
If $\tq[\emptyset]\neq 0$, 
the resulting bialgebra structure on $\Tc(\tq[\emptyset])$ is not the standard one,
and $\Tc(\tq[\emptyset])$ is not Hopf. 
Hence, $\Tc(\tq)$ is not a $q$-Hopf monoid.

In the sections that follow, and in most of the paper, we restrict
our attention to the case of positive $\tq$.

\subsection{Freeness under Hadamard products}\label{ss:had-free}

Given positive species $\tp$ and $\tq$,
define a new positive species $\tp \star \tq$ by
\begin{equation}\label{e:star}
(\tp \star \tq)[I] :=
\bigoplus_{\substack{F,G\vDash I\\F\wedge G=(I)}} \tp(F)\otimes\tq(G).
\end{equation}
The sum is over all pairs of compositions of $I$
whose meet is the minimum composition. 

The following result shows that 
the Hadamard product of two free monoids is again free
and provides an explicit description for the basis of the product 
in terms of bases of the factors.

\begin{theorem}\label{t:had-free-monoid}
For any positive species $\tp$ and $\tq$,
there is a natural isomorphism of monoids
\begin{equation}\label{e:had-free-monoid}
\Tc(\tp \star \tq) \cong \Tc(\tp)\times\Tc(\tq).
\end{equation}
\end{theorem}

Theorem~\ref{t:had-free-monoid} is proven in~\cite[Theorem~3.8]{AguMah:2012}. 
The isomorphism~\eqref{e:had-free-monoid} arises from the evident inclusion
$\tp\star\tq \into \Tc(\tp)\times\Tc(\tq)$ via the universal property of $\Tc(\tp \star \tq)$
in Theorem~\ref{t:dmunivp}.
It is shown in~\cite[Proposition~3.10]{AguMah:2012} that 
this map is an isomorphism of $0$-Hopf monoids
\[
\Tc_0(\tp \star \tq) \cong \Tc_0(\tp)\times\Tc_0(\tq);
\]
the analogous statement for $q\neq 0$ does not hold.

Let $p$ and $q\in\Kb$ be arbitrary scalars.

\begin{theorem}\label{t:freeness}
Let $\thh$ be a connected $p$-Hopf monoid.
Let $\tk$ be a $q$-Hopf monoid that is free as a monoid.
Then the connected $pq$-Hopf monoid
$\thh\times\tk$ is free as a monoid.
\end{theorem}

This is proven in~\cite[Theorem~3.2]{AguMah:2012}. 
In contrast to Theorem~\ref{t:had-free-monoid}, this result assumes freeness
from only one of the factors to conclude
it for their Hadamard product. 
On the other hand, in this situation we do not have an
explicit description for a basis of the product.

\subsection{Freeness for $0$-bimonoids}\label{ss:0free}

The structure of a connected $0$-Hopf monoid is particularly rigid, in
view of the following result.

\begin{theorem}\label{t:0free}
Let $\thh$ be a connected $0$-Hopf monoid. Then there exists an
isomorphism of $0$-Hopf monoids
\[
\thh \iso \Tc_0\bigl(\Pc(\thh)\bigr).
\]
\end{theorem}

Theorem~\ref{t:0free} is proven in~\cite[Theorem~11.49]{AguMah:2010}.
There is a parallel result for connected graded $0$-Hopf algebras which is
due to Loday and Ronco~\cite[Theorem~2.6]{LodRon:2006}.

Theorem~\ref{t:0free} implies that any connected $0$-Hopf monoid is free as a
monoid and cofree as a comonoid. In addition, if $\tq$ is finite-dimensional,
then the $0$-Hopf monoid $\Tc_0(\tq)$ is self-dual.
See~\cite[Section~11.10.3]{AguMah:2010} for more details.

\subsection{The free set-theoretic monoid}\label{ss:set-free}

Let $\rQ$ be a positive set species.
Given a composition $F=(I_1,\ldots,I_k)$ of $I$, write 
\[
\rQ(F) := \rQ[I_1]\times\cdots\times\rQ[I_k].
\]
The set species $\Tc(\rQ)$ defined by
\[
\Tc(\rQ)[I] := \coprod_{F\vDash I} \rQ(F)
\]
carries a set-theoretic monoid structure, and it is free on $\rQ$. 
There is a canonical isomorphism of monoids
\[
\Kb \Tc(\rQ) \cong \Tc(\Kb\rQ)
\]
arising from
\[
\Kb \bigl(\rQ(F)\bigr) \cong (\Kb \rQ)(F), \qquad
\BH_{(x_1,\ldots,x_k)} \leftrightarrow \BH_{x_1}\otimes\cdots\otimes\BH_{x_k},
\]
for $x_i\in\rQ[I_i]$ and $F$ as above.

The comonoid structure of $\Tc(\Kb\rQ)$ of Section~\ref{ss:freeHopf} is, in
general, not the linearization of a set-theoretic comonoid structure on $\Tc(\rQ)$.
For a case in which it is, see Section~\ref{ss:linear}.

\section{The free commutative monoid}\label{s:freecom}

We review the explicit construction of the free commutative monoid on a positive species,
following~\cite[Section~11.3]{AguMah:2010}. The discussion parallels that 
of Section~\ref{s:free}, with one important distinction: we deal exclusively
with Hopf monoids ($q=1$) as opposed to general $q$-Hopf monoids.
We briefly mention the free commutative monoid on an arbitrary species.

\subsection{The free commutative monoid on a positive species}\label{ss:freecom}

Given a positive species $\tq$ and a partition $X\vdash I$, write
\begin{equation}\label{e:sp-part}
\tq(X) := \bigotimes_{B\in X} \tq[B].
\end{equation}
It is not necessary to specify an ordering among the tensor factors above; 
the~\emph{unordered} tensor is well-defined in view of the fact that the tensor product 
of vector spaces is symmetric. Its elements are tensors
\[
\bigotimes_{B\in X} x_B,
\]
where $x_B\in\tq[B]$ for each $B\in X$.
 
Define a new species $\Sc(\tq)$ by
\[
\Sc(\tq)[I] := \bigoplus_{X\vdash I} \tq(X).
\]

When $X$ is the unique partition of $\emptyset$, we have $\tq(X)=\Kb$.
Thus, the species $\Sc(\tq)$ is connected.

Every nonempty $I$ admits a unique partition with one block; namely, $X=\{I\}$.
In this case, $\tq(X)=\tq[I]$. This yields an embedding
$
\tq[I] \into \Sc(\tq)[I]
$
and thus an embedding of species 
\[
\eta_\tq: \tq \into \Sc(\tq).
\]
On the empty set, $\eta_\tq$ is (necessarily) zero.

Given $I=S\sqcup T$ and partitions $X\vdash S$ and
$Y\vdash T$, there is a canonical isomorphism
\begin{equation*}%\label{e:unord-can}
 \tq(X)\otimes \tq(Y) \map{\cong} \tq(X\sqcup Y)
\end{equation*}
in view of the definition of unordered tensor products. Explicitly,
\[
\bigl(\bigotimes_{B\in X} x_B\bigr) \otimes \bigl(\bigotimes_{C\in Y} x_C\bigr)
\mapsto \bigotimes_{D\in X\sqcup Y} x_D.
\]
The sum of these isomorphisms over all $X\vdash S$ and
$Y\vdash T$ yields a map
\[
\mu_{S,T}: \Sc(\tq)[S]\otimes\Sc(\tq)[T] \to \Sc(\tq)[I].
\]
This turns $\Sc(\tq)$ into a monoid. 
The commutativity of the diagram
\[
\xymatrix{
 \tq(X)\otimes \tq(Y) \ar[r]^-{\cong} \ar@{<->}[d]_{\cong} & \tq(X\sqcup Y) \ar@{=}[d]\\
 \tq(Y)\otimes \tq(X) \ar[r]_-{\cong} & \tq(Y\sqcup X)
 }
 \]
implies that the monoid $\Sc(\tq)$ is commutative.
In fact, $\Sc(\tq)$ is the free
commutative monoid on the positive species $\tq$, in view of the following result (a slight reformulation of~\cite[Theorem 11.13]{AguMah:2010}).

\begin{theorem}\label{t:dcmunivp}
Let $\ta$ be a commutative monoid, $\tq$ a positive species, and
$\zeta\colon \tq \to \ta$ a morphism of species.
Then there exists a unique morphism of monoids
$\hat{\zeta}\colon \Sc(\tq) \to \ta$ such that
\begin{equation*}
\begin{gathered}
\xymatrix@C+20pt{
\Sc(\tq) \ar@{.>}[r]^-{\hat{\zeta}} & \ta\\
 \tq \ar[ru]_{\zeta}\ar[u]^{\eta_\tq}
}
\end{gathered}
\end{equation*}
commutes.
\end{theorem}

The map $\hat{\zeta}$ is as follows. On the empty set, it is the unit map of $\ta$:
\[
\Sc(\tq)[\emptyset] = \Kb \map{\iota_\emptyset} \ta[\emptyset].
\]
On a nonempty set $I$, it is the sum of the maps
\[
\tq(X) = \bigotimes_{B\in X} \tq[B] \map{\bigotimes_{B\in X} \zeta_{B}}
\bigotimes_{B\in X} \ta[B]
\map{\mu_{X}} \ta[I],
\]
where $\mu_{X}$ denotes the higher product $\mu_{I_1,\ldots,I_k}$ as in~\eqref{e:iterated-mu}
and $(I_1,\ldots,I_k)$ is any composition of $I$ with support $X$.
($\mu_X$ is well-defined by commutativity of $\ta$.)

\smallskip

When there is given an isomorphism of monoids $\ta\cong\Sc(\tq)$, we say that
the positive species $\tq$ is a \emph{basis} of the (free) commutative monoid $\ta$.

\subsection{The free commutative monoid as a Hopf monoid}\label{ss:freecomHopf}

The monoid $\Sc(\tq)$ admits a
canonical Hopf monoid structure. Given $I=S\sqcup T$,
the coproduct
\[
\Delta_{S,T}: \Sc(\tq)[I]\to\Sc(\tq)[S]\otimes\Sc(\tq)[T]
\]
is the sum of the maps 
\begin{align*}
\tq(X) & \to \tq(X|_S)\otimes \tq(X|_T)\\
\bigotimes_{B\in X} x_B & \mapsto 
\begin{cases} 
\bigl(\bigotimes_{B\in X|_S} x_B\bigr)\otimes \bigl(\bigotimes_{B\in X|_T} x_B\bigr) & \text{ if $S$ is $X$-admissible,}\\
0 & \text{ otherwise.}
\end{cases}
\end{align*}
Note that $S$ is $X$-admissible if and only if $X=X|_S\sqcup X|_T$.

This turns $\Sc(\tq)$ into a cocommutative bimonoid. Since it is connected,
it is a Hopf monoid. We have
\[
\Pc\bigl(\Sc(\tq)\bigr) = \tq;
\]
more precisely, the primitive part identifies with $\tq$ under the embedding $\eta_q$.

We return to the situation of Theorem~\ref{t:dcmunivp} in the case when
$\ta$ is in fact a commutative Hopf monoid $\thh$. Thus, we are given a
morphism of species $\zeta:\tq\to\thh$ and we consider the morphism of monoids
$\hat{\zeta}:\Sc(\tq)\to\thh$.

\begin{proposition}\label{p:dchmunivp}
Suppose that $\zeta$ maps elements of $\tq$ to primitive elements of $\thh$.
Then $\hat{\zeta}$ is a morphism of Hopf monoids.
\end{proposition}

This is a special case of~\cite[Theorem~11.14]{AguMah:2010}.

\begin{theorem}\label{t:free-com-apode}
The antipode of $\Sc(\tq)$ is given by
\begin{equation}\label{e:free-com-apode}
\apode_I(z) = (-1)^{\ell(X)} z
\end{equation} 
for any $z\in \tq(X)$, where $\ell(X)$ is the number of blocks of $X$.
\end{theorem}

One may prove this result as that in Theorem~\ref{t:free-apode}.
It is a special case of~\cite[Theorem~11.40]{AguMah:2010}.

\subsection{The free commutative monoid on an arbitrary species}\label{ss:freecom-arb}

Let $\tq$ be an arbitrary species (not necessarily positive).
The free commutative monoid $\Sc(\tq)$ on $\tq$ exists. It is the Hadamard product
\[
\Sc(\tq) := \wU_{\Sc(\tq[\emptyset])} \times \Sc(\tq_+).
\]
The second factor is the free commutative monoid (as in Section~\ref{ss:freecom})
on the positive part of $\tq$. The first factor is the monoid associated to the free
commutative algebra on the space $\tq[\emptyset]$ as in~\eqref{e:uniform}.
Theorem~\ref{t:dcmunivp} continues to hold.
One has $\Sc(\tq)[\emptyset]=\Sc(\tq[\emptyset])$;
in particular, $\Sc(\tq)$ is connected if and only if
$\tq$ is positive. 

If $\tq[\emptyset]$ is a coalgebra, then $\Sc(\tq[\emptyset])$
acquires a nonstandard bialgebra structure and $\Sc(\tq)$ is a bimonoid.
If $\tq[\emptyset]\neq 0$, then $\Sc(\tq[\emptyset])$ and $\Sc(\tq)$ are not Hopf.

In the sections that follow, and in most of the paper, we restrict
our attention to the case of positive $\tq$.

\subsection{The Hadamard product of free commutative monoids}\label{ss:had-freecom}

Let $\tp$ and $\tq$ be positive species.
Define a new positive species $\tp \diamond \tq$ by
\begin{equation}\label{e:diamond}
(\tp \diamond \tq)[I] :=
\bigoplus_{\substack{X,Y\vDash I\\X\wedge Y=\{I\}}} \tp(X)\otimes\tq(Y).
\end{equation}
The sum is over all pairs of partitions of $I$
whose meet is the minimum partition. 

\begin{theorem}\label{t:had-freecom}
For any positive species $\tp$ and $\tq$,
there is a natural isomorphism of Hopf monoids
\begin{equation}\label{e:had-freecom-monoid}
\Sc(\tp \diamond \tq) \cong \Sc(\tp)\times\Sc(\tq).
\end{equation}
\end{theorem}

Arguments similar to those in~\cite[Theorem~3.8 and Proposition~3.10]{AguMah:2012}
yield this result.

\subsection{The free commutative set-theoretic monoid}\label{ss:set-freecom}

Let $\rQ$ be a positive set species.
The free commutative set-theoretic monoid on $\rQ$ is the set species $\Sc(\rQ)$ defined by
\[
\Sc(\rQ)[I] := \coprod_{X\vdash I} \rQ(X),
\]
where 
\[
\rQ(X) := \prod_{B\in X} \rQ[B].
\]
There is a canonical isomorphism of monoids
\[
\Kb \Sc(\rQ) \cong \Sc(\Kb\rQ).
\]

The comonoid structure of $\Sc(\Kb\rQ)$ of Section~\ref{ss:freecomHopf} is, in
general, not the linearization of a set-theoretic comonoid structure on $\Sc(\rQ)$.
For a case in which it is, see Section~\ref{ss:exp}.

\subsection{The abelianization}\label{ss:abel}

Let $\tq$ be a positive species and
consider the free Hopf monoid $\Tc(\tq)$
and the free commutative Hopf monoid $\Sc(\tq)$ on it.
By freeness of the former (Theorem~\ref{t:dmunivp}),
there is a unique morphism of monoids
\[
\pi_{\tq}\colon \Tc(\tq)\to\Sc(\tq)
\]
such that
\[
\xymatrix{
\Tc(\tq) \ar[rr]^-{\pi_{\tq}} & & {\Sc(\tq)}\\
& \tq \ar[ru]_{\eta_\tq}\ar[lu]^{\eta_\tq}
}
\]
commutes, where $\eta_\tq$ denotes either universal arrow.

It is easy to see that $\pi_{\tq}$ is the abelianization of $\Tc(\tq)$.
Since the elements of $\tq$ are primitive, 
Proposition~\ref{p:dchmunivp} implies that $\pi_{\tq}$ is a morphism
of Hopf monoids.

On a finite set $I$, the morphism $\pi_{\tq}$ is the sum of the
canonical isomorphisms
\[
\tq(F) = \tq[I_1] \otimes \dots \otimes \tq[I_k] \isoto
\bigotimes_{B \in \supp F} \tq[B] = \tq(\supp F)
\]
over all compositions $F=(I_1,\ldots,I_k)$ of $I$.
The support is as in Section~\ref{ss:support}.
(The map $\pi_{\tq}$ is not an isomorphism
since a given partition supports many compositions.)

Since the support map does not preserve meets, the map
\[
\Tc(\tp\star\tq) \map{\cong} \Tc(\tp)\times\Tc(\tq) \map{\pi_\tp\times\pi_\tq} 
\Sc(\tp)\times\Sc(\tq) \map{\cong} \Sc(\tp\diamond\tq)
\]
(where the isomorphisms are as in~\eqref{e:star} and~\eqref{e:diamond})
does not send $\tp\star\tq$ to $\tp\diamond\tq$.

Abelianization is also meaningful in the set-theoretic context.

\section{The free Lie monoid}\label{s:freelie}

Recall the notion of a Lie monoid from Section~\ref{ss:liemon}.
We discuss two important universal constructions, 
that of the free Lie monoid on a positive species, 
and the universal enveloping monoid of a positive Lie monoid.

\subsection{The free Lie monoid}\label{ss:freelie}

We begin with some preliminary definitions.
A \emph{bracket sequence} $\alpha$ on $k$ letters is a way to parenthesize $k$ letters.
The concatenation $\alpha \cdot \beta$ of bracket sequences $\alpha$ and $\beta$
is defined in the obvious manner.
For instance, $(((a,b),c),(d,e))$ is the concatenation of $((a,b),c)$ and $(d,e)$.
The left bracket sequence is the one in which all brackets are to the left.
For instance, $(((a,b),c),d)$ is the left bracket sequence on $4$ letters.

A \emph{bracket composition} of a finite nonempty set $I$ is a pair $(F,\alpha)$,
where $F$ is a composition of $I$ and 
$\alpha$ is a bracket sequence on $\len(F)$ letters.
It is convenient to think of $\alpha$ as a bracket sequence on the blocks of $F$.
For example, $((S_1,S_2),(S_3,S_4))$ is a bracket composition.

\smallskip

Fix a positive species $\tq$.
We proceed to describe the free Lie monoid on $\tq$. 
For any bracket composition $(F,\alpha)$,
parenthesizing $\tq(F)$ using $\alpha$ yields a bracketed tensor product
which we denote by $\tq(F,\alpha)$.
Note that there is a canonical identification
\begin{equation}\label{eq:nonass}
\tq(F,\alpha) \otimes \tq(G,\beta) \isoto \tq(F\cdot G,\alpha\cdot \beta). 
\end{equation}
Now consider the positive species whose $I$-component is
\[
\bigoplus_{(F,\alpha)} \tq(F,\alpha),
\]
where the sum is over all bracket compositions of $I$.
The product~\eqref{eq:nonass} turns this space into a \emph{non-associative} monoid.
Let $\freelie(\tq)$ denote its quotient 
by the two-sided ideal generated by the relations:

For all $x\in \tq(F,\alpha)$, $y\in \tq(G,\beta)$, and $z\in \tq(H,\gamma)$, 
%the anti-symmetry relation
\begin{equation}\label{e:antisym-free}
x \otimes y + y \otimes x = 0, 
\end{equation}
and %the Jacobi relation
\begin{equation}\label{e:jacobi-free}
(x \otimes y) \otimes z + (z \otimes x) \otimes y + (y \otimes z) \otimes x = 0. 
\end{equation}
The relations~\eqref{e:antisym-free} and~\eqref{e:jacobi-free} imply that
the induced product satisfies~\eqref{e:antisym} and~\eqref{e:jacobi}.
So $\freelie(\tq)$ is a Lie monoid.
In fact, by construction, it is the free Lie monoid on $\tq$.
It satisfies the following universal property.

\begin{theorem}\label{t:freelie}
Let $\tg$ be a Lie monoid, $\tq$ a positive species, and
$\zeta\colon \tq \to \tg$ a morphism of species.
Then there exists a unique morphism of Lie monoids
$\hat{\zeta}\colon \freelie(\tq) \to \tg$ such that
\begin{equation*}
\begin{gathered}
\xymatrix@C+20pt{
\freelie(\tq) \ar@{.>}[r]^-{\hat{\zeta}} & \tg\\
 \tq \ar[ru]_{\zeta}\ar[u]^{\eta_\tq}
}
\end{gathered}
\end{equation*}
commutes. 
\end{theorem}

\begin{remark}
We mention that the free Lie monoid $\freelie(\tq)$ on an arbitrary species $\tq$ also exists. 
One has that $\freelie(\tq)[\emptyset]$ is the free Lie algebra 
on the vector space $\tq[\emptyset]$. 
Thus, $\freelie(\tq)$ is positive if and only if $\tq$ is positive.
\end{remark}

For any bracket composition $(F,\alpha)$, there is a map
\begin{equation}\label{eq:temp}
\tq(F,\alpha) \to \bigoplus_{G:\,\supp G=\supp F} \tq(G) 
\end{equation}
constructed by replacing each bracket in $\alpha$ by a commutator. 
For instance, if $F$ has three blocks and $\alpha$ is the left bracket sequence, then
\[
((x \otimes y) \otimes z) \mapsto 
x \otimes y \otimes z - y \otimes x \otimes z - z \otimes x \otimes y + z \otimes y \otimes x.
\]
Note that among all tensors on the right, only one begins with $x$.
This property holds whenever $\alpha$ is the left bracket sequence.
An explicit description of the action of this map on the left bracket sequence 
is provided by Lemma~\ref{l:left-brac-peakless} below.

By summing~\eqref{eq:temp} over all $(F,\alpha)$ and noting that the relations
~\eqref{e:antisym-free} and~\eqref{e:jacobi-free} map to zero,
we obtain an induced map
\begin{equation}\label{e:commutator-map}
\freelie(\tq) \to \Tc(\tq). 
\end{equation}

\begin{lemma}\label{l:fill-gap}
The map~\eqref{e:commutator-map} is injective.
\end{lemma}
\begin{proof}
We first make a general observation.
Repeated use of the Leibniz identity
\[
[x,[y,z]] = [[x,y],z]-[[x,z],y]
\]
changes any bracket sequence to a combination of left bracket sequences only, 
and antisymmetry can be used to get any specified factor to the first position.

Consider the component $\freelie(\tq)[I]$.
Note that the bracket compositions in~\eqref{e:antisym-free} or~\eqref{e:jacobi-free} 
have the same underlying set partition.
So $\freelie(\tq)[I]$ splits as a direct sum over partitions of $I$. 
For each partition $X$ of $I$, fix a block $S$ of $X$.
Using the above general observation, 
any element of $\freelie(\tq)[I]$ in the $X$-component can be expressed
as a sum of elements in those $\tq(F,\alpha)$ for which the first block of $F$ is $S$
and $\alpha$ is the left bracket sequence.
Further, the image of any such $\tq(F,\alpha)$ (see~\eqref{eq:temp})
is a linear combination of the $\tq(G)$ 
(with $G$ having the same support as $F$)
in which $\tq(F)$ appears and it is the only composition which begins with $S$.
So the images of the different $\tq(F,\alpha)$ are linearly independent 
proving injectivity.
\end{proof}

The above proof shows the following.
For each partition $X$ of $I$, fix a block $S$ of $X$.
Then $\freelie(\tq)[I]$ is isomorphic to the direct sum of 
those $\tq(F,\alpha)$ for which the first block of $F$ is $S$
and $\alpha$ is the left bracket sequence. 
As a consequence, we have
\begin{equation}
\dim_\Kb (\freelie(\tq)[I]) = \sum_{X\vdash I} \dim_\Kb(\tp(X)) (\len(X)-1)!,
\end{equation}
where $\len(X)$ denotes the number of blocks of $X$.

Recall that every monoid, and in particular $\Tc(\tq)$, is a Lie monoid
via the commutator bracket~\eqref{e:underlying-lie}.
Note that the image of $\freelie(\tq)$ under~\eqref{e:commutator-map}
is precisely the Lie submonoid of $\Tc(\tq)$ 
generated by $\tq$ (the smallest Lie submonoid containing $\tq$).

\subsection{The universal enveloping monoid}

Let $\tg$ be a positive Lie monoid.
%Recall the free monoid $\Tc(\tg)$ from Section~\ref{ss:free}.
Let $\Jc(\tg)$ be the ideal of the monoid $\Tc(\tg)$ generated by the elements
\begin{equation}\label{e:ug-rel}
x\otimes y - y\otimes x - [x,y]_{S,T} \ \in \tg(S,T) \oplus \tg(T,S) \oplus \tg[I] \subseteq \Tc(\tg)[I]
\end{equation}
for $x\in\tg[S]$ and $y\in\tg[T]$,
for all nonempty finite sets $I$ and compositions $(S,T)\vDash I$.
%Explicitly, the ideal of relations is spanned by elements
%\begin{multline}\label{e:ug-rel2}
%w\otimes x\otimes y\otimes z - w\otimes y\otimes x\otimes z - w\otimes [x,y]\otimes z\\
%\in \tg(F\cdot (S,T) \cdot G) \oplus \tg(F\cdot (T,S) \cdot G) \oplus \tg(F\cdot (S\sqcup T) \cdot G)
%\end{multline}
%for $w\in\tg(F)$, $x\in\tg[S]$, $y\in\tg[T]$, and $z\in\tg(G)$,
%as $F$, $G$, $S$ and $T$ vary.
The \emph{universal enveloping monoid} of $\tg$, denoted $\Uc(\tg)$,
is the quotient of $\Tc(\tg)$ by the ideal $\Jc(\tg)$.
It is a connected monoid.

Let
\[
\pi_\tg : \Tc(\tg) \onto \Uc(\tg)
\]
denote the quotient map, and let
\[
\eta_\tg : \tg \to \Uc(\tg)
\]
be the composition of $\pi_\tg$ with the embedding of $\tg$ in $\Tc(\tg)$.
Since~\eqref{e:ug-rel} equals $0$ in $\Uc(\tg)$,
the map $\eta_\tg$ is a morphism of Lie monoids when $\Uc(\tg)$ is viewed as
a Lie monoid via the commutator bracket~\eqref{e:underlying-lie}.

When the Lie bracket of $\tg$ is identically $0$, we have $\Uc(\tg)=\Sc(\tg)$ and
the map $\pi_\tg$ is the abelianization of Section~\ref{ss:abel}.

\begin{proposition}\label{p:ug-univ}
Let $\ta$ be a monoid, $\tg$ a positive Lie monoid, and
$\zeta\colon \tg \to \ta$ a morphism of Lie monoids.
Then there exists a unique morphism of monoids
$\hat{\zeta}\colon \Uc(\tg) \to \ta$ such that the first diagram below commutes.
\begin{align*}
& 
\begin{gathered}
\xymatrix@C+20pt{
\Uc(\tg) \ar@{.>}[r]^-{\hat{\zeta}} & \ta\\
\tg \ar[ru]_{\zeta}\ar[u]^{\eta_\tg}
}
\end{gathered}
&
\begin{gathered}
\xymatrix@C+20pt{
\Tc(\tg) \ar[dr]^{\Tilde{\zeta}} \ar[d]_{\pi_\tg}\\
\Uc(\tg) \ar[r]_-{\hat{\zeta}} & \ta
}
\end{gathered}
\end{align*}
Moreover, the second diagram above commutes as well, 
where $\Tilde{\zeta}$ is the morphism of monoids afforded by Theorem~\ref{t:dmunivp}.
\end{proposition}
\begin{proof}
By Theorem~\ref{t:dmunivp},
there is a morphism of monoids $\Tilde{\zeta} : \Tc(\tg)\to\ta$ extending $\zeta$.
By construction, its value on the element~\eqref{e:ug-rel} is
\[
\mu_{S,T}(\zeta_S(x)\otimes \zeta_T(y)) - \mu_{T,S}(\zeta_T(y)\otimes \zeta_S(x)) - \zeta_I([x,y]_{S,T}).
\]
Further, since $\zeta$ is a morphism of Lie monoids,
\[
\zeta_I([x,y]_{S,T}) = [\zeta_S(x),\zeta_T(y)]_{S,T},
\]
and hence the above value is $0$ by~\eqref{e:underlying-lie}.
Therefore, $\Tilde{\zeta} :\Tc(\tg)\to\ta$ factors through the quotient map $\pi_\tg$
to yield a morphism of monoids $\hat{\zeta}: \Uc(\tg) \to \ta$.
The uniqueness of $\hat{\zeta}$ follows from that of $\Tilde{\zeta}$
and the surjectivity of $\pi_\tg$.
\end{proof}

This explains the usage of ``universal''.
It is true but not obvious that $\eta_\tg$ is injective.
This fact is a part of the PBW theorem (Theorem~\ref{t:pbw}).
This justifies the usage of ``envelope''.

Recall from Section~\ref{ss:freeHopf} that 
$\Tc(\tg)$ carries a canonical structure of a bimonoid.
%The universal enveloping monoid carries a canonical structure of a bimonoid.

\begin{lemma}\label{l:univ-bimonoid}
The ideal $\Jc(\tg)$ of $\Tc(\tg)$ is also a coideal.
\end{lemma}
\begin{proof}
According to~\eqref{e:prim-contained}, $\tg\subseteq \Pc\bigl(\Tc(\tg)\bigr)$.
By Proposition~\ref{p:prim-lie}, $\Pc\bigl(\Tc(\tg)\bigr)$ is closed under the 
commutator bracket~\eqref{e:underlying-lie}. 
Hence, for any $x\in\tg[S]$ and $y\in\tg[T]$, 
both $x\otimes y - y\otimes x$ and $[x,y]_{S,T}$ belong to $\Pc\bigl(\Tc(\tg)\bigr)[I]$.
Thus, the ideal $\Jc(\tg)$ is generated by primitive elements, and hence is a coideal.
%This last step could be justified.
\end{proof}

As a consequence:
 
\begin{proposition}\label{p:univ-bimonoid}
There is a unique bimonoid structure on $\Uc(\tg)$ for which the map
$\pi_\tg: \Tc(\tg) \onto \Uc(\tg)$
is a morphism of bimonoids. Moreover, $\Uc(\tg)$ is cocommutative.
\end{proposition}

We return to the situation of Proposition~\ref{p:ug-univ} in the case when
$\ta$ is in fact a Hopf monoid $\thh$. Thus, we are given a
morphism of Lie monoids $\zeta:\tg\to\thh$ and we consider the morphism of monoids
$\hat{\zeta}:\Uc(\tg)\to\thh$.
Combining Propositions~\ref{p:dhmunivp} and~\ref{p:univ-bimonoid}, 
we obtain the following.

\begin{proposition}\label{p:ug-univ-ext}
Suppose that $\zeta$ maps elements of $\tg$ to primitive elements of $\thh$.
Then $\hat{\zeta}$ is a morphism of bimonoids.
\end{proposition}

It follows from here and~\cite[Theorem~IV.1.2, item (i)]{Mac:1998}
that the functors
\begin{equation}\label{e:UP}
\xymatrix@C+20pt{
\{\text{positive Lie monoids}\} \ar@<0.6ex>[r]^-{\Uc} & \ar@<0.6ex>[l]^-{\Pc} 
\{\text{connected Hopf monoids}\}
}
\end{equation}
form an adjunction (with $\Uc$ being left adjoint to $\Pc$). 
This may be composed with the adjunction
\[
\xymatrix@C+20pt{
\{\text{positive species}\} \ar@<0.6ex>[r]^-{\freelie} & \ar@<0.6ex>[l]^-{\text{forget}} 
\{\text{positive Lie monoids}\}
}
\]
(arising from Theorem~\ref{t:freelie}) 
and the result must be isomorphic to the adjunction~\eqref{e:TP}. 
Hence, for any positive species $\tq$,
\begin{equation}\label{e:tc-factorize}
\Uc(\freelie(\tq))\iso\Tc(\tq)
\end{equation}
as bimonoids.

\section{Examples of Hopf monoids}\label{s:examples}

We discuss a number of important examples of connected Hopf monoids (and $q$-Hopf monoids).
In most cases at least part of the structure is set-theoretic.
Starting with the simplest Hopf monoid (the exponential Hopf monoid),
we build on more complicated examples involving combinatorial structures such as
linear orders, set partitions, simple graphs, and generalized permutahedra.
Many of these monoids are free or free commutative.
An example involving groups of unitriangular matrices is included too.
We discuss various morphisms relating them and
provide formulas for the structure maps on one or more linear bases.
Self-duality is studied and the antipode problem is solved in all but
one instance. 

Two important examples are discussed separately in Section~\ref{s:faces}.
Additional examples can be found in~\cite[Chapters~12 and~13]{AguMah:2010}.

In this section we assume that the base field $\Kb$ is of characteristic $0$. This is 
only needed for certain statements regarding self-duality, certain basis changes,
and on a few other occasions where rational numbers intervene.

\subsection{The exponential Hopf monoid}\label{ss:exp}

Let $\rE$ be the set species defined by
\[
\rE[I] := \{I\}
\]
on all finite sets $I$.
The set $\{I\}$ is a singleton canonically associated to the set $I$.
We refer to $\rE$ as the \emph{exponential} set species. It is a
set-theoretic connected Hopf monoid
with product and coproduct defined by
\begin{align*}%\label{e:hopfE}
\rE[S]\times\rE[T] & \map{\mu_{S,T}} \rE[I] & \rE[I] &
\map{\Delta_{S,T}} \rE[S]\times\rE[T]\\
(S,T) & \xmapsto{\phantom{\mu_{S,T}}} I
& I & \xmapsto{\phantom{\Delta_{S,T}}} (S,T).
\end{align*}
The Hopf monoid axioms (Sections~\ref{ss:set-mon-com} and~\ref{ss:set-bimon})
are trivially satisfied. The Hopf monoid $\rE$ is commutative and cocommutative.

Let $\wE:=\Kb\rE$ denote the linearization of $\rE$. This agrees with~\eqref{e:expsp}.
It is a connected Hopf monoid. We follow Convention~\ref{con:bases}
by writing $\BH_I$ for the unique element of the canonical basis of $\wE[I]$. 

Let $\wX$ be the species defined by
\[
\wX[I] := \begin{cases}
\Kb & \text{if $I$ is a singleton,}\\
0 & \text{otherwise.}
\end{cases}
\]
It is positive. Note that if $X$ is a partition of $I$, then
\begin{equation*}%\label{e:XX}
\wX(X) \cong 
\begin{cases}
\Kb & \text{ if all blocks of $X$ are singletons,}\\
0 & \text{ otherwise.}
\end{cases}
\end{equation*}
Moreover, in the former case, any subset $S$ of $I$ is $X$-admissible.
It follows that $\Sc(\wX)[I]\cong \wE[I]$ for all finite sets $I$.
This gives rise to a canonical isomorphism of Hopf monoids 
\[
\Sc(\wX) \cong \wE.
\]
In particular, $\wE$ is the free commutative monoid on the species $\wX$
and the primitive part is
\[
\Pc(\wE) \cong \wX.
\]

As a simple special case of~\eqref{e:free-com-apode}, we have that the antipode of 
$\wE$ is given by
\[
\apode_I(\BH_I) = (-1)^{\abs{I}} \BH_I.
\]

The Hopf monoid $\wE$ is self-dual under $\BH_I \leftrightarrow \BM_I$.

\subsection{The Hopf monoid of linear orders}\label{ss:linear}

For any finite set $I$, let $\rL[I]$ be the set of all linear orders on $I$.
If $\ell=i_1\cdots i_n$ and $\sigma:I\to J$ is a bijection, 
then $\rL[\sigma](\ell):=\sigma(i_1)\cdots\sigma(i_n)$.
In this manner, $\rL$ is a set species, called \emph{of linear orders}.
It is a set-theoretic connected Hopf monoid with product and coproduct defined by
\begin{align*}%\label{e:hopfL}
\rL[S]\times\rL[T] & \map{\mu_{S,T}} \rL[I] & \rL[I] & \map{\Delta_{S,T}} \rL[S]\times\rL[T]\\
(\ell_1,\ell_2) & \xmapsto{\phantom{\mu_{S,T}}} \ell_1\cdot\ell_2
& \ell & \xmapsto{\phantom{\Delta_{S,T}}} (\ell |_S,\ell |_T).
\end{align*}
We have employed the operations of
concatenation and restriction from Section~\ref{ss:linear-prelim}; other notions from
that section are used below.

Consider the compatibility axiom~\eqref{e:rest-cont-prod}. 
Given linear orders $\ell_1$ on $S_1$ and $\ell_2$ on $S_2$, 
the axiom boils down to the fact that 
the concatenation of $\ell_1|_A$ and $\ell_2|_C$ agrees with 
the restriction to $T_1$ of $\ell_1\cdot \ell_2$.

The Hopf monoid $\rL$ is cocommutative but not commutative.

Let $\wL:=\Kb\rL$ denote the linearization of $\rL$. 
It is a connected Hopf monoid. The product and coproduct are
\[
\mu_{S,T}(\BH_{\ell_1}\otimes\BH_{\ell_2}) = \BH_{\ell_1\cdot\ell_2}
\qqand
\Delta_{S,T}(\BH_{\ell}) = \BH_{\ell|_S}\otimes\BH_{\ell|_T}.
\]
For the antipode of $\wL$ on the basis $\{\BH\}$, 
set $q=1$ in~\eqref{e:apode-linear} below.

For the dual Hopf monoid $\wL^*=\Kb^\rL$, we have
\[
\mu_{S,T}(\BM_{\ell_1}\otimes\BM_{\ell_2}) = \sum_{\ell:\, \text{$\ell$ a shuffle of
$\ell_1$ and $\ell_2$}} \BM_{\ell}
\]
and
\[
\Delta_{S,T}(\BM_{\ell}) = 
\begin{cases}
\BM_{\ell|_S}\otimes\BM_{\ell|_T} & \text{ if $S$ is an initial segment of $\ell$,}\\
0 & \text{ otherwise.}
\end{cases}
\]

Note that if $F$ is a composition of $I$,
\begin{equation*}%\label{e:XF}
\wX(F) \cong 
\begin{cases}
\Kb & \text{ if all blocks of $F$ are singletons,}\\
0 & \text{ otherwise.}
\end{cases}
\end{equation*}
Since a set composition of $I$ into singletons amounts to a linear order on $I$, we have 
$\Tc(\wX)[I]\cong \wL[I]$ for all finite sets $I$. Moreover, for such $F$, any 
subset $S$ of $I$ is $F$-admissible.
This gives rise to a canonical isomorphism of Hopf monoids 
\[
\Tc(\wX) \cong \wL.
\]
In particular, $\wL$ is the free monoid on the species $\wX$.

By Corollary~\ref{c:prim-in-free} below, the primitive part is
\[
\Pc(\wL)\cong \freelie(\wX),
\]
the free Lie monoid on $\wX$.
We denote this species by $\tLie$.
It is the species underlying the \emph{Lie operad}. 
(The species $\wL$ underlies the \emph{associative operad}.)
For more details, see~\cite[Section~11.9 and Appendix~B]{AguMah:2010}.

\smallskip

Fix a scalar $q\in\Kb$. 
Let $\wL_q$ denote the same monoid as $\wL$,
but endowed with the following coproduct:
\[
\Delta_{S,T}: \wL_q[I] \to \wL_q[S] \otimes \wL_q[T],
\qquad
\BH_{\ell} \mapsto q^{\area_{S,T}(\ell)}\, \BH_{\ell|_S} \otimes \BH_{\ell|_T},
\]
where the Schubert cocycle $\area_{S,T}$ is as in~\eqref{e:schubert-linear}.
Then $\wL_q$ is a connected $q$-Hopf monoid. 
The following properties of the Schubert cocycle
enter in the verification of coassociativity and of axiom~\eqref{e:comp}, respectively.
\begin{gather}\label{e:cocycle}
\area_{R,S\sqcup T}(\ell) + \area_{S,T}(\ell|_{S\sqcup T}) =
\area_{R\sqcup S,T}(\ell) + \area_{R,S}(\ell|_{R\sqcup S}),\\
\label{e:mul-schub-cocycle}
\area_{T_1,T_2}(\ell_1\cdot \ell_2) = \area_{A,B}(\ell_1)+ \area_{C,D}(\ell_2) + \abs{B}\abs{C}.
\end{gather}

There is an isomorphism of $q$-Hopf monoids
\[
\Tcq(\wX) \cong \wL_q.
\]
As a special case of~\eqref{e:free-apode}, 
we have that the antipode of $\wL_q$ is given by
\begin{equation}\label{e:apode-linear}
\apode_I(\BH_{\ell}) = (-1)^{\abs{I}} q^{\binom{\abs{I}}{2}} \BH_{\opp{\ell}},
\end{equation}
where $\opp{\ell}$ denotes the opposite linear order of $\ell$.
Other proofs of this result are discussed in~\cite[Example~8.16 and Proposition~12.3]{AguMah:2010}.

Generically, the $q$-Hopf monoid $\wL_q$ is self-dual. 
In fact, we have the following result. 
Define a map $\isolinear_q:\wL_q \to (\wL_q)^*$ by
\begin{equation}\label{e:lsdual}
\isolinear_q(\BH_{\ell}) := \sum_{\ell'} q^{\dist(\ell,\ell')}\, \BM_{\ell'}.
\end{equation}
The distance $\dist(\ell,\ell')$ between linear orders $\ell$ and $\ell'$ on $I$
is as in~\eqref{e:dist-linear}.

\begin{proposition}\label{p:lsdual}
The map $\isolinear_q:\wL_q \to (\wL_q)^*$ 
is a morphism of $q$-Hopf monoids.
If $q$ is not a root of unity, then it is an isomorphism. 
Moreover, for any $q$,
$\isolinear_q = (\isolinear_q)^*$.
\end{proposition}

Proposition~\ref{p:lsdual} is proven in~\cite[Proposition~12.6]{AguMah:2010}.
The invertibility of the map~\eqref{e:lsdual} follows from a result of Zagier; 
see~\cite[Example 10.30]{AguMah:2010} for related information.
A more general result is given in~\cite[Theorem~11.35]{AguMah:2010}.

The map $\isolinear_0$ sends $\BH_{\ell}$ to $\BM_{\ell}$
(and is an isomorphism of $0$-Hopf monoids).

The isomorphism $\isolinear_q$ gives rise to a nondegenerate pairing on $\wL_q$ 
(for $q$ not a root of unity).
The self-duality of the isomorphism translates
into the fact that the associated pairing is symmetric.

\subsection{The Hopf monoid of set partitions}\label{ss:flats}

Recall that $\rPi[I]$ denotes the set of partitions of a finite set $I$.
Below we make use of notions and notation for partitions discussed
in Section~\ref{ss:partitions}.

The species $\rPi$ \emph{of partitions} is a set-theoretic 
connected Hopf monoid with product and coproduct defined by
\begin{align*}%\label{e:hopfL}
\rPi[S]\times\rPi[T] & \map{\mu_{S,T}} \rPi[I] & \rPi[I] & \map{\Delta_{S,T}} \rPi[S]\times\rPi[T]\\
(X_1,X_2) & \xmapsto{\phantom{\mu_{S,T}}} X_1\sqcup X_2
& X & \xmapsto{\phantom{\Delta_{S,T}}} (X |_S,X |_T).
\end{align*}
We have employed the operations of union and restriction of partitions.
The Hopf monoid $\rPi$ is commutative and cocommutative.

Let $\tPi:=\Kb\rPi$ denote the linearization of $\rPi$. 
It is a connected Hopf monoid. The product and coproduct are
\[
\mu_{S,T}(\BH_{X_1}\otimes\BH_{X_2}) = \BH_{X_1\sqcup X_2}
\qqand
\Delta_{S,T}(\BH_{X}) = \BH_{X|_S}\otimes\BH_{X|_T}.
\]

The monoid $\tPi$ is the free commutative monoid on the positive part of the exponential species: 
there is an isomorphism of monoids
\[
\Sc(\wE_+) \cong \tPi, \quad
\bigotimes_{B\in X} \BH_B \leftrightarrow \BH_X.
\]
The isomorphism does not preserve coproducts, 
since $\Delta_{S,T}(\bigotimes_{B\in X} \BH_B)=0$ when $S$ is not $X$-admissible.

\begin{theorem}\label{t:apode-flat}
The antipode of $\tPi$ is given by
\[
\apode_I(\BH_X) = \sum_{Y:\, X \leq Y} (-1)^{\len(Y)} \,
(Y/X)! \, \BH_Y.
\]
\end{theorem}

A proof of Theorem~\ref{t:apode-flat} is given in~\cite[Theorem~12.47]{AguMah:2010}.

For the dual Hopf monoid $\tPi^*=\Kb^\rPi$, we have
\[
\mu_{S,T}(\BM_{X_1}\otimes\BM_{X_2}) = \sum_{X:\, \text{$X$ a quasi-shuffle of
$X_1$ and $X_2$}} \BM_{X}
\]
and
\[
\Delta_{S,T}(\BM_{X}) = 
\begin{cases}
\BM_{X|_S}\otimes\BM_{X|_T} & \text{ if $S$ is $X$-admissible,}\\
0 & \text{ otherwise.}
\end{cases}
\]

The Hopf monoid $\tPi$ is self-dual. There are in fact several isomorphisms
$\tPi\cong\tPi^*$. We highlight two particular ones. Define maps
$
\isoflat,\isograph \colon \tPi \to \tPi^*
$
by
\begin{equation}\label{e:pisdual}
\isoflat(\BH_Y) := \sum_X (X \vee Y)! \, \BM_X
\qqand
\isograph(\BH_Y) := \sum_{X:\, X\vee Y=\maxflat} \BM_X.
\end{equation}

\begin{proposition}\label{p:flat-self-dual}
The maps $\isoflat,\isograph \colon \tPi \to \tPi^*$
are isomorphisms of Hopf monoids. Moreover, $\isoflat=\isoflat^*$ and $\isograph=\isograph^*$.
\end{proposition}

The assertion about $\isoflat$ is proven in~\cite[Proposition~12.48]{AguMah:2010}.
One verifies without difficulty that $\isograph$ preserves products and coproducts.
Its invertibility follows from a result of Dowling and Wilson~\cite[Lemma~1]{DowWil:1975}
(or from the considerations below). 

\smallskip

Define a linear basis $\{\BQ\}$ of $\tPi[I]$ by
\[
\BH_X = \sum_{Y:\, X \leq Y} \BQ_Y.
\]
The basis $\{\BQ\}$ is determined by M\"obius inversion.
Let $\{\BP\}$ denote the linear basis of $\tPi^*$ dual to the basis $\{\BQ\}$ of $\tPi$.
Equivalently,
\[
\BP_Y = \sum_{X:\, X \leq Y} \BM_X.
\]

\begin{proposition}\label{p:Qflat}
We have
\begin{gather*}
\mu_{S,T}(\BQ_{X_1}\otimes\BQ_{X_2}) = \BQ_{X_1\sqcup X_2},\\
\Delta_{S,T}(\BQ_{X}) = 
\begin{cases}
\BQ_{X|_S}\otimes\BQ_{X|_T} & \text{ if $S$ is $X$-admissible,}\\
0 & \text{ otherwise,}
\end{cases}\\
\apode_I(\BQ_X) = (-1)^{\len(X)} \, \BQ_X.
\end{gather*}
\end{proposition}

This is shown in~\cite[Section~12.6]{AguMah:2010}.
In addition, the maps in~\eqref{e:pisdual} satisfy
\[
\isoflat(\BQ_X) = X\ifac \, \BP_X
\qand
\isograph(\BQ_X) = (-1)^{\abs{I}-\len(X)}\, X\ifac \, \BP_X,
\]
where the coefficient $X\ifac$ is as in~\eqref{e:Xifac}. 
The second coefficient is the M\"obius function value~\eqref{e:partmobiusPi}.

The expressions for the product and coproduct on the basis $\{\BQ\}$ show that
there is an isomorphism of Hopf monoids
\[
\Sc(\wE_+) \cong \tPi, \quad
\bigotimes_{B\in X} \BH_B \leftrightarrow \BQ_X.
\]
It follows that the primitive part satisfies
\[
\Pc(\tPi) \cong \wE_+
\]
with $\Pc(\tPi)[I]$ one-dimensional spanned by the element $\BQ_{\{I\}}$.

\subsection{The Hopf monoid of simple graphs}\label{ss:graphs}

A \emph{simple graph} on $I$ is a 
collection of subsets of $I$ of cardinality $2$.
The elements of $I$ are the \emph{vertices} and the subsets are the 
(undirected) \emph{edges}. 

Let $g$ be a simple graph on $I$ and $S\subseteq I$.
The \emph{restriction} $g|_S$ is the simple graph on $S$ consisting of those 
edges of $g$ connecting elements of $S$.
Write $I=S\sqcup T$. We say that $S$ is \emph{$g$-admissible}
if no edge of $g$ connects an element of $S$ to an element of $T$.

Given $I=S\sqcup T$ and simple graphs $g$ on $S$ and $h$ on $T$,
their \emph{union} is the simple graph $g\sqcup h$ on $I$ whose edges
belong to either $g$ or $h$.

Let $\rG[I]$ denote the set of simple graphs with vertex set $I$.
%We partially order this set by reverse inclusion of edge-sets:
%$g\leq h$ if each edge of $h$ is an edge of $g$.
%%This is a Boolean poset.
%%Going up in this partial order weakly increases the rank.
%We let $g\cap h$ denote the join in this poset.
%In other words,
%the edge-set of $g\cap h$ is the intersection of the edge-sets of $g$ and $h$.
The species $\rG$ \emph{of simple graphs} is a set-theoretic
connected Hopf monoid with product and coproduct defined by
\begin{align*}
\rG[S]\times\rG[T] & \map{\mu_{S,T}} \rG[I] & \rG[I] & \map{\Delta_{S,T}} \rG[S]\times\rG[T]\\
(g_1,g_2) & \xmapsto{\phantom{\mu_{S,T}}} g_1\sqcup g_2
& g & \xmapsto{\phantom{\Delta_{S,T}}} (g |_S,g |_T).
\end{align*}
The Hopf monoid $\rG$ is commutative and cocommutative.

We let $\tG:=\Kb\rG$ denote the linearization of $\rG$.
It is a connected Hopf monoid. The product and coproduct are
\[
\mu_{S,T}(\BH_{g_1}\otimes\BH_{g_2}) = \BH_{g_1\sqcup g_2}
\qqand
\Delta_{S,T}(\BH_{g}) = \BH_{g|_S}\otimes\BH_{g|_T}.
\]

Let $\tcG$ denote the subspecies of $\tG$ spanned by connected simple graphs.
The monoid $\tG$ is the free commutative monoid on $\tcG_+$:
there is an isomorphism of monoids
\[
\Sc(\tcG_+) \cong \tG, \qquad
\bigotimes_{B\in X} \BH_{g|_B} \leftrightarrow \BH_g.
\]
Here $X$ is the partition into connected components of $g$.
The isomorphism does not preserve coproducts, since $\Delta_{S,T}(\bigotimes_{B\in X} \BH_{g|_B})=0$ when $S$ is not $g$-admissible.

There is a pleasant solution to the antipode problem for $\tG$.
In order to state it, we set up some terminology.

Given a simple graph $g$ on $I$ and a partition $X\vdash I$,
let
\[
g|_X := \bigsqcup_{B\in X} g|_B.
\]
This is a simple graph on $I$. Let also $g/_X$ be the simple graph on $X$ such that
there is an edge between $B,C\in X$ if there is at least one edge in $g$
between an element of $B$ and an element of $C$.
In other words, $g/_X$ is obtained from $g$
by identifying all vertices in each block of $X$ and
removing all loops and multiple edges that may arise as a result.
The \emph{lattice of contractions} of $g$ is the set
\begin{equation*}%\label{e:latt-cont}
L(g) := \{X\vdash I \mid \text{the graph $g|_B$ is connected for each $B\in X$}\}.
\end{equation*}

An orientation of the edges of $g$ is \emph{acyclic} if it contains no oriented cycles.
Let $\acyc{g}$ denote the number of acyclic orientations of $g$.

\begin{theorem}\label{t:apode-graph}
The antipode of $\tG$ is given by
\begin{equation}\label{e:apode-graph}
\apode_I(\BH_g) = \sum_{X\in L(g)} (-1)^{\len(X)}\,\acyc{(g/_X)}\,\BH_{g|_X}.
\end{equation}
\end{theorem}

Note that if $X\in L(g)$, then $X$ can be recovered from $g|_X$ as the partition into
connected components. Hence there are no cancellations in the antipode
formula~\eqref{e:apode-graph}.

Theorem~\ref{t:apode-graph} is given in~\cite{AguArd:2012} 
and also (in the setting of Hopf algebras)
in work of Humpert and Martin~\cite[Theorem~3.1]{HumMar:2012}.
It is explained in these references 
how Stanley's negative one color theorem~\cite[Corollary~1.3]{Sta:1973} 
may be derived from this result. 
It is possible to proceed conversely, 
as we outline after Proposition~\ref{p:Qgraph} below.

For the dual Hopf monoid $\tG^*=\Kb^\rG$, the product is
\[
\mu_{S,T}(\BM_{g_1}\otimes\BM_{g_2}) = \sum_{g} \BM_{g}.
\]
The sum is over those simple graphs $g$ for which $g|_S=g_1$ and $g|_T=g_2$.
In other words, an edge of $g$ is either an edge of $g_1$,
or an edge of $g_2$, or connects an element of $S$ with an element of $T$.
The coproduct is
\[
\Delta_{S,T}(\BM_{g}) =
\begin{cases}
\BM_{g|_S}\otimes\BM_{g|_T} & \text{ if $S$ is $g$-admissible,}\\
0 & \text{ otherwise.}
\end{cases}
\]

The Hopf monoid $\tG$ is self-dual.
The following result provides a canonical duality.
The \emph{complement} of a simple graph $g$ on $I$ is the simple graph $g'$
with the complementary edge set. Define a map
$\isograph \colon \tG \to \tG^*$ by
\begin{equation}\label{e:gsdual}
\isograph(\BH_h) := \sum_{g:\,g\subseteq h'} \, \BM_g.
\end{equation}
The simple graphs $g$ in the sum are those without common edges with $h$.

\begin{proposition}\label{p:graph-self-dual}
The map $\isograph \colon \tG \to \tG^*$ 
is an isomorphism of Hopf monoids. Moreover, $\isograph=\isograph^*$.
\end{proposition}

One verifies that $\isograph$ preserves products and coproducts.
It is triangular with respect to inclusion, hence invertible.

\smallskip

Define a linear basis $\{\BQ\}$ of $\tG[I]$ by
\[
\BH_g = \sum_{h:\, h\subseteq g} \BQ_h.
\]
The basis $\{\BQ\}$ is determined by M\"obius inversion.
Let $\{\BP\}$ denote the linear basis of $\tG^*$
dual to the basis $\{\BQ\}$ of $\tG$.
Equivalently,
\[
\BP_h = \sum_{g:\, h\subseteq g} \BM_g.
\]

\begin{proposition}\label{p:Qgraph}
We have
\begin{gather}
\label{e:prod-graphQ}
\mu_{S,T}(\BQ_{g_1}\otimes\BQ_{g_2}) = \BQ_{g_1\sqcup g_2},\\
\label{e:coprod-graphQ}
\Delta_{S,T}(\BQ_{g}) =
\begin{cases}
\BQ_{g|_S}\otimes\BQ_{g|_T} & \text{ if $S$ is $g$-admissible,}\\
0 & \text{ otherwise,}
\end{cases}\\
\label{e:apode-graphQ}
\apode_I(\BQ_g) = (-1)^{c(g)} \, \BQ_g.
\end{gather}
\end{proposition}

Here $c(g)$ denotes the number of connected components of $g$.
In addition, the map in~\eqref{e:gsdual} satisfies
\[
\isograph(\BQ_g) = (-1)^{\abs{I}- c(g)} \, \BP_g.
\]

Formulas~\eqref{e:prod-graphQ} and~\eqref{e:coprod-graphQ} show that
there is an isomorphism of Hopf monoids
\[
\Sc(\tcG_+) \cong \tG, \qquad
\bigotimes_{B\in X} \BH_{g|_B} \leftrightarrow \BQ_g.
\]
The antipode formula~\eqref{e:apode-graphQ} is then a special case 
of~\eqref{e:free-com-apode}.
This in turn can be used to derive the antipode formula~\eqref{e:apode-graph}
using Stanley's negative one color theorem~\cite[Corollary~1.3]{Sta:1973}.

It also follows that the primitive part satisfies
\[
\Pc(\tG) \cong \tcG_+
\]
with $\Pc(\tG)[I]$ spanned by the elements $\BQ_g$ as $g$ runs over the connected
simple graphs on $I$.

\smallskip

The analogy in the analyses of the Hopf monoids $\tPi$ and $\tG$ can be
formalized. Let $k_I$ denote the complete graph on $I$, and
for a partition $X$ of $I$, let 
\[
k_X := \bigsqcup_{B\in X} k_B.
\]
Define a map $k:\tPi\to\tG$ by
\begin{equation}\label{e:Pi-G}
k(\BH_{X}) := \BH_{k_X}.
\end{equation}

\begin{proposition}\label{p:Pi-G}
The map $k:\tPi\to\tG$ is an injective morphism of Hopf monoids. Moreover, 
\[
\xymatrix{
\tG \ar[r]^-{\isograph} & \tG^* \ar[d]^{k^*}\\
\tPi \ar[u]^{k} \ar[r]_-{\isograph} & \tPi^* 
}
\]
commutes.
\end{proposition}

The commutativity of the diagram translates into the fact that $k$ is an
isometric embedding when $\tPi$ and $\tG$ are endowed with the pairings
associated to the isomorphisms $\isograph$.

One may use Proposition~\ref{p:Pi-G}
to deduce the antipode formula for $\tPi$ (Theorem~\ref{t:apode-flat})
from that for $\tG$ (Theorem~\ref{t:apode-graph}).

\subsection{The Hopf monoid of functions on unitriangular groups}\label{ss:class}

Let $\Fb$ be a fixed field, possibly different from the base field $\Kb$.
Given a finite set $I$ and a linear order $\ell$ on $I$,
let $\Mat(I)$ denote the algebra of matrices 
\[
\text{$A=(a_{ij})_{i,j\in I}$,
 $a_{ij}\in\Fb$ for all $i,j\in I$.}
\]
The general linear group $\GL(I)$ consists of the invertible matrices in $\Mat(I)$,
and the subgroup $\Un(I,\ell)$ consists of the \emph{upper $\ell$-unitriangular} matrices
\[
\text{$U=(u_{ij})_{i,j\in I}$,
$u_{ii}=1$ for all $i\in I$,
$u_{ij}=0$ whenever $i\geq j$ according to $\ell$.}
\]

Let $I=S_1\sqcup S_2$. Given $A=(a_{ij})\in\Mat(S_1)$ and $B=(b_{ij})\in\Mat(S_2)$, their \emph{direct sum} is the matrix
$A\oplus B=(c_{ij})\in\Mat(I)$ with entries
\[
c_{ij}= \begin{cases}
a_{ij} & \text{ if both $i,j\in S_1$,}\\
b_{ij} & \text{ if both $i,j\in S_2$,}\\
0 & \text{ otherwise.}
\end{cases}
\]

Let $\ell\in\rL[I]$. The direct sum of an $\ell|_{S_1}$-unitriangular and an
$\ell|_{S_2}$-unitriangular matrix is $\ell$-unitriangular.
The morphism of algebras
\[
\Mat(S_1)\times\Mat(S_2) \to \Mat(I),
\quad (A,B) \mapsto A\oplus B
\]
thus restricts to a morphism of groups
\begin{equation}\label{e:sigma}
\sigma_{S_1,S_2} :\Un(S_1,\ell|_{S_1})\times \Un(S_2,\ell|_{S_2}) \to \Un(I,\ell).
\end{equation}
(The dependence of $\sigma_{S_1,S_2}$ on $\ell$ is left implicit.)

Given $A=(a_{ij})\in\Mat(I)$, the \emph{principal minor} indexed by $S\subseteq I$
is the matrix 
\[
A_S =(a_{ij})_{i,j\in S}.
\]
If $U$ is $\ell$-unitriangular, then $U_S$ is $\ell|_S$-unitriangular. 
Moreover, the map
\[
\Un(I,\ell) \to \Un(S,\ell|_S),\quad U\mapsto U_S
\]
is a morphism of groups if and only if $S$ is an \emph{$\ell$-segment}:
if $i,k\in S$ and $i\leq j\leq k$ according to $\ell$, then also $j\in S$.

Let $I=S_1\sqcup S_2$, $\ell_i\in\rL[S_i]$, $i=1,2$.
Define a map
\begin{equation}\label{e:pi}
\pi_{S_1,S_2}: \Un(I,\ell_1\cdot\ell_2) \to \Un(S_1,\ell_1)\times \Un(S_2,\ell_2)
\end{equation}
by
\[
U \mapsto (U_{S_1},U_{S_2}).
\]
Note that $S_1$ is an initial segment for $\ell_1\cdot\ell_2$ and $S_2$ is a final 
segment for $\ell_1\cdot\ell_2$. 
Thus $\pi_{S_1,S_2}$ is a morphism of groups.

Let $I=S_1\sqcup S_2=T_1\sqcup T_2$ and let $A,B,C$ and $D$ be the pairwise intersections, 
as in~\eqref{e:4sets}.
We then have
\begin{equation}\label{e:direct-principal}
(U\oplus V)_{S_1} = U_A\oplus V_B
\end{equation}
and a similar statement for $S_2$, $C$, and $D$.

Assume from now on that the field $\Fb$ is finite. Thus, all groups
$\Un(I,\ell)$ of unitriangular matrices are finite.
Define a vector species $\FC(\Un)$ as follows. On a finite set $I$,
\[
\FC(\Un)[I] = \bigoplus_{\ell\in\rL[I]} \Kb^{\Un(I,\ell)}.
\]
In other words, $\FC(\Un)[I]$ is the direct sum of the spaces of functions (with values in 
$\Kb$) on all unitriangular groups on $I$. 

Let $I=S_1\sqcup S_2$ and $\ell_i\in\rL[S_i]$, $i=1,2$.
From the morphism $\pi_{S_1,S_2}$ in~\eqref{e:pi}
we obtain a linear map
\[
\Kb^{\Un(S_1,\ell_1)}\otimes \Kb^{\Un(S_2,\ell_2)}
\to \Kb^{\Un(I,\ell_1\cdot\ell_2)}.
\]
Summing over all $\ell_1\in\rL[S_1]$ and $\ell_2\in\rL[S_2]$, 
we obtain a linear map
\begin{equation}\label{e:f-prod}
\mu_{S_1,S_2} : \FC(\Un)[S_1]\otimes \FC(\Un)[S_2] \to \FC(\Un)[I].
\end{equation}
Explicitly, given functions $f:\Un(S_1,\ell_1)\to\Kb$ and $g:\Un(S_2,\ell_2)\to\Kb$, 
\[
\mu_{S_1,S_2}(f\otimes g): \Un(I,\ell_1\cdot\ell_2)\to\Kb
\]
is the function given by
\begin{equation*}%\label{e:f-prod2}
U \mapsto f(U_{S_1})g(U_{S_2}).
\end{equation*}

Similarly, from the map $\sigma_{S_1,S_2}$ in~\eqref{e:sigma} we obtain 
the components
\[
\Kb^{\Un(I,\ell)} \to \Kb^{\Un(S_1,\ell|_{S_1})} \otimes \Kb^{\Un(S_2,\ell|_{S_2})}
\]
(one for each $\ell\in\rL[I]$) of a linear map
\begin{equation}\label{e:f-coprod}
\Delta_{S_1,S_2}: \FC(\Un)[I] \to \FC(\Un)[S_1]\otimes \FC(\Un)[S_2].
\end{equation}
Explicitly, given a function $f:\Un(I,\ell)\to\Kb$, 
we have 
\[
\Delta_{S_1,S_2}(f)=\sum_i f^1_i\otimes f^2_i,
\] 
where 
\[
f^1_i:\Un(S_1,\ell|_{S_1})\to\Kb
\qand
f^2_i:\Un(S_2,\ell|_{S_2})\to\Kb
\]
are functions such that
\begin{equation}\label{e:f-coprod2}
f(U_1\oplus U_2) = \sum_i f^1_i(U) f^2_i(V)
\end{equation}
for all $U_1\in\Un(S_1,\ell|_{S_1})$ and $U_2\in\Un(S_2,\ell|_{S_2})$.

\begin{proposition}\label{p:f-unit}
With the operations~\eqref{e:f-prod} and~\eqref{e:f-coprod},
the species $\FC(\Un)$ is a connected Hopf monoid. 
It is cocommutative.
\end{proposition}

The compatibility between the product and the coproduct 
follows from~\eqref{e:direct-principal}. 
The construction is carried out in more detail in~\cite{ABT:2012}.

We describe the operations on the basis $\{\BM\}$ of characteristic functions
(see Convention~\ref{con:bases}). 
Let $U_i\in \Un(S_i,\ell_i)$, $i=1,2$. 
The product is
\begin{equation}\label{e:char-f-prod}
\mu_{S_1,S_2}(\BM_{U_1}\otimes \BM_{U_2}) = 
\sum_{\substack{U\in \Un(I,\ell_1\cdot\ell_2)\\U_{S_1}=U_1,\,U_{S_2}=U_2}} \BM_U.
\end{equation}
The coproduct is
\begin{equation}\label{e:char-f-coprod}
\Delta_{S_1,S_2}(\BM_{U}) = 
\begin{cases}
\BM_{U_{S_1}}\otimes \BM_{U_{S_2}} & \text{ if } U= U_{S_1} \oplus U_{S_2},\\
0 & \text{ otherwise.}
\end{cases}
\end{equation}

%The map $\wL \to \FC(\Un)$ defined by
%\[
%\BH_\ell\mapsto \sum_{U\in \Un(I,\ell)} \BM_U
%\]
%for $\ell\in\rL[I]$ is an injective morphism of Hopf monoids. The function on the right-hand side is constant with value $1$.

To a unitriangular matrix $U\in \Un(I,\ell)$ we associate a graph $g(U)$ on $I$
as follows: there is an edge between $i$ and $j$ if $i<j$ in $\ell$ and $u_{ij}\neq 0$.
For example, given nonzero entries $a,b,c\in\Fb$,
\medskip
\[
\begin{gathered}
\ell=hijk,
\quad
U=\pmat{1 & 0 & 0 & a\\
& 1 & b & c\\
& & 1 & 0\\
& & & 1}
\quad \Rightarrow \quad
g(U)=\xymatrix@R-25pt{
*{\bullet} \ar@{-}@/^3pc/[rrr] & 
*{\bullet} \ar@{-}@/^2pc/[rr] \ar@{-}@/^1pc/[r] & 
*{\bullet} & *{\bullet}\\
 h & i & j & k
}.
\end{gathered}
\]

Let $\gamma: \wL\times\tG^* \to \FC(\Un)$
be the map defined by 
\begin{equation}\label{e:f-model}
\gamma(\BH_\ell\otimes \BM_g) := \sum_{U\in\Un(I,\ell):\, g(U)=g} \BM_U.
\end{equation}

\begin{proposition}\label{p:f-model}
The map $\gamma: \wL\times\tG^* \to \FC(\Un)$ is an injective morphism of Hopf monoids.
If $\Fb$ is the field with two elements, then $\gamma$ is an isomorphism.
\end{proposition}

This result is given in~\cite[Proposition~8 and Corollary~9]{ABT:2012}.
(In this reference, $\tG$ is identified with its dual $\tG^*$.)

An injective morphism of Hopf monoids 
\[
\wL\times\tPi^* \to \wL\times\tG^*
\]
is described in~\cite[Section~4.6]{ABT:2012}. The composite
\[
\wL\times\tPi^* \into \wL\times\tG^* \map{\gamma} \FC(\Un)
\]
maps to a certain Hopf submonoid of \emph{superclass functions} on the unitriangular groups.

\subsection{The Hopf monoid of generalized permutahedra}\label{ss:gp}

Consider the Euclidean space
\[
\Rb^I\ :=\ \{\text{functions $x:I\to \Rb$}\}
\]
endowed with the standard inner product
\[
\langle x,y\rangle = \sum_{i\in I} x_i y_i.
\]

Let $P$ be a \emph{convex polytope} in $\Rb^I$~\cite{Zie:1995}.
The \emph{faces} of $P$ are then convex polytopes also.
We write $Q\leq P$ to indicate that $Q$ is a face of $P$.

Given a vector $v\in \Rb^I$, let $P_v$
denote the subset of $P$ where the function 
\[
\langle v,-\rangle :P\to\Rb
\] 
achieves its maximum value. Then $P_v$ is a face of $P$.

Given a face $Q$ of $P$, its \emph{normal cone} is the set
\[
Q^{\bot} := \{v\in\Rb^I \mid P_v=Q\}.
\]
The collection of normal cones
\[
\Nc(P) := \{Q^{\bot} \mid Q\leq P\}
\]
is the \emph{normal fan} of $P$.

Let $n:=\abs{I}$. 
The \emph{standard permutahedron} $\SP_I\subseteq \Rb^I$ 
is the convex hull of the $n!$ points
\[
\{x\in\Rb^I \mid \text{ the function $x : I \to [n]$ is bijective}\}.
\]
The faces of $\SP_I$ are in bijection with compositions $F$ of $I$.
For $F=(I_1,\ldots,I_k)$, the corresponding face $\SP_F$ of $\SP_I$ has as vertices
the bijections $x : I \to [n]$ such that
\[
x(I_1) = \{1,\ldots,i_1\}, \
x(I_2) = \{i_1+1,\ldots,i_1+i_2\},\ldots,
x(I_k) = \{i_1+\cdots+i_{k-1},\ldots,n\}.
\]
The normal cone of this face is the set of vectors $v\in\Rb^I$ such that
\begin{align*}
v_i = v_j & \text{ if $i$ and $j$ belong to the same block of $F$,}\\
v_i < v_j & \text{ if the block of $i$ precedes the block of $j$ in $F$.}
\end{align*}

When $I=\{a,b,c\}$, the standard permutahedron is a regular hexagon.
It lies on the plane $x_a+x_b+x_c=6$. 
The intersection of its normal fan with this plane is
shown next to it, below.
\begin{align*}
&
\begin{gathered}
\begin{picture}(85,95)(0,-45)
{\color{blue}
 \put(0,-15){\circle*{3}} \put(0,15){\circle*{3}} 
 \put(50,-15){\circle*{3}} \put(50,15){\circle*{3}} 
 \put(25,30){\circle*{3}} \put(25,-30){\circle*{3}} 
 \thicklines
 \put(0,-15){\line(0,1){30}} \put(50,-15){\line(0,1){30}} 
 \put(0,15){\line(5,3){25}}\put(25,-30){\line(5,3){25}}
 \put(0,-15){\line(5,-3){25}} \put(25,30){\line(5,-3){25}}
 }
 \put(-35,-20){$3,2,1$} \put(-35,15){$3,1,2$} 
 \put(56,-20){$1,3,2$} \put(56,15){$1,2,3$} 
 \put(12,-45){$2,3,1$} \put(12,38){$2,1,3$} 
\end{picture}
\end{gathered}
&
\begin{gathered}
\begin{picture}(85,95)(0,-45)
\red{
\begin{xy}
0;<1cm,0cm>:<.5cm,\halfrootthree cm>::
(-1.5,0); (1.5, 0) **[|(3)]@{-} ?<*@{<} ?>*@{>},
(0,-1.5); (0,1.5) **[|(3)]@{-} ?<*@{<} ?>*@{>},
(-1.5,1.5); (1.5,-1.5) **[|(3)]@{-} ?<*@{<} ?>*@{>},
(0,0)*{\bullet}, (2.6,0.3)*{\black{v_a<v_b=v_c}}
\end{xy}
}
\end{picture}
\end{gathered}
\end{align*}

A polytope $P$ in $\Rb^I$ is a \emph{generalized permutahedron} if its normal
fan is coarser than that of the standard permutahedron. In other words, each
cone $Q^{\bot}\in\Nc(P)$ must be a union of cones in $\Nc(\SP_I)$.

A generalized permutahedron is shown below, together with its normal fan.
\begin{align*}
&
\begin{gathered}
\begin{picture}(85,100)(0,-50)
{\color{blue}
\put(0,15){\circle*{3}} \put(0,-15){\circle*{3}}
\put(50,15){\circle*{3}} % \put(50,-15){\circle*{3}}
\put(25,30){\circle*{3}} %\put(25,-30){\circle*{3}} 
% \put(50,-29){\circle*{3}} \put(36,-38){\circle*{3}} 
\put(50,-45){\circle*{3}}
\put(25,30){\circle*{3}}
\thicklines
\put(0,-15){\line(0,1){30}} \put(50,15){\line(0,-1){60}} 
\put(0,15){\line(5,3){25}} %\put(50,-30){\line(-5,-3){13}}
\put(0,-15){\line(5,-3){50}} \put(25,30){\line(5,-3){25}}
}
\end{picture}
\end{gathered}
&
\begin{gathered}
\begin{picture}(85,95)(0,-45)
\red{
\begin{xy}
0;<1cm,0cm>:<.5cm,\halfrootthree cm>::
(-1.5,0); (1.5, 0) **[|(3)]@{-}?<*@{<} ?>*@{>},
(0,-1.5); (0,1.5) **[|(3)]@{-}?<*@{<} ?>*@{>},
(-1.5,1.5); (0,0) **[|(3)]@{-}?<*@{<},
(0,0)*{\bullet}
\end{xy}
}
\end{picture}
\end{gathered}
\end{align*}

Let $\rGP[I]$ denote the (infinite) set of generalized permutahedra in $\Rb^I$.

Fix a decomposition $I=S\sqcup T$. This yields a canonical identification
\[
\Rb^I = \Rb^S \times \Rb^T.
\]

\begin{proposition}\label{p:gp-prod}
If $P$ is a generalized permutahedron in $\Rb^S$ and $Q$ is another in $\Rb^T$,
then $P\times Q$ is a generalized permutahedron in $\Rb^I$.
\end{proposition}

This holds because the normal fan of $\SP_S\times\SP_T$ is coarser than
that of $\SP_I$. This allows us to define a map
\[
\mu_{S,T} : \rGP[S]\times\rGP[T] \to \rGP[I], \quad
(P,Q) \mapsto P\times Q.
\]

\begin{proposition}\label{p:gp-coprod}
Let $P$ be a generalized permutahedron in $\Rb^I$.
Let $v\in\Rb^I$ be a vector such that
\begin{align*}
v_i = v_j & \text{ if $i$ and $j$ belong both to $S$ or both to $T$,}\\
v_i < v_j & \text{ if $i\in S$ and $j\in T$.}
\end{align*}
Then there exist generalized permutahedra $P_1$ in $\Rb^S$ and
$P_2$ in $\Rb^T$ such that
\[
P_v = P_1\times P_2.
\]
\end{proposition}

If $S$ or $T$ are empty, then $P_v=P$. Otherwise, the conditions on $v$
express precisely that $v\in (\SP_{(S,T)})^\bot$. Therefore, if $w$ is another vector 
in $(\SP_{(S,T)})^\bot$, then both 
$v$ and $w$ belong to the same cone in the fan $\Nc(P)$, and hence
$P_v=P_w$. Thus, $P_1$ and $P_2$ are independent of the choice of $v$,
and we may define a map
\[
\Delta_{S,T} : \rGP[I] \to \rGP[S]\times\rGP[T], \quad
P \mapsto (P_1,P_2).
\]

It is shown in~\cite{AguArd:2012} that, with the above structure, the species $\rGP$ is
a set-theoretic connected Hopf monoid. 
It is commutative, but not cocommutative.

Let $\tGP:=\Kb\rGP$ denote the linearization of $\rGP$. 
It is a connected Hopf monoid.
The product and coproduct are
\[
\mu_{S,T}(\BH_P\otimes\BH_Q) = \BH_{P\times Q}
\qqand
\Delta_{S,T}(\BH_P) = \BH_{P_1}\otimes\BH_{P_2}.
\]

One of the main results in~\cite{AguArd:2012} is the following.

\begin{theorem}\label{t:apode-gp}
The antipode of $\tGP$ is
\begin{equation}\label{e:apode-gp}
\apode_I(P) = (-1)^{\abs{I}}\sum_{Q \leq P} (-1)^{\dim Q} \, Q.
\end{equation}
\end{theorem}

Each face of $P$ appears once in the sum, with coefficient $\pm 1$.
The formula is thus cancellation free.

The antipode formula~\eqref{e:apode-graph} for the Hopf monoid of simple
graphs may be deduced from~\eqref{e:apode-gp} by means of a morphism
from $\tG$ to the quotient of $\tGP$ in which generalized permutahedra
with the same normal fan are identified.

\begin{remark}
Generalized permutahedra are studied in depth by Postnikov~\cite{Pos:2009},
where they are defined as \emph{deformations} of the standard permutahedron.
The definition in terms of normal fans is from work of Postnikov, Reiner, and Williams~\cite[Proposition~3.2]{PRW:2008}, where plenty of additional information is given.
The same class of polytopes is studied in Fujishige~\cite{Fuj:2005}.
The Hopf monoid of generalized permutahedra as defined above appears
in~\cite{AguArd:2012}. Several related examples are discussed in that
reference, which builds on the work of Billera, Jia, and Reiner on Hopf algebras~\cite{BJR:2006}.
\end{remark}

\section{Higher associativity and compatibility}\label{s:higher}

Let $\thh$ be a bimonoid.
The bimonoid axioms link the structure maps $\mu_{S,T}$, $\Delta_{S,T}$, 
$\iota_\emptyset$ and $\epsilon_\emptyset$ of $\thh$.
We study the implications of these on the higher structure maps.
A remarkable fact is that the four bimonoid axioms of Section~\ref{ss:bim-hopf}
can be unified into a single axiom relating the higher product and coproduct.
Associativity and unitality are similarly unified.
The combinatorics of set decompositions underlies this study
and the discussion necessitates the notions and terminology discussed in Section~\ref{ss:decompositions}, 
including those of concatenation, refinement, splitting, 
and Tits product for set decompositions.

\subsection{Higher associativity}\label{ss:gen-asso}

Given a species $\tp$, a finite set $I$ and a decomposition $F=(I_1,\ldots,I_k)$ of $I$, write
\begin{equation}\label{e:sp-decomp}
\tp(F) := \tp[I_1] \otimes \dots \otimes \tp[I_k].
\end{equation}
(This generalizes~\eqref{e:sp-comp}.)
When $F=\emptyset^0$ (the unique decomposition of the empty set without subsets), 
we set $\tp(F):=\Kb$.

Given $I=S\sqcup T$ and decompositions $F$ of $S$ and
$G$ of $T$, there is a canonical isomorphism
\begin{equation}\label{e:conc-iso}
\tp(F)\otimes \tp(G)\map{\cong} \tp(F\cdot G)
\end{equation}
obtained by concatenating the factors in~\eqref{e:sp-decomp}.
If $F=\emptyset^0$, then this is the canonical isomorphism
\[
\Kb \otimes \tp(G)\map{\cong} \tp(G);
\] 
a similar remark applies if $G=\emptyset^0$.

Assume now that $\ta$ is a monoid and $F$ is a decomposition of $I$.
We write 
\[
\mu_F : \ta(F) \to \ta[I]
\]
for the higher product maps of $\ta$, as in~\eqref{e:iterated-mu}.
Recall that $\mu_{\emptyset^0} = \iota_\emptyset$.

Given a bijection $\sigma:I\to J$, let $J_i:=\sigma(I_i)$, $\sigma_i:=\sigma|_{I_i}$,
$\sigma(F)=(J_1,\ldots,J_k)$, and $\ta(\sigma)$ the map
\[
\begin{gathered}
\xymatrix@C+60pt{
\ta(F) \ar@{.>}[r]^-{\ta(\sigma)} \ar@{=}[d] & \ta\bigl(\sigma(F)\bigr)\\
\ta[I_1]\otimes\cdots\otimes\ta[I_k] \ar[r]_-{\ta[\sigma_1]\otimes\cdots\otimes\ta[\sigma_k]} & 
\ta[J_1]\otimes\cdots\otimes\ta[J_k]. \ar@{=}[u]
}
\end{gathered}
\]
The collection $\mu_F$ satisfies the following \emph{naturality} condition.
For any such $F$ and $\sigma$, the diagram
\begin{equation}\label{e:gen-nat}
\begin{gathered}
\xymatrix@C+30pt{
\ta(F) \ar[r]^-{\mu_F} \ar[d]_{\ta(\sigma)} & \ta[I] \ar[d]^{\ta[\sigma]}\\
\ta\bigl(\sigma(F)\bigr) \ar[r]_-{\mu_{\sigma(F)}} & \ta[J]
}
\end{gathered}
\end{equation}
commutes. This follows from~\eqref{e:nat}.

Let $F$ still be as above.
Let $G$ be a decomposition of $I$ such that $F\leq G$ and
$\gamma=(G_1,\dots ,G_k)$ a splitting of the pair $(F,G)$.
In other words, each $G_j$ is a decomposition of $I_j$ and $G=G_1\cdots G_k$. 
Define a map $\mu_{G/F}^\gamma$ by means of the diagram
\begin{equation}\label{e:muFG}
\begin{gathered}
\xymatrix@C+50pt{
\ta(G) \ar@{.>}[r]^-{\mu_{G/F}^\gamma} \ar[d]_{\cong} & \ta(F)\\
\ta(G_1)\otimes\cdots\otimes\ta(G_k) \ar[r]_-{\mu_{G_1}\otimes\cdots\otimes\mu_{G_k}} & 
\ta[I_1]\otimes\cdots\otimes\ta[I_k], \ar@{=}[u]
}
\end{gathered}
\end{equation}
where the vertical isomorphism is an iteration of~\eqref{e:conc-iso}. If $F=\emptyset^0$,
then necessarily $G=\emptyset^0$ and $\gamma=(\,)$, and we agree
that the map $\mu_{G/F}^\gamma$ is the identity of $\Kb$.

Recall that if $F$ and $G$ are compositions, then the splitting $\gamma$
is uniquely determined by $(F,G)$. In this situation
we write $\mu_{G/F}$ instead of $\mu_{G/F}^\gamma$.

\emph{Higher associativity} is the following result.

\begin{proposition}\label{p:gen-asso}
Let $\ta$ be a monoid.
Let $F$ and $G$ be decompositions of $I$ with $F\leq G$.
Then, for any splitting $\gamma$ of the pair $(F,G)$, the diagram
\begin{equation}\label{e:gen-asso}
\begin{gathered}
\xymatrix{
\ta(G) \ar[r]^-{\mu_G} \ar[d]_{\mu_{G/F}^\gamma} & \ta[I]\\
\ta(F) \ar[ru]_{\mu_F}
}
\end{gathered}
\end{equation}
commutes.
\end{proposition}

Considering on the one hand $F=(R\sqcup S,T)$, $G_1=(R,S)$, $G_2=(T)$,
and on the other $F=(R, S\sqcup T)$, $G_1=(R)$, $G_2=(S,T)$, we see that
~\eqref{e:gen-asso} implies the associativity axiom~\eqref{e:assoc}.
 The unit axioms~\eqref{e:unit} also follow from~\eqref{e:gen-asso}; 
 to obtain the left unit axiom
one chooses $F=(\emptyset,I)$, $G_1=\emptyset^0$ and $G_2=(I)$.
Conversely,~\eqref{e:gen-asso} follows by repeated use of~\eqref{e:assoc} and~\eqref{e:unit}.

Thus, higher associativity encompasses both associativity and unitality. 
In fact, the preceding discussion leads to the following result.

\begin{proposition}\label{p:gen-asso-conv}
Let $\ta$ be a species equipped with a collection of maps
\[
\mu_F: \ta(F) \to \ta[I],
\]
one map for each decomposition $F$ of a finite set $I$. Then $\ta$
is a monoid with higher product maps $\mu_F$ if and only if 
naturality~\eqref{e:gen-nat} and higher associativity~\eqref{e:gen-asso} hold.
\end{proposition}

\subsection{Higher unitality}\label{ss:gen-unit}

We continue to assume that $\ta$ is a monoid and
$F=(I_1,\ldots,I_k)$ is a decomposition of $I$.
Recall that $\pos{F}$ denotes the composition of $I$ obtained by removing
from $F$ those subsets $I_i$ which are empty. 
For any such $F$, there is a canonical isomorphism
\[
\ta(\pos{F}) \cong V_1\otimes\cdots\otimes V_k,
\]
where 
\[
V_i := \begin{cases}
\ta[I_i] & \text{ if $I_i\neq\emptyset$,}\\
\Kb & \text{ if $I_i=\emptyset$.}
\end{cases}
\]
For each $i=1,\ldots,k$, let $\iota_i:V_i\to \ta[I_i]$ be the map defined by
\[
\iota_i := \begin{cases}
\id & \text{ if $I_i\neq\emptyset$,}\\
\iota_\emptyset & \text{ if $I_i=\emptyset$,}
\end{cases}
\]
where $\iota_\emptyset$ is the unit map of $\ta$.
We then define a map $\iota_F$ as follows.
\[
\xymatrix@C+50pt{
\ta(\pos{F}) \ar@{.>}[r]^-{\iota_F} \ar[d]_{\cong} & \ta(F) \ar@{=}[d]\\
V_1\otimes\cdots \otimes V_k \ar[r]_-{\iota_1\otimes\cdots\otimes\iota_k} & \ta[I_1]\otimes\cdots\otimes\ta[I_k]
}
\]

The following result is \emph{higher unitality}.

\begin{proposition}\label{p:gen-unit}
Let $\ta$ be a monoid.
For any decomposition $F$ of $I$, the diagram
\begin{equation}\label{e:gen-unit}
\begin{gathered}
\xymatrix{
\ta(F) \ar[r]^-{\mu_F} & \ta[I]\\
\ta(\pos{F}) \ar[ru]_{\mu_{\pos{F}}} \ar[u]^{\iota_{F}} 
}
\end{gathered}
\end{equation}
commutes.
\end{proposition}

This result is in fact a special case of that in Proposition~\ref{p:gen-asso}.
To see this, note that for any decomposition $F$ of $I$
there is a canonical splitting $\can=(G_1,\dots ,G_k)$ of $(F,F_+)$;
namely, 
\[
G_j = \begin{cases}
\emptyset^0 & \text{ if $I_j=\emptyset$,}\\
(I_j) & \text{ otherwise.}
\end{cases}
\]
One then has that
\[
\iota_{F} = \mu_{F_+/F}^{\can}
\]
and therefore~\eqref{e:gen-unit} is the special case of~\eqref{e:gen-asso} in which $G=F_+$ and $\gamma=\can$.

We emphasize that, in view of the preceding discussion, higher unitality is
a special case of higher associativity.

\subsection{Higher commutativity}\label{ss:gen-comm}

Let $F=(I_1,\ldots,I_k)$ and $F'=(I'_1,\ldots,I'_k)$ be two
decompositions of $I$ which consist of the same subsets, 
possibly listed in different orders. 
Let $\sigma\in\Sr_k$ be any permutation such that 
\[
I'_i = I_{\sigma(i)}
\]
for each $i=1,\ldots,k$. The permutation $\sigma$ may not be unique due to
the presence of empty subsets. 
For any species $\tp$, there is an isomorphism
\[
\tp(F) \cong \tp(F'),\quad
x_1\otimes\cdots\otimes x_k \mapsto x_{\sigma(1)}\otimes\cdots\otimes x_{\sigma(k)}
\]
obtained by reordering the tensor factors according to $\sigma$. 
Fix a scalar $q\in\Kb$. Let
\begin{equation}\label{e:gen-braiding}
\beta_q^{\sigma}: \tp(F) \to \tp(F')
\end{equation}
be the scalar multiple of the previous map by the factor 
$
q^{\dist(F,F')},
$
with the distance function as in~\eqref{e:dist-supp}. 

If $F$ and $F'$ are compositions of $I$, then $\sigma$
is unique, and we write $\beta_q$ instead of $\beta_q^{\sigma}$.

We then have the following \emph{higher $q$-commutativity}.

\begin{proposition}\label{p:gen-comm}
Let $\ta$ be a $q$-commutative monoid.
For any decompositions $F$ and $F'$ of $I$ and permutation $\sigma$ as above,
the diagram
\begin{equation}\label{e:gen-comm}
\begin{gathered}
\xymatrix@=1.5pc{
\ta(F) \ar[dr]_{\mu_{F}} \ar[rr]^-{\beta_q^{\sigma}} & & \ta(F') \ar[dl]^{\mu_{F'}}\\
 & \ta[I] &\\
}
\end{gathered}
\end{equation}
commutes.
\end{proposition}

When $F=(S,T)$, $F'=(T,S)$ and
$\sigma$ is the transposition $(1,2)$,
the map~\eqref{e:gen-braiding} reduces to~\eqref{e:braiding}, and~\eqref{e:gen-comm}
specializes to~\eqref{e:comm}. Conversely,~\eqref{e:gen-comm} follows by
repeated use of~\eqref{e:comm}.

\begin{proposition}\label{p:gen-comm-conv}
In the situation of Proposition~\ref{p:gen-asso-conv}, the monoid $\ta$ is
commutative if and only if~\eqref{e:gen-comm} holds.
\end{proposition}

\subsection{Higher coassociativity and counitality}\label{ss:gen-coass}

Similar considerations to those in the preceding sections apply to a comonoid $\tc$. 
In particular, given a splitting $\gamma$ of the pair $(F,G)$, there is a map
\[
\Delta_{G/F}^\gamma: \tc(F) \to \tc(G).
\]
For any decomposition $F$, there is a map
\[
\epsilon_F: \tc(F) \to \tc(\pos{F});
\]
we have $\epsilon_{F} = \Delta_{F_+/F}^{\can}$.
If $F$ and $G$ are compositions, then there is a unique choice $\gamma$ and we write
$\Delta_{G/F}$ instead of $\Delta_{G/F}^\gamma$.

The diagrams dual to~\eqref{e:gen-asso} and~\eqref{e:gen-unit} commute. 
A species $\tc$ equipped with a collection of maps
\[
\Delta_F: \tc[I] \to \tc(F)
\]
is a comonoid with higher coproduct maps $\Delta_F$ if and only if the dual
conditions to~\eqref{e:gen-nat} and~\eqref{e:gen-asso} hold.
The comonoid is cocommutative if and only if the dual of~\eqref{e:gen-comm} holds.

\subsection{Higher compatibility}\label{ss:gen-comp}

We turn to the compatibility between the higher product and coproduct of a $q$-bimonoid.
The Tits product plays a crucial role in this discussion.

Let $F=(S_1,\ldots,S_p)$ and $G=(T_1,\ldots,T_q)$ 
be two decompositions of $I$ and consider the situation 
in~\eqref{e:pqsets}. We see from~\eqref{e:tits}
that the Tits products $FG$ and $GF$ may only differ in the order of the subsets.
Moreover, there is a canonical permutation `$\can$' that reorders the
subsets $A_{ij}$ in $FG$ into those in $GF$, namely $\can(i,j)=(j,i)$.

Let $\tp$ be a species and let
\[
\beta_q^{\can} : \tp(FG) \to \tp(GF)
\]
be the map~\eqref{e:gen-braiding} for this particular choice of reordering.
When $F=(S_1,S_2)$ and $G=(T_1,T_2)$, the map $\beta_q^{\can}$ 
coincides with the map in the top of diagram~\eqref{e:comp}.

For any decompositions $F$ and $G$ of $I$,
there is a canonical splitting $(H_1,\ldots,H_p)$ for $(F,FG)$.
Namely, 
\begin{equation}\label{e:row-splitting}
H_i:= (A_{i1},\ldots,A_{iq}),
\end{equation}
the $i$-th row in the matrix~\eqref{e:pqsets}.
Let `$\can$' also denote this splitting.

We then have the following \emph{higher compatibility}.

\begin{proposition}\label{p:gen-comp}
Let $\thh$ be a $q$-bimonoid.
For any decompositions $F$ and $G$ of $I$, the diagram
\begin{equation}\label{e:gen-comp}
\begin{gathered}
\xymatrix@R+1pc{
\thh(FG) \ar[rr]^-{\beta_q^{\can}} & & \thh(GF) \ar[d]^{\mu_{GF/G}^{\can}}\\
\thh(F) \ar[r]_-{\mu_F} \ar[u]^{\Delta_{FG/F}^{\can}} & \thh[I] \ar[r]_-{\Delta_G} & \thh(G)\\
}
\end{gathered}
\end{equation}
commutes.
\end{proposition}

Consider the case $F=(S_1,S_2)$ and $G=(T_1,T_2)$. 
Let $A$, $B$, $C$, and $D$ be as in~\eqref{e:4sets}. 
Then $FG=(A,B,C,D)$, $GF=(A,C,B,D)$, 
and~\eqref{e:gen-comp} specializes to the compatibility axiom~\eqref{e:comp}.

Now consider $F=\emptyset^0$ and $G=\emptyset^2$. Then $FG=GF=\emptyset^0$,
and following~\eqref{e:muFG} we find that the maps $\Delta_{FG/F}^\gamma$ and
$\mu_{GF/G}^\gamma$ are as follows.
\begin{align*}
&  
\begin{gathered}
\xymatrix@C+55pt{
\thh(\emptyset^0) \ar@{.>}[r]^-{\Delta_{FG/F}^\gamma=\Delta_{\emptyset^0/\emptyset^0}} 
\ar@{=}[d] & \thh(\emptyset^0)\\
\Kb \ar[r]_-{\id} & 
\Kb \ar@{=}[u]
}
\end{gathered}
& & 
\begin{gathered}
\xymatrix@C+55pt{
\thh(\emptyset^0) \ar@{.>}[r]^-{\mu_{GF/G}^\gamma=\mu_{\emptyset^0/\emptyset^2}^{(\emptyset^0,\emptyset^0)}} \ar[d]_{\cong} & \thh(\emptyset^2)\\
\Kb \otimes \Kb \ar[r]_-{\iota_\emptyset \otimes \iota_\emptyset } & 
\thh[\emptyset]\otimes\thh[\emptyset] \ar@{=}[u]
}
\end{gathered}
\end{align*}
Also, $\mu_F=\mu_{\emptyset^0}=\iota_{\emptyset}$ and 
$\Delta_G=\Delta_{\emptyset,\emptyset}$. It follows that
~\eqref{e:gen-comp} specializes in this case
to the right diagram in the compatibility axiom~\eqref{e:unitr}.

Similarly,
$F=\emptyset^2$ and $G=\emptyset^0$ yields the left diagram in~\eqref{e:unitr},
and $F=\emptyset^0$ and $G=\emptyset^0$ yields~\eqref{e:inverser}.

Thus, all compatibility axioms~\eqref{e:comp}--\eqref{e:inverser}
are special cases of~\eqref{e:gen-comp}. 
Conversely, the latter follows from the former.

\begin{proposition}\label{p:gen-comp-conv}
Let $\thh$ be a species equipped with two collections of maps
\[
\mu_F: \thh(F) \to \thh[I]
\qqand
\Delta_F:\thh[I] \to \thh(F),
\]
one map for each decomposition $F$ of a finite set $I$. Then $\thh$
is a bimonoid with higher (co)product maps $\mu_F$ ($\Delta_F$) if and only if 
the following conditions hold:
naturality~\eqref{e:gen-nat}, higher associativity~\eqref{e:gen-asso}, their duals, and
higher compatibility~\eqref{e:gen-comp}.
\end{proposition}

The unit conditions have been subsumed under~\eqref{e:gen-asso} and~\eqref{e:gen-comp}. 
Nevertheless, certain compatibilities involving the maps $\iota_F$ and $\epsilon_F$ are worth-stating.

\begin{proposition}\label{p:gen-comp2}
Let $\thh$ be a $q$-bimonoid.
For any decomposition $F$ of $I$, the diagram
\begin{equation}\label{e:gen-comp2}
\begin{gathered}
\xymatrix@=1.5pc{
& \thh(F) \ar[dr]^{\epsilon_F}\\
\thh(\pos{F}) \ar@{=}[rr] \ar[ur]^{\iota_F} & & \thh(\pos{F}) 
}
\end{gathered}
\end{equation}
commutes.
\end{proposition}

Axiom~\eqref{e:inverser} is the special case of~\eqref{e:gen-comp2} in which
$F=\emptyset^1$. 

For the next result, we consider a special type of refinement for set decompositions.
Let $F=(I_1,\ldots,I_k)$ be a decomposition of $I$.
Let $G$ be a decomposition of $I$ refining $F$ and with the property that
the nonempty subsets in $F$ are refined by $G$ into nonempty subsets,
and each empty set in $F$ is refined by $G$ into a nonnegative number of empty sets.
In this situation, there is a canonical choice of splitting $(G_1,\ldots,G_k)$ 
for $F\leq G$.
Namely, we choose each $G_i$ to consist either of a maximal run of contiguous nonempty sets in $G$,
or of a maximal run of contiguous empty sets. 
We denote this splitting by $\max$.

\begin{proposition}\label{p:gen-comp3}
Let $\thh$ be a $q$-bimonoid. 
For any decompositions $F$ and $G$ as above, the diagrams
\begin{align}\label{e:gen-comp3}
& 
\begin{gathered}
\xymatrix@C+2pc{
\thh(G) \ar[r]^-{\epsilon_G}\ar[d]_{\mu_{G/F}^{\mx}} &
\thh(\pos{G}) \ar[d]^{\mu_{\pos{G}/\pos{F}}}\\
\thh(F) \ar[r]_-{\epsilon_{\pos{F}}} & \thh(\pos{F})\\
}
\end{gathered}
& & 
\begin{gathered}
\xymatrix@C+2pc{
\thh(\pos{F}) \ar[d]_{\Delta_{\pos{G}/\pos{F}}} \ar[r]^-{\iota_F} & \thh(F)
\ar[d]^{\Delta_{G/F}^{\mx}}\\
\thh(\pos{G}) \ar[r]_-{\iota_G} & \thh(G)
}
\end{gathered}
\end{align}
commute.
\end{proposition}

Axioms~\eqref{e:unitr} are the special case of~\eqref{e:gen-comp3} in which
$F=(\emptyset)$ and $G=(\emptyset,\emptyset)$.

\subsection{The connected case}\label{ss:conn-higher}

The preceding discussion carries over to connected species.
The statements are simpler. There are additional results that are
specific to this setting. 

First of all, note that a connected species $\tp$ can be recovered from its positive part
by setting $\tp[\emptyset]=\Kb$.

If $\ta$ is a connected monoid, then for any decomposition $F$ 
the map $\iota_F:\ta(\pos{F})\to\ta(F)$ is invertible. 
Together with~\eqref{e:gen-unit} this implies that the collection of 
higher product maps $\mu_F$ is uniquely determined by the subcollection 
indexed by compositions $F$.

\begin{proposition}\label{p:gen-asso-conv-conn}
Let $\ta$ be a connected species equipped with a collection of maps
\[
\mu_F: \ta(F) \to \ta[I],
\]
one map for each composition $F$ of a nonempty finite set $I$. 
Then $\ta$ is a connected monoid with higher product maps $\mu_F$ if and only if 
naturality~\eqref{e:gen-nat} holds and the diagram
\begin{equation}\label{e:gen-asso-conn}
\begin{gathered}
\xymatrix{
\ta(G) \ar[r]^-{\mu_G} \ar[d]_{\mu_{G/F}} & \ta[I]\\
\ta(F) \ar[ru]_{\mu_F}
}
\end{gathered}
\end{equation}
commutes, for each compositions $F$ and $G$ of $I$ with $F\leq G$.
The monoid $\ta$ is commutative if and only if
\begin{equation}\label{e:gen-comm-conn}
\begin{gathered}
\xymatrix{
\ta(F) \ar[dr]_{\mu_{F}} \ar[rr]^-{\beta_q} & & \ta(F') \ar[dl]^{\mu_{F'}}\\
 & \ta[I] &\\
}
\end{gathered}
\end{equation}
commutes, for any compositions $F$ and $F'$ with $\supp F=\supp F'$.
\end{proposition}

This follows from Propositions~\ref{p:gen-asso-conv} and~\ref{p:gen-comm}.
Since splittings are unique for compositions, diagrams~\eqref{e:gen-asso}
and~\eqref{e:gen-comm}
simplify to~\eqref{e:gen-asso-conn} and~\eqref{e:gen-comm-conn}.

For a connected comonoid $\tc$, the map $\epsilon_F:\tc(F)\to\tc(\pos{F})$ is invertible,
and we have the dual statement to Proposition~\ref{p:gen-asso-conv-conn}.

\begin{proposition}\label{p:gen-comp-conv-conn}
Let $\thh$ be a connected species equipped with two collections of maps
\[
\mu_F: \thh(F) \to \thh[I]
\qqand
\Delta_F:\thh[I] \to \thh(F),
\]
one map for each composition $F$ of a nonempty finite set $I$. Then $\thh$
is a bimonoid with higher (co)product maps $\mu_F$ ($\Delta_F$) if and only if 
the following conditions hold:
naturality~\eqref{e:gen-nat}, higher associativity~\eqref{e:gen-asso-conn}, their duals, and the diagram
\begin{equation}\label{e:gen-comp-conn}
\begin{gathered}
\xymatrix@R+1pc{
\thh(FG) \ar[rr]^-{\beta_q} & & \thh(GF) \ar[d]^{\mu_{GF/G}}\\
\thh(F) \ar[r]_-{\mu_F} \ar[u]^{\Delta_{FG/F}} & \thh[I] \ar[r]_-{\Delta_G} & \thh(G)\\
}
\end{gathered}
\end{equation}
commutes for any pair of compositions $F$ and $G$ of a nonempty finite set $I$.
\end{proposition}

This follows from Proposition~\ref{p:gen-comp}. 
The Tits products $FG$ and $GF$ in~\eqref{e:gen-comp-conn} are those for set compositions (Section~\ref{ss:faceprod}).

Connectedness has further implications on the higher structure maps.

\begin{proposition}\label{p:hopf-split-gen2}
Let $\thh$ be a connected $q$-bimonoid.
For any decomposition $F$ of $I$, the diagram
\begin{equation}\label{e:hopf-split-gen2}
\begin{gathered}
\xymatrix@=1.5pc{
\thh(F) \ar@{=}[rr] \ar[dr]_{\epsilon_F} & & \thh(F)\\
 & \thh(\pos{F}) \ar[ur]_{\iota_F} 
}
\end{gathered}
\end{equation}
commutes.
\end{proposition}

Together with Proposition~\ref{p:gen-comp2}, this says that for a connected
$q$-bimonoid the maps $\iota_F$ and $\epsilon_F$ are inverse.

\begin{proposition}\label{p:hopf-split-gen-gen}
Let $\thh$ be a connected $q$-bimonoid.
For any compositions $F$ and $G$ of $I$ with $\supp F=\supp G$, the diagram
\begin{equation}\label{e:hopf-split-supp}
\begin{gathered}
\xymatrix@=1.5pc{
\thh(F) \ar[rr]^-{\beta_q} \ar[dr]_{\mu_F} & & \thh(G)\\
 & \thh[I] \ar[ur]_{\Delta_G }
 }
\end{gathered}
\end{equation}
commutes.
\end{proposition}

This follows from~\eqref{e:gen-comp-conn}, since in this case $FG=F$ and $GF=G$.

The following special cases are worth-stating.

\begin{corollary}\label{c:hopf-split}
Let $\thh$ be a connected $q$-bimonoid. 
For any composition $F$ of $I$, the diagrams
\begin{align*}%\label{e:hopf-split}
&
\begin{gathered}
\xymatrix@=1.5pc{
\thh(F) \ar@{=}[rr] \ar[dr]_{\mu_F} & & \thh(F)\\
& \thh[I] \ar[ur]_{\Delta_F} 
}
\end{gathered}
& &
\begin{gathered}
\xymatrix@=1.5pc{
\thh(F) \ar[rr]^-{\beta^{\omega}_q} \ar[dr]_{\mu_F} & & \thh(\opp{F})\\
& \thh[I] \ar[ur]_{\Delta_{\opp{F}} }
}
\end{gathered}
\end{align*}
commute. 
\end{corollary}

These follow from~\eqref{e:hopf-split-supp} 
by setting $G=F$ and $G=\opp{F}$.
The map $\beta^{\omega}_q$ is as in~\eqref{e:gen-braiding} for $\omega$ the 
longest permutation. Explicitly, the conditions in Corollary~\ref{c:hopf-split} are
\begin{equation}\label{e:hopf-split2}
\Delta_F\mu_F(h) = h \qqand
\Delta_{\opp{F}}\mu_F(h) = q^{\dist(F,\opp{F})}\,\opp{h}
\end{equation}
for any $h\in\thh(F)$, where, if $h=x_1\otimes\cdots\otimes x_k$, then
$\opp{h}:=x_k\otimes\cdots\otimes x_1$.

\section{The bimonoids of set (de)compositions}\label{s:faces}

We now discuss two important bimonoids $\tSig$ and $\tSigh$ 
indexed by set compositions and set decompositions respectively.
The former is connected and hence a Hopf monoid,
while the latter is not a Hopf monoid.
Both structures are set-theoretic.
We also write down formulas for the higher (co)product maps.
This is particularly compelling since the higher (co)product maps
are also indexed by set decompositions
(and the Tits product appears in the coproduct description).
 
We make use of notions and notation for (de)compositions discussed
in Section~\ref{s:prelim}. We work over a characteristic $0$ field.

\subsection{The Hopf monoid of set compositions}\label{ss:faces}

Recall that $\rSig[I]$ denotes the set of compositions of a finite set $I$.

The species $\rSig$ \emph{of compositions}
is a set-theoretic connected Hopf monoid with product and coproduct defined by
\begin{equation}\label{e:faces-def-set}
\begin{aligned}
\rSig[S]\times\rSig[T] & \map{\mu_{S,T}} \rSig[I] & \rSig[I] & \map{\Delta_{S,T}} \rSig[S]\times\rSig[T]\\
(F_1,F_2) & \xmapsto{\phantom{\mu_{S,T}}} F_1\cdot F_2 \qquad
& F & \xmapsto{\phantom{\Delta_{S,T}}} (F|_S,F|_T).
\end{aligned}
\end{equation}
We have employed the operations of
concatenation and restriction of compositions.

The Hopf monoid $\rSig$ is cocommutative but not commutative.

We let $\tSig:=\Kb\rSig$ denote the linearization of $\rSig$.
It is a connected Hopf monoid. The product and coproduct are
\begin{equation}\label{e:faces-def}
\mu_{S,T}(\BH_{F_1}\otimes\BH_{F_2}) = \BH_{F_1\cdot F_2}
\qqand
\Delta_{S,T}(\BH_{F}) = \BH_{F|_S}\otimes\BH_{F|_T}.
\end{equation}
For the antipode of $\tSig$ on the basis $\{\BH\}$, set $q=1$ in Theorem~\ref{t:apode-face} below.

For the dual Hopf monoid $\tSig^*=\Kb^\rSig$, we have
\[
\mu_{S,T}(\BM_{F_1}\otimes\BM_{F_2}) = \sum_{F:\, \text{$F$ a quasi-shuffle of
$F_1$ and $F_2$}} \BM_{F}
\]
and
\[
\Delta_{S,T}(\BM_{F}) =
\begin{cases}
\BM_{F|_S}\otimes\BM_{F|_T} & \text{ if $S$ is an initial segment of $F$,}\\
0 & \text{ otherwise.}
\end{cases}
\]

The monoid $\tSig$ is free on the positive part of the exponential species: 
there is an isomorphism of monoids
\[
\Tc(\wE_+) \cong \tSig, \quad
 \BH_{I_1}\otimes\cdots\otimes\BH_{I_k} \leftrightarrow \BH_F
\]
for $F=(I_1,\ldots,I_k)$ a composition of $I$.
The isomorphism does not preserve coproducts, since $\Delta_{S,T}( \BH_{I_1}\otimes\cdots\otimes\BH_{I_k})=0$ when $S$ is not $F$-admissible.

\smallskip

Define a linear basis $\{\BQ\}$ of $\tSig[I]$ by
\begin{equation}\label{e:H-Q}
\BH_F = \sum_{G:\,F\leq G} \frac{1}{(G/F)!}\, \BQ_G.
\end{equation}
The basis $\{\BQ\}$ is determined by triangularity. In fact, one may show that
\begin{equation}\label{e:Q-H}
\BQ_F = \sum_{G:\,F\leq G} (-1)^{\len(G)-\len(F)} \frac{1}{\len(G/F)}\, \BH_G.
\end{equation}

\begin{proposition}\label{p:Qface}
We have
\begin{gather}
\label{e:Qprod}
\mu_{S,T}(\BQ_{F_1}\otimes\BQ_{F_2}) = \BQ_{F_1\cdot F_2},\\
\label{e:Qcoprod}
\Delta_{S,T}(\BQ_{F}) =
\begin{cases}
\BQ_{F|_S}\otimes\BQ_{F|_T} & \text{ if $S$ is $F$-admissible,}\\
0 & \text{ otherwise,}
\end{cases}\\
\label{e:Qant}
\apode_I(\BQ_F) = (-1)^{\abs{F}} \, \BQ_{\opp{F}}.
\end{gather}
\end{proposition}
%The antipode formula is a special case of~\eqref{e:free-apode}.

There is an isomorphism of Hopf monoids
\[
\Tc(\wE_+) \cong \tSig, \quad
 \BH_{I_1}\otimes\cdots\otimes\BH_{I_k} \leftrightarrow \BQ_F
\]
for $F=(I_1,\ldots,I_k)$ a composition of $I$.
It follows that the primitive part of $\tSig$
is the Lie submonoid generated by the elements $\BQ_{(I)}$.
For instance,
\[
\BQ_{(I)}, \ \BQ_{(S,T)}-\BQ_{(T,S)}, \
\BQ_{(R,S,T)} - \BQ_{(R,T,S)} - \BQ_{(S,T,R)} + \BQ_{(T,S,R)}
\]
are primitive elements.

\smallskip

Fix a scalar $q\in\Kb$. Let $\tSig_q$ denote the same monoid as $\tSig$,
but endowed with the following coproduct:
\[
\Delta_{S,T}: \tSig_q[I] \to \tSig_q[S] \otimes \tSig_q[T],
\qquad
\BH_{F} \mapsto q^{\area_{S,T}(F)}\, \BH_{F|_S} \otimes \BH_{F|_T},
\]
where the Schubert cocycle $\area_{S,T}$ is as in~\eqref{e:schubert-comp}.
Then $\tSig_q$ is a connected $q$-Hopf monoid. 

\begin{theorem}\label{t:apode-face}
The antipode of $\tSig_q$ is given by
\[
\apode_I(\BH_F) = q^{\dist(F,\opp{F})} \sum_{G:\, \opp{F} \leq G} (-1)^{\len(G)} \, \BH_G.
\]
\end{theorem}

This result is proven in~\cite[Theorems~12.21 and~12.24]{AguMah:2010}.
The distance function is as in~\eqref{e:dist-opp}.

Generically, the $q$-Hopf monoid $\tSig_q$ is self-dual. 
In fact, we have the following result. 
Define a map $\isolinear_q:\tSig_q \to (\tSig_q)^*$ by
\begin{equation}\label{e:Sigsdual}
 \isolinear_q(\BH_{F}) := \sum_{F'} (FF')!\, q^{\dist(F,F')}\, \BM_{F'}.
\end{equation}

\begin{proposition}\label{p:Sigsdual}
The map $\isolinear_q:\tSig_q \to (\tSig_q)^*$ 
is a morphism of $q$-Hopf monoids.
If $q$ is not an algebraic number, then it is an isomorphism. 
Moreover, for any $q$,
$\isolinear_q = (\isolinear_q)^*$.
\end{proposition}

Proposition~\ref{p:Sigsdual} is proven in~\cite[Proposition~12.26]{AguMah:2010}.
See~\cite[Section~10.15.2]{AguMah:2010} for related information.

With $q=1$, the map in~\eqref{e:Sigsdual} is not invertible. It satisfies
\[
 \isolinear(\BQ_F) = (-1)^{\abs{I}-\len(F)}\, (\supp F)\ifac 
 \sum_{G:\,\supp G =\supp F} \BP_G,
\]
where $\{\BP\}$ is the basis of $\tSig^*$ dual to the basis $\{\BQ\}$ of $\tSig$.
The coefficient is the M\"obius function value~\eqref{e:partmobiusPi}.
The support of a composition is defined in Section~\ref{ss:support}. 

Regarding a linear order as a composition into singletons, we may view $\rL$
as a subspecies of $\rSig$.
The analogy between the present discussion and that for linear orders in
Section~\ref{ss:linear} is explained by the observation that in fact 
$\wL_q$ is in this manner a $q$-Hopf submonoid of $\tSig_q$. 
Let $j:\wL_q\into\tSig_q$ be the inclusion. The diagram
\[
\xymatrix{
\tSig_q \ar[r]^-{\isoflat_q} & (\tSig_q)^* \ar[d]^{j^*}\\
\wL_q \ar[r]_-{\isoflat_q} \ar[u]^{j} & (\wL_q)^* 
}
\]
commutes. 
The maps $\isoflat$ are as in~\eqref{e:lsdual} and~\eqref{e:Sigsdual}.

Let $\pi: \tSig \onto \tPi$ be the map defined by
\begin{equation}\label{e:Sig-Pi}
\pi(\BH_F) := \BH_{\supp F}.
\end{equation}

\begin{proposition}\label{p:Sig-Pi}
The map $\pi: \tSig \onto \tPi$ is a surjective morphism of Hopf monoids.
Moreover, we have
\[
\pi(\BQ_F) = \BQ_{\supp F}
\]
and the diagram
\[
\xymatrix{
\tSig \ar[r]^-{\isoflat} \ar[d]_{\pi} & \tSig^*\\
\tPi \ar[r]_-{\isoflat} & \tPi^* \ar[u]_{\pi^*}
}
\]
commutes.
\end{proposition}

The maps $\isoflat$ are as in~\eqref{e:pisdual} and~\eqref{e:Sigsdual} (with $q=1$).
The map $\pi$ is an instance of abelianization: 
in the notation of Section~\ref{ss:abel}, $\pi=\pi_{\wE_+}$.

\subsection{The bimonoid of set decompositions}\label{ss:dec-bimonoid}

Some of the considerations of Section~\ref{ss:faces} can be extended to all 
decompositions of finite sets, as opposed to only those into nonempty subsets (compositions).

Recall that $\rSigh[I]$ denotes the set of decompositions of a finite set $I$.
The set-theoretic bimonoid structure on $\rSig$ can be extended to $\rSigh$ by means
of the same formulas as in~\eqref{e:faces-def-set}: the product is given by concatenation
and the coproduct by restriction of decompositions.

The set $\rSigh[\emptyset]$ consists of the decompositions $\emptyset^p$, where
$p$ is a nonnegative integer.
As with any set-theoretic bimonoid, this set is an
ordinary monoid under $\mu_{\emptyset,\emptyset}$
and the map $\Delta_{\emptyset,\emptyset}$ must be the diagonal.
The latter statement is witnessed by the fact that 
$\emptyset^p|_\emptyset = \emptyset^p$ for all $p$. 
The product $\mu_{\emptyset,\emptyset}$ is concatenation, so
as discussed in~\eqref{e:empty-monoid}, $\rSigh[\emptyset]$ is isomorphic to
the monoid of nonnegative integers under addition.
Since this is not a group, the set-theoretic bimonoid $\rSigh$ is not a Hopf monoid.

Let $\tSigh:=\Kb\rSigh$. Then $\tSigh$ is a bimonoid but not a Hopf monoid.
The product and coproduct are as for $\tSig$:
\[\mu_{S,T}(\BH_{F_1}\otimes\BH_{F_2}) = \BH_{F_1\cdot F_2}
\qqand
\Delta_{S,T}(\BH_{F}) = \BH_{F|_S}\otimes\BH_{F|_T},
\]
where $F_1$, $F_2$ and $F$ are decompositions.

The map $\upsilon: \tSigh \to \tSig$ defined by
\begin{equation}\label{e:dec-comp}
\upsilon(\BH_F) := \BH_{\pos{F}}
\end{equation}
is a surjective morphism of bimonoids. (It is set-theoretic.)

\subsection{The higher (co)product in $\tSig$ and $\tSigh$}\label{ss:it-sig}

%Consider the Hopf monoid $\tSig$ of compositions from Section~\ref{ss:faces}.
The following result describes the higher product and coproduct of 
the Hopf monoid $\tSig$ of compositions, namely
\[
\mu_F: \tSig(F) \to \tSig[I]
\qand
\Delta_F: \tSig[I] \to \tSig(F).
\]
Here $F$ is a composition of $I$ and the notation is as in~\eqref{e:iterated-mu} and~\eqref{e:iterated-delta};
see also Section~\ref{ss:gen-asso}. 
It suffices to consider compositions because $\tSig$ is connected (see Section~\ref{ss:conn-higher}).

Let $F=(I_1,\ldots,I_k)\vDash I$.
Given a composition $G$ refining $F$, let $G_i:=G|_{I_i}$. 
Define an element
\[
\BH_{G/F} := \BH_{G_1}\otimes\cdots\otimes\BH_{G_k} \in \tSig[I_1]\otimes\cdots\otimes\tSig[I_k] = \tSig(F).
\]
In view of~\eqref{e:res-conc},
the elements $\BH_{G/F}$ form a linear basis of $\tSig(F)$
as $G$ varies over the refinements of $F$.

\begin{proposition}\label{p:it-sig}
For any compositions $F\leq G$ of $I$,
\begin{equation}\label{e:prod-it-sig}
\mu_F(\BH_{G/F}) = \BH_G.
\end{equation}
For any pair of compositions $F$ and $G$ of $I$,
\begin{equation}\label{e:coprod-it-sig}
\Delta_F(\BH_G) = \BH_{FG/F}.
\end{equation}
\end{proposition}

Formulas~\eqref{e:faces-def} are the special case of~\eqref{e:prod-it-sig} and~\eqref{e:coprod-it-sig} for which $F=(S,T)$. 
Note that in this case $FG=G|_S\cdot G|_T$.

\smallskip

Now consider the bimonoid $\tSigh$ of decompositions. 
The higher product and coproduct of $\tSigh$ are given by essentially 
the same formulas as~\eqref{e:prod-it-sig} and~\eqref{e:coprod-it-sig},
but adjustments are necessary. We turn to them.
 
Let $F=(I_1,\ldots,I_k)$ be a decomposition of $I$. 
Let $G$ be a decomposition of $I$ such that $F\leq G$ and 
let $\gamma=(G_1,\dots,G_k)$ be a splitting of $(F,G)$.
Define an element
\[
\BH_{G/F}^{\gamma} := \BH_{G_1}\otimes\cdots\otimes\BH_{G_k} \in \tSigh[I_1]\otimes\cdots\otimes\tSigh[I_k] = \tSigh(F).
\]
In view of~\eqref{e:conc-dec},
the elements $\BH_{G/F}^{\gamma}$ form a linear basis of $\tSigh(F)$
as $G$ varies over decompositions satisfying $F\leq G$ and 
$\gamma$ over all splittings of $(F,G)$.

Now let $F$ and $G$ be two arbitrary decompositions of $I$.
Recall from~\eqref{e:row-splitting} that
there is a canonical choice of splitting for $(F,FG)$, denoted `$\can$'.
Hence the element $\BH_{FG/F}^{\can}\in\tSigh(F)$ is defined.

\begin{proposition}\label{p:it-sigh}
Let $F$ and $G$ be decompositions of $I$ with $F\leq G$.
Then, for any splitting $\gamma$ of $(F,G)$,
\begin{equation}\label{e:prod-it-sigh}
\mu_F(\BH_{G/F}^{\gamma}) = \BH_G.
\end{equation}
For any pair of decompositions $F$ and $G$ of $I$,
\begin{equation}\label{e:coprod-it-sigh}
\Delta_F(\BH_G) = \BH_{FG/F}^{\can}.
\end{equation}
\end{proposition}

\section{Series and functional calculus}\label{s:series}

Every species gives rise to a vector space of \emph{series},
which is an algebra when the species is a monoid.
For the exponential species, this identifies with the
algebra of formal power series. When the species is a connected
bimonoid, one can define group-like series and primitive series. The former
constitutes a group and the latter a Lie algebra, and they are in correspondence 
by means of the exponential and the logarithm.

We work over a field $\Kb$ of characteristic $0$.

\subsection{Series of a species}\label{ss:series}

Let $\tq$ be a species. 
Define the space of \emph{series} of $\tq$ as
\[
\Ser(\tq) := \lim \tq,
\]
the limit of the functor $\tq:\Fset\to\Veck$. Explicitly,
a series $s$ of $\tq$ is a collection of elements 
\[
s_I\in \tq[I],
\]
one for each finite set $I$, such that
\begin{equation}\label{e:series-nat}
\tq[\sigma](s_I) = s_J
\end{equation}
for each bijection $\sigma:I\to J$. 
The vector space structure is given by
\[
(s+t)_I:= s_I+t_I \qqand (c\cdot s)_I:= c\, s_I
\]
for $s,t\in\Ser(\tq)$ and $c\in\Kb$. 

Condition~\eqref{e:series-nat} implies that 
each $s_{[n]}$ is an $\Sr_n$-invariant element of $\tq[n]$, 
and in fact there is a linear isomorphism
\begin{equation}\label{e:ser-completion}
\Ser(\tq) \cong \prod_{n\geq 0} \tq[n]^{\Sr_n},
\qquad
s \mapsto (s_{[n]})_{n\geq 0}.
\end{equation}

Let $\wE$ be the exponential species~\eqref{e:expsp}.
We have an isomorphism of vector spaces
\begin{equation}\label{e:series-mor}
\Ser(\tq) \cong \Hom_{\Ssk}(\wE,\tq).
\end{equation}
Explicitly, a series $s$ of $\tq$ corresponds to the morphism of species
\[
\wE \to \tq, \quad \BH_{I} \mapsto s_I.
\]

Let $f:\tp\to\tq$ be a morphism of species and $s$ a series of $\tp$.
Define $f(s)$ by
\[
f(s)_I := f_I(s_I).
\]
Since $f$ commutes with bijections~\eqref{e:mor-sp}, $f(s)$ is a series of $\tq$.
In this manner, $\Ser$ defines a functor from species to vector spaces.
Moreover, for any species $\tp$ and $\tq$, there is a linear map
\begin{equation}
\Ser(\tp)\otimes\Ser(\tq) \to \Ser(\tp\bdot\tq) 
\end{equation}
which sends $s\otimes t$ to the series 
whose $I$-component is 
\[
\sum_{I=S\sqcup T} s_S\otimes t_T.
\]
This turns $\Ser$ into a \emph{braided lax monoidal functor}.
Such a functor preserves monoids, commutative monoids, and Lie monoids
(see, for instance,~\cite[Propositions~3.37 and~4.13]{AguMah:2010}).
These facts are elaborated below.

Let $\ta$ be a monoid. The \emph{Cauchy product}
 $s\conv t$ of two series $s$ and $t$ of $\ta$ is defined by
\begin{equation}\label{e:convseries}
(s\conv t)_I := \sum_{I=S\sqcup T} \mu_{S,T}(s_S\otimes t_T).
\end{equation}
This turns the space of series $\Ser(\ta)$ into an algebra.
The unit is the series $u$ defined by
\begin{equation}\label{e:unitseries}
u_I = \begin{cases}
\iota_\emptyset(1) & \text{ if $I=\emptyset$,}\\
0 & \text{ otherwise.}
\end{cases}
\end{equation}
In this case, $\Hom_{\Ssk}(\wE,\ta)$ is an algebra under the
convolution product of Section~\ref{ss:convolution}, and~\eqref{e:series-mor} is 
an isomorphism of algebras 
\[
\Ser(\ta) \cong \Hom_{\Ssk}(\wE,\ta).
\]
If the monoid $\ta$ is commutative, then 
the algebra $\Ser(\ta)$ is commutative.

If $\tg$ is a Lie monoid, then $\Ser(\tg)$ is a Lie algebra.
The Lie bracket $[s, t]$
of two series $s$ and $t$ of $\tg$ is defined by
\begin{equation}\label{e:bracseries}
[s, t]_I := \sum_{I=S\sqcup T} [s_S, t_T]_{S,T}.
\end{equation}

In view of~\eqref{e:series-nat}, 
a series $s$ of $\wE$ is of the form
\[
s_I = a_{\abs{I}} \BH_I,
\]
where $a_n$ is an arbitrary scalar sequence.
The species $\wE$ is a Hopf monoid (Section~\ref{ss:exp})
and hence $\Ser(\wE)$ is an algebra. If we identify $s$ 
with the formal power series
\[
\sum_{n\geq 0} a_n \frac{\varx^n}{n!},
\]
then the Cauchy product~\eqref{e:convseries} 
corresponds to the usual product of formal power series,
and we obtain an isomorphism of algebras
\begin{equation}\label{e:ser-exp}
\Ser(\wE) \cong \Kb\llb \varx \rrb.
\end{equation}
In this sense, series of $\wE$ are \emph{exponential generating functions}.

Let $s_1$ and $s_2$ be series of $\tq_1$ and $\tq_2$, respectively.
Their \emph{Hadamard product} $s_1\times s_2$
is defined by
\[
(s_1\times s_2)_I := (s_1)_I\otimes (s_2)_I.
\]
It is a series of $\tq_1\times\tq_2$.
This gives rise to a linear map
\begin{equation}\label{e:hadseries}
\Ser(\tq_1)\otimes\Ser(\tq_2) \to \Ser(\tq_1\times\tq_2), \qquad
s_1\otimes s_2 \mapsto s_1\times s_2.
\end{equation}
In general, this map is not injective
and not a morphism of algebras under the Cauchy product
(when the $\tq_i$ are monoids). This occurs already for $\tq_1=\tq_2=\wE$;
indeed, the Hadamard product~\eqref{e:hadseries} reduces in this case 
to the familiar Hadamard product (of exponential generating functions).

\subsection{Exponential series}\label{ss:expser}

Let $\ta$ be a monoid. A series $e$ of $\ta$ is \emph{exponential} if
\begin{equation}\label{e:expser1}
\mu_{S,T}(e_S\otimes e_T) = e_I
\end{equation}
for each $I=S\sqcup T$, and
\begin{equation}\label{e:expser2}
 \iota_\emptyset(1) = e_\emptyset.
\end{equation}
 
The unit series $u$~\eqref{e:unitseries} is exponential.

Let $\Exs(\ta)$ be the set of exponential series of $\ta$.
Under~\eqref{e:series-mor}, exponential series correspond to morphisms of monoids $\wE\to\ta$. 
Since $\wE$ is generated by $\wX$, 
such a series is completely determined by the element 
\[
e_{[1]} \in \ta[1].
\]
Conversely, for an element $e_1\in \ta[1]$ to give rise to an exponential series, 
it is necessary and sufficient that
\begin{equation}\label{e:expser3}
\mu_{\{1\},\{2\}}(e_1\otimes e_2) = \mu_{\{2\},\{1\}}(e_2\otimes e_1),
\end{equation}
where $e_2:=\ta[\sigma](e_1)$ and $\sigma$ is the unique bijection $\{1\}\to\{2\}$.

If the monoid $\ta$ is free, then the only element that satisfies~\eqref{e:expser3} is $e_1=0$,
and the corresponding exponential series is $u$.
Thus,
\[
\Exs(\ta) = \{u\}
\]
in this case.

If the monoid $\ta$ is commutative, then any element satisfies~\eqref{e:expser3}.
Moreover, exponential series are closed under the Cauchy product, 
and this corresponds to addition in $\ta[1]$. Thus, in this case,
\[
(\Exs(\ta),\conv) \cong (\ta[1],+).
\]
In particular, $\Exs(\ta)$ is a group (a subgroup of the group of invertible elements
of the algebra $\Ser(\ta)$). The inverse of $e$ is given by
\[
(e^{-1})_I := (-1)^{\abs{I}} e_I.
\]
These statements can be deduced from more general properties 
of convolution products (see the end of Section~\ref{ss:bim-hopf}).

For the exponential species $\wE$, we have 
\begin{equation}\label{e:exp-exp}
(\Exs(\wE),\conv) \cong (\Kb,+).
\end{equation}
If $e(c)\in\Exs(\wE)$ corresponds to $c\in\Kb$, then
\begin{equation}\label{e:exp-c}
e(c)_I = c^{\abs{I}} \BH_I.
\end{equation}
Under the identification~\eqref{e:ser-exp}, $e(c)$
corresponds to the formal power series
\[
\exp(c\varx) = \sum_{n\geq 0} c^n \frac{\varx^n}{n!}.
\]

Suppose $\ta_1$ and $\ta_2$ are monoids.
If $e_1$ and $e_2$ are exponential series of $\ta_1$ and $\ta_2$, respectively, 
then $e_1\times e_2$ is an exponential series of $\ta_1\times\ta_2$.
In general, the map
\begin{equation}\label{e:hadexpser}
\Exs(\ta_1)\times\Exs(\ta_2) \to \Exs(\ta_1\times\ta_2), \qquad
(e_1,e_2) \mapsto e_1\times e_2
\end{equation}
is not injective and not a morphism of groups
(when the $\ta_i$ are commutative monoids).
This occurs already for $\ta_1=\ta_2=\wE$;
in this case~\eqref{e:hadexpser} identifies under~\eqref{e:exp-exp}
with the multiplication of $\Kb$.
The Hadamard product
\[
\Exs(\ta_1)\times\Ser(\ta_2) \to \Ser(\ta_1\times\ta_2)
\]
distributes over the Cauchy product:
\[
e\times(s\conv t) = (e\times s)\conv(e\times t).
\]

Recall that $\wE\times\tq\cong\tq$ for any species $\tq$.
If $s\in\Ser(\tq)$ and $c\in\Kb$, then the Hadamard product
\[
e(c)\times s\in\Ser(\wE\times\tq)
\] 
corresponds to the series of $\tq$ given by
\[
(e(c)\times s)_I := c^{\abs{I}}\, s_I.
\]
In this manner, series of $\tq$ arise in one-parameter families.

Let $\rQ$ be a finite set-theoretic comonoid and consider the monoid $\Kb^\rQ$ as 
in Section~\ref{ss:lin}.
This monoid contains a distinguished exponential series $e$ defined by
\begin{equation}\label{e:dist-expser}
e_I := \sum_{x\in\rQ[I]} \BM_x.
\end{equation}

\subsection{Group-like series}\label{ss:glike}

Let $\tc$ be a comonoid. A series $g$ of $\tc$ is \emph{group-like} if
\begin{equation}\label{e:glike1}
\Delta_{S,T}(g_I) = g_S\otimes g_T
\end{equation}
for each $I=S\sqcup T$, and
\begin{equation}\label{e:glike2}
\epsilon_\emptyset(g_\emptyset) = 1.
\end{equation}
Note that it follows from~\eqref{e:glike1} plus counitality that either~\eqref{e:glike2} holds,
or $g_I=0$ for every $I$.
 
Let $\Gls(\tc)$ be the set of group-like series of $\tc$.
Under~\eqref{e:series-mor}, group-like series correspond to morphisms of comonoids $\wE\to\tc$.

Let $\thh$ be a bimonoid. 
The unit series $u$~\eqref{e:unitseries} is group-like.
It follows from~\eqref{e:comp} that 
$\Gls(\thh)$ is closed under the Cauchy product~\eqref{e:convseries}. 
Thus, $\Gls(\thh)$ is an ordinary monoid (a submonoid of the multiplicative monoid
of the algebra $\Ser(\thh)$).
If $\thh$ is a Hopf monoid, then $\Gls(\thh)$ is a group. 
The inverse of $g$ is obtained by composing with the antipode of $\thh$:
\[
(g^{-1})_I = \apode_I(g_I).
\]
This follows from~\eqref{e:apode}. 
(These statements can also be deduced from more general properties of convolution products; 
see the end of Section~\ref{ss:bim-hopf}.)

Suppose the bimonoid $\thh$ is connected. It follows from~\eqref{e:inverser}
and~\eqref{e:hopf-split} that if a series of $\thh$ is exponential, then it is group-like. 
Thus, 
\begin{equation}\label{e:exp-glike}
\Exs(\thh)\subseteq\Gls(\thh).
\end{equation}
For the Hopf monoid $\wE$, the maps $\mu_{S,T}$ and $\Delta_{S,T}$ are inverse, and hence
a series of $\wE$ is exponential if and only if it is group-like. Therefore, under
the identification~\eqref{e:ser-exp}, group-like series of $\wE$
correspond to the formal power series of the form $\exp(c\varx)$,
where $c\in\Kb$ is an arbitrary scalar, and we have an isomorphism of groups
\begin{equation}\label{e:gl-exp}
(\Gls(\wE),\conv) \cong (\Kb,+).
\end{equation}

Suppose $\tc_1$ and $\tc_2$ are comonoids.
If $g_1$ and $g_2$ are group-like series of $\tc_1$ and $\tc_2$, respectively, 
then $g_1\times g_2$ is a group-like series of $\tc_1\times\tc_2$, and we obtain a map
\begin{equation}\label{e:had-gls}
\Gls(\tc_1)\times\Gls(\tc_2) \to \Gls(\tc_1\times\tc_2).
\end{equation}
In particular, if $g$ is a group-like series of $\tc$ and $c$ is a scalar,
then $e(c)\times g$ is another group-like series of $\tc$.
Group-like series thus arise in one-parameter families. 

Let $\rP$ be a finite set-theoretic monoid and consider the comonoid $\Kb^\rP$ as 
in Section~\ref{ss:lin}.
This comonoid contains a distinguished group-like series $g$ defined by
\begin{equation}\label{e:dist-glike-1}
g_I := \sum_{x\in\rP[I]} \BM_x.
\end{equation}

\subsection{Primitive series}\label{ss:primser}

Let $\tc$ be a comonoid and $g$ and $h$ two group-like series of $\tc$.
A series $x$ of $\tc$ is $(g,h)$-\emph{primitive} if
\begin{equation}\label{e:primser-del}
\Delta_{S,T}(x_I) = g_S\otimes x_T + x_S\otimes h_T
\end{equation}
for each $I=S\sqcup T$, and
\begin{equation}\label{e:primser-eps}
\epsilon_\emptyset(x_\emptyset) = 0.
\end{equation}

Let $\tc$ be a connected comonoid. In this case, $\tc$ possesses a distinguished
group-like series $d$ determined by 
\begin{equation}\label{e:dist-glike-2}
d_I := \begin{cases}
\epsilon_{\emptyset}^{-1}(1) & \text{ if $I=\emptyset$,}\\
0 & \text{ if $I\neq\emptyset$.}
\end{cases}
\end{equation}
A series $x$ of $\tc$ is $(d,d)$-primitive if and only if
\begin{equation}\label{e:primser1}
\Delta_{S,T}(x_I) = 0
\end{equation}
for each $I=S\sqcup T$ with $S,T\neq\emptyset$, and
\begin{equation}\label{e:primser2}
x_\emptyset = 0.
\end{equation}
In this case, we simply say that the series $x$ is primitive.

Let $\Prs(\tc)$ be the set of primitive series of $\tc$.
Under~\eqref{e:series-mor}, primitive series correspond to morphisms of species 
$\wE\to\Pc(\tc)$, where $\Pc(\tc)$ is the primitive part of $\tc$ (Section~\ref{ss:prim-ind}).
In other words,
\begin{equation}\label{e:primser-part}
\Prs(\tc) = \Ser\bigl(\Pc(\tc)\bigr).
\end{equation}

The primitive part of the exponential species is $\Pc(\wE)=\wX$; hence we have an
isomorphism of abelian Lie algebras
\[
\Prs(\wE) \cong \Kb.
\]

The Hadamard product of an arbitrary series with a primitive series is primitive.
We obtain a map
\[
\Ser(\tc_1)\times\Prs(\tc_2) \to \Prs(\tc_1\times\tc_2),
\]
where the $\tc_i$ are connected comonoids.

Let $\thh$ be a bimonoid. The series $u$ of $\thh$ is group-like, regardless
of whether $\thh$ is connected. A series $x$ of $\thh$ is $(u,u)$-primitive if and
only if~\eqref{e:primser1} holds, and in addition
\begin{equation}\label{e:primser3}
\begin{gathered}
\Delta_{\emptyset,\emptyset}(x_\emptyset) = 
x_\emptyset\otimes \iota_\emptyset(1) + \iota_\emptyset(1)\otimes x_\emptyset,\\
\Delta_{I,\emptyset}(x_I)=x_I\otimes\iota_\emptyset(1),\\
\Delta_{\emptyset,I}(x_I)=\iota_\emptyset(1)\otimes x_I,
\end{gathered}
\end{equation}
for $I$ nonempty. In particular,
$x_\emptyset$ must be a primitive element of the bialgebra $\thh[\emptyset]$.

If in addition $\thh$ is connected, then $u=d$, so
$(u,u)$-primitive series coincide with primitive series.
Note also that, in this case, the last two conditions in~\eqref{e:primser3}
follow from counitality (the duals of~\eqref{e:unit}),
and the first one is equivalent to~\eqref{e:primser2}.

\begin{proposition}\label{p:transport}
Let $\thh$ be a bimonoid. Let $f$, $g$ and $h$ be group-like series of $\thh$,
and $x$ a $(g,h)$-primitive series of $\thh$. 
Then the series $f\conv x$ and $x\conv f$ are $(f\conv g,f\conv h)$- and
$(g\conv f,h\conv f)$-primitive, respectively.
\end{proposition}
\begin{proof}
Let $I=S'\sqcup T'$. 
We calculate using~\eqref{e:comp},~\eqref{e:glike1} and~\eqref{e:primser-del}.
\begin{multline*}
\Delta_{S',T'}\bigl((f\conv x)_I\bigr) = 
\sum_{I=S\sqcup T} \Delta_{S',T'}\mu_{S,T}(f_S\otimes x_T)\\
 = \sum_{\substack{S'=A\sqcup C\\T'=B\sqcup D}}
(\mu_{A,C}\otimes\mu_{B,D})(\id\otimes\beta\otimes\id)
(\Delta_{A,B}\otimes\Delta_{C,D}) (f_S\otimes x_T)\\
= \sum_{\substack{S'=A\sqcup C\\T'=B\sqcup D}}
\mu_{A,C}(f_A\otimes g_C)\otimes\mu_{B,D}(f_B\otimes x_D)
+ \mu_{A,C}(f_A\otimes x_C)\otimes\mu_{B,D}(f_B\otimes h_D)\\
= (f\conv g)_{S'}\otimes (f\conv x)_{T'} + (f\conv x)_{S'}\otimes (f\conv h)_{T'}.
\end{multline*}
This proves~\eqref{e:primser-del} for $f\conv x$. 
Condition~\eqref{e:primser-eps} follows similarly, using~\eqref{e:unitr}.
\end{proof}

Let $\thh$ be a connected bimonoid. Since $\Pc(\thh)$ is a Lie monoid
(Proposition~\ref{p:prim-lie}), we have that $\Prs(\thh)$ is a Lie algebra
under the bracket~\eqref{e:bracseries}.
It is a Lie subalgebra of $\Ser(\thh)$, where the latter is equipped with the
commutator bracket corresponding to the Cauchy product.
Explicitly,
\begin{equation}\label{e:commbracseries}
[x, y]_I := \sum_{I=S\sqcup T} \mu_{S,T}(x_S\otimes y_T) - \mu_{T,S}(y_T\otimes x_S)
\end{equation}
for $x,y\in\Prs(\thh)$.

\subsection{Complete Hopf algebras}\label{ss:comphopf}

Let $\thh$ be a Hopf monoid. In view of~\eqref{e:ser-completion},
the space of series of $\thh$ identifies with the completion of the graded space
\[
\cKcb(\thh):= \bigoplus_{n\geq 0} \thh[n]^{\Sr_n}
\]
(with respect to the filtration by degree). 
The summand is the invariant subspace of $\thh[n]$ under the action of $\Sr_n$.
In~\cite[Chapter~15]{AguMah:2010},
it is shown that $\cKcb(\thh)$ carries a structure of graded Hopf algebra.
One may thus turn the space $\Ser(\thh)$ into a \emph{complete Hopf algebra}
in the sense of Quillen~\cite[Appendix~A.2]{Qui:1969}. 
Then $\Gls(\thh)$ and $\Prs(\thh)$ identify with the group of group-like
elements and the Lie algebra of primitive elements
of this complete Hopf algebra, respectively.

\subsection{Series of the linear order species}\label{ss:serL}

Consider the Hopf monoid $\wL$ of linear orders from Section~\ref{ss:linear}.

In view of~\eqref{e:series-nat}, 
a series $s$ of $\wL$ is of the form
\[
s_I = a_{\abs{I}} \sum_{\ell\in\rL[I]} \BH_\ell,
\]
where $a_n$ is an arbitrary scalar sequence.
If we identify $s$ 
with the formal power series
\[
\sum_{n\geq 0} a_n\,\varx^n,
\]
then the Cauchy product~\eqref{e:convseries} 
corresponds to the usual product of formal power series,
and we obtain an isomorphism of algebras
\begin{equation}\label{e:ser-L}
\Ser(\wL) \cong \Kb\llb \varx \rrb.
\end{equation}
In this sense, series of $\wL$ are \emph{ordinary generating functions}.
Contrast with~\eqref{e:ser-exp}.

Since $\wL$ is free, the only exponential series is the unit series $u$~\eqref{e:unitseries}.
The corresponding formal power series is the constant $1$.

For a series $s$ as above we have
\[
\Delta_{S,T}(s_I) = a_{\abs{I}} \binom{\abs{I}}{\abs{S}} \sum_{\substack{\ell_1\in\rL[S]\\\ell_2\in\rL[T]}} \BH_{\ell_1}\otimes\BH_{\ell_2}.
\]
For $s$ to be group-like, 
we must have 
\[
a_n a_m = \binom{n+m}{n} a_{n+m}
\] 
for all $n,m\geq 0$, and $a_0=1$.
Therefore, each scalar $c\in\Kb$ gives rise to a group-like series $g(c)\in\Gls(\wL)$ defined by
\[
g(c)_I := \frac{c^{\abs{I}}}{\abs{I}!} \sum_{\ell\in\rL[I]} \BH_{\ell},
\]
and every group-like series is of this form.
This gives rise to an isomorphism of groups
\[
(\Gls(\wL),\conv) \cong (\Kb,+).
\]

For $s$ to be primitive, we must have $a_n=0$ for every $n>1$, and $a_0=0$.
Therefore, each scalar $c\in\Kb$ gives rise to a primitive series $x(c)\in\Prs(\wL)$ defined by
\[
x(c)_I := \begin{cases}
c\, \BH_{i} & \text{ if $I=\{i\}$,}\\
0 & \text{ if not,}
\end{cases}
\]
and every primitive series is of this form.
This gives rise to an isomorphism of (abelian) Lie algebras
\[
\Prs(\wL) \cong \Kb.
\]
Note that, by~\eqref{e:primser-part},
\[
\Prs(\wL) = \Ser(\tLie),
\]
since the species $\tLie$ is the primitive part of $\wL$. 
We deduce from the above and~\eqref{e:ser-completion}
the well-known result that
\[
\tLie[n]^{\Sr_n} = 
\begin{cases}
\Kb & \text{ if $n=1$,}\\
0 & \text{ if not.}
\end{cases}
\]

The formal power series corresponding under~\eqref{e:ser-L} to $g(c)$ and $x(c)$ are,
respectively,
\[
\exp(c\varx) \qand c\varx.
\]
In Theorem~\ref{t:exp-log} below we show that, 
for any connected bimonoid $\thh$, 
group-like and primitive series of $\thh$ are in correspondence by means of
suitable extensions of the logarithm and the exponential.

\subsection{Formal functional calculus}\label{ss:fun-calc}

Let $\ta$ be a monoid. Recall from~\eqref{e:iterated-mu} 
and~\eqref{e:sp-comp} that given
a composition $F=(I_1,\ldots,I_k)$ of a finite set $I$, the space $\ta(F)$
and the map
\[
\mu_F: \ta(F) \to \ta[I]
\]
are defined (see also Section~\ref{ss:gen-asso}). 
In this situation, given a series $s$ of $\ta$, 
define an element $s_F\in\ta(F)$ by
\[
s_F := s_{I_1}\otimes\cdots\otimes s_{I_k}.
\]

Let 
\[
a(\varx) = \sum_{n\geq 0} a_n \varx^n
\]
be a formal power series. 
Given a series $s$ of $\ta$ such that 
\begin{equation}\label{e:ser-vanish}
s_\emptyset=0,
\end{equation} 
define another series $a(s)$ of $\ta$ by
\begin{equation}\label{e:calc-s}
a(s)_I := \sum_{F\vDash I} a_{\len(F)} \mu_F(s_F)
\end{equation}
for every finite set $I$. In particular, $a(s)_\emptyset = a_0\,\iota_\emptyset(1)$.
It follows that formal power series with $a_0=0$
operate in this manner on the space of series of $\ta$ satisfying~\eqref{e:ser-vanish}.

\begin{proposition}
If $a(\varx)=\varx^k$ for a nonnegative integer $k$, then $a(s)=s^{\conv k}$. 
In particular,
if $a(\varx)=1$, then $a(s)=u$, the unit series~\eqref{e:unitseries},
and if $a(\varx)=\varx$, then $a(s)=s$. 
\end{proposition}
\begin{proof}
It follows from~\eqref{e:convseries} that
\[
(s^{\conv k})_I = \sum_{\substack{F\vDash I,\\\len(F)=k}} \mu_F(s_F).
\]
Since $s_\emptyset=0$, the sum is over compositions rather than decompositions.
Now compare with~\eqref{e:calc-s}. 
\end{proof}

If $b(\varx)$ is another formal power series, the sum $(a+b)(\varx)$
and the product $(a\,b)(\varx)$ are defined. 
If in addition $b_0=0$, the composition $(a\circ b)(\varx)$ is defined.

\begin{proposition}[Functional calculus]
\label{p:fun-calc}
Let $a(\varx)$ and $b(\varx)$ be formal power series and $s$ a series of $\ta$
satisfying~\eqref{e:ser-vanish}.
\begin{enumerate}[(i)]
\item If $c$ is a scalar, then $(c\,a)(s)= c\,a(s)$.
\item $(a+b)(s) = a(s)+b(s)$.
\item\label{i:prod} 
$(a\,b)(s) = a(s)\conv b(s)$. %, the Cauchy product~\eqref{e:convseries}.
\item\label{i:comp}
If $b_0=0$, then $(a\circ b)(s) = a\bigl(b(s)\bigr)$.
\item\label{i:nat}
If $f:\ta\to\tb$ is a morphism of monoids, then $f\bigl(a(s)\bigr) = a\bigl(f(s)\bigr)$.
\end{enumerate}
\end{proposition}
\begin{proof}
We verify assertion~\eqref{i:prod}. 
According to~\eqref{e:convseries} and~\eqref{e:calc-s}, we have
\begin{align*}
\bigl(a(s)\conv b(s)\bigr)_I & = \sum_{\substack{I=S\sqcup T\\F\vDash S,\,G\vDash T}} 
\!\! a_{\len(F)} b_{\len(G)} \mu_{S,T}\bigl(\mu_F(s_F)\otimes\mu_G(s_G)\bigr) \\
& = \sum_{\substack{H\vDash I\\i+j=\len(H)}} \!\! \! a_i b_j \mu_H(s_H)\\ 
& = (a\, b)(s)_I.  
\end{align*}
We used associativity in the form $\mu_{S,T}(\mu_F\otimes\mu_G) = \mu_H$ for $H=F\cdot G$, 
and the fact that in this case $s_H = s_F\otimes s_G$.
\end{proof}

\subsection{Exponential, logarithm, and powers}\label{ss:exp-log}

Consider the formal power series
\[
\exp(\varx) = \sum_{n\geq 0} \frac{\varx^n}{n!}
\qqand
\mathrm{l}(\varx) := \log(1-\varx) = - \sum_{n\geq 1} \frac{\varx^n}{n}. 
\]
Let $\ta$ be a monoid. 
If $t$ is a series of $\ta$ such that 
\begin{equation}\label{e:ser-nonvanish}
t_\emptyset =\iota_\emptyset(1),
\end{equation} 
define
\[
\log(t) := \mathrm{l}(u-t).
\]
Then $\log(t)$ is a series satisfying~\eqref{e:ser-vanish}, and for nonempty $I$,
\begin{equation}\label{e:calc-log}
\log(t)_I = -\sum_{F\vDash I} \frac{(-1)^{\len(F)}}{\len(F)}\, \mu_F(t_F).
\end{equation}
If $s$ is a series of $\ta$ satisfying~\eqref{e:ser-vanish}, 
then $\exp(s)$ is a series satisfying~\eqref{e:ser-nonvanish},
and more generally,
\begin{equation}\label{e:calc-exp}
\exp(s)_I = \sum_{F\vDash I} \frac{1}{\len(F)!}\, \mu_F(s_F).
\end{equation}

The following is a basic property of functional calculus.

\begin{proposition}\label{p:exp-log}
Let $s_1$ and $s_2$ be series of $\ta$ satisfying~\eqref{e:ser-vanish}.
If $[s_1,s_2]=0$, then
\begin{equation}\label{e:exp-add}
\exp(s_1+s_2) = \exp(s_1)\conv\exp(s_2).
\end{equation}
Let $t_1$ and $t_2$ be series of $\ta$ satisfying~\eqref{e:ser-nonvanish}.
If $[t_1,t_2]=0$, then
\begin{equation}\label{e:log-add}
\log(t_1\conv t_2) = \log(t_1)+\log(t_2).
\end{equation}
\end{proposition}

\begin{theorem}\label{t:exp-log}
For any monoid $\ta$, the maps
\begin{equation}\label{e:exp-log1}
\xymatrix@C+15pt{
\{s\in\Ser(\ta) \mid s_\emptyset = 0\} \ar@<0.6ex>[r]^-{\exp} & \ar@<0.6ex>[l]^-{\log} 
\{t\in\Ser(\ta) \mid t_\emptyset = \iota_\emptyset(1)\} 
}
\end{equation}
are inverse bijections. Moreover, if $\thh$ is a connected bimonoid,
these maps restrict to inverse bijections
\begin{equation}\label{e:exp-log2}
\xymatrix@C+15pt{
\Prs(\thh) \ar@<0.6ex>[r]^-{\exp} & \ar@<0.6ex>[l]^-{\log} 
\Gls(\thh). 
}
\end{equation}
\end{theorem}
\begin{proof}
The first assertion follows from statement~\eqref{i:comp} in Proposition~\ref{p:fun-calc}.
The second assertion is a special case of~\eqref{exp-log-P} (the case $\tc=\wE$),
for which we provide a direct proof below.
\end{proof}

\begin{remark} 
Theorem~\ref{t:exp-log} may be seen as a special case 
of a result of Quillen for complete Hopf algebras~\cite[Appendix~A, Proposition~2.6]{Qui:1969}, 
in view of the considerations in Section~\ref{ss:comphopf}.
\end{remark}
% Quillen assumes cocommutativity of $\thh$, 
% but this is not needed since $\Prs$ and $\Gls$ do not change 
% when passing to the largest cocommutative subcomonoid of $\thh$.

Let $c\in\Kb$ be a scalar. 
Consider the formal power series
\[
p_c(x) := (1+x)^c = \sum_{n\geq 0}\binom{c}{n}\,x^n,
\]
where 
\begin{equation}\label{e:binomial}
\binom{c}{n}:=\frac{1}{n!}c(c-1)\ldots\bigl(c-(n-1)\bigr)\in\Kb.
\end{equation}
Let $\ta$ be a monoid. 
If $t$ is a series of $\ta$ satisfying~\eqref{e:ser-nonvanish}, 
define 
\[
t^c:= p_c(t-u).
\] 
Then $t^c$ is a series satisfying~\eqref{e:ser-nonvanish}, and for nonempty $I$,
\begin{equation}\label{e:calc-pow}
t^c_I = \sum_{F\vDash I} \binom{c}{\len(F)}\,\mu_F(t_F).
\end{equation}

In view of Proposition~\ref{p:fun-calc}, the identities
\[
(1+x)^{c+d} = (1+x)^c\,(1+x)^d,
\
\bigl((1+x)^c\bigr)^d = (1+x)^{cd}
\pand
(1+x)^c = \exp\bigl(c\log(1+x)\bigr)
\]
in the algebra of formal power series imply the identities
\begin{equation}\label{e:pow-add}
t^{c+d} = t^c\conv t^d,
\quad
(t^c)^d = t^{cd}
\qand
t^c = \exp(c\,\log t)
\end{equation}
in the algebra of series of $\ta$. 
In particular, $t$ is invertible in the algebra $\Ser(\ta)$,
with $t^{-1}$ satisfying~\eqref{e:ser-nonvanish} and
\[
t^{-1}_I = \sum_{F\vDash I} (-1)^{\len(F)}\,\mu_F(t_F).
\]
In addition, $t^c=t^{\conv c}$ for all integers $c$.

Let $\thh$ be a connected bimonoid and $g$ a group-like series of $\thh$.
It follows from~\eqref{e:exp-log2} and the third identity in~\eqref{e:pow-add} that 
\begin{equation}\label{e:glike-pow}
g^c\in\Gls(\thh)
\end{equation} 
for every $c\in\Kb$. 
Thus, group-like series are closed under powers by arbitrary scalars.

\subsection{Curves and derivatives}\label{ss:curves}

Let $\tp$ be a species. 
A (polynomial) \emph{curve} on $\tp$ is a function
\[
\gamma: \Kb \to \Ser(\tp)
\]
whose components $\gamma_I: \Kb \to \tp[I]$ are polynomial, for all finite sets $I$.
The component $\gamma_I$ is defined by
$\gamma_I(c):=\gamma(c)_I$ for all $c\in\Kb$.

The \emph{derivative} $\gamma'$ of a curve $\gamma$ is defined by
\[
\gamma'_I(c) := \frac{d}{dc}\gamma_I(c),
\]
where the latter is the usual derivative of a polynomial function.
When $\Kb=\Rb$ is the field of real numbers, one may also speak of smooth
curves $\gamma:(a,b)\to\Ser(\tp)$ defined on a real interval, and of their derivatives.

The derivative of a formal power series $a(\varx) = \sum_{n\geq 0} a_n \varx^n$ is
the formal power series
\begin{equation}\label{e:ser-deriv}
a'(\varx) := \sum_{n\geq 0} (n+1) a_{n+1} \varx^n. 
\end{equation}

Let $\gamma$ be a curve on a monoid $\ta$ for which each series 
$\gamma(c)$ satisfies~\eqref{e:ser-vanish}. Given a formal power series $a(\varx)$,
the curve $a\circ\gamma$ is defined (by means of functional calculus) as
\[
(a\circ\gamma)(c) := a\bigl(\gamma(c)\bigr),
\]
for each $c\in\Kb$.

Derivatives of curves and of power series are linked by the \emph{chain rule}.
This requires a commutativity assumption.

\begin{proposition}[Chain rule]\label{p:chain-rule}
Let $\gamma$ be a curve as above. Suppose that
\[
[\gamma(c),\gamma'(c)]=0
\]
for all $c\in\Kb$. 
Then
\begin{equation}\label{e:chain-rule}
(a\circ\gamma)'(c) = 
a'\bigl(\gamma(c)\bigr)\conv \gamma'(c) = \gamma'(c)\conv a'\bigl(\gamma(c)\bigr).
\end{equation}
\end{proposition}
\begin{proof}
The $\emptyset$-components are zero, so we assume $I$ to be nonempty.
Applying~\eqref{e:calc-s} to~\eqref{e:ser-deriv},
\[
a'(s)_I = \sum_{F\vDash I} \,\bigl(\len(F)+1\bigr)a_{\len(F)+1}\mu_F(s_F).
\]
Substituting $s=\gamma(c)$ and calculating using~\eqref{e:convseries},
\begin{equation}\label{e:first-calc}
\Bigl(a'\bigl(\gamma(c)\bigr)\conv \gamma'(c)\Bigr)_I 
= \sum_{I=S\sqcup T} \mu_{S,T}\bigg(\sum_{F\vDash S} \,\bigl(\len(F)+1\bigr)a_{\len(F)+1}\mu_F\bigl(\gamma_F(c)\bigr) \otimes \gamma_T'(c)\bigg).
% & = \sum_{K\vDash I} \len(K)\,a_{\len(K)} \mu_K(\gamma(c)_{I_1}\otimes \dots \otimes \gamma(c)_{I_{j-1}}\otimes \gamma'(c)_{I_j}),
\end{equation}
% where $K=(I_1,\dots,I_j)$.
% (In the second step, we wrote $K=F\cdot (T)$ and used $\mu_K=\mu_{S,T}(\mu_F\otimes\id)$.)
% We point out that only the last factor is being differentiated.
Hypothesis~\eqref{e:ser-vanish} implies that $\gamma'_\emptyset = 0$. 
Thus, we may assume that $T\neq\emptyset$ in the previous sum.

On the other hand, the derivative of $(a\circ\gamma)(c)$ is
\begin{equation}\label{e:second-calc}
(a\circ\gamma)'_I(c) = \sum_{G\vDash I} a_{\len(G)}\mu_G\bigl(\gamma_G(c)\bigr)' 
=\sum_{G\vDash I} a_{\len(G)}\mu_G\bigl((\gamma_G)'(c)\bigr).
\end{equation}
(Since $\mu_G$ is linear, it commutes with the derivative.)
If $G=(I_1,\ldots,I_k)$, then
\[
\gamma_G(c) = \gamma_{I_1}(c)\otimes\cdots\otimes\gamma_{I_k}(c)
\]
and 
\[
(\gamma_G)'(c) = \sum_{i=1}^k \gamma_{I_1}(c)\otimes\cdots\otimes\gamma_{I_i}'(c)\otimes\cdots
\otimes\gamma_{I_k}(c).
\]
Commutativity allows us to move the differentiated factor in~\eqref{e:second-calc}
to the right end, as we now verify. By hypothesis,
\[
\sum_{I=S_1\sqcup S_2} \mu_{S_1,S_2}\bigl(\gamma_{S_1}'(c)\otimes
\gamma_{S_2}(c)\bigr) = \sum_{I=S_1\sqcup S_2} \mu_{S_1,S_2}\bigl(\gamma_{S_1}(c)\otimes
\gamma_{S_2}'(c)\bigr).
\]
It follows that, for any $i=1,\ldots,k$,
\[
\sum_{G:\,\len(G)=k} \mu_G\bigl(\gamma_{I_1}(c)\otimes\cdots\otimes\gamma_{I_i}'(c)\otimes\cdots
\otimes\gamma_{I_k}(c)\bigr) = \sum_{G:\,l(G)=k} \mu_G\bigl(
\gamma_{I_1}(c)\otimes\cdots\otimes\gamma_{I_k}'(c)\bigr),
\]
and hence
\[
\sum_{G:\,l(G)=k} \mu_G\bigl(\gamma_G'(c)\bigr) = 
k \bigg(\sum_{G:\,\len(G)=k} \mu_G\bigl(
\gamma_{I_1}(c)\otimes\cdots\otimes\gamma_{I_k}'(c)\bigr)\bigg).
\]
Substituting back in~\eqref{e:second-calc} we obtain
\[
(a\circ\gamma)'_I(c) = \sum_{G\vDash I} \len(G) a_{\len(G)}
\mu_G\bigl(\gamma_{I_1}(c)\otimes\cdots\otimes\gamma_{I_k}'(c)\bigr).
\]
Letting $T=I_k$, $S=I\setminus T$, and
$F=(I_1,\ldots,I_{k-1})\vDash S$, and employing associativity, we obtain~\eqref{e:first-calc},
as needed.
\end{proof}
%By the product rule, $\gamma_G'(c)$ is a sum 
%with the tensor factors of $\gamma_G(c)$ differentiated one at a time.
%Since $\gamma(c)$ and $\gamma'(c)$ commute by hypothesis, in each term,
%we can move the differentiated factor to the end.
%Thus
%\[
%\mu_G\bigl(\gamma_G'(c)\bigr) = \sum_T \mu_{S,T}\Bigl(\mu_F\bigl(\gamma_F(c)\bigr) \otimes \gamma_T'(c)\Bigr),
%\]
%where $T$ ranges over all blocks of $G$,
%$F$ is obtained from $G$ by removing the block $T$, and $S$ is the union of the remaining blocks, so that $F\vDash S$.
%Substituting this back we see that~\eqref{e:first-calc} and~\eqref{e:second-calc} are equal.
%(There are $\len(F)+1$ ways to shuffle the blocks of $F$ with $T$, 
%so there are those many $G$'s which yield a given $F$ and $T$.)

Let $\tc$ be a comonoid and $\gamma$ a curve on $\Ser(\tc)$. 
When the series $\gamma(c)$ is group-like for every $c\in\Kb$, we say that
$\gamma$ is a group-like curve.

\begin{proposition}\label{p:glikecurve}
Let $\gamma$ be a group-like curve on a comonoid $\tc$. 
For any $c_0\in\Kb$, the series $\gamma'(c_0)$ is $(\gamma(c_0),\gamma(c_0))$-primitive.
\end{proposition}
\begin{proof} 
We argue in the smooth case, where we have
\[
\gamma'_I(c_0) = \lim_{c\to 0} \,\frac{\gamma_I(c+c_0) - \gamma_I(c_0)}{c}. 
\]
Since $\Delta_{S,T}$ is linear, it commutes with limits. 
Hence,
\begin{align*}
\Delta_{S,T}\bigl(\gamma'_I(c_0)\bigr) 
& = \lim_{c\to 0} \,\frac{1}{c} \Delta_{S,T}\bigl(\gamma_I(c+c_0) - \gamma_I(c_0)\bigr)\\
& = \lim_{c\to 0} \,\frac{1}{c} \bigl(\gamma_S(c+c_0)\otimes\gamma_T(c+c_0)
 - \gamma_S(c_0)\otimes \gamma_T(c_0) \bigr)\\
& = \lim_{c\to 0} \,\frac{1}{c}\, \gamma_S(c+c_0)\otimes\bigl(\gamma_T(c+c_0)-\gamma_T(c_0)\bigr)\\
& \qquad \qquad \qquad + \lim_{c\to 0} \,\frac{1}{c}\, \bigl(\gamma_S(c+c_0)-\gamma_S(c_0)\bigr)\otimes \gamma_T(c_0)\\
%& = \gamma'_S(c_0)\otimes \gamma_T(c_0) + \gamma_S(c_0)\otimes \gamma'_T(c_0)\\
& = \gamma_S(c_0)\otimes \gamma'_T(c_0) + \gamma'_S(c_0)\otimes \gamma_T(c_0),
\end{align*}
proving~\eqref{e:primser-del}. 
Condition~\eqref{e:primser-eps} follows similarly.
\end{proof}

\begin{corollary}\label{c:glikecurve}
Let $\thh$ be a Hopf monoid and $\gamma$ a group-like curve on $\thh$.
Fix $c_0\in\Kb$, let $g:=\gamma(c_0)$ and $x:=\gamma'(c_0)$.
Then the series
\[
g^{-1}\conv x \qand x\conv g^{-1}
\]
are $(u,u)$-primitive. 
In particular, if $\thh$ is connected, then both series are primitive.
\end{corollary}
\begin{proof}
By Proposition~\ref{p:glikecurve}, $x$ is $(g,g)$-primitive.
Transporting it by $g^{-1}$ on either side we obtain a series that is $(u,u)$-primitive,
according to Proposition~\ref{p:transport}. 
\end{proof}

Corollary~\ref{c:glikecurve} affords a construction of primitive series
out of group-like curves.
\[
\psfrag{u}{$u$}
\psfrag{g}{$g$}
\psfrag{x}{\small $x$}
\psfrag{y}{\small $g^{-1}\conv x$}
\includegraphics[width=3in]{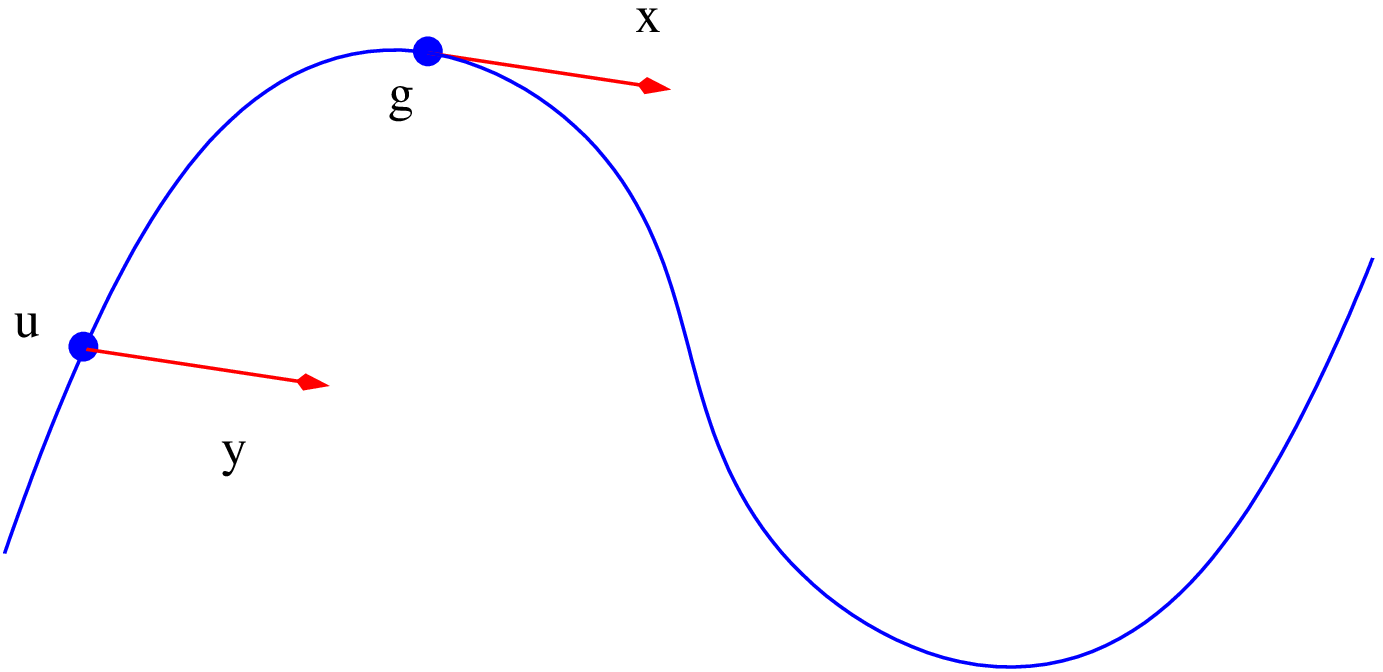}
\]

In the next section we study the result of this construction when applied to certain
special curves.

\subsection{Two canonical curves}\label{ss:can-curve}

Let $t$ be a series of a species $\tp$. There is a curve
$\cancur{t}$ on $\tp$ defined by
\begin{equation}\label{e:can-curve1}
\cancur{t}_I(c) := c^{\abs{I}}\, t_I
\end{equation}
for every finite set $I$. 
It is clearly polynomial. 

Let $\ta$ be a monoid and assume that $t$ satisfies~\eqref{e:ser-nonvanish}. 
There is then another curve $\onepar{t}$ on $\ta$ defined by
\begin{equation}\label{e:can-curve2}
\onepar{t}(c) := t^c.
\end{equation}
Formulas~\eqref{e:binomial} and~\eqref{e:calc-pow} show that
$\onepar{t}$ is polynomial. We say that $\onepar{t}$ is
the \emph{one-parameter subgroup generated by $t$}.

We have 
\[
\cancur{t}(1)= t = \onepar{t}(1) 
\qqand
\cancur{t}(0)= u = \onepar{t}(0).
\] 
This is illustrated below.

\[
\psfrag{0}{\tiny $(c=0)$}
\psfrag{1}{\tiny $(c=1)$}
\psfrag{u}{$u$}
\psfrag{t}{$t$}
\psfrag{A}{$\onepar{t}$}
\psfrag{B}{$\cancur{t}$}
\includegraphics[width=2.2in]{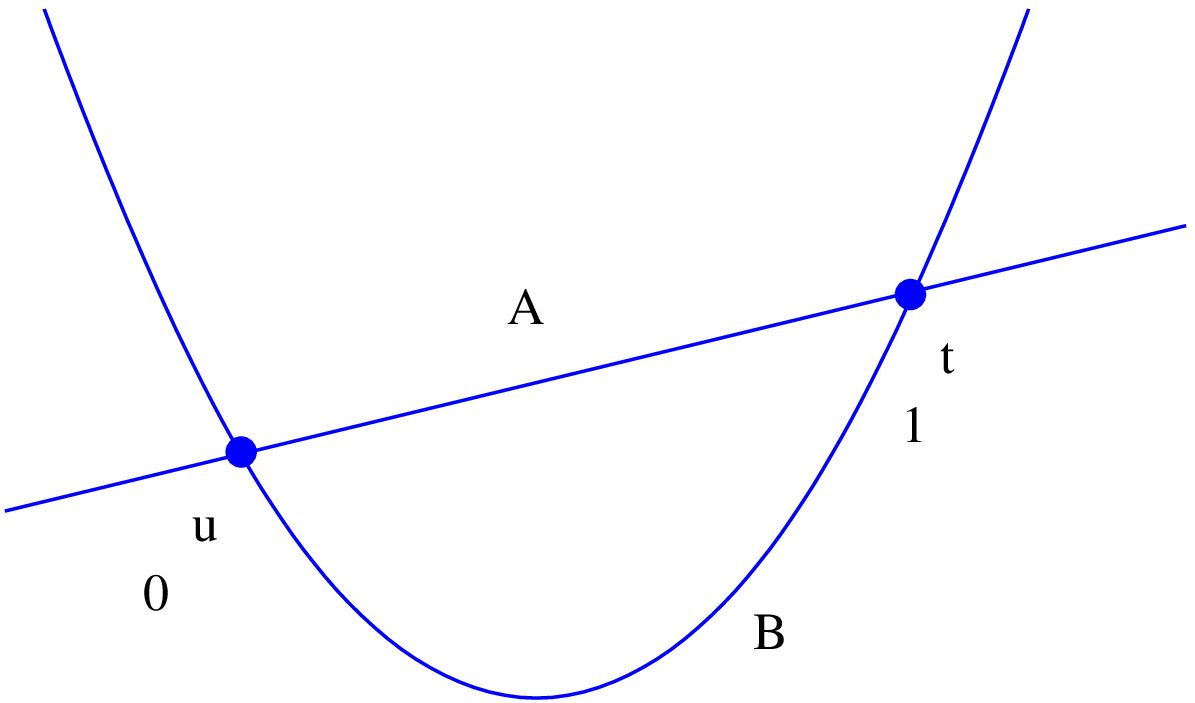}
\]
 
We proceed to apply the construction of Corollary~\ref{c:glikecurve}
to these curves.
 
We first describe the derivative of the curve $\cancur{t}$.
The \emph{number operator} is the morphism of species $\degmap:\tp\to\tp$ 
defined by
\begin{equation}\label{e:degmap}
\degmap_I:\tp[I] \to \tp[I],
\quad
\degmap_I(x) = \abs{I}\,x.
\end{equation}
Note that
\begin{equation}\label{e:degmap-der}
 \cancur{t}'(1) = \degmap(t).
\end{equation}

\begin{corollary}\label{c:glike-dynkin}
Let $g$ be a group-like series of a Hopf monoid $\thh$. 
Then the series
\[
g^{-1}\conv \degmap(g) \qand \degmap(g)\conv g^{-1}
\]
are $(u,u)$-primitive. 
In particular, if $\thh$ is connected, then both series are primitive.
\end{corollary}
\begin{proof}
We apply Corollary~\ref{c:glikecurve} to the curve $\cancur{g}$.
In view of~\eqref{e:had-gls}, this is a group-like curve.
Since $\cancur{g}(1)=g$ and $\cancur{g}'(1)=\degmap(g)$, the result follows.
\end{proof}

We turn to the derivative of the one-parameter subgroup $\onepar{t}$.

\begin{lemma}\label{l:pow-der}
Let $t$ be a series of a monoid $\ta$ satisfying~\eqref{e:ser-nonvanish}. 
Then, for any $c\in\Kb$,
\begin{equation}\label{e:pow-der}
\onepar{t}'(c) = t^c\conv \log t = \log t \conv t^c.
\end{equation}
\end{lemma}
\begin{proof}
Let $\gamma$ be the curve defined by
\[
\gamma(c) = c \log t.
\]
We have 
\[
\onepar{t}(c) = t^c = \exp(c \log t) = \exp\bigl(\gamma(c)\bigr).
\]
Since $\gamma'(c)=\log t$, we have $[\gamma(c),\gamma'(c)]=0$, and the chain
rule~\eqref{e:chain-rule} applies. We deduce that
\[
\onepar{t}'(c) = \exp\bigl(\gamma(c)\bigr) \conv \gamma'(c) = t^c\conv \log t.
\]
For the same reason, $\onepar{t}'(c) = \log t \conv t^c$.
\end{proof}

Let us apply the construction of Corollary~\ref{c:glikecurve} to the curve $\onepar{g}$,
where $g$ is a group-like series of a connected Hopf monoid $\thh$. 
First of all, the curve is group-like by~\eqref{e:glike-pow}. 
Lemma~\ref{l:pow-der} implies that $\onepar{g}'(1) = g\conv \log g$. 
The conclusion is then that $\log g$ is a primitive series of $\thh$. 
We thus recover a part of Theorem~\ref{t:exp-log}.

The curves $\onepar{t}$ enjoy special properties, 
as discussed next.

\begin{proposition}\label{p:one-par}
Let $\thh$ be a connected bimonoid and $\gamma$ a curve on $\thh$
such that $\gamma(0)=u$.
The following conditions are equivalent.
\begin{enumerate}[(i)]
\item\label{e:one-par1} $\gamma(c+d) = \gamma(c)\conv\gamma(d)$ for all $c$ and $d\in\Kb$.
\item\label{e:one-par2} There is a series $s$ satisfying~\eqref{e:ser-vanish}
such that $\gamma'(c) = \gamma(c)\conv s$ for all $c\in\Kb$.
\item\label{e:one-par3} There is a series $t$ satisfying~\eqref{e:ser-nonvanish}
such that $\gamma=\onepar{t}$.
\end{enumerate}
Assume (any) one of these conditions holds. Then
\[
t=\gamma(1), \quad s=\gamma'(0) \qand s=\log t.
\]
In addition,
\[
\text{$\gamma$ is a group-like curve} \iff t\in\Gls(\thh) \iff s\in\Prs(\thh).
\]
\end{proposition}
\begin{proof}
Assume~\eqref{e:one-par1} holds. Fixing $c$ and taking derivatives at $d=0$
we obtain~\eqref{e:one-par2}. Assume the latter holds.
If $t=\exp s$, then $\onepar{t}$ is a solution
of the differential equation in
~\eqref{e:one-par2}, according to~\eqref{e:pow-der}. 
Since $\gamma(0)=u=\onepar{t}(0)$, uniqueness of solutions of such equations
guarantees that $\gamma=\onepar{t}$,
proving that~\eqref{e:one-par3} holds. Finally, the latter
implies~\eqref{e:one-par1} by the first identity in~\eqref{e:pow-add}.

The last statement follows from~\eqref{e:exp-log2} and~\eqref{e:glike-pow}.
\end{proof}

Note also that, in the situation of Proposition~\ref{p:one-par}, 
$\gamma(c)\conv s = s\conv\gamma(c)$ for all $c\in\Kb$.

\subsection{Series of the internal Hom species}\label{ss:seriH}

Recall the functors $\iH$ and $\iE$ from Section~\ref{ss:iH}. 
Let $\tp$ and $\tq$ be species.
A series of the species $\iH(\tp,\tq)$ is precisely a morphism of
species from $\tp$ to $\tq$:
\begin{equation}\label{e:seriH}
\Ser\bigl(\iH(\tp,\tq)\bigr) = \Hom_{\Ssk}(\tp,\tq).
\end{equation}
This follows by comparing~\eqref{e:mor-sp}
and~\eqref{e:iHact}, or by using~\eqref{e:iH} and~\eqref{e:series-mor}.
We may also recover~\eqref{e:series-mor} by setting $\tp=\wE$ in~\eqref{e:seriH}.
Setting $\tp=\tq$ we obtain
\[
\Ser\bigl(\iE(\tp)\bigr) = \End_{\Ssk}(\tp).
\]

If $\ta$ is a monoid and $\tc$ is a comonoid, 
then $\iH(\tc,\ta)$ is a monoid (Section~\ref{ss:iH}) and 
hence $\Ser\bigl(\iH(\tc,\ta)\bigr)$ is an algebra under the Cauchy product~\eqref{e:convseries}.
On the other hand, the space $\Hom_{\Ssk}(\tc,\ta)$ is an algebra 
under the convolution product of Section~\ref{ss:convolution}. 
Building on~\eqref{e:seriH}, we have that these two are the same algebra:
$\Ser\bigl(\iH(\tc,\ta)\bigr) = \Hom_{\Ssk}(\tc,\ta)$.
In particular, the unit series $u$~\eqref{e:unitseries} coincides in this case with
the morphism~\eqref{e:unit-conv}.

Let $\thh$ and $\tk$ be connected bimonoids.
Assume that $\thh$ is finite-dimensional. 
As explained in Section~\ref{ss:iH}, 
$\iH(\thh,\tk)$ is then a finite-dimensional connected bimonoid. 
We then have, by~\eqref{e:exp-glike},
\[
\Exs\bigl(\iH(\thh,\tk)\bigr)\subseteq\Gls\bigl(\iH(\thh,\tk)\bigr).
\]
This inclusion can be refined as follows. 
We let 
\[
\Hom_{\Mo{\Ssk}}(\thh,\tk)
\]
denote the subset of $\Hom_{\Ssk}(\thh,\tk)$ consisting of morphisms of monoids. 
We similarly use $\Hom_{\Co{\Ssk}}(\thh,\tk)$ for comonoid morphisms.

\begin{proposition}\label{p:exp-mor-glike}
For $\thh$ and $\tk$ as above, we have
\[
\Exs\bigl(\iH(\thh,\tk)\bigr)\subseteq \Hom_{\Mo{\Ssk}}(\thh,\tk) \subseteq \Gls\bigl(\iH(\thh,\tk)\bigr)
\]
and
\[
\Exs\bigl(\iH(\thh,\tk)\bigr)\subseteq \Hom_{\Co{\Ssk}}(\thh,\tk) \subseteq \Gls\bigl(\iH(\thh,\tk)\bigr).
\]
\end{proposition}
\begin{proof}
Let $f\in \Ser\bigl(\iH(\thh,\tk)\bigr) = \Hom_{\Ssk}(\thh,\tk)$.
In view of~\eqref{e:iHprod} and~\eqref{e:iHcoprod}, the conditions 
expressing that $f$ belongs to each of the sets 
\[
\Exs\bigl(\iH(\thh,\tk)\bigr), \
\Hom_{\Mo{\Ssk}}(\thh,\tk), \ \Hom_{\Co{\Ssk}}(\thh,\tk),
\pand \Gls\bigl(\iH(\thh,\tk)\bigr)
\] 
are, respectively,
\begin{gather*}
\mu_{S,T}(f_S\otimes f_T)\Delta_{S,T} = f_I,\\
\mu_{S,T}(f_S\otimes f_T) = f_I \mu_{S,T},\\
(f_S\otimes f_T)\Delta_{S,T} = \Delta_{S,T}f_I,\\
f_S\otimes f_T = \Delta_{S,T}f_I\mu_{S,T},
\end{gather*}
in addition to $f_\emptyset=\iota_\emptyset\epsilon_\emptyset$.
The first condition implies the second and the third, and either of these implies
the fourth, in view of~\eqref{e:hopf-split}.
\end{proof}

Let $\thh$ and $\tk$ be connected bimonoids with $\thh$ finite-dimensional, 
and $f:\thh\to\tk$ a morphism of species.
Proposition~\ref{p:exp-mor-glike} implies that if $f$ is a morphism of
either monoids or comonoids,
then it is a group-like series of $\iH(\thh,\tk)$. 
The same conclusion holds if $f$ is
an antimorphism of either monoids or comonoids. 
For this one may use a similar argument 
to that in the proof of Proposition~\ref{p:exp-mor-glike}, 
using the second identity in~\eqref{e:hopf-split}.
In particular, both the identity and the antipode of $\thh$ are
group-like series of $\iE(\thh)$.

% Or we could resort to some nonsense using op and cop.
% For the antipode the reason is simpler: 
% the product in the algebra of series is convolution, 
% so the antipode is the inverse of the identity in the algebra of series, 
% hence it must be group-like.

\smallskip

Let us denote the antipodes of $\thh$, $\tk$, and $\iH(\thh,\tk)$ by
$\apode^\thh$, $\apode^\tk$, and $\apode^\iH$, respectively.

\begin{corollary}\label{c:apode-iH}
Let $\thh$ and $\tk$ be as above.
If $f:\thh\to\tk$ is a morphism of monoids, then
$\apode^\iH(f) = f \apode^\thh$.
If $f$ is a morphism of comonoids, then
$\apode^\iH(f) = \apode^\tk f$.
\end{corollary}
\begin{proof}
Since in both cases $f$ is a group-like series of $\iH(\thh,\tk)$,
the series $\apode^\iH(f)$ is its inverse with respect to convolution.
But we know from Section~\ref{ss:bim-hopf} that this inverse
is as stated.
\end{proof}

The previous result can also be shown using 
Takeuchi's formula~\eqref{e:antipode-r} and Corollary~\ref{c:hopf-split}:
For $f$ a morphism of monoids,
\begin{multline*}
\apode^\iH_I(f_I) = \sum_{F\vDash I} (-1)^{\len(F)} \mu_{F}\Delta_{F}f_I\mu_{F}\Delta_{F} 
= \sum_{F\vDash I} (-1)^{\len(F)} \mu_{F}\Delta_{F}\mu_{F}f_F\Delta_{F}\\
= \sum_{F\vDash I} (-1)^{\len(F)} \mu_{F}f_F\Delta_{F} 
= \sum_{F\vDash I} (-1)^{\len(F)} f_I\mu_{F}\Delta_{F} 
= f_I\apode^\thh_I.
\end{multline*}
Takeuchi's formula is used in the first and last steps,
$f_I\mu_{F}=\mu_{F}f_F$ is used twice, and $\Delta_{F}\mu_{F}=\id$ is used once.

\begin{remark}
Connectedness is essential in Corollary~\ref{c:apode-iH}.
If $H$ and $K$ are Hopf algebras, with $H$ finite-dimensional,
then 
\[
\Hom(H,K)\cong H^*\otimes K
\] 
is a Hopf algebra, and the antipode
sends any linear map $f:H\to K$ to $\apode^K f \apode^H$, in contrast
to Corollary~\ref{c:apode-iH}. 
\end{remark}

% I tried to obtain a
% formula for $\apode^\iH(f)$ in the nonconnected case using~\eqref{e:takeuchi-general}, and
% assuming that f is a morphism
% of (co)monoids or even the identity, but was not able to conclude anything.

{}From Proposition~\ref{p:iH-prim} we know that
\[
\iH\bigl(\thh,\Pc(\tk)\bigr) \subseteq \Pc\bigl(\iH(\thh,\tk)\bigr)
\qand
\iH\bigl(\Qc(\thh),\tk\bigr) \subseteq \Pc\bigl(\iH(\thh,\tk)\bigr).
\]
Passing to series we obtain, in view of~\eqref{e:primser-part} and~\eqref{e:seriH}, 
\[
\Hom_{\Ssk}\bigl(\thh,\Pc(\tk)\bigr) \subseteq \Prs\bigl(\iH(\thh,\tk)\bigr)
\qand
\Hom_{\Ssk}\bigl(\Qc(\thh),\tk\bigr) \subseteq \Prs\bigl(\iH(\thh,\tk)\bigr).
\]

We now show that under the $\exp$-$\log$ correspondence of Section~\ref{ss:exp-log}
\[
\xymatrix@C+15pt{
\Prs\bigl(\iH(\thh,\tk)\bigr) \ar@<0.6ex>[r]^-{\exp} & \ar@<0.6ex>[l]^-{\log} 
\Gls\bigl(\iH(\thh,\tk)\bigr)
}
\]
the space $\Hom_{\Ssk}\bigl(\thh,\Pc(\tk)\bigr)$ corresponds with
$\Hom_{\Co{\Ssk}}(\thh,\tk)$ (provided $\thh$ is cocommutative),
and the space $\Hom_{\Ssk}\bigl(\Qc(\thh),\tk\bigr)$ corresponds with
$\Hom_{\Mo{\Ssk}}(\thh,\tk)$ (provided $\tk$ is commutative).
These results hold in greater generality, as stated below.
In particular, we no longer require finite-dimensionality.

\smallskip

Let $\ta$ be a monoid and $\tc$ be a comonoid. Then $\iH(\thh,\ta)$
and $\iH(\tc,\tk)$ are monoids under convolution. The $\exp$ and $\log$
maps below act on the spaces of series of these monoids, namely
$\Hom_{\Ssk}(\thh,\ta)$ and $\Hom_{\Ssk}(\tc,\tk)$.

\begin{theorem}\label{t:exp-log-iH}
Let $\tc$ be a cocommutative comonoid and $\tk$ a connected bimonoid.
The maps $\exp$ and $\log$ restrict to inverse bijections
\begin{equation}\label{exp-log-P}
\xymatrix@C+15pt{
\Hom_{\Ssk}\bigl(\tc,\Pc(\tk)\bigr) \ar@<0.6ex>[r]^-{\exp} & \ar@<0.6ex>[l]^-{\log} 
\Hom_{\Co{\Ssk}}(\tc,\tk).
}
\end{equation}
Let $\thh$ be a connected bimonoid and $\ta$ a commutative monoid.
The maps $\exp$ and $\log$ restrict to inverse bijections
\begin{equation}\label{exp-log-Q}
\xymatrix@C+15pt{
\Hom_{\Ssk}\bigl(\Qc(\thh),\ta\bigr)\Bigr) \ar@<0.6ex>[r]^-{\exp} & \ar@<0.6ex>[l]^-{\log} 
\Hom_{\Mo{\Ssk}}(\thh,\ta).
}
\end{equation}
\end{theorem}

We provide two proofs of this result, each one with its own advantage.
The first one employs the higher forms of the bimonoid axioms of Section~\ref{s:higher}
and the combinatorics of the Tits algebra.
It extends to more general settings (beyond the scope of this paper, 
but hinted at in Section~\ref{ss:braid}) where the Tits product is defined.
The second proof relies on functional calculus and extends to a context
such as that of complete Hopf algebras.
It is essentially the same proof as those in~\cite[Appendix~A, Proposition~2.6]{Qui:1969} and~\cite[Theorem~9.4]{Sch:1994}.

\begin{firstproof}
In view of~\eqref{e:exp-log1}, it suffices to
show that $\exp$ and $\log$ map as stated. 
We show that if $f:\tc\to\Pc(\tk)$ is a morphism of species, 
then $\exp(f):\tc\to\tk$ is a morphism of comonoids. 
The remaining verifications are similar.

We have to show that
\[
\Delta_{S,T} \exp(f)_I = \bigl(\exp(f)_S\otimes \exp(f)_T\bigr)\Delta_{S,T}
\]
for $I=S\sqcup T$. We may assume that $S$ and $T$ are nonempty.

According to~\eqref{e:iHcoprod} and~\eqref{e:calc-exp},
\[
\exp(f)_I = \sum_{F\vDash I} \frac{1}{\len(F)!}\, \mu_F f_F \Delta_F.
\]
Let $G$ denote the composition $(S,T)$. 
In the following calculation, 
we make use of the higher compatibility~\eqref{e:gen-comp-conn} for $\tk$.
\[
\Delta_{S,T} \exp(f)_I = \sum_{F\vDash I} \frac{1}{\len(F)!}\, \Delta_G \mu_F f_F \Delta_F
= \sum_{F\vDash I} \frac{1}{\len(F)!}\, \mu_{GF/G} \beta \Delta_{FG/F} f_F \Delta_F.
\]
Since $f$ maps to the primitive part, $\Delta_{FG/F} f_F=0$ unless $FG=F$.
In this case, letting $F_1:=F|_S$ and $F_2:=F|_T$, we have
\[
GF = F_1\cdot F_2 
\qand
\mu_{GF/G} = \mu_{F_1}\otimes\mu_{F_2},
\]
and also
\[
\beta \Delta_{FG/F} f_F = \beta f_F = (f_{F_1}\otimes f_{F_2})\beta.
\]
Therefore,
\[
\Delta_{S,T} \exp(f)_I = \sum_{\substack{F_1\vDash S\\F_2\vDash T}}
\frac{1}{\bigl(\len(F_1) +\len(F_2)\bigr)!} (\mu_{F_1}\otimes\mu_{F_2})(f_{F_1}\otimes f_{F_2})\sum_{\substack{F:\,FG=F\\F|_S=F_1,\, F|_T=F_2}}\beta \Delta_F.
\]
Since $\tc$ is cocommutative, 
\[
\beta\Delta_F=\Delta_{F_1\cdot F_2}=(\Delta_{F_1}\otimes\Delta_{F_2})\Delta_G,
\]
by the dual of~\eqref{e:gen-comm-conn}. 
The number of terms in the last sum is
\[
\binom{\len(F_1) +\len(F_2)}{\len(F_1)}.
\]
Therefore,
\begin{align*}
\Delta_{S,T} \exp(f)_I & =
\sum_{\substack{F_1\vDash S\\F_2\vDash T}}
\frac{1}{\len(F_1)!}\,\frac{1}{\len(F_2)!}\,
\bigl((\mu_{F_1}f_{F_1}\Delta_{F_1})\otimes(\mu_{F_2}f_{F_2}\Delta_{F_2})\bigr)\Delta_G\\
& = \bigl(\exp(f)_S \otimes \exp(f)_T\bigr)\Delta_{S,T},
\end{align*}
as needed. \qed
\end{firstproof}

We remark that a morphism of species $\Qc(\thh)\to \ta$ is the same 
as a \emph{derivation} $\thh\to\ta$ when $\ta$ is endowed with the trivial
$\thh$-bimodule structure. 
Dually, a morphism of species $f:\tc\to \Pc(\tk)$ is the same as a morphism of species
$f:\tc\to\tk$ such that
\begin{equation}\label{e:derivation}
\Delta f = (f\bdot u + u\bdot f)\Delta,
\end{equation}
where $u$ is the unit in the convolution algebra $\Hom_{\Ssk}(\tc,\tk)$, as in~\eqref{e:unit-conv}. 
This says that $f$ is a coderivation from $\tc$ viewed as a trivial $\tk$-bicomodule
to $\tk$.

\begin{secondproof}
As in the first proof, we start from $f:\tc\to\Pc(\tk)$ and show that $\exp(f):\tc\to\tk$
is a morphism of comonoids. 
We deduce this fact by calculating
\begin{multline*}
\Delta \exp(f) = \exp(\Delta f) = \exp\bigl((f\bdot u + u\bdot f)\Delta\bigr) = \exp(f\bdot u + u\bdot f)\Delta =\\
\bigl(\exp(f\bdot u)\conv \exp( u\bdot f) \bigr)\Delta =
\bigl( \exp(f)\bdot \exp(f) \bigr) \Delta.
\end{multline*}

The exponentials in the middle terms are calculated in the convolution algebra
$\Hom_{\Ssk}(\tc\bdot\tc,\tk\bdot\tk)$. The Cauchy product of (co)monoids is
as in~\eqref{e:cau-mon} (with $q=1$).

The first equality holds by 
functoriality of the convolution algebra~\eqref{e:mor-conv} and naturality of functional
calculus (item~\eqref{i:nat} in Proposition~\ref{p:fun-calc}), 
since $\Delta:\tk\to\tk\bdot\tk$ is a morphism of monoids. 
The third holds by the same reason, since 
$\Delta:\tc\to\tc\bdot\tc$ is a morphism of comonoids (this uses that $\tc$ is cocommutative). 
The second holds by~\eqref{e:derivation}. 
The fourth holds by~\eqref{e:exp-add}, since
$f\bdot u$ and $u\bdot f$ commute in the convolution algebra
$\Hom_{\Ssk}(\tc\bdot\tc,\tk\bdot\tk)$ by~\eqref{e:inter-conv}. 
We complete the proof by justifying the last equality.

In this same algebra, we have
\[
(f\bdot u)^{\conv n} = f^{\conv n}\bdot u.
\]
This follows again from~\eqref{e:inter-conv}. This implies
\[
\exp(f\bdot u) = \exp(f)\bdot u \quad \text{ and similarly }\quad
\exp(u\bdot f) = u\bdot \exp(f).
\]
Hence, 
\[
\exp(f\bdot u)\conv \exp(u\bdot f) = \bigl(\exp(f)\bdot u\bigr)\conv \bigl(u\bdot \exp(f)\bigr)
= \exp(f)\bdot \exp(f)
\]
by another application of~\eqref{e:inter-conv}. \qed
\end{secondproof}

\begin{proposition}\label{p:exp-log-iH}
Let $\thh$ be a connected bimonoid and $f:\thh\to\thh$ a morphism of species
with $f_\emptyset=\iota_\emptyset\epsilon_\emptyset$.
Then $\log(f)$ and $f$ agree when restricted to the primitive part $\Pc(\thh)$, and 
also when followed by the canonical projection to $\Qc(\thh)$.
\begin{align*}
&
\begin{gathered}
\xymatrix@C+40pt{
\thh \ar@<0.6ex>[r]^-{f} \ar@<-0.6ex>[r]_-{\log(f)} & \thh\\
\Pc(\thh) \xyinc[u] \ar@/_/[ru]_{f=\log(f)}
}
\end{gathered}
&
\begin{gathered}
\xymatrix@C+40pt{
\thh \ar@<0.6ex>[r]^-{f} \ar@<-0.6ex>[r]_-{\log(f)}
 \ar@/_/[rd]_{f=\log(f)} & \thh \xyonto[d]\\
 & \Qc(\thh)
}
\end{gathered}
\end{align*}
\end{proposition}
\begin{proof}
Let $x\in\Pc(\thh)[I]$. From~\eqref{e:calc-log}, we calculate 
\[
\log(f)_I(x) = - \sum_{F\vDash I} \frac{(-1)^{\len(F)}}{\len(F)}\, \mu_F f_F \Delta_F(x) 
= f_I(x),
\]
using~\eqref{e:gen-prim}.
\end{proof}

If $f:\thh\to\thh $ is a morphism of comonoids, then it preserves the primitive part,
and we deduce from Proposition~\ref{p:exp-log-iH} that
$\log(f)$ and $f$ restrict to the same map
on $\Pc(\thh)$. 
\begin{align*}
&
\begin{gathered}
\xymatrix@C+20pt{
\thh \ar@<0.6ex>[r]^-{f} \ar@<-0.6ex>[r]_-{\log(f)} & \thh\\
\Pc(\thh) \xyinc[u] \ar[r]_{f=\log(f)} & \Pc(\thh) \xyinc[u]
}
\end{gathered}
&
\begin{gathered}
\xymatrix@C+20pt{
\thh \ar@<0.6ex>[r]^-{f} \ar@<-0.6ex>[r]_-{\log(f)} \xyonto[d] & \thh \xyonto[d]\\
\Qc(\thh) \ar[r]_{f=\log(f)} & \Qc(\thh)
}
\end{gathered}
\end{align*}
If $f$ is a morphism of monoids, then $\log(f)$ and $f$ induce the same map on the
indecomposable quotient $\Qc(\thh)$. 
Specializing further, and combining with Theorem~\ref{t:exp-log-iH}, we obtain the
following result.

\begin{corollary}\label{c:exp-log-iH}
Let $\thh$ be a connected bimonoid. 
If $\thh$ is cocommutative, then $\log(\id)$ maps to $\Pc(\thh)$ and is in fact 
a projection from $\thh$ onto $\Pc(\thh)$.
If $\thh$ is commutative, then $\log(\id)$ factors through $\Qc(\thh)$ and splits the
canonical projection $\thh\onto\Qc(\thh)$.
\end{corollary}

We provide another (though similar) proof of 
this result in Corollary~\ref{c:char-op-1euler}.

\section{The characteristic operations (Hopf powers)}\label{s:action}

The considerations of Section~\ref{s:higher} show that the combinatorics
of set decompositions is intimately linked to the notion of bimonoid,
with set compositions playing the same role in relation to connected bimonoids.
On the other hand, in Section~\ref{s:faces}
we discussed two particular bimonoids $\tSig$ and $\tSigh$ which themselves
are built out of set (de)compositions.
This double occurrence of set (de)compositions acquires formal meaning
in this section, where we show that elements of $\tSig$ give rise to operations
on connected bimonoids, and elements of $\tSigh$ to operations
on arbitrary bimonoids.

\subsection{The characteristic operations on a bimonoid}\label{ss:can-action}

Let $\thh$ be a bimonoid and $h\in\thh[I]$. 
Given an element $z\in\tSigh[I]$, write
\[
z=\sum_{F} a_F \BH_F
\]
for some $a_F\in\Kb$, with $F$ running over the decompositions of $I$
(and all but a finite number of $a_F$ equal to zero).
Define an element $z\act h\in\thh[I]$ by
\begin{equation}\label{e:can-action}
z\act h := \sum_{F} a_F (\mu_F\Delta_F)(h).
\end{equation}
Here $\mu_F$ and $\Delta_F$ denote the higher product and coproduct of $\thh$.
We refer to $z\act h$ as the \emph{characteristic operation} of $z$ on $h$. 

In particular, for a decomposition $F$ of $I$, we have
\begin{equation}\label{e:can-actionF}
\BH_F\act h := (\mu_F\Delta_F)(h).
\end{equation}

We may take $\thh=\tSigh$, and~\eqref{e:can-action} results in an operation 
on each space $\tSigh[I]$. According to the following result,
this is simply the linearization of the Tits product
of Section~\ref{ss:decompositions}.

\begin{proposition}\label{p:ast-tits}
For any decompositions $F$ and $G$ of $I$,
\begin{equation}\label{e:ast-tits}
\BH_F\act \BH_G = \BH_{FG}.
\end{equation}
\end{proposition}
\begin{proof}
We have, by~\eqref{e:prod-it-sigh} and~\eqref{e:coprod-it-sigh},
\[
\BH_F\act \BH_G = (\mu_F\Delta_F)(\BH_G) = \mu_F(\BH_{FG/F}^{\can}) = \BH_{FG}. \qedhere
\]
\end{proof}

Endow $\tSigh[I]$ with the corresponding algebra structure (the monoid algebra
of the monoid of compositions under the Tits product). The unit element is $\BH_{(I)}$
and the product is as in~\eqref{e:ast-tits}. We call $\tSigh[I]$ the \emph{Tits algebra}
of decompositions.

Consider the map
\begin{equation}\label{e:map-action}
\tSigh[I]\otimes\thh[I] \to \thh[I],
\quad
z\otimes h\mapsto z\act h.
\end{equation}
The result below shows that, when $\thh$ is cocommutative,
$\act$ is a left action of the algebra $\tSigh[I]$ on the space $\thh[I]$
(and a right action when $\thh$ is commutative). 

\begin{theorem}\label{t:can-action}
The following properties hold.
\begin{enumerate}[(i)]
\item For any $h\in\thh[I]$,
\begin{equation}\label{e:can-action-unit}
\BH_{(I)}\act h = h.
\end{equation}
\item If $\thh$ is cocommutative, then for any $z,w\in\tSigh[I]$ and
$h\in\thh[I]$,
\begin{equation}\label{e:can-action-asso}
(z\act w)\act h = z\act(w\act h).
\end{equation}
\item If $\thh$ is commutative, then for any $z,w\in\tSigh[I]$ and
$h\in\thh[I]$,
\begin{equation}\label{e:can-action-asso-op}
(z\act w)\act h = w\act(z\act h).
\end{equation}
\end{enumerate}
\end{theorem}
\begin{proof}
We show~\eqref{e:can-action-asso}. 
By linearity, we may assume $z=\BH_F$ and $w=\BH_G$, where
$F$ and $G$ are decompositions of $I$. We then have, by~\eqref{e:ast-tits},
\[
(z\act w)\act h = \BH_{FG}\act h = (\mu_{FG}\Delta_{FG})(h).
\]
On the other hand,
\[
z\act(w\act h) = (\mu_F\Delta_F)\bigl((\mu_G\Delta_G)(h)\bigr) = \bigl(\mu_F(\Delta_F\mu_G)\Delta_G\bigr)(h).
\]
By higher compatibility~\eqref{e:gen-comp}, the preceding equals
\[
\bigl(\mu_F\mu_{FG/F}\beta^{\can}\Delta_{GF/G}\Delta_G)(h)
\]
and by higher (co)associativity~\eqref{e:gen-asso} and cocommutativity (the dual of~\eqref{e:gen-comm}) this equals
\[
(\mu_{FG}\beta^{\can}\Delta_{GF})(h) = (\mu_{FG}\Delta_{FG})(h),
\]
as needed.
\end{proof}

The following result links the bimonoid structures of $\tSigh$ and $\thh$
through the characteristic operations. First, we extend~\eqref{e:map-action}
to a map
\[
\tSigh(F)\otimes\thh(F) \to \thh(F)
\]
for any decomposition $F=(I_1,\ldots,I_k)$, by defining
\begin{equation}\label{e:can-action2}
(z_1\otimes\cdots\otimes z_k)\act (h_1\otimes\cdots\otimes h_k) :=
(z_1\act h_1)\otimes\cdots\otimes (z_k\act h_k)
\end{equation}
for $z_i\in\tSigh[I_i]$ and $h_i\in\thh[I_i]$. Let $G$ be a decomposition
such that $F_+\leq G_+$. It follows that
\begin{equation}\label{e:can-actionFG}
\BH_{G/F}^{\gamma}\act h = (\mu_{G/F}^{\gamma}\Delta_{G/F}^{\gamma})(h)
\end{equation}
for all splittings $\gamma$ of $(F,G)$, and $h\in\thh(F)$.

\begin{theorem}\label{t:distributivity}
Let $\thh$ be a bimonoid and $F$ a decomposition of $I$.
The following properties are satisfied.
\begin{enumerate}[(i)]
\item For any $z\in\tSigh(F)$ and
$h\in\thh[I]$,
\begin{equation}\label{e:right-dist}
\mu_F(z)\act h = \mu_F\bigl(z\act \Delta_F(h)\bigr),
\end{equation}

\item If $\thh$ is commutative, then for any $z\in\tSigh[I]$ and
$h\in\thh(F)$, 
\begin{equation}\label{e:left-dist}
z\act\mu_F(h) = \mu_F\bigl(\Delta_F(z)\act h\bigr).
\end{equation}

\item If $\thh$ is cocommutative, then for any $z\in\tSigh[I]$ and
$h\in\thh[I]$,
\begin{equation}\label{e:delta-star}
\Delta_F(z\act h) = \Delta_F(z)\act \Delta_F(h).
\end{equation}

\end{enumerate}
\end{theorem}
\begin{proof}
We prove~\eqref{e:right-dist}. 
We may assume $z=\BH_{G/F}^{\gamma}$ for 
$\gamma$ a splitting of $(F,G)$, $F_+\leq G_+$. We have
\begin{multline*}
\mu_F\bigl(z\act \Delta_F(h)\bigr) =
(\mu_F\mu_{G/F}^{\gamma}\Delta_{G/F}^{\gamma}\Delta_F)(h)= (\mu_{G}\Delta_{G})(h) = \BH_G \act h =
\mu_F(z)\act h.
\end{multline*}
We used~\eqref{e:can-actionFG}, higher (co)associativity~\eqref{e:gen-asso} and formula~\eqref{e:prod-it-sig}.

The remaining identities can be proven by similar arguments.
\end{proof}

\begin{remark}
The conditions in Theorem~\ref{t:distributivity} may be interpreted as axioms
for a module over a ring in the setting of $2$-monoidal categories, 
with~\eqref{e:right-dist} and~\eqref{e:left-dist} playing the role of right and left distributivity of multiplication over addition. 
We do not pursue this point here. 
Related results are given by Hazewinkel~\cite[Section~11]{Haz:2009}, 
Patras and Schocker~\cite[Corollary~22]{PatSch:2006},
and Thibon~\cite[Formula~(52)]{Thi:2001}, among others.
\end{remark}

Let $\thh$ be a bimonoid and $z\in\tSigh[I]$. 
Let $\facematrix_I(z): \thh[I] \to \thh[I]$ be the map defined by
\begin{equation}\label{e:canz}
\facematrix_I(z)(h) := z\act h.
\end{equation}

We have $\facematrix_I(z)\in \End_{\Kb}(\thh[I])$, and we obtain a morphism of species
\[
\facematrix: \tSigh \to \iE(\thh).
\]
Recall from Section~\ref{ss:iH}
that $\iE(\thh)$ is a monoid.

\begin{proposition}\label{p:Psi-mor}
The map $\facematrix$ is a morphism of monoids. 
If $\thh$ is cocommutative, then each
\[
\facematrix_I : \tSigh[I] \to \End_{\Kb}(\thh[I])
\]
is a morphism from the Tits algebra to the algebra of linear endomorphisms (under
ordinary composition). 
If $\thh$ is commutative, then it is an antimorphism.
\end{proposition}
\begin{proof} 
The fact that $\facematrix$ is a morphism of monoids follows from~\eqref{e:right-dist}. 
The last statements rephrase
the fact that in these cases the Tits algebra acts (from the left or from the right)
on $\thh[I]$ (Theorem~\ref{t:can-action}).
\end{proof}

\subsection{The characteristic operations on arbitrary bialgebras}\label{ss:can-action-bialg}

Let $H$ be an arbitrary bialgebra over $\Kb$. Let $\mu,\iota,\Delta,\epsilon$
denote the structure maps. 
We employ Sweedler's notation in the form
\[
\Delta(h) = \sum h_1\otimes h_2
\]
for $h\in H$.

For each integer $n\geq 1$, the higher product
\[
\mu^{(n)} : H^{\otimes(n+1)} \to H
\]
is well-defined by associativity. 
One also defines
\[
\mu^{(-1)}=\iota \qand \mu^{(0)}=\id.
\]
The higher coproducts are defined dually.

Given $h\in H$ and $p\in\Nb$, define an element $h^{(p)}\in H$ by
\[
h^{(p)} := (\mu^{(p-1)}\Delta^{(p-1)})(h).
\]
In the recent Hopf algebra literature, the operations $h\mapsto h^{(p)}$
are called \emph{Hopf powers}~\cite{NgSch:2008} or \emph{Sweedler powers}~\cite{KMN:2012,KSZ:2006}. 
The term \emph{characteristic} is used in~\cite[Section~1]{GerSch:1991} and~\cite[Section~1]{Pat:1993} for these operations on graded bialgebras.
They enjoy the following properties. 
Let $h,k\in H$ and $p,q\in\Nb$.

\begin{enumerate}[(i)]
\item $h^{(1)} = h$.
\item \label{e:Haction}
If $H$ is either commutative or cocommutative, then 
$h^{(pq)} = (h^{(p)})^{(q)}$.
\item $h^{(p+q)} = \sum h_1^{(p)} h_2^{(q)}$.
\item If $H$ is commutative, then 
$(hk)^{(p)} = h^{(p)} k^{(p)}$.
\item If $H$ is cocommutative, then 
$\Delta(h^{(p)}) = \sum h_1^{(p)}\otimes h_2^{(p)}$.
\end{enumerate}

These properties are special cases of those in Theorems~\ref{t:can-action}
and~\ref{t:distributivity}. 
They arise by choosing $\thh=\tone_H$ (as in Section~\ref{ss:hopfalg}) and 
$F=\emptyset^2$. In this case, $\thh[\emptyset]=H$, $\mu_F=\mu$, and 
$\Delta_F=\Delta$. For instance,~\eqref{e:Haction} follows from~\eqref{e:can-action-asso}
in view of the fact that the Tits product in $\rSigh[\emptyset]$ corresponds
to multiplication in $\Nb$. 
More general properties follow by choosing $F=\emptyset^p$.

\subsection{The characteristic operations on a connected bimonoid}\label{ss:can-action-conn}

Let $\thh$ be a connected bimonoid and $h\in\thh[I]$. Given an element $z\in\tSig[I]$,
formula~\eqref{e:can-action} defines an element $z\act h\in\thh[I]$
(the sum is now over the compositions of $I$).
As before, we refer to this element as the characteristic operation of $z$ on $h$.

The set $\rSig[I]$ of compositions of $I$
is a monoid under the Tits product~\eqref{e:tits}
and $\tSig[I]$ is the algebra of this monoid.
We use $\act$ to denote its product, as in~\eqref{e:ast-tits}.
We call $\tSig[I]$ the \emph{Tits algebra} of compositions.

The results of Proposition~\ref{p:ast-tits} and Theorem~\ref{t:can-action} continue to hold
for the characteristic operations of $\tSig$ on $\thh$, with the same proofs.
In particular, when $\thh$ is cocommutative, $\act$ defines an
action of the algebra $\tSig[I]$ on the space $\thh[I]$.

To the result of Proposition~\ref{p:ast-tits} we may add the following.

\begin{proposition}\label{p:H-Q-proj}
For any compositions $F$ and $G$ of $I$,
\begin{equation}\label{e:H-Q-proj}
\BH_F\act \BQ_G=
\begin{cases}
\BQ_{FG} & \text{ if } GF=G,\\
0 & \text{ if } GF>G.
\end{cases}
\end{equation}
In particular, if $F$ and $G$ have the same support,
then $\BH_G\act \BQ_F = \BQ_G$.
\end{proposition}
\begin{proof}
This follows from Proposition~\ref{p:Qface} and~\eqref{e:ast-tits}.
\end{proof}

Consider the morphism of bimonoids $\upsilon:\tSigh \to \tSig$
of~\eqref{e:dec-comp}. As noted in~\eqref{e:dec-comp-prod}, 
the map $\upsilon:\tSigh[I] \to \tSig[I]$ preserves Tits products.
The following result shows that, when specialized to connected bimonoids, 
the characteristic operations factor through $\upsilon$.

\begin{proposition}\label{p:can-action-factor}
Let $\thh$ be a connected bimonoid.
For any $z\in\tSigh[I]$ and $h\in\thh[I]$,
\[
z\act h = \upsilon(z)\act h.
\]
\end{proposition}
\begin{proof}
We may assume $z=\BH_F$ for $F$ a decomposition of $I$. 
We then have
\[
\upsilon(z)\act h = \BH_{\pos{F}}\act h = (\mu_{\pos{F}}\Delta_{\pos{F}})(h) =
(\mu_F\iota_F\epsilon_F\Delta_F)(h)=(\mu_F\Delta_F)(h) = z\act h.
\]
We used~\eqref{e:gen-unit}, its dual, and~\eqref{e:hopf-split-gen2}. 
(The latter uses connectedness.)
\end{proof}

On connected bimonoids, the characteristic operations enjoy additional
properties to those in Theorem~\ref{t:distributivity}.

\begin{proposition}\label{p:distributivity-conn}
Let $\thh$ be a connected bimonoid and $F$ a composition of $I$.
The following properties are satisfied.
\begin{enumerate}[(i)]
\item For any $z\in\tSig(F)$ and
$h\in\thh[I]$,
\begin{equation}\label{e:left-codist}
z\act\Delta_F(h) = \Delta_F\bigl(\mu_F(z)\act h\bigr).
\end{equation}

\item For any $z\in\tSig[I]$ and
$h\in\thh(F)$,
\begin{equation}\label{e:right-codist}
\Delta_F(z)\act h = \Delta_F\bigl(z\act \mu_F(h)\bigr),
\end{equation}

\item For any $z\in\tSig(F)$ and
$h\in\thh(F)$,
\begin{equation}\label{e:mu-star}
\mu_F(z\act h) = \mu_F(z)\act \mu_F(h).
\end{equation}

\end{enumerate}
\end{proposition}
\begin{proof}
Formula~\eqref{e:left-codist} follows by applying $\Delta_F$ to both
sides of~\eqref{e:right-dist}, in view of the first formula in~\eqref{e:hopf-split2}.
Formula~\eqref{e:mu-star} follows
by replacing $h$ in~\eqref{e:right-dist} for $\mu_F(h)$
and employing~\eqref{e:hopf-split2}.

We prove~\eqref{e:right-codist}.
We may assume $z=\BH_G$ for $G$ a composition of $I$. We have
\begin{multline*}
\Delta_{F}\bigl(z\act \mu_{F}(h)\bigr) =
(\Delta_{F}\mu_{G}\Delta_{G}\mu_{F})(h) =\\
(\mu_{FG/F}\beta\Delta_{GF/G}\mu_{GF/G}\beta\Delta_{FG/F})(h) =
(\mu_{FG/F}\beta^2\Delta_{FG/F})(h) =\\
(\mu_{FG/F}\Delta_{FG/F})(h) = \BH_{FG/F}\act h =
\Delta_{F}(z)\act h.
\end{multline*}
We used the higher compatibility~\eqref{e:gen-comp} and~\eqref{e:hopf-split2}
for $GF/G$, as well as the fact that $\beta^2=\id$.
\end{proof}

Let $\thh$ be a connected bimonoid. 
By Proposition~\ref{p:can-action-factor}, the map $\facematrix: \tSigh \to \iE(\thh)$
factors through $\upsilon:\tSigh \to \tSig$. 
We also use $\facematrix$ to denote the resulting map
\[
\facematrix: \tSig \to \iE(\thh).
\]
Recall from Section~\ref{ss:iH}
that $\iE(\thh)$ is a monoid, and moreover
a Hopf monoid if $\thh$ is finite-dimensional.

\begin{proposition}\label{p:Psi-mor-conn}
The map $\facematrix$ is a morphism of monoids, and moreover of Hopf monoids if $\thh$
is finite-dimensional. 
If $\thh$ is cocommutative, then each
\[
\facematrix_I : \tSig[I] \to \End_{\Kb}(\thh[I])
\]
is a morphism from the Tits algebra to the algebra of linear endomorphisms (under
ordinary composition). 
If $\thh$ is commutative, then it is an antimorphism.
\end{proposition}
\begin{proof} 
The fact that $\facematrix$ is a morphism of comonoids follows
from~\eqref{e:right-codist}. 
The remaining statements follow from Proposition~\ref{p:Psi-mor} and the fact that $\upsilon$ preserves all the structure.
\end{proof}

\begin{remark}\label{r:patras}
Connected graded Hopf algebras also carry characteristic operations.
In this context, the role of the Hopf monoid $\tSig$
is played by the Hopf algebra of \emph{noncommutative symmetric functions}, 
and that of the Tits algebra by \emph{Solomon's descent algebra} of the symmetric group. 
The latter is the invariant subalgebra of $\tSig[n]$ under the action of $\Sr_n$ and hence
has a linear basis indexed by compositions of the integer $n$.

Noncommutative symmetric functions are studied in a series of papers
starting with \cite{GKLLRT} by Gelfand et al and including~\cite{Thi:2001}
by Thibon. The descent algebra (of a finite
Coxeter group) is introduced by Solomon in~\cite{Sol:1976}.
The theory of characteristic operations on graded Hopf algebras appears 
in work of Patras~\cite{Pat:1994}.
It also occurs implicitly in a number of places in the literature.
The operation of Solomon's descent algebra on the Hopf algebra of
noncommutative symmetric functions is considered in~\cite[Section~5.1]{GKLLRT}.
\end{remark}

\subsection{Primitive operations}\label{ss:action-prim}

We study the effect of operating on a cocommutative connected bimonoid $\thh$
by primitive elements of $\tSig$. 
We write 
\begin{equation}\label{e:typical}
z=\sum_{F\vDash I} a_F \BH_F 
\end{equation}
for a typical element of $\tSig[I]$.
 
\begin{theorem}\label{t:action-prim}
Let $\thh$ and $z$ be as above, and $h\in\thh[I]$. 
\begin{enumerate}[(i)]
\item If $z\in\Pc(\tSig)[I]$, then $z\act h\in\Pc(\thh)[I]$.
\item If in addition $h\in\Pc(\thh)[I]$, then $z\act h = a_{(I)} h$.
\end{enumerate}
\end{theorem}
\begin{proof}
Part (i) follows from~\eqref{e:delta-star}.
Regarding (ii), we have
\[
z\act h = a_{(I)} (\BH_{(I)}\act h) + \sum_{\substack{F\vDash I\\F\neq(I)}} a_F (\BH_F\act h)
= a_{(I)} h.
\]
We used~\eqref{e:can-action-unit} and the fact that $\Delta_F(h)=0$
for $F\neq(I)$,
which holds by~\eqref{e:gen-prim} since $h$ is primitive.
\end{proof} 

It is convenient to restate the above result
in terms of the map $\facematrix_I(z)$ of~\eqref{e:canz}.

\begin{theorem}\label{t:act-to-prim}
Let $\thh$ and $z$ be as in Theorem~\ref{t:action-prim}.
\begin{enumerate}[(i)]
\item If $z\in\Pc(\tSig)[I]$, then 
$\facematrix_I(z)(\thh[I])\subseteq \Pc(\thh)[I]$.
\item If in addition $a_{(I)}\neq 0$, then $\facematrix_I(a_{(I)}^{-1} z)$ is a projection from 
$\thh[I]$ onto $\Pc(\thh)[I]$.
\end{enumerate}
\end{theorem}

\begin{remark}
Statement (i) in Theorem~\ref{t:act-to-prim} says that
$\facematrix\bigl(\Pc(\tSig)\bigr)\subseteq \iH\bigl(\thh,\Pc(\thh))$.
Since morphisms of comonoids preserve primitive elements, 
Proposition~\ref{p:Psi-mor-conn} implies that (when $\thh$ is finite-dimensional)
$\facematrix\bigl(\Pc(\tSig)\bigr)\subseteq \Pc\bigl(\iH(\thh)\bigr)$.
This inclusion, however, is weaker than the previous, in view of
Proposition~\ref{p:iH-prim}.
\end{remark}

The above result has a converse:
%There is a converse to Theorem~\ref{t:act-to-prim}.

\begin{proposition}\label{p:act-to-prim-conv}
Let $z\in\tSig[I]$ be an element such that
\[
\facematrix_I(z)(\thh[I]) \subseteq \Pc(\thh)[I]
\]
for every cocommutative connected bimonoid $\thh$.
Then 
\[
z\in \Pc(\tSig)[I].
\]
\end{proposition}
\begin{proof}
Since $\tSig$ is a cocommutative connected bimonoid, we may apply the hypothesis
to $\thh=\tSig$. 
Hence $z\act \BH_G\in\Pc(\tSig)[I]$ for all compositions $G$.
But in view of~\eqref{e:ast-tits} we have $z=z\act \BH_{(I)}$,
so $z\in\Pc(\tSig)[I]$.
\end{proof}

We may consider the operation of $\tSig$ on itself.

\begin{corollary}\label{c:action-prim}
% Let $z=\sum_{F\vDash I} a_F \BH_F \in \tSig[I]$.
% If $z\in\Pc(\tSig)[I]$, then $z\act z = a_{(I)} z$.
Let $z$ be as in~\eqref{e:typical}.
If $z$ is a primitive element of $\tSig[I]$, then $z\act z = a_{(I)} z$.
\end{corollary}
\begin{proof}
This follows from item (ii) in Theorem~\ref{t:action-prim} by letting $\thh=\tSig$, $h=z$.
\end{proof}

Thus, a primitive element in the Hopf monoid $\tSig$ either satisfies $z\act z=0$, or
(if $a_{(I)}\neq 0$) is a quasi-idempotent in the Tits algebra $(\tSig[I],\act)$. 
If $a_{(I)}=1$, it is an idempotent.

The link between the Hopf monoid structure of $\tSig$ and the algebra
structure of $(\tSig[I],\act)$ is further witnessed by the following result.

\begin{corollary}\label{c:prim-deal}
The primitive part $\Pc(\tSig)[I]$ is a right ideal of the Tits algebra $(\tSig[I],\act)$.
Moreover, if $z$ as in~\eqref{e:typical} is any primitive element with $a_{(I)}\neq 0$,
then $\Pc(\tSig)[I]$ is the right ideal of $(\tSig[I],\act)$ generated by $z$.
\end{corollary}
\begin{proof}
Applying statement (i) in Theorem~\ref{t:action-prim} with $\thh=\tSig$ and $h$ an arbitrary
element of $\tSig[I]$, we deduce that the right ideal generated by $z$
is contained in $\Pc(\tSig)[I]$. Applying item (ii) to an arbitrary element $h$ of
$\Pc(\tSig)[I]$, we deduce $h = a_{(I)}^{-1} (z\act h)$, and hence the converse inclusion. 
\end{proof}

Such primitive elements do exist;
examples are given in Sections~\ref{ss:first-euler} and~\ref{ss:dynkin}.

The following is a necessary condition for an element of $\tSig$ to be primitive.

\begin{proposition}\label{p:xyz-flat-sum}
Suppose $z$ as in~\eqref{e:typical} is any primitive element with $a_{(I)}=1$.
Then for any partition $X$ of $I$,
\begin{equation}\label{e:xyz-flat-sum}
\sum_{F:\,\supp F = X} a_F = \mu(\minflat,X),
\end{equation}
where $\mu$ denotes the M\"obius function of the partition lattice~\eqref{e:partmobiusPi}.
\end{proposition}
\begin{proof}
Denote the left-hand side of~\eqref{e:xyz-flat-sum} by $f(X)$. 
Since $a_{(I)}=1$, we have $f(\minflat)=1$.
Moreover, for any $Y > \minflat$,
\[
\sum_{X:\,X \leq Y} f(X)=0.
\]
(Let $F$ be any face with support $Y$.
Then the coefficient of $\BH_F$ in $\BH_F\act z$ is the left-hand side above,
and it is zero since $z$ is primitive.)
These conditions imply $f(X)=\mu(\minflat,X)$ as required. 
\end{proof}

\subsection{The cumulants of a connected bimonoid}\label{ss:cumulant}

Let $\thh$ be a finite-dimensional connected bimonoid. 
Recall from~\eqref{e:sp-part}
that
\[
\thh(X) = \bigotimes_{B\in X} \thh[B]
\]
for each partition $X$ of $I$.
The \emph{cumulants} of $\thh$ are the integers $\eigmult{X}(\thh)$ defined by
\begin{equation}\label{e:mult}
\sum_{Y:\,Y\geq X} \eigmult{Y}(\thh) = \dim_{\Kb} \thh(X),
\end{equation}
or equivalently, by
\begin{equation}\label{e:mult-r}
\eigmult{X}(\thh) = \sum_{Y:\,Y\geq X} \mu(X,Y) \,\dim_{\Kb} \thh(Y),
\end{equation}
where $\mu$ denotes the M\"obius function of the partition lattice.
The $n$-th cumulant is
\[
\eigmult{n}(\thh) := \eigmult{\minflat}(\thh),
\]
where $\abs{I}=n$ and $\minflat$ is the partition of $I$ with one block.
Thus,
\begin{equation}\label{e:cumulant}
\eigmult{n}(\thh) = \sum_{Y\vdash I} \mu(\minflat,Y) \,\dim_{\Kb} \thh(Y).
\end{equation}
One can deduce from~\eqref{e:mobiusPi} that
\begin{equation}\label{e:mult-r-lump}
\eigmult{X}(\thh) = \prod_{B\in X} \eigmult{\abs{B}}(\thh).
\end{equation}
 
\begin{proposition}\label{p:dim-h-prim}
For any finite-dimensional cocommutative connected bimonoid $\thh$,
the dimension of its primitive part is
\begin{equation}\label{e:dim-h-prim}
\dim_{\Kb} \Pc(\thh)[I] = \eigmult{\abs{I}}(\thh).
\end{equation}
\end{proposition}
\begin{proof}
Let $z$ as in~\eqref{e:typical} be any primitive element of $\tSig[I]$ with $a_{(I)}=1$.
(As mentioned above, such elements exist.)
Then $\facematrix_I(z)$ is a projection onto $\Pc(\thh)[I]$, by Theorem~\ref{t:act-to-prim}.
Hence the dimension of its image equals its trace, and
\[
\dim_{\Kb} \Pc(\thh)[I] = \sum_F a_F \trace \facematrix_I(\BH_F).
\]
Since $\BH_F$ is idempotent, the trace of $\facematrix_I(\BH_F)$
equals the dimension of its image.
Now, $\facematrix_I(\BH_F)=\mu_F\Delta_F$ and $\Delta_F\mu_F=\id_{\thh(F)}$ by~\eqref{e:hopf-split2}. 
Hence the dimension of the image of $\facematrix_I(\BH_F)$ equals the dimension of $\thh(F)$.
So 
\[
\trace \facematrix_I(\BH_F) = \dim_{\Kb}\thh(F).
\]
Since $\thh(F)\cong \thh(\supp F)$, we have
\[
\dim_{\Kb} \Pc(\thh)[I] = \sum_F a_F \dim_{\Kb} \thh(F) = \sum_X \Bigl(\sum_{F:\,\supp F=X} a_F \Bigr)\,\dim_{\Kb} \thh(X).
\]
Now, by Proposition~\ref{p:xyz-flat-sum} and~\eqref{e:cumulant},
\[
\dim_{\Kb} \Pc(\thh)[I] = \sum_X \mu(\minflat,X) \,\dim_{\Kb} \thh(X) = \eigmult{\abs{I}}(\thh). 
\qedhere
\]
\end{proof}

It follows from~\eqref{e:mult-r-lump} and~\eqref{e:dim-h-prim} 
that the integers $\eigmult{n}(\thh)$ are nonnegative, 
a fact not evident from their definition. Here are some simple examples,
for the Hopf monoids $\wE$, $\wL$ and $\tPi$ of Section~\ref{s:examples}.
\[
\eigmult{n}(\wL)=(n-1)!, \qquad 
\eigmult{n}(\wE)= \begin{cases}
1 & \text{ if $n=1$,}\\
0 & \text{ otherwise,}
\end{cases}
\qquad
\eigmult{n}(\tPi) = 1.
\]

\begin{remark}\label{r:cumulant}
Suppose the integer $\dim_{\Kb} \thh[n]$ is the $n$-th \emph{moment} of a random variable $Z$. 
Then the integer $\eigmult{n}(\thh)$
is the $n$-th \emph{cumulant} of $Z$ in the classical sense~\cite{Fis:1929,FisWis:1931}.
Proposition~\ref{p:dim-h-prim} implies that
if $\thh$ is cocommutative, then all cumulants are nonnegative. For example,
$\dim_{\Kb} \tPi[n]$ is the $n$-th Bell number. 
This is the $n$-th moment of a Poisson variable of parameter $1$. 
Also, $\dim_{\Kb} \wL[n]=n!$ is the $n$-th moment of an exponential variable
of parameter $1$, and $\dim_{\Kb} \wE[n]=1$ is the $n$-th moment of the Dirac
measure $\delta_1$.
% the constant random variable $Z=1$
\end{remark}

\subsection{Group-like operations}\label{ss:action-glike}

Let $s$ be a series of $\tSigh$ and $\thh$ a bimonoid.
Building on~\eqref{e:canz}, the operation $\facematrix(s):\thh\to\thh$ is defined by
\[
\facematrix_I(s_I):\thh[I]\to\thh[I], \quad h\mapsto s_I\act h.
\]
If $\thh$ is connected, the same formula defines the operation of 
a series $s$ of $\tSig$ on $\thh$.

\begin{theorem}\label{t:action-glike-ser}
Let $g$ be a group-like series of $\tSigh$.
If $\thh$ is (co)commutative, 
then $\facematrix(g)$ is a morphism of (co)monoids.
\end{theorem}
\begin{proof}
The fact that $\facematrix(g)$ preserves products follows from~\eqref{e:left-dist},
and that it preserves coproducts from~\eqref{e:delta-star}. These formulas in fact
show that all higher (co)products are preserved; in particular (co)units are preserved.
\end{proof}

The \emph{universal} series $\univh$ of $\tSigh$ is defined by
\begin{equation}\label{e:expeuler}
\univh_I := \BH_{(I)}
\end{equation}
for all finite sets $I$. In particular, $\univh_\emptyset:=\BH_{\emptyset^1}$. 

\begin{lemma}\label{l:glikeH}
The universal series $\univh$ is group-like.
\end{lemma}
\begin{proof}
Indeed, by~\eqref{e:faces-def}, $\Delta_{S,T}(\BH_{(I)}) = \BH_{(S)}\otimes\BH_{(T)}$.
\end{proof}

By~\eqref{e:can-action-unit}, the universal series operates on $\thh$ as the identity:
\begin{equation}\label{e:univ-id}
\facematrix(\univh) = \id.
\end{equation}

The terminology is justified by the following results.

\begin{theorem}\label{t:univ-series}
Let $s$ be a series of a monoid $\ta$. 
Then there exists a unique morphism of monoids $\zeta:\tSigh\to\ta$ such that
\begin{equation}\label{e:univ-series1}
\zeta(\univh) = s.
\end{equation}
Moreover, $\zeta_I: \tSigh[I] \to \ta[I]$ is given by
\begin{equation}\label{e:univ-series2}
\zeta_I(\BH_F) := \mu_F(s_F)
\end{equation}
for any decomposition $F$ of $I$
\end{theorem}
\begin{proof}
This is a reformulation of the fact that 
$\tSigh$ is the free monoid on the species $\wE$.
\end{proof}

\begin{theorem}\label{t:univ-series2}
Let $g$ be a group-like series of a bimonoid $\thh$. 
Then the unique morphism of monoids $\zeta:\tSigh\to\thh$ such that
\begin{equation}\label{e:univ-series3}
\zeta(\univh) = g
\end{equation}
is in fact a morphism of bimonoids. Moreover, 
$\zeta_I: \tSigh[I] \to \thh[I]$ is given by
\begin{equation}\label{e:univ-series4}
\zeta_I(z) := z\act g_I.
\end{equation}
\end{theorem}
\begin{proof}
Let $F$ be a decomposition of $I$. 
We have, by~\eqref{e:univ-series3},
\[
\zeta_I(\BH_F) = \mu_F(g_F) = \mu_F\Delta_F(g_I) = \BH_F\act g_I,
\]
and this proves~\eqref{e:univ-series4}.
To prove that $\zeta$ is a morphism of comonoids, we may assume
that $\thh$ is cocommutative, since $g$ belongs to the 
coabelianization of $\thh$. 
In this case we may apply~\eqref{e:delta-star} and this yields
\[
\Delta_F\zeta_I(z) = \Delta_F(z\act g_I) = \Delta_F(z)\act\Delta_F(g_I) =
\Delta_F(z)\act g_F = \zeta_F\Delta_F(z),
\]
as required.
\end{proof}

From Theorem~\ref{t:univ-series} we deduce a bijection between
the set of series of $\ta$ and the set of monoid morphisms from $\tSigh$ to $\ta$
given by
\[
\Hom_{\Mo{\Ssk}}(\tSigh,\ta) \to \Ser(\ta), \quad \zeta \mapsto \zeta(\univh).
\]
The inverse maps a series $s$ to the morphism $\zeta$ in~\eqref{e:univ-series2}.

The set $\Ser(\ta)$ is an algebra under the Cauchy product~\eqref{e:convseries}
and the space of species morphisms $\Hom_{\Ssk}(\tSigh,\ta)$ 
is an algebra under the convolution product~\eqref{e:convolution}. 
These operations do not correspond to each other.
Indeed, the subset $\Hom_{\Mo{\Ssk}}(\tSigh,\ta)$ is not closed under 
either addition or convolution. 

Similarly, from Theorem~\ref{t:univ-series2} we deduce a bijection between
the set of group-like series of $\thh$ and 
the set of bimonoid morphisms from $\tSigh$ to $\thh$ given by
\begin{equation}\label{mor-glike}
\Hom_{\Bo{\Ssk}}(\tSigh,\thh) \to \Gls(\thh), \quad \zeta \mapsto \zeta(\univh).
\end{equation}
If $\thh$ is commutative, then $\Hom_{\Bo{\Ssk}}(\tSigh,\thh)$ is a group under
convolution, and the above bijection is an isomorphism of monoids.
This follows from~\eqref{e:left-dist}.

\smallskip

The results of this section hold for group-like series of $\tSig$
in relation to connected bimonoids $\thh$. The role of the universal series
is played by $\univ := \upsilon(\univh)$.
As in~\eqref{e:expeuler}, we have
\begin{equation}\label{e:expeuler3}
\univ_I := \BH_{(I)}
\end{equation}
for nonempty $I$, and $\univ_\emptyset=\BH_{\emptyset^0}$
(where
$\emptyset^0$ is the composition of $\emptyset$ with no blocks).

\begin{remark}
Theorem~\ref{t:univ-series2} is analogous to~\cite[Theorem~4.1]{ABS:2006}.
The latter result is formulated in dual terms and for graded connected Hopf algebras.
In this context, the expression for the morphism $\zeta$ in terms of the
characteristic operations appears in~\cite[Proposition~3.7]{HerHsi:2009}.
Many of these ideas appear in work of Hazewinkel~\cite[Section~11]{Haz:2009}.
(Warning: Hazewinkel uses \emph{curve} for what we call group-like series.)
One may consider the bijection~\eqref{mor-glike} in the special case $\thh=\tSigh$.
The analogue for graded connected Hopf algebras is then a bijection between
endomorphisms of the Hopf algebra of noncommutative symmetric functions,
and group-like series therein. This is used extensively in~\cite{KLT:1997} under
the name \emph{transformations of alphabets}.
\end{remark}

\section{Classical idempotents}\label{s:convpower}

This section introduces a number of idempotent elements 
of the Tits algebra of compositions. 
We start by studying the first Eulerian idempotent.
This is in fact a primitive element of $\tSig$ and it operates
on cocommutative connected bimonoids as the logarithm of the identity.

We then introduce a complete orthogonal family of idempotents for the Tits algebra. 
The family is indexed by set partitions. We call these the Garsia-Reutenauer idempotents.
The idempotent indexed by the partition with one block is the first Eulerian.
Lumping according to the length of the partitions
yields the higher Eulerian idempotents.

In the Tits algebra of compositions,
we define elements indexed by integers $p$,
whose characteristic operation on a connected bimonoid $\thh$
is the $p$-th convolution power of the identity. 
There is a simple expression for these elements 
in terms of the higher Eulerian idempotents which diagonalizes 
their operation on $\thh$.

%We pay particular attention to the case $p=-1$:
%the characteristic operation of $\BH^{(-1)}_+$ is the antipode,
%which is the inverse of the identity under convolution.

We introduce the Dynkin quasi-idempotent and describe its operation on connected bimonoids.
We establish the Dynkin-Specht-Wever theorem for connected Hopf monoids.

\subsection{The first Eulerian idempotent}\label{ss:first-euler}

We proceed to define a primitive series of $\tSig$.
We call it the \emph{first Eulerian idempotent} and denote it by $\euler{}$.
%For simplicity, we denote its components by $\euler{I}$ instead of $(\euler{1})_I$.

Set $\euler{\emptyset}=0$, and for each nonempty finite set $I$
define an element $\euler{I}\in\tSig[I]$ by
\begin{equation}\label{e:first-euler}
\euler{I} := \sum_{F\vDash I} (-1)^{\len(F)-1} \frac{1}{\len(F)} \,\BH_F.
\end{equation}

Recall the group-like series $\univ$ of $\tSig$~\eqref{e:expeuler3}. 
It is given by $\univ_I := \BH_{(I)}$. 
It follows that
\begin{equation}\label{e:euler-expeuler}
\euler{}=\log(\univ),
\end{equation}
with the logarithm of a series as in~\eqref{e:calc-log}.

\begin{proposition}\label{p:primE}
The series $\euler{}$ of $\tSig$ is primitive and
the element $\euler{I}$ is an idempotent of the Tits algebra $(\tSig[I],\act)$.
\end{proposition}
\begin{proof}
As the logarithm of a group-like series, the series $\euler{}$ is primitive 
(Theorem~\ref{t:exp-log}).
Thus, each $\euler{I}\in\Pc(\tSig)[I]$, and by Corollary~\ref{c:action-prim}
the element $\euler{I}$ is an idempotent of $(\tSig[I],\act)$ (since the coefficient of $\BH_{(I)}$
in $\euler{I}$ is $1$).
\end{proof}

Comparing~\eqref{e:first-euler} 
with the definition of the basis $\{\BQ\}$ of $\tSig$~\eqref{e:Q-H},
we see that, for $I$ nonempty,
\begin{equation}\label{e:first-euler-Q}
\euler{I} = \BQ_{(I)}.
\end{equation}
%This equality is a problem if $I$ is empty.

The primitivity of the $\euler{I}$ may be deduced in other ways.
We may derive it from~\eqref{e:first-euler-Q} and the coproduct formula~\eqref{e:Qcoprod}, 
using the fact that no proper subset of $I$ is $(I)$-admissible.
Alternatively, applying~\eqref{e:H-Q-proj} with $G=(I)$,
we have $\mu_F(\Delta_F(\euler{I})) = \BH_F \act \euler{I} = 0$ for $F\not=(I)$,
so $\euler{I}$ is primitive since $\mu_F$ is injective by Corollary~\ref{c:hopf-split}.

We turn to the characteristic operation~\eqref{e:can-action} of $\euler{I}$
on connected bimonoids.

\begin{corollary}\label{c:char-op-1euler}
For any connected bimonoid $\thh$, 
\begin{equation}\label{e:car-op-1euler}
\facematrix(\univ) = \id
\qand 
\facematrix(\euler{}) = \log(\id).
\end{equation}
Moreover, if $\thh$ is cocommutative, then $\log(\id):\thh\to\thh$ is a projection onto 
the primitive part $\Pc(\thh)$. 
\end{corollary}
\begin{proof}
The assertion about $\univ$ follows from~\eqref{e:univ-id}.
Formula~\eqref{e:car-op-1euler} then follows from naturality of functional calculus
(statement~\eqref{i:nat} in Proposition~\ref{p:fun-calc}), 
since $\facematrix$ is a morphism of monoids (Proposition~\ref{p:Psi-mor-conn}). 
The last statement follows from Theorem~\ref{t:act-to-prim}.
\end{proof}
 
\begin{remark}
Recall from~\eqref{e:xyz-flat-sum} that for an element of $\tSig[I]$ to be primitive, 
the sum of the coefficients of compositions with support $X$ must be $\mu(\minflat,X)$.
So a natural way to construct a primitive is to let 
these coefficients be equal.
Since there are $\len(X)!$ such compositions, 
each coefficient must be
\[
\frac{\mu(\minflat,X)}{\len(X)!} = (-1)^{\len(X)-1} \frac{1}{\len(X)}
\]
by~\eqref{e:partmobiusPi}. 
This is precisely the coefficient used to define $\euler{I}$ in~\eqref{e:first-euler}. 
This observation may be used to extend the definition of the
first Eulerian idempotent to more general settings (beyond the scope of this paper, 
but hinted at in Section~\ref{ss:braid}).
\end{remark}

\begin{remark}
The first Eulerian idempotent belongs to Solomon's descent algebra
(and hence to the symmetric group algebra). 
As such, it appears in work of Hain~\cite{Hai:1986}.
Its action on graded Hopf algebras is considered by Schmitt~\cite[Section~9]{Sch:1994}.
Patras and Schocker~\cite[Definition~31]{PatSch:2006} consider this element in the same context as here.
It is the first among the higher Eulerian idempotents. 
These elements are discussed in Section~\ref{ss:higher-euler} below, 
and additional references to the literature 
(which also pertain to the first Eulerian idempotent) are given there.
\end{remark}

\subsection{The Garsia-Reutenauer idempotents}

Fix a nonempty set $I$.
The element $\euler{I}$ is a part of a family of elements indexed by partitions of $I$:
For each $X\vdash I$, let $\euler{X}\in\tSig[I]$ be the element defined by
\begin{equation}\label{e:gar-reu}
\euler{X} := \frac{1}{\len(X)!} \sum_{F:\,\supp F=X} \BQ_F.
\end{equation}
We call these elements the \emph{Garsia-Reutenauer idempotents}.
It follows from~\eqref{e:Q-H} that
\begin{equation}
\euler{X} = \frac{1}{\len(X)!} \sum_{F:\,\supp F=X} \ \sum_{G:\,F\leq G} (-1)^{\len(G)-\len(F)} \frac{1}{\len(G/F)}\, \BH_G.
\end{equation}
Setting $X=\minflat$ recovers the first Eulerian: 
$\euler{\minflat}=\euler{I}$.

\begin{theorem}\label{t:idempotent-family}
For any nonempty set $I$, we have
\begin{gather}\label{e:idem-sum}
\sum_{X\vdash I} \euler{X} = \BH_{\minflat},\\
\label{e:idem-ortho}
\euler{X}\act\euler{Y}=
\begin{cases}
\euler{X} & \text{ if } X=Y,\\
0 & \text{ if } X\not=Y.
\end{cases}
\end{gather}
\end{theorem}

%Some of this can be seen from~\eqref{e:H-P-proj}.
%But the full proof is nontrivial.
Theorem~\ref{t:idempotent-family} states that the family $\{\euler{X}\}$,
as $X$ runs over the partitions of $I$,
is a complete system of orthogonal idempotents in
the Tits algebra $(\tSig[I],\act)$. 
(They are moreover \emph{primitive} in the sense that 
they cannot be written as a sum of nontrivial orthogonal idempotents.)

\smallskip

Proposition~\ref{p:H-Q-proj} and~\eqref{e:gar-reu} imply that
for any composition $F$ with support $X$,
\begin{equation}\label{e:proj-euler}
\BH_F \act \euler{X} = \BQ_F.
\end{equation}
Note from~\eqref{e:supp-tits} and~\eqref{e:Q-H} that 
$\BQ_G \act \BH_F = \BQ_G$ whenever $F$ and $G$ have the same support;
so we also have $\euler{X} \act \BH_F = \euler{X}$.
One can now deduce that $\BQ_F$ is an idempotent and
\begin{equation}\label{e:E-Q-iso}
\euler{X}\act\BQ_F=\euler{X} \qand \BQ_F\act\euler{X}=\BQ_F.
\end{equation}

\begin{remark}
The Garsia-Reutenauer idempotents $\euler{X}$ (indexed by set partitions)
possess analogues indexed by integer partitions 
which play the same role for graded Hopf algebras as the $\euler{X}$ do for Hopf monoids. 
They are elements of Solomon's descent algebra and 
appear in~\cite[Theorem~3.2]{GarReu:1989} and~\cite[Section~9.2]{Reu:1993}.
 
Brown~\cite[Equations (24) and (27)]{Bro:2000} and Saliola~\cite[Section~1.5.1]{Sal:2006} or~\cite[Section~5.1]{Sal:2009} construct many families of orthogonal idempotents
for the Tits algebra of a central hyperplane arrangement.
The Garsia-Reutenauer idempotents constitute one particular family 
that occurs for the braid arrangement. 
Brown worked in the generality of left regular bands; 
for further generalizations, see the work of Steinberg~\cite{Ste:2006}.
\end{remark}

\subsection{The higher Eulerian idempotents}\label{ss:higher-euler}

Let $k\geq 1$ be any integer. 
Set $\euler{k,\emptyset}=0$, and for each nonempty set $I$ define 
\begin{equation}\label{e:higher-euler}
\euler{k,I} := \sum_{X:\,\len(X)=k} \euler{X}.
\end{equation}
The sum is over all partitions of $I$ of length $k$.
This defines the series $\euler{k}$ of $\tSig$.
We call it the $k$-th Eulerian idempotent.
The series $\euler{1}$ is the first Eulerian idempotent
as defined in Section~\ref{ss:first-euler}. 
By convention, we set $\euler{0}$ to be the unit series~\eqref{e:unitseries}.

The Cauchy product~\eqref{e:convseries} of the series $\euler{1}$ with itself 
can be computed using~\eqref{e:Qprod} and~\eqref{e:first-euler-Q}:
\[
(\euler{1} \conv \euler{1})_I = 
\sum_{\substack{I=S\sqcup T\\ S,T\neq\emptyset}} \mu_{S,T}(\BQ_{\{S\}}\otimes\BQ_{\{T\}}) = 
\sum_{\substack{I=S\sqcup T\\ S,T\neq\emptyset}} \BQ_{(S,T)} = 2!\,\euler{2,I}.
\]
The last step used~\eqref{e:gar-reu} and~\eqref{e:higher-euler}.
More generally, by the same argument, we deduce 
\begin{equation}\label{e:power-1euler}
\euler{k} = \frac{1}{k!}\, \euler{1}^{\conv k}. 
\end{equation}
The right-hand side involves the $k$-fold Cauchy power of $\euler{1}$.

Recall from Corollary~\ref{c:char-op-1euler} that 
the characteristic operation of $\euler{1}$ on a connected bimonoid is $\log \id$.
Since $\facematrix$ is a morphism of monoids (Proposition~\ref{p:Psi-mor-conn}), 
it preserves Cauchy products of series and we obtain the following result.

\begin{proposition}\label{p:conv-power-log-id}
The characteristic operation of $\euler{k}$ on a connected bimonoid is
\begin{equation}\label{e:conv-power-log-id}
\facematrix(\euler{k}) = \frac{1}{k!} (\log\id)^{\conv k}. 
\end{equation}
\end{proposition}

\begin{remark}
The higher Eulerian idempotents appear in the works of 
Gerstenhaber and Schack~\cite[Theorem~1.2]{GerSch:1987},
~\cite[Section~3]{GerSch:1991},
Hanlon~\cite{Han:1990}, Loday~\cite[Proposition~2.8]{Lod:1989}, 
Patras~\cite[Section~II.2]{Pat:1991},
and Reutenauer~\cite[Section~3]{Reu:1986},~\cite[Section~3.2]{Reu:1993},
among other places. 
A related idempotent (the sum of all higher Eulerian idempotents except the first) 
goes back to Barr~\cite{Bar:1968}; see~\cite[Theorem~1.3]{GerSch:1987}.

These idempotents give rise to various \emph{Hodge-type} decompositions including
the $\lambda$-decomposition of Hochschild and cyclic homology of
commutative algebras~\cite{FeiTsy:1987,GerSch:1991,Lod:1989} and
a similar decomposition of Hadwiger's group of polytopes~\cite{Pat:1991}.
\end{remark}

\subsection{The convolution powers of the identity}

In the Tits algebra of decompositions $\tSigh[I]$,
define for any nonnegative integer $p$,
\begin{equation}\label{e:loday-lift}
\BHh_{p,I} := \sum_{F:\,\len(F)=p} \BH_F.
\end{equation}
The sum is over all decompositions $F$ of $I$ of length $p$. 
The resulting series of $\tSigh$ is denoted $\BHh_p$.

\begin{proposition}
We have
\begin{equation}\label{e:loday-lift-identity}
\BHh_{p,I}\act \BHh_{q,I} = \BHh_{pq,I}.
\end{equation} 
\end{proposition}
\begin{proof}
Let $F$ be a decomposition of $I$ of length $p$, 
and $G$ one of length $q$.
Then their Tits product $FG$ has length $pq$.
If we arrange the blocks of $FG$ in a matrix as in~\eqref{e:pqsets},
we recover $F$ by taking the union of the blocks in each row of 
this matrix, and we similarly recover $G$ from the columns.
Thus, the Tits product sets up a bijection between pairs of decompositions
of lengths $p$ and $q$, and decompositions of length $pq$.
The result follows.
\end{proof}

It is clear from the definitions~\eqref{e:can-action} and~\eqref{e:loday-lift} that
the characteristic operation of $\BHh_p$ is the $p$-th convolution power of the identity:

\begin{proposition}\label{p:conv-power-id}
For any bimonoid, we have
\begin{equation}\label{e:conv-power-id}
\facematrix(\BHh_p) = \id^{\conv p}.
\end{equation}
\end{proposition}

We deduce that when composing convolution powers of the identity of a bimonoid
that is either commutative or cocommutative, the exponents multiply:
\begin{equation}\label{e:conv-power-pq}
\id^{\conv p} \id^{\conv q} = \id^{\conv pq}.
\end{equation}
This follows by applying $\facematrix$ to~\eqref{e:loday-lift-identity}, 
using~\eqref{e:conv-power-id}, 
and the fact that in this situation 
\[
\facematrix_I:\tSig[I]\to\End_{\Kb}(\thh[I])
\] 
either preserves or reverses products (Proposition~\ref{p:Psi-mor}).

Formula~\eqref{e:conv-power-pq} generalizes property~\eqref{e:Haction} of Hopf powers 
(Section~\ref{ss:can-action-bialg}).

Under the morphism $\upsilon: \tSigh[I] \onto \tSig[I]$, 
the element $\BHh_{p,I}$ maps to
\begin{equation}\label{e:loday}
\BH_{p,I} := \sum_{F\vDash I} \,\binom{p}{\len(F)} \,\BH_F.
\end{equation}
The sum is over all compositions $F$ of $I$; 
the binomial coefficient accounts for
the number of ways to turn $F$ into a decomposition of $I$ of length $p$ by adding empty blocks.
Since $\upsilon$ preserves Tits products, 
it follows from~\eqref{e:loday-lift-identity} that
\begin{equation}\label{e:loday-identity}
\BH_{p,I}\act \BH_{q,I} = \BH_{pq,I}.
\end{equation}
%This argument is in Ken's notes and is also probably BD's original argument.

It follows from~\eqref{e:conv-power-id} that on a connected bimonoid,
\begin{equation}\label{e:conv-power-id-conn}
\facematrix(\BH_p) = \id^{\conv p}.
\end{equation}

Comparing with~\eqref{e:expeuler3}, we see that $\BH_1=\univ$, and more
generally
\begin{equation}\label{e:HG}
\BH_p = \univ^{\conv p}.
\end{equation}
The latter is the Cauchy product in the algebra of series of $\tSig$.

\begin{theorem}\label{t:loday-diag}
We have
\begin{equation}\label{e:loday-diag}
\BH_p = \sum_{k\geq 0} \,p^k \,\euler{k}.
\end{equation}
\end{theorem}
\begin{proof}
Consider the identity 
\[
x^{p} = \exp(p\log x) = \sum_{k\geq 0} \frac{p^k}{k!} (\log x)^k
\]
 in the algebra of formal power series.
Now apply functional calculus with $x=\univ$ in the algebra of series
of $\tSig$, 
and use~\eqref{e:euler-expeuler},~\eqref{e:power-1euler}, and~\eqref{e:HG}.
\end{proof}

Observe that~\eqref{e:loday} defines $\BH_p$ for any integer $p$.
In particular, for $p=-1$,
\begin{equation}\label{e:loday-1}
\BH_{-1,I} = \sum_{F\vDash I} \,\binom{-1}{\len(F)} \,\BH_F = 
\sum_{F\vDash I} \,(-1)^{\len(F)} \,\BH_F.
\end{equation}
By polynomiality, the identities~\eqref{e:loday-identity}--\eqref{e:loday-diag}
continue to hold for all scalars $p$ and $q$. In particular, from~\eqref{e:loday-diag}
we deduce that
\begin{equation}\label{e:loday-diag-1}
\BH_{-1} = \sum_{k\geq 1} (-1)^k \euler{k},
\end{equation}
and from~\eqref{e:conv-power-id-conn} that
this element operates on a connected bimonoid as the antipode:
\begin{equation}\label{e:char-apode}
\facematrix(\BH_{-1}) = \apode.
\end{equation}
Together with~\eqref{e:loday-1}, 
this yields another proof of Takeuchi's formula~\eqref{e:antipode-r}.

\begin{remark}
The elements $\BH_{p,[n]}$ belong to Solomon's descent algebra
(and hence to the symmetric group algebra). 
In connection to the higher Eulerian idempotents, they are considered in
~\cite[Section~1]{GerSch:1991},~\cite[Definition~1.6 and Theorem~1.7]{Lod:1989},
and~\cite[Section~II]{Pat:1991}.
They are also considered in~\cite[Lemma~1]{BayDia:1992} 
and~\cite[Proposition~2.3]{BHR:1999} in connection to riffle shuffles. 
The connection to convolution powers of the identity of a graded Hopf algebra
is made in~\cite[Section~1]{GerSch:1991},
~\cite[Section~4.5]{Lod:1998} and~\cite[Section~1]{Pat:1993}.
A discussion from the point of view of $\lambda$-rings is given in~\cite[Section~5]{Pat:2003}.
\end{remark}

\subsection{The Dynkin quasi-idempotent}\label{ss:dynkin}

Fix a nonempty finite set $I$. 
Given a composition $F=(I_1,\ldots,I_k)$ of $I$, let
\[
\omega(F) := I_k
\]
denote its last block.

For each $i\in I$, define an element $\pdynkin_i\in\tSig[I]$ by
\begin{equation}\label{e:pre-dynkin}
\pdynkin_i := \sum_{F:\,i\in\omega(F)} \,(-1)^{\len(F) - 1} \,\BH_F.
\end{equation}
The sum is over all compositions $F$ of $I$ whose last block contains $i$.
The sum of these elements defines another element of $\tSig[I]$
\begin{equation}\label{e:dynkin}
\dynkin_I := \sum_{i\in I} \pdynkin_i.
\end{equation}
We call it the \emph{Dynkin quasi-idempotent}.
It follows that
\begin{equation}
\dynkin_I = \sum_{F\vDash I} \,(-1)^{\len(F) - 1} \,\abs{\omega(F)}\, \BH_F, 
\end{equation}
where $\abs{\omega(F)}$ denotes the size of the last block of $F$.
%The sign in front is the negative of the parity of the number of blocks of $F$.
The elements $\dynkin_I$ define a series $\dynkin$ of $\tSig$. 
By convention, $\dynkin_\emptyset=0$.
The following result shows that this series is primitive.

\begin{proposition}\label{p:dynkin-proj}
The elements $\pdynkin_i$ (and hence $\dynkin_I$) are primitive elements of $\tSig[I]$.
\end{proposition}
\begin{proof}
Let $I=S\sqcup T$ be a decomposition with $S,T\not=\emptyset$.
Let $F\vDash S$ and $G\vDash T$. 
Then by~\eqref{e:faces-def}, 
the coefficient of $\BH_F\otimes\BH_G$ in $\Delta_{S,T}(\pdynkin_i)$ is
\[
\sum_K \,(-1)^{\len(K) - 1}.
\]
The sum is over all quasi-shuffles $K$ of $F$ and $G$ such that
$i\in\omega(K)$.
(Quasi-shuffles are defined in Section~\ref{ss:oper}.) 
We show below that this sum is zero.

We refer to the $K$ which appear in the sum as admissible quasi-shuffles.
If $i$ does not belong to either $\omega(F)$ or $\omega(G)$,
then there is no admissible $K$ and the sum is zero.
So suppose that $i$ belongs to (say) $\omega(G)$.
For any admissible quasi-shuffle $K$ in which $\omega(F)$ appears by itself,
merging this last block of $F$ with the next block 
(which exists and belongs to $G$) yields another admissible quasi-shuffle say $K'$.
Note that $K$ can be recovered from $K'$.
Admissible quasi-shuffles can be paired off in this manner and
those in a pair contribute opposite signs to the above sum.
\end{proof}

Combining Proposition~\ref{p:dynkin-proj}
with Theorem~\ref{t:act-to-prim} and Corollary~\ref{c:action-prim}
we obtain the following result.

\begin{corollary}\label{c:dynkin-proj}
For each $i\in I$, the element $\pdynkin_i$ is an idempotent of the Tits algebra $\tSig[I]$.
The element $\dynkin_I$ is a quasi-idempotent of the Tits algebra:
\[
\dynkin_I\act \dynkin_I=\abs{I}\,\dynkin_I.
\]
For any cocommutative connected bimonoid $\thh$,
the characteristic operations
$\facematrix_I(\pdynkin_i)$ and $\facematrix_I(\frac{1}{\abs{I}}\dynkin_I)$ 
are projections from $\thh[I]$ onto $\Pc(\thh)[I]$.
\end{corollary}

The operations of $\pdynkin_i$ and $\dynkin_I$ on $\thh$ 
can be described in terms of the antipode of $\thh$; 
these results do not require cocommutativity.

\begin{proposition}\label{p:pdynkin-action}
For any connected bimonoid and any $i\in I$,
\begin{equation}\label{e:pdynkin-action}
\facematrix_I(\pdynkin_i) = \sum_{\substack{I=S\sqcup T\\ i\in T}}
\mu_{S,T}(\apode_S\otimes\id_T)\Delta_{S,T}.
\end{equation}
\end{proposition}
\begin{proof}
According to~\eqref{e:pre-dynkin}, we have
\[
\facematrix_I(\pdynkin_i) = \sum_{F:\,i\in\omega(F)} (-1)^{\len(F)-1} \mu_F\Delta_F,
\]
where the sum is over those compositions $F$ of $I$ for which $i$ lies
in the last block. Given such $F$, let $T:=\omega(F)$ be its last block and 
let $S$ be
the union of the other blocks.
Note that $S$ may be empty.
Then $F$ consists of a composition $G$ of $S$ followed by $T$. By
(co)associativity~\eqref{e:gen-asso-conn},
\[
\mu_F = \mu_{S,T}(\mu_G\otimes\id_T)
\qand
\Delta_F = (\Delta_G\otimes\id_T)\Delta_{S,T}.
\]
Therefore,
\begin{align*}
\facematrix_I(\pdynkin_i) & = \sum_{\substack{I=S\sqcup T\\ i\in
T}}\sum_{G} (-1)^{\len(G)}
\mu_{S,T}(\mu_G\otimes\id_T)(\Delta_G\otimes\id_T)\Delta_{S,T}\\
& = \sum_{\substack{I=S\sqcup T\\ i\in
T}}\mu_{S,T}\Bigl(\sum_{G}(-1)^{\len(G)}(\mu_G\Delta_G)\otimes\id_T\Bigr)\Delta_{S,T}\\
& = \sum_{\substack{I=S\sqcup T\\ i\in T}}
\mu_{S,T}(\apode_S\otimes\id_T)\Delta_{S,T},
\end{align*}
by Takeuchi's formula~\eqref{e:antipode-r}.
\end{proof}

%The \emph{number operator} is the morphism of species $\degmap:\thh\to\thh$ 
%defined by
%\[
%\degmap_I:\thh[I] \to \thh[I],
%\quad
%\degmap_I(x) = \abs{I}\,x.
%\]

\begin{corollary}\label{c:dynkin-action}
For any connected bimonoid,
\begin{equation}\label{e:dynkin-action}
\facematrix(\dynkin) = \apode\conv \degmap,
\end{equation}
the convolution product~\eqref{ss:convolution} between the antipode and
the number operator~\eqref{e:degmap}.
\end{corollary}
\begin{proof}
In view of~\eqref{e:pdynkin-action},
\begin{align*}
\facematrix_I(\dynkin_I) & =\sum_{i\in I} \sum_{\substack{I=S\sqcup T\\ i\in T}}
\mu_{S,T}(\apode_S\otimes\id_T)\Delta_{S,T}\\
& = \sum_{I=S\sqcup T} \sum_{i\in T} \mu_{S,T}(\apode_S\otimes\id_T)\Delta_{S,T}\\
& = \sum_{I=S\sqcup T} \mu_{S,T}(\apode_S\otimes \degmap_T)\Delta_{S,T}\\
& = (\apode\conv \degmap)_I. \qedhere
\end{align*}
\end{proof}

It follows from~\eqref{e:char-apode} and~\eqref{e:dynkin-action} that
\begin{equation}\label{e:dynkin-univ}
\dynkin = \univ^{-1}\conv \degmap(\univ).
\end{equation}
In other words, the Dynkin quasi-idempotent arises from the
construction of Corollary~\ref{c:glike-dynkin} applied to the universal series.

\begin{remark}
The Dynkin quasi-idempotent is classical.
The analogue of formula~\eqref{e:dynkin-action} for graded Hopf algebras
appears in work of Patras and Reutenauer~\cite[Section~3]{PatReu:2002}
and (less explicitly) in work of von Waldenfels~\cite{Wal:1966}.
The analogue of~\eqref{e:dynkin-univ} is the definition of the Dynkin quasi-idempotent
by Thibon et al~\cite[Section~2.1]{KLT:1997},~\cite[Section~2.1]{Thi:2001}.
(In these references, this element is also called the \emph{power sum of the first kind}.)
Aubry studies the operator $\facematrix_I(\dynkin_I)$ in~\cite[Section~6]{Aub:2010}.
The operators $\facematrix_I(\pdynkin_i)$ are considered in unpublished notes
of Nantel Bergeron; related ideas appear in Fisher's thesis~\cite{Fis:2010}.
The idempotents $\pdynkin_i$ are specific to the setting of Hopf monoids.
\end{remark}

\subsection{The Dynkin-Specht-Wever theorem}\label{ss:bracketing}

Recall the bimonoid of linear orders $\wL$ (Section~\ref{ss:linear}).
We proceed to describe the characteristic operation 
of the Dynkin quasi-idempotent on $\wL$.
In the discussion, $I$ always denotes a nonempty set.

To begin with, note that the commutator bracket~\eqref{e:underlying-lie} of $\wL$
satisfies
\begin{equation}\label{eq:comm-bracket-L}
[\BH_{l_1},\BH_{l_2}] := \BH_{l_1\cdot l_2} - \BH_{l_2\cdot l_1} 
\end{equation}
for any $I=S\sqcup T$, and linear orders $\ell_1$ on $S$ and $\ell_2$ on $T$.

\begin{lemma}\label{l:left-brac-peakless}
Let $\ell=i_1i_2\cdots i_n$ be a linear order on $I$. Then
\begin{equation}\label{e:left-brac-peakless}
[\dots[\BH_{i_1},\BH_{i_2}],\dots,\BH_{i_n}] 
= \sum_{\substack{I=S\sqcup T\\ i_1\in T}} (-1)^{\abs{S}}\,\BH_{\opp{\ell|_S}\cdot \ell|_{T}}.
\end{equation}
\end{lemma}
\begin{proof}
Expanding $\bigl[[\BH_{i_1},\BH_{i_2}],\ldots,\BH_{i_n}\bigr]$ using~\eqref{eq:comm-bracket-L} 
we obtain a sum of $2^{n-1}$ elements of the form $\pm\BH_{\ell'}$.
The linear orders $\ell'$ are characterized by the following property:
In $\ell'$, any $i_j$ either precedes all of $i_1,\ldots,i_{j-1}$ or succeeds all of them.
Write $I=S\sqcup T$, 
where $S$ consists of those $i_j$ for which $i_j$ preceeds $i_1,\ldots,i_{j-1}$.
By convention, $i_1\in T$, so the sign in front of $\ell'$ is the parity of $\abs{S}$.
Also note that $\ell' = \opp{\ell|_S}\cdot \ell|_{T}$,
that is, in $\ell'$, the elements of $S$ appear first
but reversed from the order in which they appear in $\ell$,
followed by the elements of $T$ appearing in the same order as in $\ell$.
%The only condition on $(S,T)$ is that $i_1\in T$.
\end{proof}

Recall the element $\pdynkin_i$ from~\eqref{e:pre-dynkin}.

\begin{proposition}\label{p:kdynkin-action}
Let $i_1i_2\cdots i_n$ be a linear order on $I$. Then, for any $j\in I$,
\begin{equation}\label{e:pdynkin-linear}
\pdynkin_j \act \BH_{i_1i_2\cdots i_n} =
\begin{cases}
[\dots[\BH_{i_1},\BH_{i_2}],\dots,\BH_{i_n}] & \text{ if } i_1=j,\\
0 & \text{ otherwise}.
\end{cases}
\end{equation}
\end{proposition}
\begin{proof}
Put $\ell:=i_1i_2\cdots i_n$.
By~\eqref{e:pdynkin-action},
\[
\pdynkin_j\act \BH_{\ell} = \sum_{\substack{I=S\sqcup T\\ j\in T}}
\mu_{S,T}(\apode_S\otimes\id_T)\Delta_{S,T}(\BH_{\ell})
= \sum_{\substack{I=S\sqcup T\\ j\in T}} (-1)^{\abs{S}}\,\BH_{\opp{\ell|_S}\cdot \ell|_{T}}.
\]
The second step used the explicit formulas
for the product, coproduct and antipode of $\wL$.
The claim in the case $i_1=j$ follows from Lemma~\ref{l:left-brac-peakless}.
Now suppose $i_1\neq j$.
Then $i_1$ may appear either in $S$ or in $T$.
In the former case, $i_1$ is the last element of $\opp{\ell|_S}$,
while in the latter, $i_1$ is the first element of $\ell|_{T}$.
Thus, for two decompositions $I=S\sqcup T$ and $I=S'\sqcup T'$ 
which differ only in the location of $i_1$, we have
$\opp{\ell|_S}\cdot \ell|_{T} = \opp{\ell|_{S'}}\cdot \ell|_{T'}$
and $(-1)^{\abs{S}}+(-1)^{\abs{S'}}=0$.
Thus $\pdynkin_j\act \BH_{\ell} = 0$.
\end{proof}

Combining~\eqref{e:dynkin} and~\eqref{e:pdynkin-linear}, 
we obtain the following result.

\begin{theorem}\label{t:dynkin-action}
For any linear order $i_1i_2\cdots i_n$ on $I$,
\begin{equation}\label{e:dynkin-linear}
\dynkin_I \act \BH_{i_1i_2\cdots i_n} = [\dots[\BH_{i_1},\BH_{i_2}],\dots,\BH_{i_n}].
\end{equation}
\end{theorem}

This result shows that $\dynkin_I$ operates on $\wL[I]$ as the \emph{left bracketing}
\[
\BH_{i_1i_2\cdots i_n} \mapsto [\dots[\BH_{i_1},\BH_{i_2}],\dots,\BH_{i_n}].
\]
This operator on $\wL[I]$
(rather than the element $\dynkin_I$)
is sometimes taken as the definition of the Dynkin quasi-idempotent.

As a special case of Corollary~\ref{c:dynkin-proj}, we have the following result.
Recall that $\tLie[I]=\Pc(\wL)[I]$.

\begin{corollary}\label{c:DSW}
The image of the left bracketing operator on $\wL[I]$ is $\tLie[I]$.
In addition, $\tLie[I]$ is the eigenspace of eigenvalue $\abs{I}$ of this operator.
\end{corollary}

This is the analogue for Hopf monoids of the classical \emph{Dynkin-Specht-Wever} theorem
~\cite[Theorem~2.3]{Ree:1958}.

\smallskip

We turn to the left bracketing operator on an arbitrary connected bimonoid $\thh$.

Given a composition $F=(I_1,\ldots,I_k)$ of a finite set $I$
and elements $x_i\in\thh[I_i]$. Recall that $\mu_F:\thh(F)\to\thh[I]$
denotes the higher product of $\thh$~\eqref{e:iterated-mu}.
Let us write 
\[
(x)_F:= \mu_F(x_1\otimes \cdots \otimes x_k).
\]
Thus, $(x)_F\in\thh[I]$ is the product of the $x_i$ in the order specified by $F$.
The \emph{left bracketing} of the $x_i$ is the element $[x]_F\in\thh[I]$ defined by
\[
[x]_F := [\dots[x_{1},x_{2}],\dots,x_{k}].
\]
It is obtained by iterating the commutator bracket~\eqref{e:underlying-lie} of $\thh$.

Let $\alpha(F):=I_1$ denote the first block of $F$. 

We then have the following results, 
generalizing~\eqref{e:left-brac-peakless}--\eqref{e:dynkin-linear}. 
(Similar proofs apply.)

\begin{equation}\label{e:left-brac-gen}
[x]_F = \sum (-1)^{\abs{S}}\, (x)_{\opp{F|_S}\cdot F|_{T}},
\end{equation}
where the sum is over all decompositions $I=S\sqcup T$ 
for which $T$ is $F$-admissible and $\alpha(F)\subseteq T$. (Note that $S$ may be empty.)
If the $x_i$ are primitive, then, for any $j\in I$,
\begin{equation}\label{e:pdynkin-gen}
\pdynkin_j \act (x)_F =
\begin{cases}
[x]_F & \text{ if } j\in\alpha(F),\\
0 & \text{ otherwise}.
\end{cases}
\end{equation}
Under the same hypothesis,
\begin{equation}\label{e:dynkin-gen}
\dynkin_I \act (x)_F = \abs{\alpha(F)}\,[x]_F.
\end{equation}

Recall the cocommutative Hopf monoid $\Tc(\tq)$ from Section~\ref{ss:freeHopf}.

\begin{corollary}\label{c:prim-in-free-1}
The primitive part $\Pc(\Tc(\tq))$ is the smallest Lie submonoid of $\Tc(\tq)$ containing $\tq$. 
\end{corollary}
\begin{proof}
Recall from Proposition~\ref{p:prim-lie} that $\Pc(\Tc(\tq))$ is a Lie submonoid of $\Tc(\tq)$.
Moreover, $\tq\subseteq \Pc(\Tc(\tq))$, 
so it only remains to show that every primitive element 
belongs to the Lie submonoid generated by $\tq$.
Accordingly, pick a primitive element $y$.
The elements $(x)_F \in\tq(F)$ span $\Tc(\tq)$, as the $x_i\in\tq[I_i]$ vary.
So express $y$ as a linear combination of these elements.
Now apply $\facematrix(\dynkin_I)$ to this equation and use Corollary~\ref{c:dynkin-proj}
and~\eqref{e:dynkin-gen}. 
This expresses $\abs{I}\,y$ (and hence $y$) 
as a linear combination of the $[x]_F$.
\end{proof}

\begin{remark}
There is a ``right version'' of the Dynkin quasi-idempotent.
It is defined as follows.
Let $\opp{\pdynkin}_i$ be as in~\eqref{e:pre-dynkin},
where the sum is now over all set compositions $F$
whose first block $\alpha(F)$ contains $i$.
Let $\opp{\dynkin}_I$ be the sum of these over all $i$:
\[
\opp{\dynkin}_I=\sum_{F\vDash I} \,(-1)^{\len(F) - 1} \,\abs{\alpha(F)}\, \BH_F.
\]
This defines another primitive series $\opp{\dynkin}$ of $\tSig$.
It satisfies
\begin{equation}\label{e:dynkin-univ-opp}
\opp{\dynkin} = \degmap(\univ) \conv \univ^{-1}.
\end{equation} 
Its characteristic operation on a product of primitive elements 
is given by ``right bracketing''. 
\end{remark}

\section{The Poincar\'e-Birkhoff-Witt
and Cartier-Milnor-Moore theorems}\label{s:structure}
 
We establish analogues of the Poincar\'e-Birkhoff-Witt (PBW) and Cartier-Milnor-Moore (CMM) theorems for connected Hopf monoids.
These results appeared in the work of Joyal~\cite{Joy:1986} and Stover~\cite{Sto:1993},
with a precursor in the work of Barratt~\cite{Bar:1978}.
%Barratt shows that the free Lie monoid embeds as a submonoid of the free monoid.
%His proof is a little sketchy and uses the Dynkin-Specht-Wever operators.
We provide here an approach based on the Garsia-Reutenauer idempotents.
For connected Hopf algebras, proofs of these classical results can be
found in~\cite[Section~7]{MilMoo:1965} and \cite[Appendix~B]{Qui:1969}. 

Let $\Kb$ be a field of characteristic $0$.

\subsection{The canonical decomposition of a cocommutative connected bimonoid}\label{ss:can-decomp}

Let $\thh$ be a cocommutative connected bimonoid. 
Let $I$ be any nonempty set.
Consider the characteristic operation 
of the first Eulerian idempotent $\euler{I}$ on $\thh$.
%As discussed in Section~\ref{ss:first-euler}, $\euler{I}$ is a primitive element,
By Corollary~\ref{c:char-op-1euler},
\begin{equation}\label{e:prelim}
\facematrix_I(\euler{I})(\thh[I]) = \Pc(\thh)[I]. 
\end{equation}
%so $\facematrix_I(\euler{\minflat})$ is a projection onto $\Pc(\thh)[I]$.
More generally, the characteristic operation 
of the Garsia-Reutenauer idempotent $\euler{X}$ on $\thh$
is related to the primitive part of $\thh$ as follows.

\begin{proposition}\label{p:higher-euler-prim}
For any composition $F$ of $I$ with support $X$, 
the map $\Delta_F$ restricts to an isomorphism
\[
\facematrix_I(\euler{X})(\thh[I]) \map{\iso} \Pc(\thh)(F).
\] 
\end{proposition}
\begin{proof}
Write $F=(I_1,\dots,I_k)$.
Since $\BQ_F=\BQ_F\act \BH_F$,
\begin{align*}
\BQ_F\act h & =\BQ_F\act \BH_F \act h\\
& =\mu_F(\BQ_{(I_1)}\otimes \dots \otimes \BQ_{(I_k)}) \act (\mu_F(\Delta_F(h)))\\
& =\mu_F((\BQ_{(I_1)}\otimes \dots \otimes \BQ_{(I_k)})\act \Delta_F(h))
\end{align*}
for any $h\in\thh[I]$.
This used~\eqref{e:Qprod}, \eqref{e:can-actionF} and~\eqref{e:mu-star}.
Now by Corollary~\ref{c:hopf-split}, $\Delta_F\mu_F$ is the identity,
so $\Delta_F$ is surjective (onto $\thh(F)$) and $\mu_F$ is injective, and
by~\eqref{e:first-euler-Q} and~\eqref{e:prelim},
\[
\facematrix_{I_j}(\BQ_{(I_j)})(\thh[I_j]) = \Pc(\thh)[I_j].
\] 
One can deduce from here that 
the maps $\Delta_F$ and $\mu_F$ restrict to inverse isomorphisms:
\[
\Delta_F : \facematrix_I(\BQ_F)(\thh[I]) \map{\iso} \Pc(\thh)(F)
\qand
\mu_F : \Pc(\thh)(F) \map{\iso} \facematrix_I(\BQ_F).
\]
We know from~\eqref{e:E-Q-iso} that 
$\facematrix_I(\euler{X})$ and $\facematrix_I(\BQ_F)$ are canonically isomorphic under
$\euler{X}\act h \leftrightarrow \BQ_F\act h$.
Further,
\[
\Delta_F(\euler{X}\act h) = \Delta_F\mu_F\Delta_F(\euler{X}\act h) = \Delta_F(\BH_F\act\euler{X}\act h) = \Delta_F(\BQ_F\act h).
\]
The last step used~\eqref{e:proj-euler}.
The claim follows.
\end{proof}

\begin{theorem}\label{t:euler-decomp}
For any cocommutative connected bimonoid $\thh$,
\begin{equation}\label{e:euler-decomp}
\thh[I] = \bigoplus_X \facematrix_I(\euler{X})(\thh[I]),
\end{equation}
and, if $\thh$ is finite-dimensional, then
\begin{equation}\label{e:euler-decomp-dim}
\dim \facematrix_I(\euler{X})(\thh[I]) = \eigmult{X}(\thh),
\end{equation}
with $\eigmult{X}(\thh)$ as in~\eqref{e:mult-r}.
\end{theorem}
\begin{proof}
The decomposition~\eqref{e:euler-decomp} follows from Theorem~\ref{t:idempotent-family}.
For the second part, 
let $F=(I_1,\dots,I_k)$ be any composition with support $X$.
By Proposition~\ref{p:higher-euler-prim}, 
the dimension of the $X$-summand is the dimension of $\Pc(\thh)(F)$.
The latter is the product of the dimensions of $\Pc(\thh)[I_j]$.
Applying Proposition~\ref{p:dim-h-prim} and~\eqref{e:mult-r-lump} 
yields~\eqref{e:euler-decomp-dim}. 
\end{proof}

\begin{theorem}\label{t:pbw+cmm}
For any cocommutative connected bimonoid $\thh$, 
there is an isomorphism of comonoids
\begin{equation}\label{e:pbw+cmm-r}
\Sc(\Pc(\thh)) \isoto \thh.
\end{equation}
\end{theorem}
Explicitly, the isomorphism is given by the composite
\begin{equation}\label{e:pbw+cmm+explicit}
\Pc(\thh)(X) \into \thh(X) \map{\frac{1}{\len(X)!}\sum_F \beta_F}
\bigoplus_{F:\,\supp F=X} \thh(F) \map{\sum_F \mu_F} \thh[I],
\end{equation}
where, for $F=(I_1,\ldots,I_k)$ with $\supp F = X$,
the map $\beta_F$ is the canonical isomorphism
\begin{equation}\label{e:betaF}
\thh(X) = \bigotimes_{B \in X} \thh[B] \isoto
\thh[I_1] \otimes \dots \otimes \thh[I_k] = \thh(F).
\end{equation}
%(This lifts Quillen's map.)
\begin{proof}
Let $V$ be the (isomorphic) image of $\Pc(\thh)(X)$ 
inside the direct sum in~\eqref{e:pbw+cmm+explicit}.
%So $V$ is isomorphic to $\Pc(\thh)(X)$.
Consider the composite map
\[
\facematrix_I(\euler{X})(\thh[I]) \map{(\frac{1}{\len(F)!}\Delta_F)_{F}} V 
\map{\sum_{F} \mu_F} \facematrix_I(\euler{X})(\thh[I]).
\]
In both maps $F$ ranges over all compositions with support $X$.
The first map is an isomorphism by Proposition~\ref{p:higher-euler-prim}.
Using~\eqref{e:gar-reu} and~\eqref{e:proj-euler},
one can deduce that the composite is the identity:
\begin{align*}
\sum_{F:\,\supp F=X} \frac{1}{\len(F)!}\, \mu_F\Delta_F (\euler{X} \act h) & = 
\sum_{F:\,\supp F=X} \frac{1}{\len(F)!}\, \BH_F \act \euler{X} \act h\\
& = \sum_{F:\,\supp F=X} \frac{1}{\len(F)!}\, \BQ_F \act h = \euler{X} \act h. 
\end{align*}
So the second map is an isomorphism as well.
It follows that the image of~\eqref{e:pbw+cmm+explicit} is precisely $\facematrix_I(\euler{X})(\thh[I])$.
%In particular, $\Pc(\thh)$ is the image of the first Eulerian idempotent.
Combining with the decomposition~\eqref{e:euler-decomp}
establishes the isomorphism~\eqref{e:pbw+cmm-r}.

To check that this is a morphism of comonoids, 
we employ the coproduct formula of $\Sc(\Pc(\thh))$ given in Section~\ref{ss:freecomHopf}.
The calculation is split into two cases, 
depending on whether $S$ is $X$-admissible or not.
Both cases can be handled using the higher compatibility axiom~\eqref{e:gen-comp} 
for $F$ arbitrary and $G=(S,T)$.
We omit the details.
\end{proof}

\begin{remark}
The analogue of Theorem~\ref{t:pbw+cmm} for graded Hopf algebras 
can be proved using the analogues of the Garsia-Reutenauer idempotents 
indexed by integer partitions.
A proof of a similar nature is given by 
Patras~\cite[Proposition~III.5]{Pat:1993} and Cartier~\cite[Section~3.8]{Car:2007}. 
\end{remark}

\subsection{PBW and CMM}\label{ss:PBW-CMM}

Let $\tg$ be any positive species.
For any partition $X$, consider the map
\[
\tg(X) \map{\frac{1}{\len(X)!}\sum_F \beta_F} \bigoplus_{F:\,\supp F=X} \tg(F),
\]
where $\beta_F$ is as in~\eqref{e:betaF}.
Summing over all $X$ yields a map
\begin{equation}\label{e:norm-norm}
\Sc(\tg) \to \Tc(\tg).
\end{equation}
It is straightforward to check that this is a morphism of comonoids.

Now let $\tg$ be a positive Lie monoid.
Composing the above map with the quotient map $\Tc(\tg)\to\Uc(\tg)$ yields a morphism
\begin{equation}\label{e:pbw}
\Sc(\tg) \to \Uc(\tg)
\end{equation}
of comonoids.

\begin{theorem}[PBW]\label{t:pbw}
For any positive Lie monoid $\tg$,
the map~\eqref{e:pbw} is an isomorphism of comonoids.
In particular, the canonical map $\eta_\tg : \tg \to \Uc(\tg)$ is injective.
\end{theorem}
\begin{proof}
The first step is to construct a surjective map $\psi_\tg$
fitting into a commutative diagram
\[
\xymatrix@=1pc{
& \Tc(\tg) \ar@{.>>}[ld]_{\psi_\tg}\xyonto[rd]^{\pi_\tg}\\
\Sc(\tg) \ar[rr] & & \Uc(\tg).
}
\]
The map on each summand $\tg(F)$ is defined by an induction on the rank of $F$.
We omit the details.
The existence of $\psi_\tg$ shows that~\eqref{e:pbw} is surjective.

To show that~\eqref{e:pbw} is injective,
we need to show that $\Jc(\tg)$, the kernel of $\pi_\tg$,
is contained in the kernel of $\psi_\tg$.
Suppose $\tg' \onto \tg$ is a surjective map of Lie monoids,
and suppose that this result holds for $\tg'$.
Then the result holds for $\tg$ as well.
To see this, consider the commutative diagram
\[
\xymatrix{
\Jc(\tg') \ar[r]\xyonto[d] & \Tc(\tg') \ar[r]^{\psi_{\tg'}}\xyonto[d] & \Sc(\tg') \xyonto[d]\\
\Jc(\tg) \ar[r] & \Tc(\tg) \ar[r]_{\psi_\tg} & \Sc(\tg).
}
\]
The key observation is that the map $\Jc(\tg') \to \Jc(\tg)$ is surjective.
(This can be deduced using the generating relations~\eqref{e:ug-rel}.)
Then the top composite being zero implies that
the bottom composite is also zero as required.

Next note that for any positive Lie monoid $\tg$,
the canonical map $\freelie(\tg) \to \tg$ 
(defined using the Lie structure of $\tg$) is a surjective morphism of Lie monoids.
So it suffices to prove injectivity for free Lie monoids.
Accordingly, let $\tg=\freelie(\tq)$.
The map~\eqref{e:pbw} is the composite
\[
\Sc(\tg) \to \Sc(\Pc(\Uc(\tg))) \map{\iso} \Uc(\tg).
\]
The second map is an isomorphism by Theorem~\ref{t:pbw+cmm}.
The first map is induced by $\tg \to \Pc(\Uc(\tg))$.
In the free Lie monoid case, $\Uc(\tg)=\Tc(\tq)$ by~\eqref{e:tc-factorize}.
So by Lemma~\ref{l:fill-gap}, this map (and hence the induced map) is injective.
\end{proof}

\begin{theorem}[CMM]\label{t:cmm}
The functors 
\[
\xymatrix@C+20pt{
\{\textup{positive Lie monoids}\} \ar@<0.6ex>[r]^-{\Uc} & \ar@<0.6ex>[l]^-{\Pc} 
\{\textup{cocommutative connected bimonoids}\}
}
\]
form an adjoint equivalence.
\end{theorem}
In other words, the functors form an adjunction for which 
the unit and counit are isomorphisms
\begin{equation}\label{e:adj-equiv}
\Uc(\Pc(\thh)) \isoto \thh
\qand
\tg \isoto \Pc(\Uc(\tg)).
\end{equation}
\begin{proof}
The adjunction follows from~\eqref{e:UP}, 
since (connected) cocommutative bimonoids form
a full subcategory of the category of (connected) bimonoids.

The isomorphism~\eqref{e:pbw+cmm-r} can be expressed as the composite
\[
\Sc(\Pc(\thh)) \to \Uc(\Pc(\thh)) \to \thh.
\]
The first map is obtained by setting $\tg:=\Pc(\thh)$ in~\eqref{e:pbw}
and hence is an isomorphism by Theorem~\ref{t:pbw}.
Hence the second map which is as in~\eqref{e:adj-equiv}
is also an isomorphism.
%The second map uses the adjunction between $\Uc$ and $\Pc$.
%Observe that this composite map coincides with~\eqref{e:pbw+cmm-r}
%and is in particular, an isomorphism.

By the adjunction property,
\[
\Uc(\tg) \to \Uc(\Pc(\Uc(\tg))) \to \Uc(\tg)
\]
is the identity.
From the above, we know that the second map is an isomorphism.
Hence so is the first.
Now by Theorem~\ref{t:pbw}, a Lie monoid embeds in its universal enveloping monoid.
So this map restricts to an isomorphism $\tg \isoto \Pc(\Uc(\tg))$ as required.
\end{proof}

Applying $\Pc$ to~\eqref{e:tc-factorize} and using~\eqref{e:adj-equiv}, we deduce:

\begin{corollary}\label{c:prim-in-free}
There is an isomorphism of Lie monoids
\begin{equation}\label{e:prim-in-free}
\Pc(\Tc(\tq))\iso\freelie(\tq).
\end{equation} 
\end{corollary}

More precisely, the image of the map~\eqref{e:commutator-map}
identifies $\freelie(\tq)$ with the primitive part of $\Tc(\tq)$.
In particular, $\Pc(\Tc(\tq))$ is the Lie submonoid of $\Tc(\tq)$ generated by $\tq$.
This latter result was obtained by different means in Corollary~\ref{c:prim-in-free-1}.

\begin{remark}
Theorem~\ref{t:cmm} is due to Stover~\cite[Proposition~7.10 and Theorem~8.4]{Sto:1993}.
This result does not require characteristic $0$.
The map~\eqref{e:pbw} in the PBW isomorphism
is the analogue of Quillen's map in the classical theory
~\cite[Appendix~B, Theorem~2.3]{Qui:1969}.
Theorem~\ref{t:pbw} is due to Joyal~\cite[Section~4.2, Theorem~2]{Joy:1986}.
He deduces it from the classical PBW; he does not mention comonoids though.
Stover~\cite[Theorem~11.3]{Sto:1993} uses a different map to
show that $\Sc(\tg)$ and $\Uc(\tg)$ are isomorphic;
his map is not a morphism of comonoids.
\end{remark}

%Joyal's argument is as follows.
%The functor $\Kcb_V$ from species to vector spaces sending $\tp$ to $\tp\circ V$
%is braided bistrong. So it sends (Lie) monoids to (Lie) algebras.
%There is a canonical isomorphism
%\[
%\Uc(\tg) \circ V = \Uc(\tg\circ V).
%\]
%($\tg\circ V$ is a Lie algebra, and the rhs is its usual universal enveloping algebra.)
%There is a similar isomorphism with $\Sc$ instead of $\Uc$.
%From classical theory,
%Quillen's map $\Sc(\tg\circ V) \to \Uc(\tg\circ V)$ is an isomorphism of coalgebras.
%This arises from a map $\Sc(\tg)\to\Uc(\tg)$ which is then an isomorphism of comonoids.
%[This coincides with our map.
%Joyal does not say anything explicit about using Quillen's map or about comonoids though.]

\section{The dimension sequence of a connected Hopf monoid}
\label{s:genfun}

In this section, all species are assumed to be finite-dimensional.

We consider three formal power series associated to a species $\tp$:
\begin{align*}
%\label{e:expo-gen}
\expo{\tp}{\varx} &:= \sum_{n\geq 0} \dim_{\Kb} \tp[n]\, \frac{\varx^n}{n!},\\
%\label{e:type-gen}
\type{\tp}{\varx} &: = \sum_{n\geq 0} \dim_{\Kb} \tp[n]_{S_n}\, \varx^n,\\
%\label{e:ordi-gen}
\ordi{\tp}{\varx} &: = \sum_{n\geq 0} \dim_{\Kb} \tp[n]\, \varx^n.
\end{align*}
They are the \emph{exponential}, \emph{type}, and \emph{ordinary generating functions}, respectively. 
In the second function, the coefficient of $t^n$ 
is the dimension of the space of coinvariants of $\tp[n]$ under the action of $\Sr_n$.

For example,
\begin{align*}
\expo{\wE}{\varx}& = \exp(\varx), & 
\expo{\wL}{\varx} & = \frac{1}{1-\varx}, & 
\expo{\tPi}{\varx} & = \exp\bigl(\exp(\varx)-1),\\
\type{\wE}{\varx} & =\frac{1}{1-\varx}, & 
\type{\wL}{\varx} & =\frac{1}{1-\varx}, & 
\type{\tPi}{\varx} & =\prod_{k\geq 1} \frac{1}{1-\varx^k}.
\end{align*}
These results are given in~\cite{BerLabLer:1998}.
More generally, for any positive species $\tq$,
\begin{align*}%\label{e:type-free}
\expo{\Sc(\tq)}{\varx} & = \exp\bigl(\expo{\tq}{\varx}\bigr), &
\expo{\Tc(\tq)}{\varx} & = \frac{1}{1-\expo{\tq}{\varx}},\\
\type{\Sc(\tq)}{\varx} & = \prod_{k\geq 1} \exp\bigl(\frac{1}{k}\type{\tq}{\varx^k}\bigr), &
\type{\Tc(\tq)}{\varx} & = \frac{1}{1-\type{\tq}{\varx}}.
\end{align*}
These follow from~\cite[Theorem~1.4.2]{BerLabLer:1998}.

As for the ordinary generating function, note that for any species $\tp$
\[
\ordi{\tp}{\varx} = \expo{\tp\times \wL}{\varx} = \type{\tp\times \wL}{\varx}.
\]

The following results assert that certain power series associated to a connected
Hopf monoid $\thh$ have nonnegative coefficients. 
They impose nontrivial inequalities on the dimension sequence of $\thh$.

\begin{theorem}\label{t:ordi-boolean}
Let $\thh$ be a connected $q$-Hopf monoid. Then
\begin{equation*}%\label{e:ordi-hopf}
1-\frac{1}{\ordi{\thh}{\varx}} \in \Nb\llb \varx\rrb.
\end{equation*}
\end{theorem}

Theorem~\ref{t:ordi-boolean} is given in~\cite[Theorem~4.4]{AguMah:2012}. 
The proof goes as follows.
By Theorem~\ref{t:freeness}, $\thh\times\wL$ is free. 
Hence there is a positive species $\tq$ such that
\[
\ordi{\thh}{\varx} = \type{\thh\times \wL}{\varx} = \type{\Tc(\tq)}{\varx} = \frac{1}{1-\type{\tq}{\varx}},
\]
and this gives the result.
It can be stated as follows: 
the \emph{Boolean transform} of the dimension sequence of $\thh$ is nonnegative.

\begin{theorem}\label{t:ordi-type}
Let $\thh$ be a connected Hopf monoid. 
Then
\begin{equation*}
\ordi{\thh}{\varx} \Big/ \type{\thh}{\varx} \in \Nb\llb \varx\rrb.
\end{equation*}
\end{theorem}

This result is given in~\cite[Corollary~3.4]{AguLau:2012}. 
For the proof one applies the dual of Theorem~\ref{t:lagrange} 
to the surjective morphism of Hopf monoids
\[
\thh\times \wL \map{\id\times\pi_{\wX}} \thh\times\wE \cong \thh. 
\]
The same argument yields
\begin{equation*}
\ordi{\thh}{\varx} \Big/ \expo{\thh}{\varx} \in \Qb_{\geq 0}\llb \varx\rrb,
\end{equation*}
but as shown in~\cite[Section~4.6]{AguMah:2012},
this result follows from Theorem~\ref{t:ordi-boolean}.

\begin{theorem}\label{t:binomial-set}
Let $\rH$ be a nontrivial set-theoretic connected Hopf monoid and $\thh=\Kb\rH$. 
Then
\begin{equation*}
\exp(-\varx) \expo{\thh}{\varx} \in \Qb_{\geq 0}\llb \varx\rrb 
\qqand
(1-\varx) \type{\thh}{\varx} \in \Nb\llb \varx\rrb.
\end{equation*}
\end{theorem}

This is given in~\cite[Corollary~3.6]{AguLau:2012}. 
Nontrivial means that $\rH\neq 1$. 
In this case one must have $\rH[I]\neq\emptyset$ for all $I$. 
Then $\thh$ surjects onto $\wE$ canonically (as Hopf monoids) and one may again apply
Theorem~\ref{t:lagrange} to conclude the result.
The first condition states that the \emph{binomial transform} of the
sequence $\dim_{\Kb} \thh = \abs{\rH[n]}$ is nonnegative. 
The second that the sequence $\dim_{\Kb} \thh[n]_{\Sr_n} = \abs{\rH[n]^{\Sr_n}}$ is weakly increasing.

\begin{theorem}\label{t:log-nonneg}
Let $\thh$ be a cocommutative connected Hopf monoid. Then
\[
\log \expo{\thh}{\varx} \in \Qb_{\geq 0}\llb \varx\rrb.
\]
\end{theorem}
Theorem~\ref{t:pbw+cmm} implies
$\log \expo{\thh}{\varx} = \expo{\Pc(\thh)}{\varx}$, whence the result.

\begin{remark}
Let $\thh$ be a connected Hopf monoid. 
Suppose the integer $\dim_{\Kb} \thh[n]$ is the $n$-th \emph{moment} 
of a random variable $Z$ (possibly noncommutative) in the classical sense. 
Theorem~\ref{t:log-nonneg} may be rephrased as follows:
if $\thh$ is cocommutative, then the \emph{classical cumulants} of $Z$ are nonnegative. 
See the remark at the end of Section~\ref{r:cumulant}.
%this will refer to the right the section and jump to the right place when clicking
There is a notion of \emph{Boolean cumulants} and another of \emph{free cumulants}~\cite{Leh:2002,SpeWor:1995}. 
Lehner obtains expressions for these as sums of classical cumulants~\cite[Theorem~4.1]{Leh:2002}; 
hence their nonnegativity holds for cocommutative $\thh$.
However, Theorem~\ref{t:ordi-boolean} yields the stronger assertion
that the Boolean cumulants are nonnegative for any connected Hopf monoid $\thh$.
We do not know if nonnegativity of the free cumulants continues to hold for arbitrary $\thh$.
For information on free cumulants, see~\cite{NicSpe:2006}.
\end{remark}

%-------------------------------------------------------------------
% Begin BIBLIOGRAPHY
%-------------------------------------------------------------------

\bibliographystyle{amsplain}

\end{document}